\documentclass[12pt]{article}

\usepackage{amsfonts,amscd,amssymb,amsmath,amsthm,mathrsfs,multirow,xcolor,lscape,longtable,pbox,lipsum}
\usepackage[thinlines]{easytable}
\usepackage{microtype}
\usepackage[normalem]{ulem}
\usepackage{etex}
\usepackage{tikz-cd,xr,multicol,microtype}
\usepackage{color}
\usepackage[colorlinks,linkcolor=blue,anchorcolor=blue,citecolor=blue,backref=page]{hyperref}
\usepackage{hyperref}
\usepackage{phaistos}
\usepackage{todonotes}
\usepackage{afterpage,graphicx,threeparttable}
\usepackage[thinlines]{easytable}
\usepackage{microtype}

\newtheorem{thm}{Theorem}
\newtheorem{defi}[thm]{Definition}
\newtheorem{lemma}[thm]{Lemma}
\newtheorem{prop}[thm]{Proposition}
\newtheorem{cor}[thm]{Corollary}

\newtheorem{conjecture}[thm]{Conjecture}

\usepackage[all,cmtip]{xy}
\usepackage{tikz}
\usetikzlibrary{arrows}
\usetikzlibrary{matrix}
\usetikzlibrary{calc}

\newcommand{\Z}{\mathbb{Z}}

\newcommand{\Q}{\mathbb{Q}}

\newcommand{\R}{\mathbb{R}}

\newcommand{\eqr}[1]{\mbox{(\ref{eq:#1})}}

\newcommand{\mr}[1]{\mathrm{#1}}

\newcommand{\C}{\mathbb{C}}

\begin{document}

\title{Computations of Cohomology of Arithmetic Groups, Part 1}
\author{Ivan Horozov}
\maketitle

\abstract{
The key part of the current paper is the computation of boundary and Eisenstein cohomology of $GL_4({\mathbb Z})$ with coefficient in any highest weight representations. The method we develop let us compute in an alternative way the cohomology of $SL_3({\mathbb Z})$ and of $GL_3({\mathbb Z})$ with coefficients in any highest weight representation. This is done in a simpler, faster and in a more structured way compared to \cite{BHHM}. We state a duality for the boundary cohomology of $GL_m({\mathbb Z})$ of the type of Serre's duality, where the dualizing sheaf is a power of the determinant representation. We refine this duality to a duality on the level of the spectral sequence for the boundary cohomology $E_\infty^{p,q}$. We state it as a conjecture. However, all the computations, 30 different families of representations, satisfy this conjecture. We compute the Eisenstein cohomology of $GL_4({\mathbb Z})$ with coefficients in the symmetric powers and their twist by the determinant representation. For several other representations, we compute the Eisenstein cohomology, based a few conjectures. Based on those conjectured, one can compute the Eisenstein cohomology in most of the cases. They will be included in the next version of the paper.
}

{\bf Subjclass[2010]} 11F75; 11F70; 11F22; 11F06

\maketitle

\newpage
\tableofcontents
\section{Introduction}
Now, let us recall what is boundary cohomology. Let $\Gamma$ be $GL_3(\Z)$ or $GL_4(\Z)$; let $G$ be $GL_3$ or $GL_4$; and let $K$ be a maximal compact group in $G(\R)$. 
We can take $K$ to be $O_3(\R)$ or $O_4(\R)$, respectively. Let $S$ be the locally symmetric space defined by the double quotient
\[
S=\Gamma\backslash G(\R)/(K\times \R_{>0})
\]
Then the group cohomology of $\Gamma$ with coefficients in a representation $V$ is isomorphic to the cohomology of $S$ with coefficient in the corresponding local system $\tilde{V}$.
That is
\[H^q(\Gamma,V)=H^q(S,\tilde{V}).\] 
The Borel-Serre compactification of $S$ (see \cite{BoSe}), which we will denote by $\overline{S}$, is a compactification obtain by glueing strata $S_P$ corresponding to each parabolic subgroup $P$ of $G$.  
The glueing is done in the following way. Strata corresponding to maximal parabolic subgroup (rank 1) are glued along strata corresponding to parabolic groups of higher rank. For instance for $GL_3$, we have two maximal parabolic subgroups $Q_{12}$ and $Q_{23}$ and a minimal parabolic subgroup $Q_0$. Then the strata $S_{Q_{12}}$ and $S_{Q_{23}}$ are glued together along $S_{Q_0}$.
For $Q_{12}$, $Q_{23}$ and $Q_0$, we take the following representatives 

{\bf Maximal parabolic subgroups (rank 1):}
\[ 
Q_{12} = \left( \begin{array}{cccccc}
\ast	& \ast 	& \ast \\
\ast  	& \ast 	& \ast  \\
 0 		& 0		& \ast \\ 
\end{array}  \right) , 
Q_{23} =\left( \begin{array}{cccccc}
\ast	& \ast 	& \ast \\
0		& \ast 	& \ast  \\
0 		& \ast	& \ast \\ 
\end{array}  \right) .
\]

{\bf Minimal parabolic subgroup (rank 2):}
\[ 
Q_{0} = \left( \begin{array}{cccccc}
\ast	& \ast 	& \ast \\
0		& \ast 	& \ast  \\
0 		& 0		& \ast \\ 
\end{array}  \right).
\]

The compactification $\overline{S}$ has the same homotopy type as the locally symmetric space $S$. Therefore, we have $H^q(\overline{S},i_*\tilde{V})=H^q(S,\tilde{V}),$ where $i:S\rightarrow \overline{S}$ is the inclusion.
Let $\partial \overline{S}$ be the boundary of $\overline{S}$. The boundary cohomology of $\Gamma$ is the cohomology of $\partial S$ with coefficients in the restriction of $i_*\tilde{V}$. Formally, the boundary cohomology is defined as
\[H^q_\partial(\Gamma,V):=H^q(\partial\overline{S},j^*i_*\tilde{V}),\] where $j:\partial\overline{S} \rightarrow \overline{S}$ is the inclusion of the boundary into the compactification. 
If $\Gamma=GL_3(\Z)$, then the boundary cohomology $H^q_\partial(\Gamma,V)$ can be computed via Mayer-Vietoris type of exact sequence
\begin{eqnarray*}
&&\rightarrow H^{q-1}(Q_0,V)\rightarrow\\
&&\rightarrow H^q_\partial(GL_3(\Z),V) \rightarrow  H^{q}(Q_{12},V) + H^{q}(Q_{23},V) \rightarrow H^{q}(Q_{0},V)
\end{eqnarray*}

For $GL_4(\Z)$, we consider the following representatives for each rank of parabolic subgroups. 

{\bf Maximal parabolic subgroups (rank 1):}
\[ 
P_{13} = 
\left( \begin{array}{ccccccc}
\ast 	& \ast 	& \ast  	& \ast  \\
\ast 	& \ast 	& \ast  	& \ast \\
\ast 	& \ast 	& \ast 	& \ast  \\
0 	& 0 		& 0 		& \ast   \\ 
 \end{array}  \right)
 \,
P_{12,34} =
 \left( \begin{array}{cccccc}
\ast & \ast & \ast & \ast \\
\ast & \ast & \ast & \ast \\
0 & 0 & \ast & \ast  \\
0 & 0 & \ast & \ast \\ 
\end{array}  \right)
P_{24} = 
\left( \begin{array}{ccccccc}
\ast & \ast & \ast  & \ast\\
0 & \ast & \ast  & \ast  \\
0 & \ast & \ast & \ast  \\
0 & \ast & \ast & \ast  \\ 
\end{array}\right). 
\]

{\bf Intermediate parabolic subgroups (rank 2):}
 \[ 
P_{12} =
\left( \begin{array}{ccccccc}
\ast & \ast & \ast  & \ast  \\
\ast & \ast & \ast  & \ast  \\
0 & 0 & \ast & \ast  \\
0 & 0 & 0 & \ast  \\ 
\end{array}  \right) 
P_{23} = 
\left( \begin{array}{cccccc}
\ast & \ast & \ast & \ast \\
0 & \ast & \ast & \ast \\
0 & \ast & \ast & \ast  \\
0 & 0 & 0 & \ast \\ 
\end{array}  \right) 
P_{34} = 
\left( \begin{array}{ccccccc}
\ast & \ast & \ast  & \ast\\
0 & \ast & \ast  & \ast\\
0 & 0 & \ast & \ast\\
0 & 0 & \ast & \ast\\ 
 \end{array} \right).
 \]
 
{\bf Minimal parabolic subgroup (rank 3): } 
\[
B = 
\left( \begin{array}{ccccc}
\ast & \ast & \ast  & \ast  \\
0 & \ast & \ast  & \ast  \\
0 & 0& \ast & \ast  \\
0 & 0 & 0& \ast  \\ 
\end{array}  \right). 
\]

For $\Gamma=GL_4(\Z)$,
one can compute the boundary cohomology in terms of a spectral sequence.
Let \[E_1^{p,q}=\bigoplus H^q(P,V),\] where the direct sum is taken over parabolic subgroups $P$ of rank $p+1$.
From the glueing of the strata we naturally have the map
\[d_1:E_1^{p,q}\rightarrow E_1^{p+1,q},\]
by restricting the cohomology of a parabolic subgroup $P$ to the cohomology of the parabolic subgroups contained it $P$ whose rank is with $1$ larger than the parabolic rank of $P$.
For any spectral sequence there is a standard way of constructing the $d_2$-map
\[d_2:E_2^{p,q}\rightarrow E_2^{p+2,q-1}.\]

Computation of $E_1$ and $E_2$ terms is straightforward. However, the spectral sequence might not degenerate at the $E_2$ level. That depends only on $d_2$ if all the $d_2$'s in the spectral sequence are trivial, then the spectra sequence degenerated at the $E_2$ level.

What are ghost classes and what are potentially ghost classes? Potentially ghost classes are defined as those cohomological classes in the boundary cohomology that come from higher rank parabolic subgroups (not from the maximal parabolic subgroups). They can be defined via a filtration of the spectral sequence; however, we do not need it in such generality for $GL_3(\Z)$. For our purposes, potentially ghost classes for $GL_3(\Z)$ are the ones in the image of the connecting homomorphism 
\[pGh^q(GL_3(\Z),V):=im\left[H^{q-1}(Q_0,V)
\rightarrow H^q_\partial(GL_3(\Z),V)\right].\]

Now, we are going to define Eisenstein cohomology. It is the image of the compositions 
$H^q(\Gamma,V)=H^q(S,\tilde{V})=H^q(\overline{S},i_*\tilde{V})$ and restriction homomorphism $j^*:H^q(\overline{S},i_*\tilde{V}) \rightarrow H^q(\partial\overline{S},j^*i_*\tilde{V})=H^q_\partial(GL_3(\Z),V)$.
Then the ghost classes are defined as the intersection of potentially ghost classes and the Eisenstein cohomology
\[Gh^q(GL_3(\Z),V):=pGh^q(GL_3(\Z),V)\cap H^q_{Eis}(GL_3(\Z),V)\]

We compute cohomology of the parabolic subgroups by using Kostant's formula and the Leray-Serre spectral sequence.
Let $P$ be a parabolic subgroup of $G$. Let $N_P$ be the nilpotent radical of $P$ and let $M_P=P/N_P$ be the Levi quotient.
Then, we have the Leray-Serre spectral sequence 
\[H^j(M_P,H^i(N_P,V))=>H^{i+j}(P,V).\]
This gives a sort of induction from lower rank reductive groups to ones with higher rank, (since each Levi quotient $M_P$ is a product of groups of lower rank.)
To start all that we need to compute the cohomology of the nilpotent radical $N_P$. It is done using the Kostant formula.

\begin{thm} (Kostant)
\label{Kostant}  Let $V$ be a representation of highest
weight $\lambda$. Let $P$ be a fixed parabolic subgroup of $G$. 
Let $N_P$ be the nilpotent radical of the parabolic group $P$,
and let $\rho$ be half of the sum of the positive roots. Then
\[H^i(N_P, V)= \oplus _{\omega} L_{{\omega}(\lambda + \rho) - \rho},\]
where the sum is taken over the representatives of the quotient
$W_P \backslash W$ with minimal length such that their length is
exactly $i$. In the above notation $L_\lambda$ means
representation of $N_P$ with highest weight $\lambda$;
and $W$ and $W_P$ are the Weyl groups of $G$ and of $P$, respectively.
\end{thm}

We will introduce a combinatoric notation that will be useful for computations of  the cohomology of $GL_4(\Z)$ with coefficients in (almost) every highest weight representation, for $GL_5(\Z)$ with coefficients in some family of representations. For the computations of the cohomology of $\Gamma_1(3,p)$ and $\Gamma_1(4,p)$, the notation from this section will be useful. For them we would need to introduce more combinatoric notation.

First, we introduce our notation for permutations. It is very useful for finding explicitly, a permutation of minimal length of certain type depending on a parabolic subgroup.

For $GL_3$ the Weyl group is $S_3$  - the permutation group of three elements. Let the elements that will be permuted be $1,2$ and $3$. If $1$ and $2$ are interchanged the notation would be $213$. It means the following. The element $2$ goes to the $1$-st place, the element $1$ goes to the $2$-nd place and the element $3$ goes (stays at) the $3$-rd place.

Another example, the permutation $231$ means that and the element $1$ goes to the $3$-rd place, the element $2$ goes to the $1$-st place and the element $3$ goes to the second place.

More generally, if $w$ is a permutation of the ordered set $\{1,2,3\}$, 
then $w$ acting on $(a_1,a_2,a_3)$ is defined as 
\[w(a_1,a_2,a_3)=(a_{w^{-1}(1)},a_{w^{-1}(2)},a_{w^{-1}(3)}).\]

In particular, if $w=(2,3,1)$ then $1$ goes to the $3$-rd place, $w(1)=3$, $2$ goes to the $1$-st place, $w(2)=1$, and $3$ goes to the $2$-nd place, $w(3)=2$. Therefore, $w^{-1}(1)=2$, $w^{-1}(2)=3$ and $w^{-1}(3)=1$. Then $w(a_1,a_2,a_3)=(a_{w^{-1}(1)},a_{w^{-1}(2)},a_{w^{-1}(3)})=(a_2,a_3,a_1)$.

Besides permutations, we would need a length of a permutation. Each permutation can be written as a product of transpositions. A length of permutation $w$ is the least number of transpositions whose product gives the permutation $w$.

We will give a constructive way of determining the length of a permutation.

Let $w$ be an automorphism of $\{1,2,\dots,n \}$. We will use the following two rules to determine the length of a permutation

(a) Assume that $w(n)=n$. We denote by $w|_{n-1}$ the restriction of the function $w$ to the sent $\{1,2,\dots, n-1\}$, which is also a permutation.

(b) $l(n,1,2,\dots,n-1)=n-1$

Let $w=(w^{-1}(1),\dots,w^{-1}(n))$ and $s=(s^{-1}(1),\dots,s^{-1}(n))$ be permutations. 
Let $w^{-1}(k)=n$

Let $s=(1,2,\dots, k-1,k+1,\dots,n,k)$ be a permutation. Then $s^{-1}(n)=k$.
From rule (b) we have $l(s)=n-k$.  Also $(sw)^{-1}(n)=w^{-1}s^{-1}(n)=w^{-1}(k)=n$.
Then $l(w)=l(s)+l(sw)=n-k+l(sw|_{n-1})$,
which reduces the problem to finding the length of a permutation of fewer elements.

The following table list the permutations of three elements together with their length.

\[
\begin{tabular}{ccccc}
$w$: permutations & $l(w)$: length\\
\\
$123$ & $0$\\
$132$ & $1$\\
$213$ & $1$\\
$231$ & $2$\\
$312$ & $2$\\
$321$ & $3$\\
\end{tabular}
\]

Recall that the parabolic subgroup $Q_{12}$ contains $GL_2$ in its intersection of first two rows and first two columns.
The parabolic subgroup $Q_{23}$ contains $GL_2$ in the intersection of the second and third rows with the second and third columns. The parabolic subgroup $Q_0$ is the intersection of the above two maximal parabolic subgroups.

Let $W$ be the Weyl group for $GL_3$. Let $W(Q_{12})$ be the Weyl group of $Q_{12}$. Then $W(Q_{12})$ consists of $123$ and $213$.  In order to use Kostant formula 
we need to find permutation of minimal length that represent the cosets $W(Q_{12})\backslash W$. That is $123$, $132$ and $231$. The have lengths $0,1,2$, respectively. Note that the other representatives of the cosets are $213$, $312$ and $321$, which have lengths $1,2,3$, respectively. 

Therefore, the representatives of minimal length of $W(Q_{12})\backslash W$ together with their length are
\[
\begin{tabular}{ccccc}
$w \in  W(Q_{12})\backslash W$ & length\\
\\
$123$ & $0$\\
$132$ & $1$\\
$231$ & $2$\\
\end{tabular}
\]
Note that a representative of the coset is of minimal length if the $1$-st and the $2$-nd entries of a permutation of minimal length representing $W(Q_{12})\backslash W$ are in increasing order. For the list of the $1$-st and the $2$-nd entries of the permutations, we have $12$, $13$ and $23$. 

Similarly, for $Q_{23}$ we have that the elements of $W(Q_{23})$ are $123$ and $132$.
Then the minimal length of element in the quotient $W(Q_{23})\backslash W$ are $123$, $213$ and $312$ of length $0,1,2$ respectively. The remaining elements in each of the cosets are $132$, $231$ and $321$ have length $1$, $2$ and $3$.  
We can organize it in a table
\[
\begin{tabular}{ccccc}
$w \in  W(Q_{23})\backslash W$ & length\\
\\
$123$ & $0$\\
$213$ & $1$\\
$312$ & $2$\\
\end{tabular}
\]
Note that the minimal representative of each quotient $W(Q_{23})\backslash W$ has the following property. The $2$-nd and the $3$-rd digit of $w$ are in increasing order. For these representatives the $2$-nd and the $3$-rd entries are
$12$, $13$, $23$.

Let $V(a,b,c)$ is a representation of $GL_3$ of highest weight $\lambda=(a,b,c)$. We can restrict it to a parabolic subgroup $Q$. 
Then $H^q(Q,V_\lambda)=\bigoplus_{w\in W(Q)\backslash W}H^{q-length(w)}(M_Q,V_{w(\lambda+\rho)-\rho})$.

Therefore we need to have explicitly the weights of the Levi quotients $M_Q$.

Let $Q=Q_{12}$. Then

\[
\begin{tabular}{ccccc}
$w \in  W(Q_{12})\backslash W$	& length	& $w((a,b,c)+\rho)-\rho$\\
\\
$123$ 						& $0$	& $(a,b|c)$\\
$132$ 						& $1$	& $(a,c-1|b+1)$\\
$231$						& $2$	& $(b-1,c-1|a+2)$
\end{tabular}
\]
The resulting weight for the Levi quotient is, for instance $(a,c-1,b+1)$. We use the vertical line in the notation, that is $(a,c-1|b+1)$
to signify that the Levi quotient has $GL_2$ in the upper left corner and ${\mathbb G}_m$ in the lower right corner.

Note that it is consistent with the highest weights. If $(a,b,c)$ is a highest weight for $GL_3$, then $a\geq b\geq c$. Then, for the first two entries of $w((a,b,c)+\rho)-\rho$, we have $a\geq b$, $a\geq c-1$ and $b-1\geq c-1$.

We consider $Q_{23}$ in a similar way. 
\[
\begin{tabular}{ccccc}
$w \in  W(Q_{23})\backslash W$	& length	& $w((a,b,c)+\rho)-\rho$\\
\\
$123$ 						& $0$	& $(a|b,c)$\\
$213$ 						& $1$	& $(b-1|a+1,c)$\\
$312$						& $2$	& $(c-2|a+1,b+1)$
\end{tabular}
\]
Again, the vertical line separated the first from the second and the third entry. Let's consider the second and the third columns and rows of $M_{23}$. They form the group $GL_2$. Again, we have that $(a,b,c)$ is a highest weight for $GL_3$, that is $a\geq b\geq c$ where $a$, $b$ and $c$ are integers. The second and the third entry form a highest weight representation of 
$GL_2$, namely, $(b,c)$, $(a+1,c)$ and $(a+1,b+1)$. Note that $b\geq c$, $a+1 \geq c$ and $a+1\geq b+1$.

In general, this is a recipe for  finding all the representatives with minimal length of $W_{Q_{12}}\backslash W$. Consider a regular highest weight representation $\lambda=(a_1,a_2,a_3)$, that is $a_1>a_2>a_3$. The consists exactly those $w\in W$ such that 
$a_{w^{-1}(1)}>a_{w^{-1}(2)}$. In practice, this means that the first and the second entries of the permutation $w$ are arranged in  increasing order. Those permutations are exactly,
$(123)$, 
$(132)$
and
$(231)$.

Similarly, this is a recipe for  finding all the representatives with minimal length of $W_{Q_{23}}\backslash W$. Consider a regular highest weight representation $\lambda=(a_1,a_2,a_3)$, that is $a_1>a_2>a_3$. They consists exactly those $w\in W$ such that 
$a_{w^{-1}(2)}>a_{w^{-1}(3)}$. In practice this means that the second and the third entries permutation of $w$ are arranged in increasing order. 
Those permutations are exactly,
 $(123)$, 
$(213)$
and
$(312)$.

Let us consider an example of a parabolic subgroup $P_{13}$ of $GL_4$. The parabolic subgroup $P_{13}$ contains $GL_3$ in the intersection of its first three rows and first three columns. To find the Weil elements from the cosets $W_{P_{13}}\backslash W$. Let $\lambda=(a_1,a_2,a_3,a_4)$ be a regular highest weight representation of $GL_4$, that is $a_1>a_2>a_3>a_4$.  
They consists exactly those $w\in W$ such that 
$a_{w^{-1}(2)}>a_{w^{-1}(2)}>a_{w^{-1}(3)}$. In practice this means that the first, the  second and the third entries of the permutation of $w$ are arranged in increasing order. 
Those permutations are exactly,
 $(1234)$, 
$(1243)$,
$(1342)$
and
$(2341)$.

The dualities that we use either of the type of Serre's duality. We compare boundary cohomology and Eisenstein cohomology with coefficient in a given representation $V$ to to the corresponding cohomology with coefficients in $V\otimes det$, or $V^*$, or $V^*\otimes det$.

The other techniques that we mentioned is homological Euler characteristics. Intuitively it is the sum of alternating dimensions of cohomology groups; however it is computed via considering torsion elements of the arithmetic group. A homological Euler characteristic gives an extra invariant to work with. It useful in the process of finding the Eisenstein cohomology after having the boundary cohomology.

\begin{conjecture} For every highest weight representation V of $GL_m$ there is a non-degenerate pairing for the boundary cohomology
\[H^p_\partial(GL_m(\Z),V)\otimes H^{N-p}_\partial(GL_m(\Z),V\otimes det^{m-1})\rightarrow H^{N}_\partial(GL_m(\Z),det^{m-1})=\C,\]
where $N=m(m+1)/2-1$ is the dimension of the boundary component.
\end{conjecture}

\begin{conjecture} \[H^p_\partial(GL_m(\Z),V)=H^p_{Eis}GL_m(\Z),V)+H^{N-p}_{Eis}(GL_m(\Z),V\otimes det^{m-1}).\]
\end{conjecture}
 
\begin{defi}
Let $P$ be a parabolic subgroup with Levi quotient $S$ and nilpotent radical $N$.
We define inner cohomology of $P$ to be 
\[H^q_!(P,V)=\bigoplus_{i+j=q}H^i_!(S,H^j(N,V)),\]
where 
$H^q_!(GL_1(\Z),V)=
H^q(GL_1(\Z),V)$.
\end{defi}

\begin{defi}
Let $P$ be a parabolic subgroup with Levi quotient $S$ and nilpotent radical $N$.
We define Eisenstein cohomology of $P$ to be the cokernel 
\[H^q_{Eis}(P,V)=coker[H^q_!(P,V)\rightarrow [H^q(P,V)].\]
\end{defi}

\begin{conjecture} Up to semisimplicity, for every $q$, the boundary cohomology $H^q_\partial(GL_m(\Z),V)$ a direct sum of inner cohomologies of Levi quotients $S_P$ of certain representations $V'$ of $S_P$.
\end{conjecture}

The following the definition is essential part of the conjecture duality between cohomologies of higher rank parabolic subgroups.
\begin{defi}
Let $E_1^{p,q}(V)=\oplus_{rank(P)=p+1}H^q(P,V)$ be the spectral sequence that converges to to the boundary cohomology $H^{p+q}_\partial(GL_m(\Z),V)$. 
Consider parabolic subgroups of rank $r$. Consider $E_\infty^{p,q}(V)$ up to semisimplicity. Let $E^{p,q}_\infty(V,r)$ be the direct summand of $E_\infty^{p,q}$  consisting of all inner cohomologies $H^i(S_P,V')$ where $S_P$ is the Levi quotient of a parabolic group $P$ of rank $r$.
\end{defi}

\begin{conjecture}
\[E^{p,q}_\infty(V,r)=E^{p-r+1,N-q+r-1}_\infty(V\oplus det^{m-1},r),\]
where $N=m(m+1)/2-1$.
\end{conjecture}

\section{A Theorem on Boundary Cohomology of $GL_4(\Z)$}

\begin{thm}
The boundary cohomology of $GL_4(\Z)$ with coefficients in any highest weight representation is given by

(A1) The family of representations $A1$ are highest weight representations with weights $(2a,2b,2c,2d)$
\[H^q_\partial(GL_4(\Z),A1)=
\left\{
\begin{tabular}{llll}
$\left\{
	\begin{tabular}{llll}
	$E_\infty^{2,0}
		=\C=(2a|2b|2c|2d)$\\
	$E_\infty^{1,1}
		=
		\left\{
		\begin{tabular}{lll}
		$(2a,2b|2c|2d)=S_{2a-2b+2}$\\
		$(2a|2b,2c|2d)=S_{2b-2c+2}$\\
		$(2a|2b|2c,2d)=S_{2c-2d+2}$
		\end{tabular}
		\right\}
		$\\
	$E_\infty^{0,2}
		=
		\left\{
		\begin{tabular}{ll}
		$(2a,2b|2c,2d)=S_{2a-2b+2}S_{2c-2d+2}$\\
		$(\overline{2a,2b,2c}|2d)=S_{2a,2b,2c}(SL_3(\Z))$\\
		$(2a|\overline{2b,2c,2d})=S_{2b,2c,2d}(SL_3(\Z))$
		\end{tabular}
		\right\}$
	\end{tabular}
\right\}$ 
			& $q=2$\\
	$E_\infty^{0,3}
	=
	\left\{
	\begin{tabular}{lll}
	$(\overline{2a,2b,2c}|2d)=S_{2a,2b,2c}(SL_3(\Z))$\\
	$(2a|\overline{2b,2c,2d})=S_{2b,2c,2d}(SL_3(\Z))$
	\end{tabular}
	\right\}$ & $q=3$\\
	$E_\infty^{0,4}=\left\{
	\begin{tabular}{lll}
	$(2a,2d-2|\overline{2b+1,2c+1})=S_{2a_2d+4}S_{2b-2c+2}$\\
	$(\overline{2b-1,2c-1}|2a+2,2d)=S_{2a-2d+4}S_{2b-2c+2}$\\
	$(2c-2|\overline{2a+1,2d-1}|2b+2)=S_{2a-2d+4}$
	\end{tabular}
	\right\}$ & $q=4$\\
	$0$ & $q=5$\\
	$E_\infty^{0,6}=(2c-2,2d-2|2a+2,2b+2)=S_{2a-2b+2}S_{2c-2d+2}$ & $q=6$
\end{tabular}
\right.
\]


(A1') The family of representations $A1'$ are highest weight representations with weights $(2a+1,2b+1,2c+1,2d+1)$
\[
H^q_\partial(GL_4(\Z),A1')
=
\left\{
\begin{tabular}{lll}
	$E_\infty^{0,2}=(\overline{2a+1,2b+1}|\overline{2c+1,2d+1})=S_{2a-2b+2}S_{2c-2d+2}$
		& $q=2$\\
	$0$	
		& $q=3$\\
	$\left\{
	\begin{tabular}{lll}
		$E_\infty^{0,4}
		=
			\left\{
			\begin{tabular}{lll}
			$(\overline{2a+1,2d-1}|2b+2,2c+2)=S_{2a-2d+4}S_{2b-2c+2}$\\
			$(2b,2c|\overline{2a+3,2d+1})=S_{2a-2d+4}S_{2b-2c+2}$
			\end{tabular}
			\right\}$\\
		$E_\infty^{1,3}=(2b|2a+2,2d|2c+2)=S_{2a-2d+4}$
	\end{tabular}
	\right\}$
		& $q=4$\\
	$E_\infty^{0,5}
	=
		\left\{
		\begin{tabular}{lll}
		$(\overline{2b,2c,2d}|2a+4)=S_{2b,2c,2d}(SL_3(\Z))$\\
		$(2d-2|\overline{2a+2,2b+2,2c+2})=S_{2a,2b,2c}(SL_3(\Z))$
		\end{tabular}
		\right\}$
			& $q=5$\\
	$E_\infty^{0,6}
	=
		\left\{
		\begin{tabular}{lll}
		$(\overline{2b,2c,2d}|2a+4)=S_{2b,2c,2d}(SL_3(\Z))$\\
		$(2d-2|\overline{2a+2,2b+2,2c+2})=S_{2a,2b,2c}(SL_3(\Z))$\\
		$(\overline{2c-1,2d-1},\overline{2a+3,2b+3})=S_{2c-2d+2}S_{2a-2b+2}$\\
		$(2d-2|2c|\overline{2a+3,2b+3})=S_{2a-2b+2}$\\
		$(2d-2|\overline{2b+1,2c+1}|2a+4)=S_{2b-2c+2}$\\
		$(\overline{2c-1,2d-1}|2b+2|2a+4)=S_{2c-2d+2}$\\
		$(2d-2|2c|2b+2|2a+4)=\C$
		\end{tabular}
		\right\}$
			& $q=6$
\end{tabular}
\right.
\]

(A2) The family of representations $A1'$ are highest weight representations with weights $(2a,2b,2c,2c)$, (We put restriction $(c=d)$)
\[
H^q_\partial(GL_4(\Z),A2)=
\left\{
\begin{tabular}{llll}
	$\left\{
	\begin{tabular}{llll}
	$E_\infty^{1,1}=(2a|2b, 2c|2c)=S_{2b-2c+2}$\\
	$E_\infty^{0,2}=(\overline{2a,2b,2c}|2c)=S_{2a,2b,2c}(SL_3(\Z))$
	\end{tabular}
	\right\}$
		& $q=2$\\
	$E_\infty^{0,3}=\overline{2a,2b,2c}|2c)=S_{2a,2b,2c}(SL_3(\Z))$
		 & $q=3$\\
	$E_\infty^{0,4}=
	\left\{
	\begin{tabular}{lll}
	$(2a,2c-2|\overline{2b+1,2c+1})=S_{2a-2c+4}S_{2b-2c+2}$\\
	$(\overline{2b-1,2c-1}|2a+2,2c)=S_{2a-2c+4}S_{2b-2c+2}$\\
	$(2c-2|\overline{2a+1,2c-1}|2b+2)=S_{2a-2c+4}$
	\end{tabular}
	\right\}$
		& $q=4$\\
	$E_\infty^{0,5}=(2c-2|2c-|2a+2,2b+2)=S_{2a-2b+2}$
		& $q=5$\\
	$E_\infty^{0,6}=0$ & $q=6$
\end{tabular}
\right.
\]

(A2')The family of representations $A2'$ are highest weight representations with weights $(2a+1,2b+1,2c+1,2c+1)$ with the restriction $c=d$.
\[
H^q_\partial(GL_4(\Z),A2')
=
\left\{
\begin{tabular}{lll}
	$E_\infty^{1,2}=(\overline{2a+1,2b+1}|2c|2c+2)$	
		& $q=3$\\
	$\left\{
	\begin{tabular}{lll}
		$E_2^{0,4}=
		\left\{
		\begin{tabular}{lll}
		$(\overline{2a+1,2c-1}|2b+2,2c+2)=S_{2a-2c+4}S_{2b-2c+2}$\\
		$(2b,2c|\overline{2a+3,2c+1})=S_{2a-2c+4}S_{2b-2c+2}$
		\end{tabular}
		\right\}$\\
		$E_\infty^{1,3}=(2b|2a+2,2c|2c+2)=S_{2a-2c+4}$
	\end{tabular}
	\right\}$
		& $q=4$\\
	$E_\infty^{0,5}
		=(2d-2|\overline{2a+2,2b+2,2c+2})=S_{2a,2b,2c}(SL_3(\Z))$
			& $q=5$\\
	$E_\infty^{0,6}=
	\left\{
	\begin{tabular}{lll}
	$(2c-2|\overline{2a+2,2b+2,2c+2})=S_{2a,2b,2c}(SL_3(\Z))$\\
	$(2c-2|\overline{2b+1,2c+1}|2a+4)=S_{2b-2c+2}$
	\end{tabular}
	\right\}$
		& $q=6$
\end{tabular}
\right.
\]

(A3) The family of representations $A3$ are highest weight representations with weights $(2a,2b,2b,2d)$, with the restriction $b=c$.
\[H^q_\partial(GL_4(\Z),A3)=
\left\{
\begin{tabular}{llll}
$\left\{
	\begin{tabular}{llll}
	$E_2^{1,1}=
		\left\{
		\begin{tabular}{lll}
		$(2a,2b|2b|2d)=S_{2a-2b+2}$\\
		$(2a|2b|2b,2d)=S_{2b-2d+2}$
		\end{tabular}
		\right\}$\\
	$E_\infty^{0,2}=(2a,2b|2b,2d)=S_{2a-2b+2}S_{2b-2d+2}$
	\end{tabular}
		\right\}$
		& $q=2$\\
	$0$ & $q=3,4$\\
	$E_\infty^{1,4}=(2b-2|\overline{2a+1,2d-1}|2b+2)=S_{2a-2d+4}$ & $q=5$\\
	$E_\infty^{0,6}=(2b-2,2d-2|2a+2,2b+2)=S_{2a-2b+2}S_{2b-2d+2}$ & $q=6$
\end{tabular}
\right.
\]

(A3') 
The family of representations $A3'$ are highest weight representations with weights $(2a+1,2b+1,2b+1,2d+1)$ with the restriction $b=c$.
\[
H^q_\partial(GL_4(\Z),A3')
=
\left\{
\begin{tabular}{lll}
	$E_\infty^{0,2}=(\overline{2a+1,2b+1}|\overline{2b+1,2d+1})=S_{2a-2b+2}S_{2b-2d+2}$
		& $q=2$\\
	$E_\infty^{0,3}=(2b|2a+2,2d|2b+2)=S_{2a-2d+4}$
		& $q=3$\\
	$0$
		& $q=4,5$\\
	$E_\infty^{0,6}
	=
		\left\{
		\begin{tabular}{lll}
		$(\overline{2b-1,2d-1}|\overline{2a+3,2b+3})=S_{2b-2d+2}S_{2a-2b+2}$\\
		$(\overline{2b-1,2d-1}|2b+2|2a+4)=S_{2b-2d+2}$\\
		$(2d-2|2b|\overline{2a+3,2b+3})=S_{2a-2b+2}$
		\end{tabular}
		\right.$
			& $q=6$
\end{tabular}
\right.
\]

(A4) The family of representations $A4$ are highest weight representations with weights $(2a,2a,2c,2d)$ with the restriction $a=b$
\[H^q_\partial(GL_4(\Z),A4)=
\left\{
\begin{tabular}{llll}
	$\left\{
	\begin{tabular}{llll}
	$E_\infty^{1,1}=(2a|2a,2c|2d)=S_{2a-2c+2}$\\
	$E_\infty^{0,2}
		=
		\left\{
		\begin{tabular}{ll}
		$(2a,2b|2c,2d)=S_{2a-2b+2}S_{2c-2d+2}$\\
		$(\overline{2a,2b,2c}|2d)=S_{2a,2b,2c}(SL_3(\Z))$\\
		$(2a|\overline{2b,2c,2d})=S_{2b,2c,2d}(SL_3(\Z))$
		\end{tabular}
		\right\}$
	\end{tabular}
	\right\}$ 
			& $q=2$\\
	$E_\infty^{0,3}
	=
	\left\{
	\begin{tabular}{lll}
	$(\overline{2a,2b,2c}|2d)=S_{2a,2b,2c}(SL_3(\Z))$\\
	$(2a|\overline{2b,2c,2d})=S_{2b,2c,2d}(SL_3(\Z))$
	\end{tabular}
	\right\}$ & $q=3$\\
	$E_\infty^{0,4}
	=
		\left\{
		\begin{tabular}{lll}
			$(2a,2d-2|\overline{2a+1,2c+1})=S_{2a_2d+4}S_{2a-2c+2}$\\
			$(\overline{2a-1,2c-1}|2a+2,2d)=S_{2a-2d+4}S_{2a-2c+2}$\\
			$(2c-2|\overline{2a+1,2d-1}|2b+2)=S_{2a-2d+4}$
		\end{tabular}
		\right\}$ 
			& $q=4$\\
	$E_\infty^{0,5}=(2c-2,2d-2|2a+2,2a+2)=S_{2c-2d+2}$ 
		& $q=5$\\
	$0$ & $q=3,6$
\end{tabular}
\right.
\]

(A4')
The family of representations $A4'$ are highest weight representations with weights $(2a+1,2a+1,2c+1,2d+1)$ with the restriction $a=b$.
\[
H^q_\partial(GL_4(\Z),A4')
=
\left\{
\begin{tabular}{lll}
	$E_\infty^{1,2}=(2a|2a+2|\overline{2c+1,2d+1})=S_{2c-2d+2}$ & $q=3$\\
	$\left\{
	\begin{tabular}{lll}
	$E_\infty^{0,4}
		=
		\left\{
		\begin{tabular}{lll}
		$(\overline{2a+1,2d-1}|2a+2,2c+2)=S_{2a-2d+4}S_{2a-2c+2}$\\
		$(2a,2c|\overline{2a+3,2d+1})=S_{2a-2d+4}S_{2a}-2c+2$
		\end{tabular}
		\right\}$\\
	$E_\infty^{1,3}=(2a|2a+2,2d|2c+2)=S_{2a-2d+4}$
	\end{tabular}
	\right\}$ 
		& $q=4$\\
	$0$ & $q=5$\\
	$E_\infty^{0,6}=(2d-2|\overline{2a+1,2c+1}|2a+4)=S_{2a-2c+2}$
		& $q=6$
\end{tabular}
\right.
\]

(A5) Symmetric powers is the family of representations $A5$ are highest weight representations with weights $(2a,2b,2b,2b)$, with the restriction $b=c=d$.
The boundary cohomology is non-trivial only in degree 5.

\[H^5_\partial(GL_4(\Z),Sym^{2a-2b})=
H^5_\partial(GL_4(\Z),A5)
=
\left\{
\begin{tabular}{lll}
	$E_\infty^{0,5}=(2b-2,2b-2|2a+2,2b+2)=S_{2a-2b+2}$\\
	$E_\infty^{1,4}=(2b-2|\overline{2a+1,2b-1}|2b+2)=S_{2a-2b+4}$\\
	$E_\infty^{2,3}=(2a|2b - 2|2b|2b + 2)=\C$
\end{tabular}
\right.
	\]

(A5') Symmetric powers twisted by the determinant representation is the family of representations $A5'$ are highest weight representations with weights $(2a+1,2b+1,2b+1,2b+1)$, with the restriction $b=c=d$.
The boundary cohomology is non-trivial only in degree 3.
\[H^3_\partial(GL_4(\Z),Sym^{2a-2b})=H^3_\partial(GL_4(\Z),A5')
=
\left\{
\begin{tabular}{lll}
	$E_\infty^{0,3}=
		\left\{
		\begin{tabular}{lll}
		$(2b|2a+2,2b|2b+2)=S_{2a-2b+4}$\\
		$(2b|2b|2b|2a+4)=\C$
		\end{tabular}
		\right\}$\\
	$E_\infty^{1,2}=(\overline{2a+1,2b+1}|2b|2b+2)=S_{2a-2b+2}$
\end{tabular}
\right.
\]
	
(A6) Family of representations $A6$ are highest weight representations with weights $(2a,2a,2c,2c)$, with the restrictions $a=b$ and $c=d$.
\[H^q_\partial(GL_4(\Z),A6)=
\left\{
\begin{tabular}{llll}
$E_\infty^{1,1}=(2a|2a,2c|2c)=S_{2a-2c+2}$ 
		& $q=2$\\
$E_\infty^{0,4}=
	\left\{
	\begin{tabular}{lll}
	$(\overline{2a-1,2c-1}|2a+2,2c)=S_{2a-2c+2}\otimes S_{2a-2c+4}$\\
	$(2a,2c-2|\overline{2a+1,2c+1})=S_{2a-2c+2}\otimes S_{2a-2c+4}$\\	
	$(2c-2|2a+1,2c-1|2a+2)=S_{2a-2c+4}$
	\end{tabular}
	\right\}$
		& $q=4$
\end{tabular}
\right.
\]

(A6') Family of representations $A6$ are highest weight representations with weights $(2a+1,2a+1,2c+1,2c+1)$, with the restrictions $a=b$ and $c=d$.
\[H^q_\partial(GL_4(\Z),A6')
=
\left\{
\begin{tabular}{lll}
	$\left\{
	\begin{tabular}{lll}
	$E_\infty^{0,4}
		=
		\left\{
		\begin{tabular}{lll}
		$(\overline{2a+1,2c-1}|2a+2,2c+2)=S_{2a-2c+4}\otimes S_{2a-2c+2}$\\
		$(2a,2c|\overline{2a+3,2c+1})=S_{2a-2c+4}\otimes S_{2a-2c+2}$
		\end{tabular}
		\right.$\\
	$E_\infty^{1,3}=(2a|2a+2,2c|2c+2)=S_{2a-2c+4}$\\
	$E_\infty^{2,2}=(2a|2a+2|2c|2c+2)=\C$
	\end{tabular}
	\right\}$
		& $q=4$\\
	$E_\infty^{0,6}=(2c-2|\overline{2a+1,2c+1}|2a+4)=S_{2a-2c+2}$ 
		& $q=6$
\end{tabular}
\right.
\]

(A7) Dual of the symmetric power. Restriction $a=b=c$. The cohomology is non-trivial only for $q=5$
\[H^5_\partial(GL_4(\Z),(2a,2a,2a,2d))=
\left\{
\begin{tabular}{llll}
	$E_\infty^{0,5}=(2a-2,2d-2|2a+2,2a+2)=S_{2a-2d+2}$\\
	$E_\infty^{1,4}=(2a-2|\overline{2a+1,2d-1}|2a+2=S_{2a-2d+4}$\\
	$E_\infty^{2,3} = (2a-2|2a|2a+2|2d)=\C$
\end{tabular}
\right.
\]

(A7') Dual of the symmetric power twisted by the determinant. Representations with highest weight $(2a+1,2a+1,2a+1,2d+1)$. Restriction $a=b=c$. The cohomology is non-trivial only for $q=3$
\[H^3(GL_4(\Z),A7')
=
\left\{
\begin{tabular}{lll}
$E_\infty^{0,3}
	=
	\left\{
	\begin{tabular}{lll}
	$(2d-2|2a+2|2a+2|2a+2)=\C$\\
	$S_{2a-2d+4}$
	\end{tabular}
	\right\}$\\
$E_\infty^{1,2}=(2a|2a+2|\overline{2a+1,2d+1})=S_{2a-2d+2}$
\end{tabular}
\right.
\]

(A8) Trivial representation
\[H^q_\partial(GL_4(\Z),\C)=
\left\{
\begin{tabular}{llll}
	$E_\infty^{0,0}=(0|0|0|0)=\C$ 	& $q=0$\\
	$E_\infty^{2,3}=(-2|0|2|0)=\C$	& $q=5$\\
	$0$ & $q\neq 0, 5$\\
\end{tabular}
\right.
\]

(A8') Determinant representation
\[H^q_\partial(GL_4(\Z),det)=
\left\{
\begin{tabular}{llll}
	$E_\infty^{0,3}=(0|0|0|4)=\C$ & $q=3$\\
	$E_\infty^{2,6}=(-2|0|2|4)=\C$ & $q=8$\\
	$0$ & $q\neq 3, 8$\\
\end{tabular}
\right.
\]

(B1) Family of representations with with highest weight $(2a+1,2b+1,2c,2d)$
\[H^q_\partial(GL_4(\Z),B1)=
\left\{
\begin{tabular}{llll}
$\left\{
	\begin{tabular}{lll}
		$E_\infty^{0,2}=(\overline{2a + 1, 2b + 1}|2c, 2d)=S_{2a-2b+2}\otimes S_{2c-2d+2}$\\
		$E_\infty^{1,1}=(\overline{2a + 1, 2b + 1}|2c|2d)=S_{2a-2b+2}$
	\end{tabular}
	\right\}$ 
		& $q=2$\\
$E_\infty^{0,3}=(\overline{2a + 1,2c - 1}|2b + 2,2d)=S_{2a-2c+4}\otimes S_{2b-2d+4}$
		& $q=3$\\
$E_\infty^{0,5}=\Delta_2+(2b,2d - 2|\overline{2a + 3,2c + 1})=S_{2b-2d+4}+S_{2a-2c+4}\otimes S_{2b-2d+4}$
		& $q=5$\\
$E_\infty^{0,6}=(2c-2,2d-2|2a+3,2b+3)=S_{2a-2b+2}\otimes S_{2c-2d+2}+S_{2c-2d+2}$
		& $q=6$
\end{tabular}
	\right.
\]

(B1') Family of representations with with highest weight $(2a,2b,2c-1,2d-1)$.
\[H^q_\partial(GL_4(\Z),B1')=
\left\{
\begin{tabular}{llll}
$\left\{
	\begin{tabular}{lll}
	$E_\infty^{0,2}=(2a, 2b|\overline{2c - 1, 2d - 1})=S_{2a-2b+2}\otimes S_{2c-2d+2}$\\
	$E_\infty^{1,1}=(2a|2b|\overline{2c - 1, 2d - 1})=S_{2c-2d+2}$
	\end{tabular}
	\right\}$
		& $q=2$\\
	$\left\{
	\begin{tabular}{lll}
	$E_\infty^{0,3}=(2a,2c - 2|\overline{2b + 1,2d - 1})=S_{2a-2c+4}\otimes S_{2b-2d+4}$\\
	$E_\infty^{1,2}=(2a|2b, 2d - 2|2c)=S_{2b-2d+4}$
	\end{tabular}
	\right\}$
		& $q=3$\\
	$0$
		& $q=4$\\
	$E_\infty^{0,5}=\left\{
	\begin{tabular}{lll}
	$(\overline{2b - 1,2d - 3}|2a + 2,2c)=S_{2a-2c+4}\otimes S_{2b-2d+4}$\\
	$(2d - 4|2b|2a + 2, 2c)=S_{2a-2c+4}$
	\end{tabular}
	\right\}$
		& $q=5$\\
	$E_\infty^{0,6}=\left\{
	\begin{tabular}{lll}
	$(\overline{2c-3,2d-3}|2a+2,2b+2)=S_{2a-2b+2}\otimes S_{2c-2d+2}$\\
	$(2d - 4|2b|2a + 2, 2c)=S_{2a-2b+2}$
	\end{tabular}
	\right\}$
		& $q=6$
\end{tabular}
\right.
\]

(B2)  Family of representations with highest weight $(2a+1,2b+1,2c,2c)$, restricting to $c=d$.
\[H^q_\partial(GL_4(\Z),B2)
=
\left\{
\begin{tabular}{lll}
	$\left\{
	\begin{tabular}{lll}
	$E_\infty^{0,3}=(\overline{2a + 1,2c - 1}|2b + 2,2c)=S_{2a-2c+4}\otimes S_{2b-2c+4}$\\
	$E_\infty^{1,2}=S_{2a-2c+4}$
	\end{tabular}
	\right\}$
		& $q=3$\\
$0$
		& $q=4$\\
$E_\infty^{0,5}=
	\left\{
	\begin{tabular}{ll}
	$\Delta_1+\Delta_0$=$S_{2b-2c+4}+\C$\\
	$(2b,2c - 2|\overline{2a + 3,2c + 1}=S_{2b-2c+4}\otimes S_{2a_2c+4}$\\
	$(2c-2,2c-2|\overline{2a+3,2b+3})=S_{2b-2c+4}\otimes S_{2a_2c+4}$
	\end{tabular}
	\right.$
		& $q=5$\\
$0$ 		& $q=2,4,6$
\end{tabular}
\right.
\]

(B2') Family of representations with highest weight $(2a,2b,2c-1,2c-1)$, restricting to $c=d$.

\[H^q_\partial(GL_4(\Z),B2')
=
\left\{
\begin{tabular}{lll}
	$\left\{
	\begin{tabular}{lll}
	$E_\infty^{0,3}=(2a,2c - 2|\overline{2b + 1,2c - 1})=S_{2a-2c+4}\otimes S_{2b-2c+4}$\\
	$E_\infty^{1,2}=
		\left\{
		\begin{tabular}{ll}
		$(2a, 2b|2c + 2|2c)=S_{2a-2b+2}$\\
		$(2a|2b, 2c - 2|2c)=S_{2b-2c+4}$
		\end{tabular}
		\right\}$\\
	$E_\infty^{2,1}=(2a|2b|2c - 2|2c)=\C$
	\end{tabular}
		\right\}$
			& $q=3$\\

$E_\infty^{0,5}=\left\{
	\begin{tabular}{lll}
	$(\overline{2b - 1,2c - 3}|2a + 2,2c)=S_{2b-2c+4}\otimes S_{2a-2c+4}$\\
	$(2c - 4|2b|2a + 2, 2c)=S_{2a-2c+4}$
	\end{tabular}
	\right\}$
		& $q=5$\\	
$0$ & $q\neq 3,5$	
\end{tabular}
\right.
\]

(B3) Family of representations with highest weight $(2a+1,2a+1,2c,2d)$, restricting to $a=b$.
\[H^q_\partial(GL_4(\Z),B3)
=
\left\{
\begin{tabular}{lll}
	$\left\{
	\begin{tabular}{lll}
	$E_\infty^{0,3}=(\overline{2a + 1,2c - 1}|2a + 2,2d)=S_{2a-2c+4}\otimes S_{2a-2d+4}$\\
	$E_\infty^{1,2}=
		\left\{
		\begin{tabular}{ll}
		$(2a|2a + 2, 2c|2d)=S_{2a-2c+4}$\\
		$(2a|2a + 2|2c, 2d)=S_{2c-2d+2}$
		\end{tabular}
		\right\}$\\
	$E_\infty^{2,1} = (2a|2a + 2|2c|2d)=\C$
	\end{tabular}
	\right\}$
		& $q=3$\\
$E_\infty^{0,5}=
	\left\{
	\begin{tabular}{ll}
	$(2a,2d - 2|\overline{2a + 3,2c + 1})=S_{2a-2d+4}\otimes S_{2a-2c+4}$\\
	$(2a,2d-2|2c|2a+4)=S_{2a-2d+4}$
	\end{tabular}
	\right\}$
		& $q=5$\\
	$0$ 
		& $q\neq 3, 5$
\end{tabular}
\right.
\]

(B3') Family of representations with highest weight $(2a,2a,2c-1,2d-1)$, restricting to $a=b$.

\[H^q_\partial(GL_4(\Z),B3')
=
\left\{
\begin{tabular}{lll}
	$\left\{
	\begin{tabular}{lll}
	$E_\infty^{0,3}=(2a,2c - 2|\overline{2a + 1,2d - 1})=S_{2a-2c+4}\otimes S_{2a-2d+4}$\\
	$E_\infty^{1,2}=(2a|2a, 2d - 2|2c)=S_{2a-2d+4}$
	\end{tabular}
	\right\}
	$
		& $q=3$\\
$E_\infty^{0,5}=\left\{
	\begin{tabular}{lll}
	$(\overline{2a - 1,2d - 3}|2a + 2,2c)=S_{2a-2c+4}\otimes S_{2a-2d+4}$\\
	$(\overline{2c-3,2d-3}|2a+2,2a+2)=S_{2c-2d+2}$\\
	$(2d - 4|2a|2a + 2, 2c)=S_{2a-2c+4}$\\
	$(2d-4|2c-2|2a+2|2a+2)=\C$
	\end{tabular}
	\right\}$
				& $q=5$\\
	$0$
		& $q\neq 3, 5$
\end{tabular}
\right.
\]

(B4) Family of representations with highest weight $(2a+1,2a+1,2c,2c)$, restricting to $a=b$ and $c=d$.
\[H^q_\partial(GL_4(\Z),B4)
=
\left\{
\begin{tabular}{lll}
	$\left\{
	\begin{tabular}{lll}	
	$E_\infty^{0,3}=(\overline{2a + 1,2c - 1}|2a + 2,2c)=S_{2a-2c+4}\otimes S_{2a-2c+4}$\\
	$E_\infty^{1,2}=(2a|2a + 2, 2c|2c)=S_{2a-2c+4}$
	\end{tabular}
	\right\}
	$	
			& $q=3$\\	
$E_\infty^{0,5}=
	\left\{
	\begin{tabular}{ll}
	$(2a,2c - 2|\overline{2a + 3,2c + 1})=S_{2a-2c+4}\otimes S_{2a-2c+4}$\\
	$(2a,2c-2|2c|2a+4)=S_{2a-2c+4}$
	\end{tabular}
	\right\}
	$
		& $q=5$
$0$
		& $q\neq 3, 5$
\end{tabular}
\right.
\]

(B4') Family of representations with highest weight $(2a,2a,2c-1,2d-1)$, restricting to $a=b$ and $c=d$.

\[H^q_\partial(GL_4(\Z),B4')
=
\left\{
\begin{tabular}{lll}
	$\left\{
	\begin{tabular}{lll}
	$E_\infty^{0,3}=(2a,2c - 2|\overline{2a + 1,2c - 1})=S_{2a-2c+4}\otimes S_{2a-2c+4}$\\
	$E_\infty^{1,2}=(2a|2a, 2c - 2|2c)=S_{2a-2c+4}$
	\end{tabular}
	\right\}$	
			& $q=3$\\	
$E_\infty^{0,5}=\left\{
	\begin{tabular}{lll}
	$(\overline{2a - 1,2c - 3}|2a + 2,2c)=S_{2a-2c+4}\otimes S_{2a-2c+4}$\\
	$(2c - 4|2a|2a + 2, 2c)=S_{2a-2c+4}$
	\end{tabular}
	\right.$
		& $q=5$
$0$
		& $q\neq 3, 5$
\end{tabular}
\right.
\]

(C1) Family of representations with highest weight $(2a+1,2b,2c,2d-1)$
\[H^q_\partial(GL_4(\Z),C1)
=
\left\{
\begin{tabular}{lll}
	$0$
		& $q=2$\\
	$E_\infty^{0,3}=(\overline{2a+1,2c-1}|\overline{2b+1,2d-1})=S_{2a-2c+4}\otimes S_{2b-2d+4}$
		& $q=3$\\
	$E_\infty^{04}=\left\{
	\begin{tabular}{lll}
	$(\overline{2a+1, 2d-3}|\overline{2b+1,2c+1})=S_{2a-2d+6}\otimes S_{2b-2c+2}$\\
	$(\overline{2b-1,2c-1}|\overline{2a+3,2d-1})=S_{2a-2d+6}\otimes S_{2b-2c+2}$\\
	$(2c-2|2a+2,2d-2|2b+2)=S_{2a-2d+6}$
	\end{tabular}
	\right\}$
		& $q=4$\\
$E_\infty^{0,5}=\left\{
	\begin{tabular}{lll}
	$(\overline{2b-1,2d-3}|\overline{2a+3,2c+1})=S_{2b-2d+4}\otimes S_{2a-2c+4}$\\
	$(2d - 4|2b|2a + 3,2c + 1)=S_{2a-2c+4}$\\
	$(2b - 1,2d - 3|2c|2a + 4)=S_{2b-2d+4}$
	\end{tabular}
	\right.$
		& $q=5$\\
$E_\infty^{0,6}=(2d - 4|2b, 2c|2a + 4)=S_{2b-2c+2}$
		& $q=6$
\end{tabular}
\right.
\]

(C1') Family of representations with highest weight $(2a+2,2b+1,2c+1,2d)$
\[
H^q_\partial(GL_4(\Z),C1')
=
\left\{
\begin{tabular}{lll}
$E_\infty^{1,1}=(2a+2|\overline{2b+1,2c+1}|2d)=S_{2b-2c+2}$
		& $q=2$\\
	$\left\{
	\begin{tabular}{lll}
	$E_\infty^{0,3}=(2a+2,2c|2b+2,2d)=S_{2a-2c+4}\otimes S_{2b-2d+4}$\\
	$E_\infty^{1,2}=\left\{
		\begin{tabular}{lll}
		$(2a + 2, 2c|2b + 2|2d)=S_{2a-2c+4}$\\
		$(2a + 2|2c|b + 2, 2d)=S_{2b-2d+4}$
		\end{tabular}
		\right\}$
	\end{tabular}
	\right\}$
		& $q=3$\\
	$\left\{
	\begin{tabular}{lll}
	$E_\infty^{0,4}=
		\left\{
		\begin{tabular}{lll}
		$(2a + 2, 2d - 2|2b + 2, 2c + 2)=S_{2a-2d+6}\otimes S_{2b-2c+2}$\\
		$(2b, 2c|2a + 4, 2d)=S_{2a-2d+6}\otimes S_{2b-2c+2}$ 
		\end{tabular}
		\right.$\\	
	$E_\infty^{1.3}=(2b|\overline{2a + 3,2d - 1}|2c + 2)=S_{2a-2d+6}$
	\end{tabular}
	\right\}$
		& $q=4$\\
$E_\infty^{0,5}=(2b,2d-2|2a+4,2c+2)=S_{2b-2d+4}\otimes S_{2a-2c+4}$
		& $q=5$\\
$0$
		& $q=6$
\end{tabular}
\right.
\]

(C2) Family of representations with highest weight $(2a+1,2b,2b,2d-1)$ restricting to $b=c$.
\[H^q_\partial(GL_4(\Z),C2)
=
\left\{
\begin{tabular}{lll}
	$E_\infty^{0,3}=(\overline{2a+1,2b-1}|\overline{2b+1,2d-1})=S_{2a-2b+4}\otimes S_{2b-2d+4}$
		& $q=3$\\
	$\left\{
	\begin{tabular}{lll}
	$E_\infty^{0,5}=\left\{
		\begin{tabular}{lll}
		$(\overline{2b-1,2d-3}|\overline{2a+3,2b+1})=S_{2b-2d+4}\otimes S_{2a-2b+4}$\\
		$(\overline{2b - 1,2d - 3}|2b|2a + 4)=S_{2b-2d+4}$\\
		$(2d - 4|2b|\overline{2a + 3,2b + 1})=S_{2a-2b+4}$\\
		$(2d - 4|2b|2b|2a + 4)\C$
		\end{tabular}
		\right\}$\\
	$E_\infty^{1,4}=(2b-2|2a+2,2d-2|2b+2)=S_{2a-2d+6}$
	\end{tabular}
	\right\}$
		& $q=5$\\
$0$
		& $q\neq 3, 5$
\end{tabular}
\right.
\]

(C2') Family of representations with highest weight $(2a+2,2b+1,2c+1,2d)$, restricted to $b=c$
\[H^q_\partial(GL_4(\Z),C2')
=
\left\{
\begin{tabular}{lll}
	$\left\{
	\begin{tabular}{lll}
	$E_\infty^{0,3}=
		\left\{
		\begin{tabular}{lll}
		$(2a+2,2b|2b+2,2d)=S_{2a-2b+4}\otimes S_{2b-2d+4}$\\
		$(2b|2a + 3,2d - 1|2b + 2)=S_{2a-2d+6}$
		\end{tabular}
		\right\}$\\
	$E_\infty^{1,2}=
		\left\{
		\begin{tabular}{lll}
		$(2a + 2, 2b|2b + 2|2d)=S_{2a-2b+4}$\\
		$(2a + 2|2b|b + 2, 2d)=S_{2b-2d+4}$
		\end{tabular}
		\right\}$\\
	$E_\infty^{2,1}=(2a+2|2b|2b+2|2d)=\C$
	\end{tabular}
	\right\}$
	$E_\infty^{0,5}=(2b,2d-2|2a+4,2b+2)S_{2b-2d+4}\otimes S_{2a-2b+4}$
		& $q=5$\\
$0$
		& $q\neq 3, 5$
\end{tabular}
\right.
\]

(D1) We expect that for the representation $D1$ the boundary cohomology will vanish. However, what we obtain is the following
\[H^q_\partial(GL_4(\Z),(2a+1,2b,2c-1,2d-2))
=
\left\{
\begin{tabular}{lll}
$E_\infty^{0,2}=(H^2_!(2a+1,2b,2c-1)|2d-2)$
		& $q=2$\\
$E_\infty^{0,3}=\left\{
	\begin{tabular}{ll}
	$(H^3_!(2a+1,2b,2c-1)|2d-2)$\\
	$(H^2_!(2a+1,2b,2d-3)|2c)$
	\end{tabular}
	\right\}$
		& $q=3$\\	
$E_\infty^{0,4}=
	\left\{
	\begin{tabular}{ll}
	$(H^3_!(2a+1,2b,2d-3)|2c)$\\
	$(H^2_!(2a+1,2c-2,2d-3)|2b+2)$
	\end{tabular}
	\right\}$	
		& $q=4$\\
$E_\infty^{0,5}=\left\{
	\begin{tabular}{ll}
	$(H^3_!(2a+1,2c-2,2d-3)|2b+2)$\\
	$(H^2_!(2a+1,2c-2,2d-3)|2b+2)$\\
	\end{tabular}
	\right\}$
		& $q=5$\\ 
$E_\infty^{0,6}=(H^3_!(2a+1,2c-2,2d-3)|2b+2)$
		& $q=6$
\end{tabular}
\right.
\]

(D1') We expect that for the representation $D1'$ the boundary cohomology will vanish. However, what we obtain is the following.
\[H^q_\partial(GL_4(\Z),(2a+2,2b+1,2c,2d-1))
=
\left\{
\begin{tabular}{lll}
$E_\infty^{0,2}=(2a+2|H^2_!(2b+1,2c,2d-1))$
		& $q=2$\\
$E_\infty^{0,3}=\left\{
	\begin{tabular}{ll}
	$(2a+2|H^3_!(2b+1,2c,2d-1))$\\
	$(2b|H^2_!(2a+3,2c,2d-1))$
	\end{tabular}
	\right.$
		& $q=3$\\	
$E_\infty^{0,4}=
	\left\{
	\begin{tabular}{ll}
	$(2b|H^3_!(2a+3,2c,2d-1))$\\
	$(2c-2|H^2_!(2a+3,2b+2,2d-1))$
	\end{tabular}
	\right.$	
		& $q=4$\\
$E_\infty^{0,5}=\left\{
	\begin{tabular}{ll}
	$(2c-2|H^3_!(2a+3,2b+2,2d-1))$\\
	$(2d-4|H^2_!(2a+3,2b+2,2c+1))$
	\end{tabular}
	\right.$
		& $q=5$\\ 
$E_\infty^{0,6}=(2d-4|H^3_!(2a+3,2b+2,2c+1))$
		& $q=6$
\end{tabular}
\right.
\]

\end{thm}

\section{Conjectures on Eisenstein cohomology of $GL_4(\Z)$ }

\begin{conjecture}

(A1)
\[H^q_{Eis}(GL_4(\Z),(2a,2b,2c,2d))
=
\left\{
\begin{tabular}{lll}
$0$ 
		& $q=3$\\
$\left\{
	\begin{tabular}{lll}
	$(2a,2d-2|\overline{2b+1,2c+1})$\\
	$(2c-2|\overline{2a+1,2d-1}|2b+2)$
	\end{tabular}
\right\}$	& $q=4$\\
$0$ 
		& $q=5$\\
$(2c-2,2d-2|2a+2,2b+2)$ 
		& $q=6$
\end{tabular}
\right.
\]

(A1') 
\[H^q_{Eis}(GL_4(\Z),(2a+1,2b+1,2c+1,2d+1))
=
\left\{
\begin{tabular}{lll}
$0$
		& $q=3$\\	
$S_{2a-2d}\otimes S_{2b-2c+2}$
		& $q=4$\\
$\left\{
	\begin{tabular}{lll}
	$S_{(2b,2c,2d)}(GL_3(\Z))$\\
	$S_{(2a+2,2b+2,2c+2)}(GL_3(\Z))$
	\end{tabular}
	\right\}$
		& $q=5$\\ 
$\left\{
	\begin{tabular}{lll}
	$S_{(2b,2c,2d)}(GL_3(\Z))$\\
	$S_{(2a+2,2b+2,2c+2)}(GL_3(\Z))$\\
	$M_{2c-2d+2}\otimes M_{2a-2a+2}$\\
	$S_{2b-2c+2}$
	\end{tabular}
	\right\}$
		& $q=6$
\end{tabular}
\right.
\]

(A2)
\[H^q_{Eis}(GL_4(\Z),(2a,2b,2c,2c))
=
\left\{
\begin{tabular}{lll}
$S_{2a+2c+4}S_{2b-2c+2}+S_{2a-2c+4}$
		& $q=4$\\
$S_{2a-2b+2}$
		& $q=5$\\ 
$0$
		& otherwise
\end{tabular}
\right.
\]

(A2')
\[H^q_{Eis}(GL_4(\Z),(2a+ 1,2b+ 1,2c+ 1,2c+ 1))
=
\left\{
\begin{tabular}{lll}
$S_{2a-2c+4}S_{2b-2c+2}$
		& $q=4$\\
$S_{2a,2b,2c}(SL_3(\Z))$
		& $q=5$\\ 
$S_{2a,2b,2c}(SL_3(\Z))+S_{2b-2c+2}$
		& $q=6$\\
$0$
		& otherwise
\end{tabular}
\right.
\]

(A3)
\[H^q_{Eis}(GL_4(\Z),(2a,2b,2b,2d))=
\left\{
\begin{tabular}{llll}
	$S_{2a-2b+2}S_{2b-2d+2}$ 	& $q=6$\\
	$0$						& otherwise
\end{tabular}
\right.
\]

(A3')
\[H^q_{Eis}(GL_4(\Z),(2a+ 1,2b+ 1,2b+ 1,2d+ 1))
=
\left\{
\begin{tabular}{lll}
$S_{2a-2d+4}$
		& $q=3$\\	
$S_{2a-2b+2}S_{2b-2d+2}
+S_{2b-2d+2}
+S_{2a-2b+2}$
		& $q=6$\\
$0$
		& otherwise\\
		
\end{tabular}
\right.
\]

(A4)
\[H^q_{Eis}(GL_4(\Z),(2a,2a,2c,2d))=
\left\{
\begin{tabular}{llll}
$S_{2a-2d+4}S_{2a_2c+2}+S_{2a-2d+4}$
		& $q=4$\\
$S_{2c-2d+2}$ 
		& $q=5$\\
$0$ 		& otherwise
\end{tabular}
\right.
\]

(A4')
\[H^q_{Eis}(GL_4(\Z),(2a+ 1,2a+ 1,2c+ 1,2d+ 1))
=
\left\{
\begin{tabular}{lll}
$S_{2a-2d+4}S_{2a-2c+2}+S_{2a-2d+4}$
		& $q=4$\\
$S_{2a,2c,2d}(SL_3(\Z))$
		& $q=5$\\
$S_{2a,2c,2d}(SL_3(\Z))+S_{2a-2c+2}$
		& $q=6$\\
$0$
		& otherwise
\end{tabular}
\right.
\]

(A5) (Theorem for symmetric powers)
\[H^q_{Eis}(GL_4(\Z),Sym^{2a-2b})=
\left\{
\begin{tabular}{llll}
	$S_{2a-2b+2}$
 		& $q=5$\\
	$0$ 
		& otherwise
	\end{tabular}
\right.
\]

(A5') (Theorem for a twist of the symmetric powers)
\[H^q_{Eis}(GL_4(\Z),Sym^{2a-2b}\otimes det)=
\left\{
\begin{tabular}{llll}
	$S_{2a-2b+4}+\C$
 		& $q=3$\\
	$0$ 
		& otherwise
	\end{tabular}
\right.
\]

(A6)
\[H^q_{Eis}(GL_4(\Z),(2a,2a,2c,2c))=
\left\{
\begin{tabular}{llll}
$S_{2a-2c+4}\otimes S_{2a_2c+2}+S_{2a-2c+4}$
		& $q=4$\\
$0$ 		& otherwise
\end{tabular}
\right.
\]

(A6')
\[H^q_{Eis}(GL_4(\Z),(2a+ 1,2a+ 1,2c+ 1,2c+ 1))
=
\left\{
\begin{tabular}{lll}
$S_{2a-2d+4}\otimes S_{2a-2c+2}$
		& $q=4$\\
$S_{2a-2c+2}$
		& $q=6$\\
$0$
		& otherwise
\end{tabular}
\right.
\]

(A7) (Theorem for the dual of the symmetric powers)
\[H^q_{Eis}(GL_4(\Z),(Sym^{2a-2b})^*)=
\left\{
\begin{tabular}{llll}
	$S_{2a-2b+2}$
 		& $q=5$\\
	$0$ 
		& otherwise
	\end{tabular}
\right.
\]

(A7') (Theorem for a twist of the dual of symmetric powers)
\[H^q_{Eis}(GL_4(\Z),(Sym^{2a-2b})^*\otimes det)=
\left\{
\begin{tabular}{llll}
	$S_{2a-2b+4}+\C$
 		& $q=3$\\
	$0$ 
		& otherwise
	\end{tabular}
\right.
\]

(A8) (Theorem)
\[H^q_{Eis}(GL_4(\Z),\C)=
\left\{
\begin{tabular}{llll}
	$\C$
 		& $q=0$\\
	$0$ 
		& otherwise
	\end{tabular}
\right.
\]

(A8') (Theorem)
\[H^q_{Eis}(GL_4(\Z),det)=
\left\{
\begin{tabular}{llll}
	$\C$
 		& $q=3$\\
	$0$ 
		& otherwise
	\end{tabular}
\right.
\]

\end{conjecture}

\section{Cohomology of $GL_3(\Z)$}

In this section we compute the cohomology of $GL_3(\Z)$ with coefficients in any finite dimensional highest weight representation.
The results are known from (joint paper with Harder).  However, there are several important reason for including those result here.
First, the method we use here is an improved version which uses certain dualities instead of computing 144 Euler characteristics. 
Second, it introduces essential notation that we will use when we compute the cohomology of $GL_4(\Z)$, $GL_5(\Z)$ and certain congruence subgroups. Last but not least, for convenience of the reader, we include this computations in order to make the book more self-contained.

The weights of $GL_3$ can be written as a triple of integers $\lambda=(a,b,c)$ such that $a\geq b \geq c$. Let us denote by $V_\lambda$ the irreducible highest weight representation with weight $\lambda$. If minus the identity matrix, $-I$, acts nontrivially on $V_\lambda$ then $H^q(GL_3(\Z),V_\lambda)=0$.

Note that $-I$ acts on $V_\lambda$ by multiplication by $(-1)^{a+b+c}$. Therefore, $H^q(GL_3(\Z),V_\lambda)\neq 0$ for come $q$ only if $a+b+c$ is even. 
All possible options for such a weight $\lambda$ occur when $(a,b,c)$ is $(even,even,even)$, or $(even,odd,odd)$, or $(odd,odd,even)$, or $(odd,even,odd)$.
Each of those cases may branch into several cases in the following way: If we have a regular representation, that is $a>b>c$, then we are in one of the sub-cases. The other sub-cases occur for irregular representations, when $a=b>c$, or $a>b=c$, or $a=b=c$. 
The list of cases and sub-cases is the following

Let $a$, $b$ and $c$ be integers such that $a>b>c$. The list of all cases and sub-cases is the following

\begin{enumerate}
\item[(A)] $(even,even,even)$

\begin{enumerate}
\item[(1)] $(2a,2b,2c)$

\item[(2)] $(2a,2b,2b)$, when $b=c$; then $V_\lambda=Sym^{2a-2b}V$

\item[(3)] $(2a,2a,2c)$, when $a=b$; then $V_\lambda=Sym^{2a-2c}V^*$

\item[(4)] $(2a,2a,2a)$, when $a=b=c$; then $V_\lambda=\C$

\end{enumerate}

\item[(B)] $(even,odd,odd)$

\begin{enumerate}
\item[(1)]  $(2a,2b-1,2c-1)$

\item[(2)]  $(2a,2b-1,2b-1)$, when $b=c$; then  $V_\lambda=Sym^{2a-2b+1}V\otimes det$
\end{enumerate}

\item[(C)] $(odd,odd,even)$

\begin{enumerate}
\item[(1)]  $(2a+1,2b+1,2c)$

\item[(2)]  $(2a+1,2a+1,2c)$ when $a=b$; then $V_\lambda=Sym^{2a-2c+1}V^*\otimes det$
\end{enumerate}

\item[(D)] $(odd,even,odd)$

\begin{enumerate}
\item[(1)]  $(2a+1,2b,2c-1)$
\end{enumerate}

\end{enumerate}

The reason for the above order fo types and cases is the following. Part of the computation for type $I$, cases $B$ and $C$ follows from the computation of type $I$, case $A$. Similarly, part of the computation of type $I$, case $D$ follows from the computation of type $I$, cases $B$ and $IC$.

For case $II$,  part of the computation for type $II$, case $B$  follows from the computation of type $II$, case $A$.
And, For case $III$,  part of the computation for type $III$, case $B$  follows from the computation of type $III$, case $A$.

In comparison to the cases $1$ through $9$ from (joint paper with Harder), we have the following correspondence.

A1 - case 4,

A2 -  case 2,

A3 - case 3,

A4 - case 1,

B1 - case 6,

B2 - case 5,

C1 - case 8,

C2 - case 7,

D1  - case 9.




\subsection{Cohomology of $GL_3(\Z)$. Highest weight $(2a,2b,2c)$.  Case $A1$.}
Let $\lambda=(2a,2b,2c)$.
For each element $w$ of the Weyl group, let we have to compute the weight $w(\lambda+\rho)-\rho$ of a Levi quotient, where $w$ acts by permutation and $\rho$ is the half of the sum of the positive roots. Explicitly, 
\[\rho=\frac{1}{2}[(1,-1,0)+(0,1,-1)]=\left(\frac{1}{2},0,-\frac{1}{2}\right).\] 

We separate $w(\lambda+\rho)-\rho$ into to parts: $w(\lambda)$ and $w(\rho)-\rho$.

$\begin{tabular}{lllll}
$w$ 		& $l$	 	& $w(\rho)-\rho$\\
$123$ 	& $0$	& $(0,0,0)$\\
$132$ 	& $1$	& $(0,-1,1)$\\
$213$ 	& $1$	& $(-1,1,0)$\\
$231$ 	& $2$	& $(-1,-1,2)$\\
$312$ 	& $2$	& $(-2,1,1)$\\
$321$ 	& $3$	& $(-2,0,2)$\\
\end{tabular}
$

Note that for each $w$, we have that $w(\lambda)$ is again of the form $(even,even,even)$, since it is just a permutation of $(2a,2b,2c)$. Note also that $w(\rho)-\rho$ has only even entries when $w=(123)$ and $w=(321)$. In all other cases there is an odd entry. If we have an odd entry say $(2k+1,2l+1,2m)$ Then $H^0(M_0,(2k+1,2l+1,2m))$ for $M_0={\mathbb G}_m\times{\mathbb G}_m\times{\mathbb G}_m$  be separated component-wise 
$H^0({\mathbb G}_m(\Z),sign^{2k+1})\otimes H^0({\mathbb G}_m(\Z),sign^{2l+1})\otimes H^0({\mathbb G}_m(\Z),sign^{2m})$.
If the exponent is even then $sign^{2m}=\C$ is the trivial representation and $H^0({\mathbb G}_m(\Z),sign^{2m})=H^0({\mathbb G}_m(\Z),\C)=\C$. If the exponent is odd, then $sign^{2k+1}=sign$ and  $H^0({\mathbb G}_m(\Z),sign^{2k+1})=H^0({\mathbb G}_m(\Z),sign)=0$.

Therefore, the only cohomology of the minimal parabolic subgroup $P_0$ with coefficients in $V_\lambda$ are 
\[H^0(P_0,V_\lambda)=H^0(M_0,V_\lambda)=H^0(M_0,(2a,2b,2c)),\]
and
\[H^3(P_0,V_\lambda)=H^0(M_0,V_{w(\lambda+\rho)-\rho})=H^0(M_0,(2c-2,2b,2a-2)).\]
We have $H^3$ because the element $w=(321)$ from the Weyl group has length $3$.

From now on very often we are going to use the notation $(2c-2|2b|2a-2)$, which means $H^0(M_0,(2c-2,2b,2a-2))$, where 
$M_0={\mathbb G}_m\times{\mathbb G}_m\times{\mathbb G}_m$.

$Q_0$: 
$\begin{tabular}{lllll}
$w$ 		& $l$	& 	& $w(\lambda+\rho)-\rho$\\
$123$ 	& $0$&	& $(2a|2b|2c)$\\
$321$ 	& $3$&	& $(2c-2|2b|2a+2)$\\
\end{tabular}
$

For the maximal parabolic subgroup $Q_{12}$, we have that the corresponding Levi quotient is 
$M_{12}=GL_2\times {\mathbb G}_m$. The weight for ${\mathbb G}_m$ has to be even or else the cohomology will be zero. 
Any element of the Weyl group acts by permutation. Since $\lambda=(even,even,even)$, we have that any permutation of it would have the same form $(even,even,even)$. Thus, in Case I, the parity of $w(\lambda+\rho)-\rho$ is the same as the parity $w(\rho)-\rho$.

In order that the last entry of $w(\rho)-\rho$ be even, which would be the weight of ${\mathbb G}_m$, we must have a permutation $w$ with last entry $1$ or $3$.
The elements of minimal length from the quotient $W_{12}\backslash W$ with last elements $1$ or $3$ are
 $(123)$   and $(231)$ of lengths $0$ and $2$, respectively. 
Note that the first two entries of $w$ should be in increasing order so that when the act on decreasing elements of the weight the result would be a decreasing elements which is the possible weight for $GL_2$.

For $w=(123)$, we have $w(\lambda+\rho)-\rho=\lambda=(2a,2b,2c)$.
Then we have $H^1(Q_{12},(2a,2b,2c))=H^1(M_{12},(2a,2b,2c)=H^1(GL_2(\Z),(2a,2b))\otimes H^0({\mathbb G}_m(\Z),(2c))$
$H^1(GL_2(\Z),(2a,2b))=S_{2a-2b}(SL_2(\Z))$ is the space of cusp forms of the modular group $SL_2(\Z)$.
The notation for $H^1(GL_2(\Z),(2a,2b))\otimes H^0({\mathbb G}_m(\Z),(2c))$ that we are going to use is $(2a,2b|2c)$.

For $w=(231)$, we have $w(\lambda)=(2b,2c,2a)$ and $w(\rho)-\rho=(-1,-1,2)$. Thus, $w(\lambda+\rho)-\rho=(2b-1,2c-1,2a+2)$.
Then, 
$H^3(Q_{12},(2a,2b,2c))=H^{1+length(w)}(Q_{12},(2a,2b,2c))=H^1(GL_2(\Z),(2b-1,2c-1))\otimes H^0({\mathbb G}_m(\Z),(2a+2))$.
The notation for the last expression will be $(2b-1,2c-1|2a+2)$.

$Q_{12}$: 
$\begin{tabular}{lllll}
$w$ 		& $l$	& 	& $w(\lambda+\rho)-\rho$\\
$123$ 	& $0$&	& $(2a,2b|2c)$\\
$231$ 	& $2$&	& $(2b-1,2c-1|2a+2)$\\
\end{tabular}
$

For the maximal parabolic subgroup $Q_{23}$, we have that the corresponding Levi quotient is 
$M_{12}= {\mathbb G}_m\times GL_2$. The weight for ${\mathbb G}_m$ has to be even or else the cohomology will be zero. 
\
Any element of the Weyl group acts by permutation. Since $\lambda=(even,even,even)$, we have that any permutation of it would have the same form $(even,even,even)$. Thus, in Case I, the parity of $w(\lambda+\rho)-\rho$ is the same as the parity $w(\rho)-\rho$.

In order that the first entry of $w(\rho)-\rho$ be even, which would be the weight of ${\mathbb G}_m$, we must have a permutation $w$ with first entry $1$ or $3$.
The elements of minimal length from the quotient $W_{23}\backslash W$ with last elements $1$ or $3$ are
 $(123)$   and $(312)$ of lengths $0$ and $2$, respectively. 
Note that the last two entries of $w$ should be in increasing order so that when the act on decreasing elements of the weight the result would be a decreasing elements which is the possible weight for $GL_2$.

For $w=(123)$, we have $w(\lambda+\rho)-\rho=\lambda=(2a,2b,2c)$.
Then we have $H^1(Q_0,(2a,2b,2c))=H^1(M_0,(2a,2b,2c)=H^0({\mathbb G}_m(\Z),(2a))\otimes H^1(GL_2(\Z),(2b,2c))$
$H^1(GL_2(\Z),(2b,2c))=S_{2b-2c}(SL_2(\Z))$ is the space of cusp forms of the modular group $SL_2(\Z)$.
The notation for $H^0({\mathbb G}_m(\Z),(2a)) \times H^1(GL_2(\Z),(2b,2c))$ that we are going to use is $(2a|2b,2c)$.

For $w=(312)$, we have $w(\lambda)=(2c,2a,2b)$ and $w(\rho)-\rho=(-2,1,1)$. Thus, $w(\lambda+\rho)-\rho=(2c-2,2a+1,2b+1)$.
Then, 
$H^3(Q_{23},(2a,2b,2c))=H^{1+length(w)}(Q_{23},(2a,2b,2c))= H^0({\mathbb G}_m(\Z),(2c-2))\times H^1(GL_2(\Z),(2a+1,2b+1))$.
The notation for the last expression will be $(2c-2|2a+1,2b+1)$.

$Q_{23}$: 
$\begin{tabular}{lllll}
$w$ 		& $l$	& 	& $w(\lambda+\rho)-\rho$\\
$123$ 	& $0$&	& $(2a|2b,2c)$\\
$312$ 	& $2$&	& $(2c-2|2a+1,2b+1)$\\
\end{tabular}
$

Now, we compute the boundary cohomology $H^q_\partial$ via Mayer-Vietoris-type of exact sequence.
We have
$H^1(Q_0,(2a,2b,2c))=0$ and $H^2(Q_{12},(2a,2b,2c))=H^2(Q_{23},(2a,2b,2c))=0$. Therefore, $H^2_\partial(2a,2b,2c)=0$
For $H^1_\partial$, we have
$H^0(Q_0,(2a,2b,2c))\rightarrow H^1_\partial \rightarrow (2a,2b|2c)+(2a|2b,2c)\rightarrow 0$.
Therefore, $H^1_\partial=(2a|2b|2c)+ (2a,2b|2c)+(2a|2b,2c)$.
And for $H^3_\partial$, we have
$0\rightarrow H^3_\partial \rightarrow (2b-1,2c-1|2a+2)+(2c-2|2a+1,2b+1)\rightarrow (2c-2|2b|2a+2)\rightarrow 0$

Therefore, $H^3_\partial = (\overline{2b-1,2c-1}|2a+2)+(2c-2|\overline{2a+1,2b+1}) + (2c-2|2b|2a+2)$.
Note that $H^3_\partial$ and $H^1_\partial$ are dual. Since $H^1(GL_3(\Z),V)=0$, we have that $H^1_{Eis}(GL_3(\Z),V)=0$
Therefore,
\[H^3_{Eis}(GL_3(\Z),(2a,2b,2c))=(\overline{2b-1,2c-1}|2a+2)+(2c-2|\overline{2a+1,2b+1}) + (2c-2|2b|2a+2).\]
Also, $H^q_{Eis}(GL_3(\Z),(2a,2b,2c))=0$ for $q\neq 3$, since the corresponding boundary cohomology vanishes.

\begin{thm}
\[H^3_{Eis}(GL_3(\Z),(2a,2b,2c))=(\overline{2b-1,2c-1}|2a+2)+(2c-2|\overline{2a+1,2b+1}) + (2c-2|2b|2a+2).\]
$H^q_{Eis}(GL_3(\Z),(2a,2b,2c))=0$ for $q\neq 3$.
\end{thm}

\begin{thm}
\[H^3_{Eis}(GL_3(\Z),(2a,2b,2c))=S_{2b-2c+2}+S_{2a-2b+2}+\C.\]
$H^q_{Eis}(GL_3(\Z),(2a,2b,2c))=0$ for $q\neq 3$,
where $S_{n}$ is the space of holomorphic cusp forms of weight $k$ for the full modular group $SL_2(\Z)$.
\end{thm}




\subsection{Cohomology of $GL_3(\Z)$. Highest weight $(2a,2b,2b)$. Case $A2$.}
Let $\lambda=(2a,2b,2b)$.

$Q_0$: 
$\begin{tabular}{lllll}
$w$ 		& $l$	& 	& $w(\lambda+\rho)-\rho$\\
$123$ 	& $0$&	& $(2a|2b|2b)$\\
$321$ 	& $3$&	& $(2b-2|2b|2a+2)$\\
\end{tabular}
$

$Q_{12}$: 
$\begin{tabular}{lllll}
$w$ 		& $l$	& 	& $w(\lambda+\rho)-\rho$\\
$123$ 	& $0$&	& $(2a,2b|2b)$\\
$231$ 	& $2$&	& $0$\\
\end{tabular}
$

$Q_{23}$: 
$\begin{tabular}{lllll}
$w$ 		& $l$	& 	& $w(\lambda+\rho)-\rho$\\
$123$ 	& $0$&	& $(2a|2b,2b)$\\
$312$ 	& $2$&	& $(2b-2|2a+1,2b+1)$\\
\end{tabular}
$

The last entry for $Q_{12}$ vanishes because for $(2b-1,2b-1)=det$, we have $H^1(GL_2(\Z),det)=0$.

We have that for $H^2_\partial$ we have.

$H^1(Q_0,(2a,2b,2b))=0$ and $H^2(Q_{12},(2a,2b,2b))=H^2(Q_{23},(2a,2b,2b))=0$. Therefore, $H^2_\partial(2a,2b,2b)=0$
For $H^1_\partial$, we have
$H^0(Q_{23},(2a,2b,2b))\rightarrow H^0(Q_0,(2a,2b,2b))\rightarrow H^1_\partial \rightarrow(2a,2b|2b)\rightarrow 0$.
Therefore, $H^1_\partial= (2a,2b|2b)$.
And for $H^3_\partial$, we have
$0\rightarrow H^3_\partial \rightarrow (2b-2|2a+1,2b+1)\rightarrow (2b-2|2b|2a+2)\rightarrow 0$

Therefore, $H^3_\partial = (2b-2|\overline{2a+1,2b+1})$.
Note that $H^3_\partial$ and $H^1_\partial$ are dual. Since $H^1(GL_3(\Z),V)=0$, we have that $H^1_{Eis}(GL_3(\Z),V)=0$
Therefore,
\[H^3_{Eis}(GL_3(\Z),(2a,2b,2b))=(2b-2|\overline{2a+1,2b+1}).\]
Also, $H^q_{Eis}(GL_3(\Z),(2a,2b,2b))=0$ for $q\neq 3$, since the corresponding boundary cohomology vanishes.

\begin{thm}
\[H^3_{Eis}(GL_3(\Z),(2a,2b,2b))=(2b-2|\overline{2a+1,2b+1}).\]
Also, $H^q_{Eis}(GL_3(\Z),(2a,2b,2b))=0$ for $q\neq 3$.
\end{thm}

\begin{thm}
\[H^3_{Eis}(GL_3(\Z),(2a,2b,2b))=S_{2a-2b+2}.\]
Also, $H^q_{Eis}(GL_3(\Z),(2a,2b,2b))=0$ for $q\neq 3$.
\end{thm}



\subsection{Cohomology of $GL_3(\Z)$.  Highest weight $(2a,2a,2c)$ Case $A3$.}
Let $\lambda=(2a,2a,2c)$.
From the previous subsubsection, we have that

$Q_0$: 
$\begin{tabular}{lllll}
$w$ 		& $l$	& 	& $w(\lambda+\rho)-\rho$\\
$123$ 	& $0$&	& $(2a|2a|2c)$\\
$321$ 	& $3$&	& $(2c-2|2a|2a+2)$\\
\end{tabular}
$

$Q_{12}$: 
$\begin{tabular}{lllll}
$w$ 		& $l$	& 	& $w(\lambda+\rho)-\rho$\\
$123$ 	& $0$&	& $(2a,2a|2c)$\\
$231$ 	& $2$&	& $(2a-1,2c-1|2a+2)$\\
\end{tabular}
$

$Q_{23}$: 
$\begin{tabular}{lllll}
$w$ 		& $l$	& 	& $w(\lambda+\rho)-\rho$\\
$123$ 	& $0$&	& $(2a|2a,2c)$\\
$312$ 	& $2$&	& $0$\\
\end{tabular}
$

The last entry is zero because $(2c-2|2a+1,2a+1)=0$.

We have that for $H^2_\partial$ we have.

$H^1(Q_0,(2a,2a,2c))=0$ and $H^2(Q_{12},(2a,2a,2c))=H^2(Q_{23},(2a,2a,2c))=0$. Therefore, $H^2_\partial(2a,2a,2c)=0$
For $H^1_\partial$, we have
$H^0(Q_{12},(2a,2a,2c)) \rightarrow H^0(Q_0,(2a,2a,2c))\rightarrow H^1_\partial \rightarrow (2a|2a,2c)\rightarrow 0$.
Therefore, $H^1_\partial=(2a|2a,2c)$.
And for $H^3_\partial$, we have
$0\rightarrow H^3_\partial \rightarrow (2a-1,2c-1|2a+2) \rightarrow (2c-2|2a|2a+2)\rightarrow 0$

Therefore, $H^3_\partial = (\overline{2a-1,2c-1}|2a+2)$.
Note that $H^3_\partial$ and $H^1_\partial$ are dual. Since $H^1(GL_3(\Z),V)=0$, we have that $H^1_{Eis}(GL_3(\Z),V)=0$
Therefore,
\[H^3_{Eis}(GL_3(\Z),(2a,2a,2c))=(\overline{2a-1,2c-1}|2a+2).\]
Also, $H^q_{Eis}(GL_3(\Z),(2a,2a,2c))=0$ for $q\neq 3$, since the corresponding boundary cohomology vanishes.

\begin{thm}
\[H^3_{Eis}(GL_3(\Z),(2a,2a,2c))=(\overline{2a-1,2c-1}|2a+2).\]
Also, $H^q_{Eis}(GL_3(\Z),(2a,2a,2c))=0$ for $q\neq 3$.
\end{thm} 

\begin{thm}
\[H^3_{Eis}(GL_3(\Z),(2a,2a,2c))=S_{2a-2c+2}.\]
Also, $H^q_{Eis}(GL_3(\Z),(2a,2a,2c))=0$ for $q\neq 3$.
\end{thm} 



\subsection{Cohomology of $GL_3(\Z)$. Highest weight $(2a,2a,2a)$.  Case $A4$. }
Let $\lambda=(2a,2a,2a)$.
From the previous subsubsection, we have that

$Q_0$: 
$\begin{tabular}{lllll}
$w$ 		& $l$	& 	& $w(\lambda+\rho)-\rho$\\
$123$ 	& $0$&	& $(2a|2a|2a)$\\
$321$ 	& $3$&	& $(2a-2|2a|2a+2)$\\
\end{tabular}
$

$Q_{12}$: 
$\begin{tabular}{lllll}
$w$ 		& $l$	& 	& $w(\lambda+\rho)-\rho$\\
$123$ 	& $0$&	& $(2a,2a|2a)$\\
$231$ 	& $2$&	& $0$\\
\end{tabular}
$

$Q_{23}$: 
$\begin{tabular}{lllll}
$w$ 		& $l$	& 	& $w(\lambda+\rho)-\rho$\\
$123$ 	& $0$&	& $(2a|2a,2a)$\\
$312$ 	& $2$&	& $0$\\
\end{tabular}
$

We have the following exact sequence
$0
\rightarrow H^0_\partial(2a,2a,2c)
\rightarrow H^0(Q_{12},(2a,2a,2c))
+H^0(Q_{23},(2a,2a,2c)) 
\rightarrow H^0(Q_0,(2a,2a,2c)) 
\rightarrow  0$

We have that for the $H^q$ of any of the parabolic subgroups vanishes for $q=1,2$. For $q=3$, the only one that does. not vanish is
$H^3(Q_0,(2a,2a,2a))$. Therefore, $H^4_\partial=(2a-2|2a|2a+2)$

Therefore, $H^q_\partial(2a,2a,2a)=0$ for $q=1,2,3$

Note that $H^0_\partial$ and $H^4_\partial$ are dual. Since $H^4(GL_3(\Z),V)=0$, we have that $H^4_{Eis}(GL_3(\Z),V)=0$
Therefore,
\begin{thm}
\[H^0_{Eis}(GL_3(\Z),(2a,2a,2a))=(2a|23a|2a)=\C.\]
and $H^q_{Eis}(GL_3(\Z),(2a,2a,2a))=0$ for $q\neq 0$.
\end{thm}





\subsection{Cohomology of $GL_3(\Z)$. Highest weight $(2a,2b-1,2c-1)$.  Case $B1$}
Let $\lambda=(2a,2b-1,2c-1)$.

We are going to find which of the Weyl elements produce a nontrivial cohomology of the minimal parabolic subgroup. For such elements $2a$ should be at the first or at the third as $2a+2$. Note that at the second place it would appear as $2a+1$, which is an odd number. Therefore, the corresponding cohomology class will vanish.
Similarly, the second element $2b-1$ could be at the first place as $2b-2$ or at the third place as $2b$. We have two options

$Q_0$: 
$\begin{tabular}{lllll}
$w$ 		& $l$	& 	& $w(\lambda+\rho)-\rho$\\
$132$ 	& $1$&	& $(2a|2c-2|2b)$\\
$231$ 	& $2$&	& $(2b-2|2c-2|2a+2)$\\
\end{tabular}
$

$Q_{12}$: 
$\begin{tabular}{lllll}
$w$ 		& $l$	& 	& $w(\lambda+\rho)-\rho$\\
$132$ 	& $1$&	& $(2a,2c-2|2b)$\\
$231$ 	& $2$&	& $(2b-2,2c-2|2a+2)$\\
\end{tabular}
$

$Q_{23}$: 
$\begin{tabular}{lllll}
$w$ 		& $l$	& 	& $w(\lambda+\rho)-\rho$\\
$123$ 	& $0$&	& $(2a|2b-1,2c-1)$\\
$213$ 	& $1$&	& $(2b-2|2a+1,2c-1)$\\
\end{tabular}
$

Then, we have the wfollowing exact sequence
$0\rightarrow H^1_\partial \rightarrow  H^1(Q_{23})\rightarrow H^1(Q_0) \rightarrow 0$. Explicitly, it can be expressed as
$0\rightarrow H^1_\partial \rightarrow  (2a|2b-1,2c-1)\rightarrow (2a|2c-2|2b) \rightarrow 0$.
Therefore $H^1_\partial = (2a|\overline{2b-1,2c-1})$.

For $H^2_\partial$, we have the following
$0\rightarrow H^2_\partial \rightarrow  H^2(Q_{12})+H^2(Q_{23})\rightarrow H^2(Q_0) \rightarrow 0$. 
Explicitly, it can be expressed as
$0\rightarrow H^2_\partial \rightarrow (2a,2c-2|2b)+ (2b-2|2a+1,2c-1) \rightarrow (2b-2|2c-2|2a+2) \rightarrow 0$.
Therefore, 
$H^2_\partial = (2a,2c-2|2b)+ (2b-2|\overline{2a+1,2c-1})$. 

For $H^3_\partial$, we have the following
$0\rightarrow H^3_\partial \rightarrow  H^3(Q_{12})\rightarrow  0$. 
Explicitly, 
$H^3_\partial=  H^3(Q_{12}) = (2b-2,2c-2|2a+2)$

Using dualities, we obtain the following.
\begin{thm}
\[H^q_{Eis}(GL_3(\Z),(2a,2b-1,2c-1))
=
\left\{
\begin{tabular}{lllll}
$\Delta$	& $q=2$\\
$(2b-2,2c-2|2a+2)$ 								& $q=3$,\\
\end{tabular}
\right.
\]
where $\Delta\subset (2a,2c-2|2b)+ (2b-2|\overline{2a+1,2c-1})$ is a diagonal embedding. 
\end{thm}
The ``slope" is can be expresses as a quotient of special values of $L$-functions.

\begin{thm}
\[H^q_{Eis}(GL_3(\Z),(2a,2b-1,2c-1))
=
\left\{
\begin{tabular}{lllll}
$S_{2a-2c+4}$		& $q=2$\\
$S_{2b-2c+2}$ 		& $q=3$.\\
\end{tabular}
\right.
\]
\end{thm}




\subsection{Cohomology of $GL_3(\Z)$. Highest weight $((2a,2b-1,2b-1)$.  Case $B2$}
Let $\lambda=(2a,2b-1,2b-1)$.

Claim: $V_\lambda$ is a twist of an odd symmetric power.

Note that $V_3=(1,0,0)$ and $Sym^{2k+1}V_3=(2k+1,0,0)$. Also,  $det=(1,1,1)$ and $det(-1,-1,-1)$. 
Then
$Sym^{2k+1}V_3\otimes det^{-1}=(2k,-1,-1)$.
Therefore, $(2a,2b-1,2b-1)=(2a-2b,-1,-1)=Sym^{2a-2b+1}V_3\otimes det^{-1}$.
For the parabolic subgroups $Q_0$, $Q_{12}$ and $Q_{23}$, we have

$Q_0$: 
$\begin{tabular}{lllll}
$w$ 		& $l$	& 	& $w(\lambda+\rho)-\rho$\\
$132$ 	& $1$&	& $(2a|2b-2|2b)$\\
$231$ 	& $2$&	& $(2b-2|2b-2|2a+2)$\\
\end{tabular}
$

$Q_{12}$: 
$\begin{tabular}{lllll}
$w$ 		& $l$	& 	& $w(\lambda+\rho)-\rho$\\
$132$ 	& $1$&	& $(2a,2b-2|2b)$\\
$231$ 	& $2$&	& $(2b-2,2b-2|2a+2)$\\
\end{tabular}
$

$Q_{23}$: 
$\begin{tabular}{lllll}
$w$ 		& $l$	& 	& $w(\lambda+\rho)-\rho$\\
$123$ 	& $0$&	& $0$\\
$213$ 	& $1$&	& $(2b-2|2a+1,2b-1)$\\
\end{tabular}
$

The zero occurs for $Q_{23}$ because the cohomology of ${\mathbb G}_m\times GL_2$ with coefficients $(2a|2b-1,2b-1)$ vanishes. Note that $H^1(GL_2(\Z),(2b-1,2b-1))=H^1(GL_2(\Z),det)=0$.

Then, we have thew following $H^1_\partial = H^1(Q_{23}) =  0$.

For $H^2_\partial$, we have the following
$0\rightarrow H^1(Q_0) \rightarrow H^2_\partial \rightarrow  H^2(Q_{12})+H^2(Q_{23})\rightarrow H^2(Q_0) \rightarrow 0$. 
Note that 
$H^1(Q_0)=(2a|2b-2|2b)$, 
$H^2(Q_{12})=(2a,2b-2|2b)+(2b-2,2b-2|2a+2)$ and $H^2(Q_{23})=(2b-2|2a+1,2b-1)$
 and
$H^2(Q_0)=(2b-2|2b-2|2a+2)$.

Therefore,
$H^2_\partial(GL_3(\Z),(2a,2b-1,2b-1))=(2a|2b-2|2b)+ 
(2a,2b-2|2b)+(2b-2,2b-2|2a+2)+(2b-2|\overline{2a+1,2b-1})$ and
$H^q_\partial=0$ for $q\neq 2$.

We obtain the following
\begin{thm}
Then \[H^2_{Eis}(GL_3(\Z),(2a,2b-1,2b-1))=\Delta_0+\Delta_1,\]
where $\Delta_0$ is a diagonal in $(2a|2b-2|2b) + (2b-2|2b-2|2a+2)$,
and 
$\Delta_1$ is a diagonal in $(2a,2b-2|2b) + (2b-2|\overline{2a+1,2b-1})$.
And
\[H^q_{Eis}=0\mbox{ for }q\neq 2.\]
\end{thm}

\begin{thm}
Then \[H^2_{Eis}(GL_3(\Z),(2a,2b-1,2b-1))=S_{2a-2b+4}+\C\]
And
\[H^q_{Eis}=0\mbox{ for }q\neq 2.\]
\end{thm}





\subsection{Cohomology of $GL_3(\Z)$. Highest weight $(2a+1,2b+1,2c)$.  Case $C1$}
Let $\lambda=(2a+1,2b+1,2c)$.

First we have to find the elements of the Weyl group which give a nontrivial cohomology of the minimal parabolic subgroup via the Kostant formula.

Note that the third element $2c$ can be either in third place as $2c$ or in first place as $2c-2$. Note that in second place it will appear as $2c-1$ which is odd.

The first element should appear only in the second place as $2a+2$. If it appears in first or third places, the corresponding numbers would be $2a+1$ or $2a+3$, respectively, which are odd.
We have two options

$Q_0$: 
$\begin{tabular}{lllll}
$w$ 		& $l$	& 	& $w(\lambda+\rho)-\rho$\\
$213$ 	& $1$&	& $(2b|2a+2|2c)$\\
$312$ 	& $2$&	& $(2c-2|2a+2|2b+2)$\\
\end{tabular}
$

To find the elements of the Weyl group, we do the following. First, we find what could be stay at the third place. It could be $2$ or $3$, but not $1$. The first and second entry should be in increasing order. Thus, if the last element is $2$ then the first two are $13$, that is, $w=(132)$. If the last element is $3$ then the first two are $12$, that is, $w=(123)$.

$Q_{12}$: 
$\begin{tabular}{lllll}
$w$ 		& $l$	& 	& $w(\lambda+\rho)-\rho$\\
$123$ 	& $0$&	& $(2a+1,2b+1|2c)$\\
$132$ 	& $1$&	& $(2a+1,2c-1|2b+2)$\\
\end{tabular}
$

First, we find what could be stay at the first place place. It could be $2$ or $3$, but not $1$. The second and the third entry should be in increasing order. Thus, if the first element is $2$ then the last two are $13$, that is, $w=(213)$. If the first element is $3$ then the last two are $12$, that is, $w=(312)$.

$Q_{23}$: 
$\begin{tabular}{lllll}
$w$ 		& $l$	& 	& $w(\lambda+\rho)-\rho$\\
$213$ 	& $1$&	& $(2b|2a+2,2c)$\\
$312$ 	& $2$&	& $(2c-2|2a+2,2b+2)$\\
\end{tabular}
$

Then, we have the wfollowing exact sequence
$0\rightarrow H^1_\partial \rightarrow  H^1(Q_{12})\rightarrow H^1(Q_0) \rightarrow 0$. Explicitly, it can be expressed as
$0\rightarrow H^1_\partial \rightarrow  (2a+1,2b+1|2c)\rightarrow (2b,2a+2|2c) \rightarrow 0$.
Therefore $H^1_\partial = (\overline{2a+1,2b+1}|2c)$.

For $H^2_\partial$, we have the following
$0\rightarrow H^2_\partial \rightarrow  H^2(Q_{12})+H^2(Q_{23})\rightarrow H^2(Q_0) \rightarrow 0$. 
Explicitly, it can be expressed as
$0\rightarrow H^2_\partial \rightarrow (2a+1,2c-1|2b+2)+(2b|2a+2,2c)  \rightarrow (2c-2|2a+2|2b+2) \rightarrow 0$.
Therefore, 
$H^2_\partial = (\overline{2a+1,2c-1}|2b+2)+(2b|2a+2,2c)$. 

For $H^3_\partial$, we have the following
$0\rightarrow H^3_\partial \rightarrow  H^3(Q_{23})\rightarrow  0$. 
Explicitly, 
$H^3_\partial=  H^3(Q_{23}) = (2c-2|2a+2,2b+2)$

Using dualities, we obtain the following.
\begin{thm}
\[H^q_{Eis}(GL_3(\Z),(2a+1,2b+1,2c))
=
\left\{
\begin{tabular}{lllll}
$\Delta$	& $q=2$\\
$ (2c-2|2a+2,2b+2)$ 								& $q=3$,\\
\end{tabular}
\right.
\]
where $\Delta\subset(\overline{2a+1,2c-1}|2b+2)+(2b|2a+2,2c)$ is a diagonal embedding. 
The ``slope" is can be expresses as a quotient of special values of $L$-functions.
\end{thm}

\begin{thm}
\[H^q_{Eis}(GL_3(\Z),(2a+1,2b+1,2c))
=
\left\{
\begin{tabular}{lllll}
$S_{2a-2c+4}$	& $q=2$\\
$S_{2a-2b+2}$ 								& $q=3$,\\
\end{tabular}
\right.
\]
\end{thm}



\subsection{Cohomology of $GL_3(\Z)$. Highest weight $(2a+1,2a+1,2c)$.  Case $C2$}
Let $\lambda=(2a+1,2a+1,2c)$.

Claim: $V_\lambda$ is a twist of the dual of the odd symmetric power.

Note that $V_3=(1,0,0)$, $Sym^{2k+1}V_3=(2k+1,0,0)$, $Sym^{2k+1}V^*_3=(0,0,-2k-1)$ and
$Sym^{2k+1}V_3\otimes det^{2k+1}=(2k+1,2k+1,0)$ and $Sym^{2k+1}V_3\otimes det=(2k+1,2k+1,0)$
Therefore, $(2a+1,2a+1,2c)=(2a-2c+1,2a-2c+1,0)=Sym^{2a-2c+1}V^*_3\otimes det$.

From Case II, part A, we have

$Q_0$: 
$\begin{tabular}{lllll}
$w$ 		& $l$	& 	& $w(\lambda+\rho)-\rho$\\
$213$ 	& $1$&	& $(2a|2a+2|2c)$\\
$312$ 	& $2$&	& $(2c-2|2a+2|2a+2)$\\
\end{tabular}
$

$Q_{12}$: 
$\begin{tabular}{lllll}
$w$ 		& $l$	& 	& $w(\lambda+\rho)-\rho$\\
$123$ 	& $0$&	& $0$\\
$132$ 	& $1$&	& $(2a+1,2c-1|2a+2)$\\
\end{tabular}
$

$Q_{23}$: 
$\begin{tabular}{lllll}
$w$ 		& $l$	& 	& $w(\lambda+\rho)-\rho$\\
$213$ 	& $1$&	& $(2a|2a+2,2c)$\\
$312$ 	& $2$&	& $(2c-2|2a+2,2a+2)$\\
\end{tabular}
$

The zero occurs for $Q_{12}$ because the cohomology of $ GL_2\times{\mathbb G}_m$ with coefficients $(2a+1,2a+1|2c)$ vanishes. Note that $H^1(GL_2(\Z),(2a+1,2a+1))=H^1(GL_2(\Z),det)=0$.

Then, we have thew following $H^1_\partial = H^1(Q_{23}) =  0$.

For $H^2_\partial$, we have the following
$0\rightarrow H^1(Q_0) \rightarrow H^2_\partial \rightarrow  H^2(Q_{12})+H^2(Q_{23})\rightarrow H^2(Q_0) \rightarrow 0$. 
Note that 
$H^1(Q_0)=(2a|2a+2|2c)$, 
$H^2(Q_{12})=(2a+1,2c-1|2a+2)$, 
$H^2(Q_{23})=(2a|2a+2,2c)+(2c-2|2a+2,2a+2)$
 and
$H^2(Q_0)=(2c-2|2a+2|2a+2)$.

Therefore,
$H^2_\partial(GL_3(\Z),(2a+1,2b+1,2c))=(2a|2a+2|2c)+ 
(\overline{2a+1,2c-1}|2a+2)+(2a|2a+2,2c)+(2c-2|2a+2,2a+2)$
and
$H^q_\partial=0$ for $q\neq 2$.

We obtain:
\begin{thm} \[H^2_{Eis}(GL_3(\Z),(2a,2b-1,2b-1))=\Delta_0+\Delta_1.\]
where $\Delta_0$ is a diagonal in $(2a|2a+2|2c)+ (2c-2|2a+2,2a+2)$
and 
$\Delta_1$ is a diagonal in 
$(\overline{2a+1,2c-1}|2a+2)+(2a|2a+2,2c)$
and
$H^q_\partial=0$ for $q\neq 2$.
\end{thm}

\begin{thm} \[H^2_{Eis}(GL_3(\Z),(2a,2b-1,2b-1))=S_{2a-2c+4}+\C.\]
and
\[H^q_\partial=0\mbox{ for }q\neq 2.\]
\end{thm}




\subsection{Cohomology of $GL_3(\Z)$. Highest weight $(2a+1,0,-2c-1)$.  Case $D1$.}
The action from any elements of the Weyl group via $\lambda\mapsto w(\lambda+\rho)-\rho$ send $(odd,even,odd)$ to 
$(odd,even,odd)$. Therefore the cohomology of any of the parabolic subgroups vanish. As a result, the boundary cohomology vanishes and the Eisenstein cohomology vanishes, since it is a subspace of the boundary cohomology. 

However, there could be inner cohomology. Recall that the inner cohomology is the kernel of the map
$H^q_!(GL_3(\Z),V)=ker(H^q(GL_3(\Z),V)\rightarrow H^q_\partial(GL_3(\Z),V))$

Then  $H^q(GL_3(\Z),V_\lambda)=H^q_!(GL_3(\Z),V_\lambda)$ for $q=2,3$
and
 $H^q(GL_3(\Z),V_\lambda)=0$ for $q=0,1$.

{\bf Question:} {\it Is it true that the inner cohomology vanishes for non-self-dual representation.}

\begin{thm}
$H^q_{Eis}GL_3(\Z),V_\lambda)=0$ all $q$, where the weights are $(2a+1,0,-2c-1)$.
\end{thm}


\section{Boundary Cohomology of $GL_4(\Z)$} 

\subsection{Types of weight that produce non-zero cohomology classes.}
In this section we compute various types of cohomology of $GL_4(\Z)$ with coefficients in finite dimensional highest weight representations.

Let $\lambda=(a,b,c,d)$ a weight for $GL_4$, where $a\geq b\geq c\geq d$ are integers. 
In order for the cohomology $H^q(GL_4(Z),V_\lambda)\neq 0$ for some $q$, 
we must have certain parity for the coefficients $a,b,c,d$. 
If $-I$ acts non-trivially on $V_\lambda$ then $H^q(GL_4(\Z),V_\lambda)=0$ for all $q$. 
Note that the action of $-$ on $V_\lambda$ is multiplication by $(-1)^{a+b+c+d}$.

We have the following cases
\begin{enumerate}
\item[($A$)] $(even,even,even,even)$ and $(odd,odd,odd,odd)$
\begin{enumerate}
\item[$(1)$]  $(2a,2b,2c,2d)$
\item[$(1')$]  $(2a+1,2b+1,2c+1,2d+1)$
\item[$(2)$]  $(2a,2b,2c,2c)$ when $c=d$
\item[$(2')$]  $(2a+1,2b+1,2c+1,2c+1)$
\item[$(3)$]  $(2a,2b,2b,2d)$  when $b=c$
\item[$(3')$]  $(2a+1,2b+1,2b+1,2d+1)$
\item[$(4)$]  $(2a,2a,2c,2d)$ when $a=b$
\item[$(4')$]  $(2a+1,2a+1,2c+1,2d+1)$
\item[$(5)$]  $(2a,2b,2b,2b)$,  when $b=c=d$; then   $V_\lambda=Sym^{2a-2b}V$
\item[$(5')$]  $(2a+1,2b+1,2b+1,2b+1)$; then   $V_\lambda=Sym^{2a-2b}V\otimes det$
\item[$(6)$]  $(2a,2a,2a,2d)$  when $a=b=c$; then   $V_\lambda=Sym^{2a-2d}V^*$
\item[$(6')$]  $(2a+1,2a+1,2a+1,2d+1)$; then   $V_\lambda=Sym^{2a-2d}V^*\otimes det$
\item[$(7)$]  $(2a,2a,2c,2c)$, when  $a=b$ and $c=d$
\item[$(7')$]  $(2a+1,2a+1,2c+1,2c+1)$; then  \item[$(8)$]  $(2a,2a,2a,2a)$, when $a=b=c=d$; then   $V_\lambda=\C$
\item[$(8')$]  $(2a+1,2a+1,2a+1,2a+1)$; then  $V_\lambda=det$
\end{enumerate}
\item[($B$)] $(even,even,odd,odd)$ and $(odd,odd,even,even)$
\begin{enumerate}
\item[$(1)$] $(2a,2b,2c-1,2d-1)$
\item[$(1')$]  $(2a+1,2b+1,2c,2d)$
\item[$(2)$]  $(2a,2b,2c-1,2c-1)$, when $c=d$
\item[$(2')$]  $(2a+1,2b+1,2c,2c)$ 
\item[$(3)$]  $(2a,2a,2c-1,2d-1)$, when $a=b$
\item[$(3')$]  $(2a+1,2a+1,2c,2d)$
\item[$(4)$] $(2a,2a,2c-1,2c-1)$, when $a=b$ and $c=d$
\item[$(4')$]  $(2a+1,2a+1,2c,2c)$; 
\end{enumerate}
\item[($C$)] $(even,odd,odd,even)$ and $(odd,even,even,odd)$
\begin{enumerate}
\item[$(1)$]  $(2a+2,2b+1,2c+1,2d)$
\item[$(1')$]  $(2a+1,2b,2c,2d-1)$
\item[$(2)$]  $(2a+2,2b+1,2b+1,2d)$, when $b=c$
\item[$(2')$]  $(2a+1,2b,2b,2d-1)$
\end{enumerate}
\item[($D$)] $(even,odd,even,odd)$ and $(odd,even,odd,even)$
\begin{enumerate}
\item[$(1)$]  $(2a+2,2b+1,2c,2d-1)$
\item[$(1')$]  $(2a+1,2b,2c-1,2d-2)$
\end{enumerate}
\end{enumerate}






\subsection{Cohomology of $GL_4(\Z)$. Highest weight $(2a,2b,2c,2d)$.  Case $A1$}


\subsubsection{Cohomology of the parabolic subgroups. Highest weight $(2a,2b,2c,2d)$. Case $A1$}
Let $\lambda=(2a,2b,2c,2d)$.

For the minimal parabolic subgroup $P_0$ we have to consider only the permutations that interchange $1$ and $3$, and $2$ and $4$.
All possible such permutations are: $(1234)$, $(1432)$, $(3214)$ and $(3412)$. For them we have the following weights and lengths.

$P_0$: 
$\begin{tabular}{lllll}
$w$ 		& $l$	& 	& $w(\lambda+\rho)-\rho$\\
$1234$ 	& $0$&	& $(2a|2b|2c|2d)$\\
$1432$ 	& $3$&	& $(2a|2d-2|2c|2b+2)$\\
$3214$	& $3$&	& $(2c-2|2b|2a+2|2d)$\\
$3412$	& $4$&	& $(2c-2|2d-2|2a+2|2b+2)$\\
\end{tabular}
$

For the intermediate parabolic subgroup $P_{12}$ we have to consider only the permutations that have $1$ or $3$ at the third place and $2$ or $4$ at the fourth place. Together with that the first and the second entry should be in increasing order.
For the last two elements of the permutation, we have $34$,  $32$, $14$ and $12$. Since the first two elements of the permutation have to be in increasing order, the permutations are $(1234)$, $(1432)$, $(2314)$ and $(3412)$. 
For them we have the following weights and lengths.

$P_{12}$: 
$\begin{tabular}{lllll}
$w$ 		& $l$	& 	& $w(\lambda+\rho)-\rho$\\
$1234$ 	& $0$&	& $(2a,2b|2c|2d)$\\
$1432$ 	& $3$&	& $(2a,2d-2|2c|2b+2)$\\
$2314$	& $2$&	& $(2b-1,2c-1|2a+2|2d)$\\
$3412$	& $4$&	& $(2c-2,2d-2|2a+2|2b+2)$\\
\end{tabular}
$

For the intermediate parabolic subgroup $P_{23}$ we have to consider only the permutations that have $1$ or $3$ at the first place and $2$ or $4$ at the fourth place. Together with that the second and the third entry should be in increasing order.
The permutations are $(1234)$, $(1342)$, $(3124)$ and $(3142)$. 
For them we have the following weights and lengths.

$P_{23}$: 
$\begin{tabular}{lllll}
$w$ 		& $l$	& 	& $w(\lambda+\rho)-\rho$\\
$1234$ 	& $0$&	& $(2a|2b,2c|2d)$\\
$1342$ 	& $2$&	& $(2a|2c-1,2d-1|2b+2)$\\
$3124$	& $2$&	& $(2c-2|2a+1, 2b+1|2d)$\\
$3142$	& $3$&	& $(2c-2|2a+1,2d-1|2b+2)$\\
\end{tabular}
$

For the intermediate parabolic subgroup $P_{34}$ we have to consider only the permutations that have $1$ or $3$ at the first place and $2$ or $4$ at the second place. Together with that the third and the fourth entry should be in increasing order.
For the first two elements of the permutation, we have $12$, $14$, $32$ and $34$. Since the last two elements of the permutation have to be in increasing order, the permutations are $(1234)$, $(1423)$, $(3214)$ and $(3412)$. 
For them we have the following weights and lengths.

$P_{34}$: 
$\begin{tabular}{lllll}
$w$ 		& $l$	& 	& $w(\lambda+\rho)-\rho$\\
$1234$ 	& $0$&	& $(2a|2b|2c,2d)$\\
$1423$ 	& $2$&	& $(2a|2d-2|2b+1,2c+1)$\\
$3214$	& $3$&	& $(2c-2|2b|2a+2,2d)$\\
$3412$	& $4$&	& $(2c-2|2d-2|2a+2,2b+2)$\\
\end{tabular}
$

For the maximal parabolic subgroup $P_{13}$ we have to consider only the permutations that have $2$ or $4$ at the fourth place. Together with that the first, the second  and the third entry should be in increasing order.
For the last element of the permutation, we have $2$ or $4$ Since the first three elements of the permutation have to be in increasing order, the permutations are $(1234)$, $(1342)$
For them we have the following weights and lengths.

$P_{13}$: 
$\begin{tabular}{lllll}
$w$ 		& $l$	& 	& $w(\lambda+\rho)-\rho$\\
$1234$ 	& $0$&	& $(2a,2b,2c|2d)$\\
$1342$ 	& $2$&	& $(2a,2c-1,2d-1|2b+2)$\\
\end{tabular}
$

For the maximal parabolic subgroup $P_{12,34}$ we have to consider only the permutations that have $1$ or $3$ at the first place. Together with that the first and the second, entry should be in increasing order. Also the third and the fourth entry should be in increasing order. Also the first and the second entry should have opposite parity. All the possibilities for the first two entries are: $12$, $14$, $23$, $34$. The corresponding permutations are $(1234)$, $(1423)$, $(2314)$ and $(3412)$

$P_{12,34}$: 
$\begin{tabular}{lllll}
$w$ 		& $l$	& 	& $w(\lambda+\rho)-\rho$\\
$1234$ 	& $0$&	& $(2a,2b|2c,2d)$\\
$1423$ 	& $2$&	& $(2a, 2d-2|2b+1,2c+1)$\\
$2314$ 	& $2$&	& $(2b-1,2c-1|2a+2,2d)$\\
$3412$	& $4$&	& $(2c-2,2d-2|2a+2,2b+2)$\\
\end{tabular}
$

For the maximal parabolic subgroup $P_{24}$ we have to consider only the permutations that have $1$ or $3$ at the first place. Together with that the second, the third and the fourth entry should be in increasing order.
For the first element of the permutation, we have $1$ or $3$ Since the first three elements of the permutation have to be in increasing order, the permutations are $(1234)$, $(3124)$
For them we have the following weights and lengths.

$P_{24}$: 
$\begin{tabular}{lllll}
$w$ 		& $l$	& 	& $w(\lambda+\rho)-\rho$\\
$1234$ 	& $0$&	& $(2a|2b,2c,2d)$\\
$3124$ 	& $2$&	& $(2c-2|2a+1,2b+1,2d)$\\
\end{tabular}
$


\subsubsection{$E_1$-page. Highest weight $(2a,2b,2c,2d)$.  Case $A1$.}

{\begin{center}
\begin{small}
\scriptsize\renewcommand{\arraystretch}{2.2}
\begin{longtable}{|c|c|c|c|c|c|c|c|c|}
\hline
$E_1^{0,0}=0$ 
& $E_1^{1,0}=0$ & 
$E_1^{2,0}=(2a|2b|2c|2d)$
\\
\hline
$E_1^{0,1}=0$
	& $E_1^{1,1}=
	\left\{
	\begin{tabular}{lll}
	$(2a,2b|2c|2d)$\\
	$(2a|2b,2c|2d)$\\
	$(2a|2b|2c,2d)$
	\end{tabular}
	\right.
	$
& $E_1^{2,1}=0$
\\
\hline
$E_1^{0,2}=(2a,2b|2c,2d)$	
& $E_1^{1,2}=0$ 
& $E_1^{2,2}=0$
\\
\hline
$E_1^{0,3}=
	\left\{
	\begin{tabular}{ll}
	$(\overline{2a,2b,2c}|2d)$\\
	$(2a|\overline{2b,2c,2d})$
	\end{tabular}
	\right.
	$
& $E_1^{1,3}=
	\left\{
	\begin{tabular}{ll}
	$(2b-1,2c-1|2a+2|2d)$\\
	$(2c-2|2a+1,2b+1|2d)$\\
	$(2a|2d-2|2b+1,2c+1)$\\
	$(2a|2c-1,2d-1|2b+2)$
	\end{tabular}
	\right.
	$
& $E_1^{2,3}=
	\left\{
	\begin{tabular}{ll}
	$(2c-2|2b|2a+2|2d)$\\
	$(2a|2d-2|2c|2b+2)$
	\end{tabular}
	\right.
	$\\
\hline
$E_1^{0,4}=
	\left\{
	\begin{tabular}{ll}
	$(H^2(2a,2c-1,2d-1)|2b+2)$\\
	$(2a,2d-2|2b+1,2c+1)$\\
	$(2b-1,2c-1|2a+2,2d)$\\
	$(2c-2|H^2(2a+1,2b+1,2d))$
	\end{tabular}
	\right.$
& $E_1^{1,4}=
	\left\{
	\begin{tabular}{ll}
	$(2a,2d-2|2c|2b+2)$\\
	$(2c-2|2a+1,2d-1|2b+2)$\\
	$(2c-2|2b|2a+2,2d)$
	\end{tabular}
	\right.$
& $E_1^{2,4}=(2c-2|2d-2|2a+2|2b+2)$\\
\hline
$E_1^{0,5}=\left\{
	\begin{tabular}{ll}
	$(H^3(2a,2c-1,2d-1)|2b+2)$\\
	$(2c-2|H^3(2a+1,2b+1,2d))$
	\end{tabular}
	\right.$
& $E_1^{1,5}=\left\{
	\begin{tabular}{ll}
	$(2c-2,2d-2|2a+2|2b+2)$\\
	$(2c-2|2d-2|2a+2,2b+2)$
	\end{tabular}
	\right.$
& $E_1^{2,5}=0$\\
\hline
$E_1^{0,6}=(2c-2,2d-2|2a+2,2b+2)$
& $E_1^{1,6}=0$
& $E_1^{1,6}=0$\\
\hline
\end{longtable}
\end{small}
\end{center}


\subsubsection{Certain cohomology groups of $GL_3(\Z)$ that appear on the $E_1$-page in the case $A1$.}
From the computations of the cohomology of $GL_3(\Z)$ from the previous section, we have
{\begin{center}
\scriptsize\renewcommand{\arraystretch}{2.2}
\begin{longtable}{|c|c|c|}
\hline
$H^2(GL_3(\Z),(2a,2b,2c))=(\overline{2a,2b,2c})$
& $H^3(GL_3(\Z),(2a,2b,2c))
	=
	\left\{	\begin{tabular}{llll}
	$(\overline{2a,2b,2c})$\\
	$(\overline{2b-1,2c-1}|2a+2)$\\
	$(2c-2|\overline{2a+1,2b+1})$\\
	$(2c-2|2b|2a+2)$
	\end{tabular}
	\right.$\\
\hline
$H^2(GL_3(\Z),(2b,2c,2d))=(\overline{2b,2c,2d})$
& $H^3(GL_3(\Z),(2b,2c,2d))
	=
	\left\{	\begin{tabular}{llll}
	$(\overline{2b,2c,2d})$\\
	$(\overline{2c-1,2d-1}|2c+2)$\\
	$(2d-2|\overline{2b+1,2c+1})$\\
	$(2d-2|2c|2b+2)$
	\end{tabular}
	\right.$\\
\hline
$H^2(GL_3(\Z),(2a,2c-1,2d-1))
	=
	\Delta
	\subset
	\left\{	\begin{tabular}{llll}
	$(2a,2d-2|2c)$\\
	$(2c-2|\overline{2a+1,2d-1})$	
	\end{tabular}
	\right.$
	& $H^3(GL_3(\Z),(2a,2c-1,2d-1))=(2c-2,2d-2|2a+2)$
\\
\hline
$H^2(GL_3(\Z),(2a+1,2b+1,2d))
	=
	\Delta
	\subset
	\left\{	\begin{tabular}{llll}
	$(\overline{2a+1,2d-1}|2b+2)$\\
	$(2b|2a+2,2d)$
	\end{tabular}
	\right.$
&$H^3(GL_3(\Z),(2a+1,2b+1,2d))=(2d-2|2a+2,2b+2)$\\
\hline

\end{longtable}
\end{center}


\subsubsection{$E_2$-page. Highest weight $(2a,2b,2c,2d)$. Case $A1$}

Consider the $E_1^{0,3}$ term that comes from the maximal parabolic subgroup $P_{13}$. 
The corresponding summand is $(H^3(2a,2b,2c)|2d)$. 
It maps naturally to the cohomology of $P_{12}$ and $P_{23}$ which in turn map to the minimal parabolic subgroup $P_0$. 
Assume the representation $(2a,2b,2c)$ is regular. For it we have, $H^3(GL_3(\Z),(2a,2b,2c))=H^3_!(GL_3(\Z),(2a,2b,2c))+H^3_{Eis}(GL_3(\Z),(2a,2b,2c))=(\overline{2a,2b,2c})+H^3_\partial(GL_3(\Z),(2a,2b,2c))$. 
From the above computation of  $H^3_\partial(GL_3(\Z),(2a,2b,2c))$, 
we have that it maps to 
$H^3(Q_{12},(2a,2b,2c))+H^3(Q_{23},(2a,2b,2c))=(2b-1,2c-1|2a+2) + (2c-2|2a+1,2b+1)$. 
Both terms map to $H^3(Q_0,(2a,2b,2c))=(2c-2|2b|2a+2)$.
\[0\rightarrow H^3(GL_3(\Z),(2a,2b,2c)\rightarrow (2b-1,2c-1|2a+2) + (2c-2|2a+1,2b+1)\rightarrow (2c-2|2b|2a+2)\rightarrow 0\]

The cohomology of $GL_3(\Z)$ induces cohomology of $P_{13}$. From the last short exact sequence, we obtain the short exact sequence for 
$H^3(P_{13},(2a,2b,2c,2d))=(H^3(2a,2b,2c)|2d)$. We have that it maps to $H^3(P_{12},(2a,2b,2c,2d))+H^3(P_{23},(2a,2b,2c,2d))=(2b-1,2c-1|2a+2|2d) + (2c-2|2a+1,2b+1|2d)$, in turn, both terms map to $H^3(P_0,(2a,2b,2c,2d))=(2c-2|2b|2a+2|2d)$. Thus the short exact sequence is 
\begin{eqnarray}
\label{P_{13},2a,2b,2c,2d}
0\rightarrow (H^3(2a,2b,2c)|2d)\rightarrow \\
\rightarrow (2b-1,2c-1|2a+2|2d) + (2c-2|2a+1,2b+1|2d)\rightarrow (2c-2|2b|2a+2|2d)\rightarrow 0.
\end{eqnarray}
Similarly, we obtain the short exact sequence 
\begin{eqnarray}
\label{P_{13},2a,2b,2c,2d}
0\rightarrow (2a|H^3(2b,2c,2d))\rightarrow\\
	\rightarrow 
	(2a|2c-1,2d-1|2c+2)+
	(2a|2d-2|2b+1,2c+1)
	\rightarrow
	(2a|2d-2|2c|2b+2)
	\rightarrow 0.
\end{eqnarray}	

The direct sum of the last two short exact sequences is exactly.
\[0\rightarrow E_1^{0,3}\rightarrow E_1^{1,3}\rightarrow E_1^{2,3}\rightarrow 0.\]
Therefore, $E_2^{0,3}=E_2^{1,3}=E_2^{2,3}=0$.

For $E_1^{p,4}$ terms we have the following:
\[H^4(P_{13}) \rightarrow H^4(P_{12})+ H^4(P_{23})\] is injective, imbedding diagonally.
\[H^4(P_{24}) \rightarrow H^4(P_{23})+ H^4(P_{34})\] is injective, imbedding diagonally.
Also
\[H^4(P_{12,34}) \rightarrow H^4(P_{12})+ H^4(P_{34})\] is surjective.
Therefore, $E_2^{0,4}=ker[E_1^{0,4}\rightarrow E_1^{1,4}]=ker[H^4(P_{13}) + H^4(P_{24}) \rightarrow H^4(P_{23})]$ which is isomorphic to $ker(H^4(P_{23})\rightarrow H^4(P_0)=(2c-2|\overline{2a+2,2d-1}|2b+2)$.
Also, \[coker[E_1^{0,4}\rightarrow E_1^{1,4}]=(2c-2|2d-2|2a+2|2b+2)=\C\]. Therefore,
$E_2^{1,4}=0$ and $E_2^{2,4}=0.$

The argument that $E_2^{p,5}$ also vanish is similar.
 We have that $H^3(GL_3(\Z),(2a,2c-1,2d-1))=(2c-2,2d-2|2a+2)$. 
 We have also that 
 $H^3(GL_3(\Z),(2a,2c-1,2d-1))=H^3+{Eis}(GL_3(\Z),(2a,2c-1,2d-1)=H^3_\partial(GL_3(\Z),(2a,2c-1,2d-1)$, and that $H^3(Q_{12},(2a,2c-1,2d-1))=(2c-2,2d-2|2a+2)$. Therefore, the map
$H^3_\partial(GL_3(\Z),(2a,2c-1,2d-1))\rightarrow H^3(Q_{12},(2a,2c-1,2d-1))$ is an isomorphism. This isomorphism induces and isomorphism  on cohomology of parabolic subgroups of $GL_4$, namely.
$(H^3(GL_3(\Z),(2a,2c-1,2d-1)|2b+2)\rightarrow (2c-2,2d-2|2a+2|2b+2)$.
From the table with $E_1$ terms, we have that this is exactly the isomorphism
$H^5_\partial(P_{13},(2a,2b,2c,2d))\rightarrow H^5(P_{12},(2a,2b,2c,2d))$.

Similarly, one obtains that 
$H^5_\partial(P_{24},(2a,2b,2c,2d))\rightarrow H^5(P_{34},(2a,2b,2c,2d))$. Is an isomorphism. 

The direct sum of the last two isomorphisms gives is
\[0\rightarrow E_1^{0,5}\rightarrow E_1^{1,5}\rightarrow E_1^{2,5}\rightarrow 0,\]
where the first nontrivial arrow is an isomorphism and $E_1^{2,5}=0$.
Therefore, $E_2^{p,5}=0$

The $E_1^{p,6}$ has only one term. Therefore, $E_2{0,6}=E_1^{0,6}=(2c-2,2d-2|2a+2,2b+2)$.

The computations for $E_2^{p,4}$ rely on duality. For that reason, first we need to find the $E_1$-page of the spectral sequence for the boundary cohomology of $GL_4(\Z)$ with coefficients in $(2a,2b,2c,2d)\otimes det=(2a+1,2b+1,2c+1,2d+1)$.

So far we have the computed all of the $E_2$ terms except the $E_2^{p,4}$ terms.
This is written in the following table. The empty boxes denote that the corresponding $E_2^{p,q}$ term vanishes.

{\begin{center}
\begin{small}
\scriptsize\renewcommand{\arraystretch}{2.2}
\begin{longtable}{|c|c|c|c|c|c|c|c|c|}
\hline
& $p=0$
& $p=1$
& $p=2$\\
\hline
$q=0$
&  
&  
& $E_2^{2,0}=(2a|2b|2c|2d)$
\\
\hline
$q=1$
&
& $E_2^{1,1}=
	\left\{
	\begin{tabular}{lll}
	$(2a,2b|2c|2d)$\\
	$(2a|2b,2c|2d)$\\
	$(2a|2b|2c,2d)$
	\end{tabular}
	\right.
	$
&
\\
\hline
$q=2$
&
$E_2^{0,2}=
	\left\{
	\begin{tabular}{ll}
	$(2a,2b|2c,2d)$\\
	$(\overline{2a,2b,2c}|2d)$\\
	$(2a|\overline{2b,2c,2d})$
	\end{tabular}
	\right.
	$
&  & 
\\
\hline
$q=3$
& $E_2^{0,3}=\left\{
	\begin{tabular}{ll}
	$(\overline{2a,2b,2c}|2d)$\\
	$(2a|\overline{2b,2c,2d})$
	\end{tabular}
	\right.
	$
& 
& \\
\hline
$q=4$
&$E_2^{0,4}=\left\{
	\begin{tabular}{lll}
	$(2a,2d-2|\overline{2b+1,2c+1})$\\
	$(\overline{2b-1,2c-1}|2a+2,2d)$\\
	$(2c-2|\overline{2a+1,2d-1}|2b+2)$
	\end{tabular}
	\right.$
&
& \\
\hline
$q=5$
&
& 
&\\
\hline
$q=6$
&
$E_2^{0,6}=(2c-2,2d-2|2a+2,2b+2)$
& 
& \\
\hline
\end{longtable}
\end{small}
\end{center}


\subsubsection{Boundary cohomology. Highest weight $(2a,2b,2c,2d)$.  Case $A1$.}

By definition, the boundary cohomology is the one to which the spectral sequence converges.
That is, $\bigoplus_{prk(P)=p+1}H^q(P,V_\lambda)=>H^{p+q}_\partial(GL_4(\Z),V_\lambda)$.
From the previous subsection, it is clear that in the case 
$A1$
the spectral sequence degenerates at the $E_2$-page.
Therefore, up to semisimplicity, 
$H^{n}_\partial(GL_4(\Z),V_\lambda)=\bigoplus_{p+q=n} E_2^{p,q}$.
Since we know the $E_2$ terms which are of the form, we can find the  boundary cohomology using our computation of the $E_2$ terms
Let us denote temporarily $H^q=H^q_\partial(GL_4(\Z),(2a,2b,2c,2d))$.
Then, we can summarize the result for the boundary cohomology in the following table.

\begin{thm}
$H^q=
\left\{
\begin{tabular}{llll}
	$0$ & $q=0$\\
	$0$ & $q=1$\\
	$E_2^{0,2}+E_2^{1,1}+E_2^{2,0}=\left\{
	\begin{tabular}{lll}
	$(\overline{2a,2b,2c}|2d)$\\
	$(2a|\overline{2b,2c,2d})$\\
	$(2a,2b|2c,2d)$\\
	$(2a,2b|2c|2d)$\\
	$(2a|2b,2c|2d)$\\
	$(2a|2b|2c,2d)$\\
	$(2a|2b|2c|2d)$
	\end{tabular}
	\right\}$ & $q=2$\\
	$E_2^{0,3}=\left\{
	\begin{tabular}{lll}
	$(\overline{2a,2b,2c}|2d)$\\
	$(2a|\overline{2b,2c,2d})$
	\end{tabular}
	\right\}$ & $q=3$\\
	$E_2^{0,4}=\left\{
	\begin{tabular}{lll}
	$(2a,2d-2|\overline{2b+1,2c+1})$\\
	$(\overline{2b-1,2c-1}|2a+2,2d)$\\
	$(2c-2|\overline{2a+1,2d-1}|2b+2)$
	\end{tabular}
	\right\}$ & $q=4$\\
	$0$ & $q=5$\\
	$E_2^{0,6}=(2c-2,2d-2|2a+2,2b+2)$ & $q=6$\\
	$0$ & $q=7$\\
	$0$ & $q=8$
\end{tabular}
\right.
$
\end{thm}

\begin{thm}

\begin{eqnarray*}
H^2_\partial(GL_4(\Z),(2a,2b,2c,2d))=&&H^2_!(GL_3(\Z),(2a,2b,2c)) + H^2_!(GL_3(\Z),2b,2c,2d))\\
							&&S_{2b-2a+2}S_{2c-2d+2} + S_{2a-2b+2} + S_{2b-2c+2} + S_{2c-2d+2}+\C\\
H^3_\partial(GL_4(\Z),(2a,2b,2c,2d))=&&H^3_!(GL_3(\Z),(2a,2b,2c)) + H^3_!(GL_3(\Z),(2b,2c,2d))\\
H^4_\partial(GL_4(\Z),(2a,2b,2c,2d))
	=&&
	S_{2a-2d+4}S_{2b-2c+2}
	+S_{2b-2c+2}S_{2a-2d+4}
	+S_{2a-2d+4}\\
	H^5_\partial(GL_4(\Z),(2a,2b,2c,2d))=&&0\\
H^6_\partial(GL_4(\Z),(2a,2b,2c,2d))
	= &&S_{2c-2d+2}S_{2a-2b+2}
\end{eqnarray*}
	
\end{thm}

\begin{thm}

\begin{eqnarray*}
H^2_\partial(GL_4(\Z),(2a,2b,2c,2d))=&&H^2_!(GL_3(\Z),(2a,2b,2c)) + H^2_!(GL_3(\Z),(2b,2c,2d))\\
							&&S_{2b-2a+2}S_{2c-2d+2} + S_{2a-2b+2} + S_{2b-2c+2} + S_{2c-2d+2}+\C\\
H^3_\partial(GL_4(\Z),(2a,2b,2c,2d))=&&H^3_!(GL_3(\Z),(2a,2b,2c)) + H^3_!(GL_3(\Z),(2b,2c,2d))\\
H^4_\partial(GL_4(\Z),(2a,2b,2c,2d))
	=&&
	S_{2a-2d+4}S_{2b-2c+2}
	+S_{2b-2c+2}S_{2a-2d+4}
	+S_{2a-2d+4}\\
H^5_\partial(GL_4(\Z),(2a,2b,2c,2d))=&&0\\
H^6_\partial(GL_4(\Z),(2a,2b,2c,2d))
	= &&S_{2c-2d+2}S_{2a-2b+2}
\end{eqnarray*}
	
\end{thm}








\subsection{Cohomology of $GL_4(\Z)$. Highest weight $(2a+1,2b+1,2c+1,2d+1)$.  Case $A1'$}


\subsubsection{Cohomology of the parabolic subgroups. Highest weight $(2a+1,2b+1,2c+1,2d+1)$. Case $A1'$}
For that case the weights are $(2a+1,2b+1,2c+1,2d+1)$.
We have to find the elements $w$ of the Weyl group that give a non-trivial cohomology of the minimal parabolic subgroup. The first entry has to go to the second or fourth place in order to be even. Similarly, third entry should go to the second or fourth place. Also, the second and the four elements have to go the first and the third elements. 
Thus, all permutations are $2143$, $2341$, $4123$ and $4321$.

$P_0$: 
$\begin{tabular}{lllll}
$w$ 		& $l$	& 	& $w(\lambda+\rho)-\rho$\\
$2143$ 	& $2$&	& $(2b|2a+2|2d|2c+2)$\\
$2341$ 	& $3$&	& $(2b|2c|2d|2a+4)$\\
$4123$	& $3$&	& $(2d-2|2a+2|2b+2|2c+2)$\\
$4321$	& $6$&	& $(2d-2|2c|2b+2|2a+4)$\\
\end{tabular}
$

For the intermediate parabolic subgroup $P_{12}$ we have to consider only the permutations that have $2$ or $4$ at the third place and $1$ or $3$ at the fourth place. Together with that the first and the second entry should be in increasing order.
For the last two elements of the permutation, we have $43$,  $23$, $41$ and $21$. Since the first two elements of the permutation have to be in increasing order, the permutations are $1243$, $1423$, $2341$ and $3421$. 
For them we have the following weights and lengths.

$P_{12}$: 
$\begin{tabular}{lllll}
$w$ 		& $l$	& 	& $w(\lambda+\rho)-\rho$\\
$1243$ 	& $1$&	& $(2a+1,2b+1|2d|2c+2)$\\
$1423$ 	& $2$&	& $(2a+1,2d-1|2b+2|2c+2)$\\
$2341$	& $3$&	& $(2b,2c|2d|2a+4)$\\
$3421$	& $5$&	& $(2c-1,2d-1|2b+2|2a+4)$\\
\end{tabular}
$

For the intermediate parabolic subgroup $P_{23}$ we have to consider only the permutations that have $1$ or $3$ at the fourth place and $2$ or $4$ at the first place. Together with that the second and the third entry should be in increasing order.
The permutations are $2143$, $2341$, $4123$ and $4231$. 
For them we have the following weights and lengths.

$P_{23}$: 
$\begin{tabular}{lllll}
$w$ 		& $l$	& 	& $w(\lambda+\rho)-\rho$\\
$2143$ 	& $2$&	& $(2b|2a+2,2d|2c+2)$\\
$2341$	& $3$&	& $(2b|2c,2d|2a+4)$\\
$4123$	& $3$&	& $(2d-2|2a+2, 2b+2|2c+2)$\\
$4231$	& $5$&	& $(2d-2|2b+1,2c+1|2a+4)$\\
\end{tabular}
$

For the intermediate parabolic subgroup $P_{34}$ we have to consider only the permutations that have $2$ or $4$ at the first place and $1$ or $3$ at the second place. Together with that the third and the fourth entry should be in increasing order.
For the first two elements of the permutation, we have $12$, $14$, $32$ and $34$. Since the last two elements of the permutation have to be in increasing order, the permutations are $2134$, $2314$, $4123$ and $4312$. 
For them we have the following weights and lengths.

$P_{34}$: 
$\begin{tabular}{lllll}
$w$ 		& $l$	& 	& $w(\lambda+\rho)-\rho$\\
$2134$ 	& $1$&	& $(2b|2a+2|2c+1,2d+1)$\\
$2314$	& $2$&	& $(2b|2c|2a+3,2d+1)$\\
$4123$	& $3$&	& $(2d-2|2a+2|2b+2,2c+2)$\\
$4312$ 	& $5$&	& $(2d-2|2c|2a+3,2b+3)$\\
\end{tabular}
$

For the maximal parabolic subgroup $P_{13}$ we have to consider only the permutations that have $1$ or $3$ at the fourth place. Together with that the first, the second  and the third entry should be in increasing order.
For the last element of the permutation, we have $1$ or $3$ Since the first three elements of the permutation have to be in increasing order, the permutations are$1243$, $2341$.
For them we have the following weights and lengths.

$P_{13}$: 
$\begin{tabular}{lllll}
$w$ 		& $l$	& 	& $w(\lambda+\rho)-\rho$\\
$1243$ 	& $1$&	& $(2a+1,2b+1,2d|2c+2)$\\
$2341$ 	& $3$&	& $(2b,2c,2d|2a+4)$\\
\end{tabular}
$

For the maximal parabolic subgroup $P_{12,34}$ we have to consider only the permutations that have $1$ or $3$ at the first place. Together with that the first and the second, entry should be in increasing order. Also the third and the fourth entry should be in increasing order. Also the first and the second entry should have opposite parity. All the possibilities for the first two entries are: $12$, $14$, $23$, $34$. The corresponding permutations are $(1234)$, $(1423)$, $(2314)$ and $(3412)$

$P_{12,34}$: 
$\begin{tabular}{lllll}
$w$ 		& $l$	& 	& $w(\lambda+\rho)-\rho$\\
$1234$ 	& $0$&	& $(2a+1,2b+1|2c+1,2d+1)$\\
$1423$ 	& $2$&	& $(2a+1, 2d-1|2b+2,2c+2)$\\
$2314$ 	& $2$&	& $(2b,2c|2a+3,2d+1)$\\
$3412$	& $4$&	& $(2c-1,2d-1|2a+3,2b+3)$\\
\end{tabular}
$

For the maximal parabolic subgroup $P_{24}$ we have to consider only the permutations that have $2$ or $4$ at the first place. Together with that the second, the third and the fourth entry should be in increasing order.
For the first element of the permutation, we have $2$ or $4$ Since the first three elements of the permutation have to be in increasing order, the permutations are  $2134$, $4123$
For them we have the following weights and lengths.

$P_{24}$: 
$\begin{tabular}{lllll}
$w$ 		& $l$	& 	& $w(\lambda+\rho)-\rho$\\
$2134$ 	& $1$&	& $(2b|2a+2,2c+1,2d+1)$\\
$4123$ 	& $3$&	& $(2d-2|2a+2,2b+2,2c+2)$\\
\end{tabular}
$


\subsubsection{$E_1$-page. Highest weight $(2a+1,2b+1,2c+1,2d+1)$.  Case $A1'$}

{\begin{center}
\begin{small}
\scriptsize\renewcommand{\arraystretch}{2.2}
\begin{longtable}{|c|c|c|c|c|c|c|c|c|}
\hline
$E_1^{0,0}=0$ 
& $E_1^{1,0}=0$ & 
$E_1^{2,0}=0$
\\
\hline
$E_1^{0,1}=0$
& $E_1^{1,1}=0$
& $E_1^{2,1}=0$
\\
\hline
$E_1^{0,2}=(2a+1,2b+1|2c+1,2d+1)$ 
& $E_1^{1,2}=
	\left\{
	\begin{tabular}{ll}
	$(2a+1,2b+1|2d|2c+2)$\\
	$(2b|2a+2|2c+1,2d+1)$
	\end{tabular}
	\right.$ 
& $E_1^{2,2}=(2b|2a+2|2d|2c+2)$
\\
\hline
$E_1^{0,3}=
	\left\{
	\begin{tabular}{ll}
	$(H^2(2a+1,2b+1,2d)|2c+2)$\\
	$(2b|H^2(2a+2,2c+1,2d+1))$
	\end{tabular}
	\right.
	$
& $E_1^{1,3}=
	\left\{
	\begin{tabular}{ll}
	$(2a+1,2d-1|2b+2|2c+2)$\\
	$(2b|2a+2,2d|2c+2)$\\
	$(2b|2c|2a+3,2d+1)$
	\end{tabular}
	\right.
	$
& $E_1^{2,3}=
	\left\{
	\begin{tabular}{ll}
	$(2b|2c|2d|2a+4)$\\
	$(2d-2|2a+2|2b+2|2c+2)$
	\end{tabular}
	\right.
	$\\
\hline
$E_1^{0,4}=
	\left\{
	\begin{tabular}{ll}
	$(H^3(2a+1,2b+1,2d)|2c+2)$\\
	$(2a+1,2d-1|2b+2,2c+2)$\\
	$(2b,2c|2a+3,2d+1)$\\
	$(2b|H^3(2a+2,2c+1,2d+1))$
	\end{tabular}
	\right.$
& $E_1^{1,4}=
	\left\{
	\begin{tabular}{ll}
	$(2b,2c|2d|2a+4)$\\
	$(2b|2c,2d|2a+4)$\\
	$(2d-2|2a+2,2b+2|2c+2)$\\
	$(2d-2|2a+2|2b+2,2c+2)$
	\end{tabular}
	\right.$
& $E_1^{2,4}=0$\\
\hline
$E_1^{0,5}=\left\{
	\begin{tabular}{ll}
	$(H^2(2b,2c,2d)|2a+4)$\\
	$(2d-2|H^2(2a+2,2b+2,2c+2))$
	\end{tabular}
	\right.$
& $E_1^{1,5}=0$
& $E_1^{2,5}=0$\\
\hline
$E_1^{0,6}=
	\left\{
	\begin{tabular}{ll}
	$(H^3(2b,2c,2d)|2a+4)$\\
	$(2c-1,2d-1|2a+3,2b+3)$\\
	$(2d-2|H^3(2a+2,2b+2,2c+2))$
	\end{tabular}
	\right.$
& $E_1^{1,6}=
	\left\{
	\begin{tabular}{ll}
	$(2c-1,2d-1|2b+2|2a+4)$\\
	$(2d-2|2b+1,2c+1|2a+4)$\\
	$(2d-2|2c|2a+3,2b+3)$
	\end{tabular}
	\right.$
& $E_1^{2,6}=(2d-2|2c|2b+2|2a+4)$\\
\hline
\end{longtable}
\end{small}
\end{center}


\subsubsection{Certain cohomology groups of $GL_3(\Z)$ that appear on the $E_1$-page in the case $A1'$.}
From the computations of the cohomology of $GL_3(\Z)$ from the previous section, we have
{\begin{center}
\scriptsize\renewcommand{\arraystretch}{2.2}
\begin{longtable}{|c|c|c|}
\hline
$H^2(GL_3(\Z),(2a+1,2b+1,2d))
	=
	\Delta_1
	$
&$H^3(GL_3(\Z),(2a+1,2b+1,2d))=(2d-2|2a+2,2b+2)$\\
\hline
$H^2(GL_3(\Z),(2a+2,2c+1,2d+1))
	=
	\Delta_2
	$
& $H^3(GL_3(\Z),(2a+2,2c+1,2d+1))=(2c,2d|2a+4)$
\\
\hline
$H^2(GL_3(\Z),(2b,2c,2d))=(\overline{2b,2c,2d})$
& $H^3(GL_3(\Z),(2b,2c,2d))
	=
	\left\{	\begin{tabular}{llll}
	$(\overline{2b,2c,2d})$\\
	$(\overline{2c-1,2d-1}|2c+2)$\\
	$(2d-2|\overline{2b+1,2c+1})$\\
	$(2d-2|2c|2b+2)$
	\end{tabular}
	\right.$\\
\hline
$H^2(GL_3(\Z),(2a+2,2b+2,2c+2))=(\overline{2a+2,2b+2,2c+2})$
& $H^3(GL_3(\Z),(2a+2,2b+2,2c+2))
	=
	\left\{	\begin{tabular}{llll}
	$(\overline{2a+2,2b+2,2c+2})$\\
	$(\overline{2b+1,2c+1}|2a+4)$\\
	$(2c|\overline{2a+3,2b+3})$\\
	$(2c|2b+2|2a+4)$
	\end{tabular}
	\right.$\\
	\hline
\end{longtable}
\end{center}
where
$\Delta_1 = (\overline{2a+1,2d-1}|2b+2) + (2b|2a+2,2d)$.
	and
$\Delta_2 = (2a+2,2d|2c+2) + (2c|\overline{2a+3,2d+1})$.


\subsubsection{$E2$ page. Highest weight $(2a+1,2b+1,2c+1,2d+1)$.  Case $A1'$}

From consideration of Euler characteristics we can conclude that the map $(2a+1,2b+1)\rightarrow (2b|2a+2)$ is surjective for $a>b$, which corresponds to the Eisenstein series $E_{2a-2b+2}$ for the classical modular group $SL_2(\Z)$.

From the surjectivity of the above map we have that the following sequence 
$E_1^{0,2}\rightarrow E_1^{1,2} \rightarrow E_1^{1,2} \rightarrow 0$
is exact.
Therefore, 
$E_2^{0,2}=(\overline{2a+1,2b+1}|\overline{2c+1,2d+1})$
and $E_2^{1,2}=E_1^{2,2} =0$.

Now let us examine the third line $E_1^{p,3}$ of the $E_1$-page. 
First, let us concentrate to on the terms coming from the inner cohomology $H^1_!(GL_2(\Z),V_\lambda)$.
$H^2(2a+1,2b+1,2d)$ embeds diagonally in $(2a+1,2d-1|2b+2)+(2b|2a+2,2d)$. Therefore,
$(H^2(2a+1,2b+1,2d)|2c+2)$ embeds diagonally in $(\overline{2a+1,2d-1}|2b+2|2c+2)+(2b|2a+2,2d|2c+2)$.
Similarly, $(2b|H^2(2a+2,2c+1,2d+1))$  embeds diagonally in 
$(2b|2a+2,2d|2c+2)+(2b|2c|\overline{2a+3,2d+1})$.
Therefore, the kernel $ker[E_1^{0,3}\rightarrow E_1^{1,3}]$ does not contain a copy of an inner cohomology $H^1_!(GL_2(\Z),V_\lambda)$,
and $coker[E_1^{0,3}\rightarrow E_1^{1,3}]$  contains one copy of an inner cohomology of $GL_2(\Z)$, isomorphic to $(2b|2a+2,2d|2c+2)$.
Therefore, $E_2^{1,3}=(2b|2a+2,2d|2c+2)$.

Now, we can examine the terms on the third line that come from cohomology of $GL_1={\mathbb G}_m$. Such cohomology we will denote briefly by $\C$.
There are two copies of $\C$ in $E_1^{1,3}$ that map isomorphically to $E_1^{2,3}$.
Therefore, $E_2^{0,3}=0$ and $E_2^{2,3}=0$.

For the same reason, for $E_1^{p,4}$, 
we have that the maps 
$(2a+1,2d-1|2b+2,2c+2)\rightarrow (2d-2|2a+2|2b+2,2c-2)$ 
and
$(2b,2c|2a+3,2d+1)\rightarrow (2b,2c|2d,2a+4)$ 
are surjective.

From cohomology of $GL_3(\Z)$ with coefficients in $(2a+1,2b+1,2d)$, we have that 
$H^3(2a+1,2b+1,2d)=(2d-2|2a+2,2b+2)$.
Therefore,
$(H^3(2a+1,2b+1,2d)|2c+2)=(2d-2|2a+2,2b+2|2c+2)$.
Similarly, from the cohomology of $GL_3(\Z)$ with coefficients in $(2a+2,2c+1,2d+1)$l, we have that
$H^3(2a+2,2c+1,2d+1)=(2c,2d|2a+4)$. Therefore,
$(2b|H^3(2a+2,2c+1,2d+1))=(2b|2c,2d|2a+4)$.

Then $E_2^{0,4}=(\overline{2a+1,2d-1}|2b+2,2c+2) + (2b,2c|\overline{2a+3,2d+1})$
and $E_2^{1,4}=E_2^{2,4}=0$.

For $E_1^{p,6}$ we have that the following.
Denote by
$H^q_!(P,V)=\oplus_{i+j=q}H^i_!(S_P,H^j(M_P,V)$
Let $H^q_{Eis}(P,V)=coker[H^q_!(P,V)\rightarrow H^q(P,V)].$

From the previous subsection, we have the following isomorphisms
\[H^6(P_{13},V)=ker[H^6(P_{12},V)+ H^6(P_{23},V)\rightarrow H^6(P_0,V)\]
\[H^q(P_{24},V)=H^q(P_{23},V)+ H^q(P_{34},V)\rightarrow H^6(P_0,V)\]
\[H^q_{Eis}(P_{12,34},V)=H^q(P_{12},V)+ H^q(P_{34},V)\rightarrow H^6(P_0,V)\]
Note also that 
$H^6_!(P_{12,34},V)=(\overline{2c-1,2d-1},\overline{2a+3,,2b+3})$
Also $E_1^{1,6}\rightarrow E_1^{2,6}$ is surjective.
Therefore

$E_2^{0,6}=\C+H^q_!(P_{12},V)+ H^q_!(P_{23},V)+ H^q_!(P_{34},V)+H^6_!(P_{13},V)+H^6_!(P_{24},V)+H^6_!(P_{12,34},V)$.

So far we have the computed all of the $E_2$ terms except the $E_2^{p,3}$ terms.
This is written in the following table. The empty boxes denote that the corresponding $E_2^{p,q}$ term vanishes.

{\begin{center}
\begin{small}
\scriptsize\renewcommand{\arraystretch}{2.2}
\begin{longtable}{|c|c|c|c|c|c|c|c|c|}
\hline
& $p=0$
& $p=1$
& $p=2$\\
\hline
$q=0$
&  
&  
& \\
\hline
$q=1$
&
& 
&
\\
\hline
$q=2$
& $E_2^{0,2}=(\overline{2a+1,2b+1}|\overline{2c+1,2d+1})$
&  
& 
\\
\hline
$q=3$
&
& $E_2^{1,3}=(2b|2a+2,2d|2c+2)$
&\\
\hline
$q=4$
& $E_2^{0,4}=
	\left\{
	\begin{tabular}{lll}
	$(\overline{2a+1,2d-1}|2b+2,2c+2)$\\
	$(2b,2c|\overline{2a+3,2d+1})$
	\end{tabular}
	\right.$
& 
& \\
\hline
$q=5$
& $E_2^{0,5}=\left\{
	\begin{tabular}{lll}
	$(\overline{2b,2c,2d}|2a+4)$\\
	$(2d-2|H^2_!(2a+2,2b+2,2c+2))$
	\end{tabular}
	\right.$
& 
&\\
\hline
$q=6$
&
$E_2^{0,6}=\left\{
	\begin{tabular}{lll}
	$(\overline{2b,2c,2d}|2a+4)$\\
	$(2d-2|H^3_!(2a+2,2b+2,2c+2))$\\
	$(2d-2|\overline{2b+1,2c+1}|2a+4)$\\
	$(\overline{2c-1,2d-1},\overline{2a+3,,2b+3})$\\
	$(2d-2|2c|\overline{2a+3,2b+3})$\\
	$(\overline{2c-1,2d-1}|2b+2|2a+4)$\\
	$(2d-2|2c|2b+2|2a+4)$
	\end{tabular}
	\right.$
& 
& \\
\hline
\end{longtable}
\end{small}
\end{center}


\subsubsection{Boundary cohomology. Highest weight $(2a+1,2b+1,2c+1,2d+1)$.  Case $A1'$.}
Again the boundary cohomology degenerates at the $E_2$-page. Therefore, up to semisimplicity we have
$H^n_\partial(GL_4(\Z),V_\lambda)=\bigoplus_{p+q=n}E_2^{p,q}$.
More systematically, we have

\begin{thm}
Let $H^q=H^q_\partial(GL_4(\Z),V_\lambda)$ be the cohomology of $GL_4(\Z)$ with coefficients in the highest weight representation with weight 
$\lambda=(2a+1,2b+1,2c+1,2d+1)$ written as a character of the split maximal torus. Then,
\[H^q
=
\left\{
\begin{tabular}{lll}
$E_2^{0,2}=(\overline{2a+1,2b+1}|\overline{2c+1,2d+1})$ 
		& $q=2$\\
$0$
		& $q=3$\\	
$E_2^{0,4}+E_2^{1,3}=
	\left\{
	\begin{tabular}{lll}
	$(\overline{2a+1,2d-1}|2b+2,2c+2)$\\
	$(2b,2c|\overline{2a+3,2d+1})$\\
	$(2b|2a+2,2d|2c+2)$
	\end{tabular}
	\right\}$
		& $q=4$\\
$E_2^{0,5}=\left\{
	\begin{tabular}{lll}
	$(\overline{2b,2c,2d}|2a+4)$\\
	$(2d-2|H^2_!(2a+2,2b+2,2c+2))$
	\end{tabular}
	\right\}$
		& $q=5$\\ 
$E_2^{0,6}=\left\{
	\begin{tabular}{lll}
	$(\overline{2b,2c,2d}|2a+4)$\\
	$(2d-2|H^3_!(2a+2,2b+2,2c+2))$\\
	$(2c-1,2d-1|2a+3,2b+3)$\\
	$(2d-2|\overline{2b+1,2c+1}|2a+4)$
	\end{tabular}
	\right\}$
		& $q=6$
\end{tabular}
\right.
\]
\end{thm}

Since the above $H^4_\partial(A1')$ contains $E_2^{p,q}$ with $p>0$, we have that $H^4_\partial(A1')$ contains a potentially ghost class, namely,
\[pGh^4(GL_4(\Z),(2a+1,2b+1,2c+1,2d+1))=E_2^{1,3}=(2b|2a+2,2d|2c+2).\]

This is important not for the cohomology of $GL_4(\Z)$ but for the cohomology of $GL_5(\Z)$. In a similar way as the potentially ghost classes in $GL_3(\Z)$ have importance for computing the cohomology of $GL_4(\Z)$. The next representation of $GL_4$, namely $(2a,2b,2c,2c)$, or case $A2$, exhibits exactly the use of potentially ghost classes in $GL_3(\Z)$ and their effect on the cohomology of $GL_4(\Z)$.

In terms of spaces of modular forms, we have the following

\begin{thm}
With the notation $H^q=H^q_\partial(GL_4(\Z),V_\lambda)$ for the highest weight 
$\lambda=(2a+1,2b+1,2c+1,2d+1)$, we have

\begin{eqnarray*}
H^2_\partial(GL_4(\Z),(2a+1,2b+1,2c+1,2d+1))
=
&&
S_{2a-2b+2}S_{2c-2d+2}\\
H^3_\partial(GL_4(\Z),(2a+1,2b+1,2c+1,2d+1))
=&&0\\	
H^4_\partial(GL_4(\Z),(2a+1,2b+1,2c+1,2d+1))=
&&
	S_{2a-2d+4}S_{2b_2c+2}\\
&&
	+
	S_{2b_2c+2}S_{2a-2d+4}\\
&&
	+
	S_{2a-2d+4}\\
H^5_\partial(GL_4(\Z),(2a+1,2b+1,2c+1,2d+1))=
&&
	H^2_!(GL_3(\Z),(2b,2c,2d))\\
&&
	+
	H^2_!(GL_3(\Z),(2a+2,2b+2,2c+2))\\
H^6_\partial(GL_4(\Z),(2a+1,2b+1,2c+1,2d+1))=
&&
	H^3_!(GL_3(\Z),(2b,2c,2d))\\
&&
	+
	H^3_!(GL_3(\Z),(2a+2,2b+2,2c+2))\\
&&
	+
	S_{2b-2c+2}\\
&&
	+
	S_{2c-2d+2}S_{2a-2b+2}\\
&&	
	+
	S_{2c-2d+2}\\
&&	
	+
	S_{2a-2b+2}\\
&&	
	+
	\C
\end{eqnarray*}
\end{thm}

\subsection{Eisenstein cohomology of $GL_4(\Z)$ with coefficients in the families of representations $A1$ and $A1'$ }




\subsection{Cohomology of $GL_4(\Z)$. Highest weight $(2a,2b,2c,2c)$.  Case $A2$, $(a>b>c=d)$}


\subsubsection{Cohomology of the parabolic subgroups. Highest weight $(2a,2b,2c,2c)$. Case $A2$}

From the case $A1$, we know which elements of the Weyl group could produce non-trivial cohomology classes. On top of that we have that 
\[H^0(GL_2(\Z),(2c,2c))=H^0(GL_2(\Z),\C)=\C\]
\[H^1(GL_2(\Z),(2c,2c))=H^1(GL_2(\Z),\C)=0\] 
and \[H^1(GL_2(\Z),(2c+1,2c+1))=H^1(GL_2(\Z),det)=0\]
Notation \[(2c,2c)=H^0(GL_2(\Z),(2c,2c))\]

$P_0$: 
\begin{tabular}{|c|c|l|}
\hline
$w$ 		& $l$	 	& $w(\lambda+\rho)-\rho$\\
\hline
$1234$ 	& $0$	& $(2a|2b|2c|2c)$\\
\hline
$1432$ 	& $3$	& $(2a|2c-2|2c|2b+2)$\\
\hline
$3214$	& $3$	& $(2c-2|2b|2a+2|2c)$\\
\hline
$3412$	& $4$	& $(2c-2|2c-2|2a+2|2b+2)$\\
\hline
\end{tabular}
$P_{12}$: 
\begin{tabular}{|c|c|l|}
\hline
$w$ 		& $l$	 	& $w(\lambda+\rho)-\rho$\\
\hline
$1234$ 	& $0$	& $(2a,2b|2c|2c)$\\
\hline
$1432$ 	& $3$	& $(2a,2c-2|2c|2b+2)$\\
\hline
$2314$	& $2$	& $(2b-1,2c-1|2a+2|2c)$\\
\hline
$3412$	& $4$	& $(2c-2,2c-2|2a+2|2b+2)$\\
\hline
\end{tabular}

\vspace{.4cm}

$P_{23}$: 
\begin{tabular}{|c|c|l|}
\hline
$w$ 		& $l$	 	& $w(\lambda+\rho)-\rho$\\
\hline
$1234$ 	& $0$	& $(2a|2b,2c|2c)$\\
\hline
$1342$ 	& $2$	& $(2a|2c-1,2c-1|2b+2)$\\
\hline
$3124$	& $2$	& $(2c-2|2a+1, 2b+1|2c)$\\
\hline
$3142$	& $3$	& $(2c-2|2a+1,2c-1|2b+2)$\\
\hline
\end{tabular}
$P_{34}$: 
\begin{tabular}{|c|c|l|}
\hline
$w$ 		& $l$	 	& $w(\lambda+\rho)-\rho$\\
\hline
$1234$ 	& $0$	& $(2a|2b|2c,2c)$\\
\hline
$1423$ 	& $2$	& $(2a|2c-2|2b+1,2c+1)$\\
\hline
$3214$	& $3$	& $(2c-2|2b|2a+2,2c)$\\
\hline
$3412$	& $4$	& $(2c-2|2c-2|2a+2,2b+2)$\\
\hline
\end{tabular}

\vspace{.4cm}

$P_{13}$: 
\begin{tabular}{|c|c|l|}
\hline
$w$ 		& $l$	 	& $w(\lambda+\rho)-\rho$\\
\hline
$1234$ 	& $0$	& $(2a,2b,2c|2c)$\\
\hline
$1342$ 	& $2$	& $(2a,2c-1,2c-1|2b+2)$\\
\hline
\end{tabular}
$P_{24}$: 
\begin{tabular}{|c|c|l|}
\hline
$w$ 		& $l$		& $w(\lambda+\rho)-\rho$\\
\hline
$1234$ 	& $0$	& $(2a|2b,2c,2c)$\\
\hline
$3124$ 	& $2$	& $(2c-2|2a+1,2b+1,2c)$\\
\hline
\end{tabular}

\vspace{.4cm}

$P_{12,34}$: 
\begin{tabular}{|c|c|l|}
\hline
$w$ 		& $l$	 	& $w(\lambda+\rho)-\rho$\\
\hline
$1234$ 	& $0$	& $(2a,2b|2c,2c)$\\
\hline
$1423$ 	& $2$	& $(2a, 2c-2|2b+1,2c+1)$\\
\hline
$2314$ 	& $2$	& $(2b-1,2c-1|2a+2,2c)$\\
\hline
$3412$	& $4$	& $(2c-2,2c-2|2a+2,2b+2)$\\
\hline
\end{tabular}


\subsubsection{$E_1$-page. Highest weight $(2a,2b,2c,2c)$.  Case $A2$. $(a>b>c=d)$}

{\begin{center}
\begin{small}
\scriptsize\renewcommand{\arraystretch}{2.2}
\begin{longtable}{|c|c|c|c|c|c|c|c|c|}
\hline
$E_1^{0,0}=0$ 
& $E_1^{1,0}=(2a|2b|2c,2c)$
& 
$E_1^{2,0}=(2a|2b|2c|2c)$
\\
\hline
$E_1^{0,1}=(2a,2b|2c,2c)$
	& $E_1^{1,1}=
	\left\{
	\begin{tabular}{lll}
	$(2a,2b|2c|2c)$\\
	$(2a|2b,2c|2c)$
	\end{tabular}
	\right.
	$
& $E_1^{2,1}=0$
\\
\hline
$E_1^{0,2}=(\overline{2a,2b,2c}|2c)$ 
& $E_1^{1,2}=0$ 
& $E_1^{2,2}=0$
\\
\hline
$E_1^{0,3}=
	\left\{
	\begin{tabular}{ll}
	$(\overline{2a,2b,2c}|2c)$\\
	$(2a|H^3(2b,2c,2c))$
	\end{tabular}
	\right.
	$
& $E_1^{1,3}=
	\left\{
	\begin{tabular}{ll}
	$(2b-1,2c-1|2a+2|2c)$\\
	$(2c-2|2a+1,2b+1|2c)$\\
	$(2a|2c-2|2b+1,2c+1)$\\
	$0=(2a|2c-1,2c-1|2b+2)$
	\end{tabular}
	\right.
	$
& $E_1^{2,3}=
	\left\{
	\begin{tabular}{ll}
	$(2c-2|2b|2a+2|2c)$\\
	$(2a|2c-2|2c|2b+2)$
	\end{tabular}
	\right.
	$\\
\hline
$E_1^{0,4}=
	\left\{
	\begin{tabular}{ll}
	$(H^2(2a,2c-1,2c-1)|2b+2)$\\
	$(2a,2c-2|2b+1,2c+1)$\\
	$(2b-1,2c-1|2a+2,2c)$\\
	$(2c-2|H^2(2a+1,2b+1,2c))$
	\end{tabular}
	\right.$
& $E_1^{1,4}=
	\left\{
	\begin{tabular}{ll}
	$(2a,2c-2|2c|2b+2)$\\
	$(2c-2|2a+1,2c-1|2b+2)$\\
	$(2c-2|2b|2a+2,2c)$\\
	$(2c-2,2c-2|2a+2|2b+2)$
	\end{tabular}
	\right.$
& $E_1^{2,4}=(2c-2|2c-2|2a+2|2b+2)$\\
\hline
$E_1^{0,5}=\left\{
	\begin{tabular}{ll}
	$(H^3(2a,2c-1,2c-1)|2b+2)$\\
	$(2c-2|H^3(2a+1,2b+1,2c))$\\
	$(2c-2,2c-2|2a+2,2b+2)$
	\end{tabular}
	\right.$
& $E_1^{1,5}=(2c-2|2c-2|2a+2,2b+2)$
& $E_1^{2,5}=0$\\
\hline
$E_1^{0,6}=0$
& $E_1^{1,6}=0$
& $E_1^{1,6}=0$\\
\hline
\end{longtable}
\end{small}
\end{center}


\subsubsection{Cohomology groups of $GL_3(\Z)$ that appear on the $E_1$-page in the case $(GL_4,A2)$.}
From the computations of the cohomology of $GL_3(\Z)$ from the previous section, we have
{\begin{center}
\scriptsize\renewcommand{\arraystretch}{2.2}
\begin{longtable}{|c|c|c|}
\hline
$H^2(GL_3(\Z),(2a,2b,2c))=\overline{2a,2b,2c}$
& $H^3(GL_3(\Z),(2a,2b,2c))
	=
	\left\{	\begin{tabular}{llll}
	$\overline{2a,2b,2c}$\\
	$(\overline{2b-1,2c-1}|2a+2)$\\
	$(2c-2|\overline{2a+1,2b+1})$\\
	$(2c-2|2b|2a+2)$
	\end{tabular}
	\right.$\\
	\hline
$H^2(GL_3(\Z),(2b,2c,2c))
	=0$
& $H^3(GL_3(\Z),(2b,2c,2c))
	=(2c-2|\overline{2b+1,2c+1})$
	\\
\hline
$H^2(GL_3(\Z),(2a,2c-1,2c-1))
	=
	\Delta_0+\Delta_1$
	& $H^3(GL_3(\Z),(2a,2c-1,2c-1))=0$
\\
\hline
$H^2(GL_3(\Z),(2a+1,2b+1,2c))
	=
	\Delta
	\subset
	\left\{	\begin{tabular}{llll}
	$(\overline{2a+1,2c-1}|2b+2)$\\
	$(2b|2a+2,2c)$
	\end{tabular}
	\right.$
&$H^3(GL_3(\Z),(2a+1,2b+1,2c))=(2c-2|2a+2,2b+2)$\\
\hline
\end{longtable}
\end{center}
In the above Table we denote by
$\Delta_0\subset (2a|2c -2|2c)+(2c-2|2c-2|2a+2)$
and
$\Delta_1\subset (2a,2c-2|2c)+(2c-2|\overline{2a+1,2c-1})$.	
diagonal embeddings.


\subsubsection{The $E_2$-page for the case $A2$.}

The map $E_1^{1,0}\rightarrow E_1^{1,1}$ is an isomorphism. Since it is equivalent to the map $(2a|2b|2c, 2c) \rightarrow (2a|2b|2c|2c)$, 
which is equivalent to the map $H^0(GL_2(\Z),(2c,2c))\rightarrow H^0(GL_1(\Z),(2c))\otimes H^0(GL_1(\Z),(2c))$. The last map is equivalent to 
$H^0(GL_2(\Z),\C)\rightarrow H^0(GL_1(\Z),\C)\otimes H^0(GL_1(\Z),\C)$, which is an isomorphism since the left hand side is the $0$-th Eisenstein cohomology of $GL_2(\Z)$ which maps isomorphically to the $0$-th boundary cohomology.
Therefore, $E_2^{p,0}=0$ for $p=0,1,2$.

For the same reason, $E_1^{0,1}$ embeds in $E_1^{1,1}$. Therefore $E_2^{p,1}=0$ for $p=0,2$ and $E_2^{1,1}=(2a|2b, 2c|2c)$.

For $q=2$, we have $E_1^{p,2}=0$ for $p=1,2$, therefore $E_2^{0,2}=\overline{2a,2b,2c}$.

From the previous subsection $E_2^{p,3}=0$ for $p=1,2$, therefore $E_2^{0,3}=\overline{2a,2b,2c}$.

For $q=4$, we have that the map $(2c-2|2a+1,2c-1|2b+2)\rightarrow (2c-2|2c-2|2a+2|2b+2)$ is surjective. Also the maps 
$(2a,2c - 2|2b + 1,2c + 1)\rightarrow(2a, 2c-2|2c|2b + 2)$ and 
$(2b - 1,2c - 1|2a + 2,2c)\rightarrow (2c - 2|2b|2a + 2, 2c)$ are surjective for the same reason.
Finally, $(H^2(2a,2c-1,2c-1)|2b+2)$ projects isomorphically to $(2c-2|\overline{2a+1,2c-1}|2b+2)+ (2c-2,2c-2|2a+2|2b+2)$.
Therefore,
$E_2^{p,4}=0$ for $p=1,2$ 
and $E_2^{0,4}=(2a,2c-2|\overline{2b+1,2c+1})+(\overline{2b-1,2c-1}|2a+2,2c)+\Delta$, where
$\Delta=
	\subset
	(\overline{2a+1,2c-1}|2b+2)+(2b|2a+2,2c).
	$

For $q=5$, we have an isomorphism $(2c-2,2c-2|2a+2,2b+2)\rightarrow (2c-2|2c-2|2a+2,2b+2)$. Therefore,
$E_2^{p,5}=0$ for $p=1,2$ 
and
$E_2^{0,5}=(2c-2|H^3(GL_3(\Z),(2a+1,2b+1,2c))=(2c-2|2c-2|2a+2,2b+2)$

{\begin{center}
\begin{small}
\scriptsize\renewcommand{\arraystretch}{2.2}
\begin{longtable}{|c|c|c|c|c|c|c|c|c|}
\hline
& $p=0$
& $p=1$
& $p=2$\\
\hline
$q=0$
&  
&  
&
\\
\hline
$q=1$
&
& $E_2^{1,1}=(2a|2b, 2c|2c)$.
&
\\
\hline
$q=2$
& $E_2^{0,2}=(\overline{2a,2b,2c}|2c)$.
&  
& 
\\
\hline
$q=3$
& $E_2^{0,3}=(\overline{2a,2b,2c}|2c)$.
& 
& \\
\hline
$q=4$
&$E_2^{0,4}=\left\{
	\begin{tabular}{lll}
	$(2a,2c-2|\overline{2b+1,2c+1})$\\
	$(\overline{2b-1,2c-1}|2a+2,2c)$\\
	$(2c-2|\overline{2a+1,2c-1}|2b+2)$
	\end{tabular}
	\right.$
&
& \\
\hline
$q=5$
& $E_2^{0,5}=(2c-2|2c-|2a+2,2b+2)$
& 
&\\
\hline
$q=6$
&
& 
& \\
\hline
\end{longtable}
\end{small}
\end{center}


\subsubsection{Boundary cohomology of $GL_4(\Z)$ with coefficients in case $A2$, $(2a,2b,2c,2c)$}

From the $E_2$-page it is clear that the spectral sequence has trivial $d_2$ maps. Therefore $E_\infty^{p,q}=E_2^{p,q}$. Up to semisiplicity, we have 
$H^q_\partial(A2)=\bigoplus_{i+j=q} E_2{i,j}$. From the previous subsection we obtain

\begin{thm}
\[H^q_\partial(A2)
=
\left\{
\begin{tabular}{lll}
	$\left\{
	\begin{tabular}{ll}
	$E_2^{1,1}=(2a|2b, 2c|2c)$\\
	$E_2^{0,2}=(\overline{2a,2b,2c}|2c)$
	\end{tabular}
	\right\}$
		& $q=2$\\
$E_2^{0,3}=(\overline{2a,2b,2c}|2c)$
		& $q=3$\\	
$E_2^{0,4}=
	\left\{
	\begin{tabular}{lll}
	$(2a,2c-2|\overline{2b+1,2c+1})$\\
	$(\overline{2b-1,2c-1}|2a+2,2c)$\\
	$(2c-2|\overline{2a+1,2c-1}|2b+2)$
	\end{tabular}
	\right\}$
		& $q=4$\\
$E_2^{0,5}=(2c-2|2c-2|2a+2,2b+2)$
		& $q=5$\\ 
$0$
		& $q=6$
\end{tabular}
\right.
\]
\end{thm}

\begin{thm}
\begin{eqnarray*}
H^2_\partial(GL_4(\Z), (2a,2b,2c,2c))
=
&&
	S_{2b-2c+2}+H^2_!(GL_3(\Z),(2a,2b,2c))\\
H^3_\partial(GL_4(\Z), (2a,2b,2c,2c))
=
&&
	H^3_!((GL_3(\Z),(2a,2b,2c))\\
H^4_\partial(GL_4(\Z), (2a,2b,2c,2c))
=
&&
	S_{2a-2c+4}S_{2b-2c+2}
	+
	S_{2b-2c+2}S_{2a-2c+4}
	+
	S_{2a-2c+4}\\
H^5_\partial(GL_4(\Z), (2a,2b,2c,2c))
=
&&
	S_{2a-2b+2}\\
H^6_\partial(GL_4(\Z), (2a,2b,2c,2c))
=
&&0
\end{eqnarray*}
\end{thm}








\subsection{Cohomology of $GL_4(\Z)$. Highest weight $(2a+1,2b+1,2c+1,2c+1)$.  Case $A2'$}


\subsubsection{Cohomology of the parabolic subgroups. Highest weight $(2a+1,2b+1,2c+1,2c+1)$. Case $A2'$}
For that case the weights are $(2a+1,2b+1,2c+1,2c+1)$.
The choice of Weyl elements depend on the parity of the highest weight. Therefore, we can use the weights from the case $A1'$. For them we have

$P_0$: 
$\begin{tabular}{lllll}
$w$ 		& $l$	& 	& $w(\lambda+\rho)-\rho$\\
$2143$ 	& $2$&	& $(2b|2a+2|2c|2c+2)$\\
$2341$ 	& $3$&	& $(2b|2c|2c|2a+4)$\\
$4123$	& $3$&	& $(2c-2|2a+2|2b+2|2c+2)$\\
$4321$	& $6$&	& $(2c-2|2c|2b+2|2a+4)$\\
\end{tabular}
$

Similarly, for the other parabolic subgroups - we consider the same Weyl elements as for the case $A1'$. In the output, instead of $d$ we write $c$. 
$P_{12}$: 
$\begin{tabular}{lllll}
$w$ 		& $l$	& 	& $w(\lambda+\rho)-\rho$\\
$1243$ 	& $1$&	& $(2a+1,2b+1|2c|2c+2)$\\
$1423$ 	& $2$&	& $(2a+1,2c-1|2b+2|2c+2)$\\
$2341$	& $3$&	& $(2b,2c|2c|2a+4)$\\
$3421$	& $5$&	& $(2c-1,2c-1|2b+2|2a+4)$\\
\end{tabular}
$

$P_{23}$: 
$\begin{tabular}{lllll}
$w$ 		& $l$	& 	& $w(\lambda+\rho)-\rho$\\
$2143$ 	& $2$&	& $(2b|2a+2,2c|2c+2)$\\
$2341$	& $3$&	& $(2b|2c,2c|2a+4)$\\
$4123$	& $3$&	& $(2c-2|2a+2, 2b+2|2c+2)$\\
$4231$	& $5$&	& $(2c-2|2b+1,2c+1|2a+4)$\\
\end{tabular}
$

$P_{34}$: 
$\begin{tabular}{lllll}
$w$ 		& $l$	& 	& $w(\lambda+\rho)-\rho$\\
$2134$ 	& $1$&	& $(2b|2a+2|2c+1,2c+1)$\\
$2314$	& $2$&	& $(2b|2c|2a+3,2c+1)$\\
$4123$	& $3$&	& $(2c-2|2a+2|2b+2,2c+2)$\\
$4312$ 	& $5$&	& $(2c-2|2c|2a+3,2b+3)$\\
\end{tabular}
$

$P_{13}$: 
$\begin{tabular}{lllll}
$w$ 		& $l$	& 	& $w(\lambda+\rho)-\rho$\\
$1243$ 	& $1$&	& $(2a+1,2b+1,2c|2c+2)$\\
$2341$ 	& $3$&	& $(2b,2c,2c|2a+4)$\\
\end{tabular}
$

$P_{12,34}$: 
$\begin{tabular}{lllll}
$w$ 		& $l$	& 	& $w(\lambda+\rho)-\rho$\\
$1234$ 	& $0$&	& $(2a+1,2b+1|2c+1,2c+1)$\\
$1423$ 	& $2$&	& $(2a+1, 2c-1|2b+2,2c+2)$\\
$2314$ 	& $2$&	& $(2b,2c|2a+3,2c+1)$\\
$3412$	& $4$&	& $(2c-1,2c-1|2a+3,2b+3)$\\
\end{tabular}
$

For the maximal parabolic subgroup $P_{24}$ we have to consider only the permutations that have $2$ or $4$ at the first place. Together with that the second, the third and the fourth entry should be in increasing order.
For the first element of the permutation, we have $2$ or $4$ Since the first three elements of the permutation have to be in increasing order, the permutations are  $2134$, $4123$
For them we have the following weights and lengths.

$P_{24}$: 
$\begin{tabular}{lllll}
$w$ 		& $l$	& 	& $w(\lambda+\rho)-\rho$\\
$2134$ 	& $1$&	& $(2b|2a+2,2c+1,2c+1)$\\
$4123$ 	& $3$&	& $(2c-2|2a+2,2b+2,2c+2)$\\
\end{tabular}
$


\subsubsection{$E_1$-page. Highest weight $(2a+1,2b+1,2c+1,2c+1)$.  Case $A2'$}

This case is very similar to the case $A1'$. However, there are some minor differences. First, when we set $c=d$, we might encounter cohomology of the representation $(2c+1,2c+1)$. Both the zeroth and the first cohomology of that representation vanish. If we encounter $(2c,2c)$, then its first cohomology vanishes. However, its zeroth does not. For all other cohomology groups the change is only writing $c$ instead of $d$.

{\begin{center}
\begin{small}
\scriptsize\renewcommand{\arraystretch}{2.2}
\begin{longtable}{|c|c|c|c|c|c|c|c|c|}
\hline
$E_1^{0,0}=0$ 
& $E_1^{1,0}=0$ & 
$E_1^{2,0}=0$
\\
\hline
$E_1^{0,1}=0$
& $E_1^{1,1}=0$
& $E_1^{2,1}=0$
\\
\hline
$E_1^{0,2}=(2a+1,2b+1|2c+1,2c+1)=0$ 
& $E_1^{1,2}=
	\left\{
	\begin{tabular}{ll}
	$(2a+1,2b+1|2c|2c+2)$\\
	$(2b|2a+2|2c+1,2c+1)=0$
	\end{tabular}
	\right.$ 
& $E_1^{2,2}=(2b|2a+2|2c|2c+2)$
\\
\hline
$E_1^{0,3}=
	\left\{
	\begin{tabular}{ll}
	$(H^2(2a+1,2b+1,2c)|2c+2)$\\
	$(2b|H^2(2a+2,2c+1,2c+1))$
	\end{tabular}
	\right.
	$
& $E_1^{1,3}=
	\left\{
	\begin{tabular}{ll}
	$(2a+1,2c-1|2b+2|2c+2)$\\
	$(2b|2a+2,2c|2c+2)$\\
	$(2b|2c,2c|2a+4)$\\
	$(2b|2c|2a+3,2c+1)$
	\end{tabular}
	\right.
	$
& $E_1^{2,3}=
	\left\{
	\begin{tabular}{ll}
	$(2b|2c|2c|2a+4)$\\
	$(2c-2|2a+2|2b+2|2c+2)$
	\end{tabular}
	\right.
	$\\
\hline
$E_1^{0,4}=
	\left\{
	\begin{tabular}{ll}
	$(H^3(2a+1,2b+1,2c)|2c+2)$\\
	$(2a+1,2c-1|2b+2,2c+2)$\\
	$(2b,2c|2a+3,2c+1)$\\
	$(2b|H^3(2a+2,2c+1,2c+1))$
	\end{tabular}
	\right.$
& $E_1^{1,4}=
	\left\{
	\begin{tabular}{ll}
	$(2b,2c|2c|2a+4)$\\
	$(2c-2|2a+2,2b+2|2c+2)$\\
	$(2c-2|2a+2|2b+2,2c+2)$
	\end{tabular}
	\right.$
& $E_1^{2,4}=0$\\
\hline
$E_1^{0,5}=(2c-2|H^2(2a+2,2b+2,2c+2))$
& $E_1^{1,5}=0$
& $E_1^{2,5}=0$\\
\hline
$E_1^{0,6}=
	\left\{
	\begin{tabular}{ll}
	$(H^3(2b,2c,2c)|2a+4)$\\
	$(2c-1,2c-1|2a+3,2b+3)=0$\\
	$(2c-2|H^3(2a+2,2b+2,2c+2))$
	\end{tabular}
	\right.$
& $E_1^{1,6}=
	\left\{
	\begin{tabular}{ll}
	$(2c-1,2c-1|2b+2|2a+4)=0$\\
	$(2c-2|2b+1,2c+1|2a+4)$\\
	$(2c-2|2c|2a+3,2b+3)$
	\end{tabular}
	\right.$
& $E_1^{2,6}=(2c-2|2c|2b+2|2a+4)$\\
\hline
\end{longtable}
\end{small}
\end{center}


\subsubsection{Cohomology groups of $GL_3(\Z)$ that appear on the $E_1$-page in the case $(GL_4,A2')$.}
From the computations of the cohomology of $GL_3(\Z)$ from the previous section, we have
{\begin{center}
\scriptsize\renewcommand{\arraystretch}{2.2}
\begin{longtable}{|c|c|c|}
\hline
$H^2(GL_3(\Z),(2a+2,2b+2,2c+2))=H^2_!(GL_3(\Z),(2a+2,2b+2,2c+2))$
& $H^3(GL_3(\Z),(2a+2,2b+2,2c+2))
	=
	\left\{	\begin{tabular}{llll}
	$H^3_!(GL_3(\Z),(2a+2,2b+2,2c+2))$\\
	$(\overline{2b+1,2c+1}|2a+4)$\\
	$(2c|\overline{2a+3,2b+3})$\\
	$(2c|2b+2|2a+4)$
	\end{tabular}
	\right.$\\
	\hline
$H^2(GL_3(\Z),(2b,2c,2c))
	=0$
& $H^3(GL_3(\Z),(2b,2c,2c))
	=(2c-2|\overline{2b+1,2c+1})$
	\\
\hline
$H^2(GL_3(\Z),(2a+2,2c+1,2c+1))
	=
	\Delta_0+\Delta_1$
	& $H^3(GL_3(\Z),(2a+2,2c+1,2c+1))=0$
\\
\hline
$H^2(GL_3(\Z),(2a+1,2b+1,2c))
	=
	\Delta
	$
&$H^3(GL_3(\Z),(2a+1,2b+1,2c))=(2c-2|2a+2,2b+2)$\\
\hline
\end{longtable}
\end{center}
In the above Table we denote by
$\Delta_0\subset (2a+2|2c|2c+2)+(2c|2c|2a+4)$,
$\Delta_1\subset (2a+2,2c|2c+2)+(2c|\overline{2a+3,2c+1})$
and
$\Delta
	\subset
	(\overline{2a+1,2c-1}|2b+2)+(2b|2a+2,2c)$
are diagonal embeddings.


\subsubsection{$E2$ page. Highest weight $(2a+1,2b+1,2c+1,2c+1)$.  Case $A2'$}

From consideration of Euler characteristics we can conclude that the map $(2a+1,2b+1)\rightarrow (2b|2a+2)$ is surjective for $a>b$, which corresponds to the Eisenstein series $E_{2a-2b+2}$ for the classical modular group $SL_2(\Z)$.

From the surjectivity of the above map we have that the following map
$E_1^{1,2} \rightarrow E_1^{1,2}$
is si surjective.
Therefore, 
$E_2^{0,2}=E_1^{2,2} =0$
and $E_2^{1,2}=(\overline{2a + 1,2b + 1}|2c|2c + 2)$.

Now let us examine the third line $E_1^{p,3}$ of the $E_1$-page. 
First, let us concentrate to on the terms coming from the inner cohomology $H^1_!(GL_2(\Z),V_\lambda)$.
$H^2(2a+1,2b+1,2c)$ embeds diagonally in $(2a+1,2c-1|2b+2)+(2b|2a+2,2c)$. Therefore,
$(H^2(2a+1,2b+1,2c)|2c+2)$ embeds diagonally in $(\overline{2a+1,2c-1}|2b+2|2c+2)+(2b|2a+2,2c|2c+2)$.
Similarly, $(2b|H^2_!(2a+2,2c+1,2c+1))$  embeds diagonally in 
$(2b|2a+2,2c|2c+2)+(2b|2c|\overline{2a+3,2c+1})$.
Therefore, the kernel $ker[E_1^{0,3}\rightarrow E_1^{1,3}]$ does not contain a copy of an interior cohomology of $GL_2(\Z)$, $H^1_!(GL_2(\Z),V_\lambda)$,
and $coker[E_1^{0,3}\rightarrow E_1^{1,3}]$  contains one copy of an interior cohomology of $GL_2(\Z)$, isomorphic to $(2b|2a+2,2d|2c+2)$.
Therefore, $E_2^{1,3}=(2b|2a+2,2c|2c+2)$.

Now, we can examine the terms on the third line that come from cohomology of $GL_1={\mathbb G}_m$. Such cohomology we will denote briefly by $\C$
There is one copy of $\C$ in $E_1^{0,3}$. There are three copies of $\C$ in $E_1^{1,3}$ which  map surjectively to $\C^2$ to $E_1^{2,3}$
Therefore, $E_2^{0,3}=0$ and $E_2^{2,3}=0$

For the same reason, for $E_1^{p,4}$, 
we have that the maps 
$(2a+1,2c-1|2b+2,2c+2)\rightarrow (2c-2|2a+2|2b+2,2c-2)$ 
and
$(2b,2c|2a+3,2c+1)\rightarrow (2b,2c|2c,2a+4)$ 
are surjective.

From cohomology of $GL_3(\Z)$ with coefficients in $(2a+1,2b+1,2c)$, we have that 
$H^3(2a+1,2b+1,2c)=(2c-2|2a+2,2b+2)$.
Therefore,
$(H^3(2a+1,2b+1,2c)|2c+2)=(2c-2|2a+2,2b+2|2c+2)$.
Similarly, from the cohomology of $GL_3(\Z)$ with coefficients in $(2a+2,2c+1,2c+1)$l, we have that
$H^3(2a+2,2c+1,2c+1)=0$. 

Then $E_2^{0,4}=(\overline{2a+1,2c-1}|2b+2,2c+2) + (2b,2c|\overline{2a+3,2c+1})$
and $E_2^{1,4}=E_2^{2,4}=0$.

For $E_1^{p,6}$ we have that the following.
From the cohomology of $GL_3(\Z)$ with coefficients in $2a+2,2b+2,2c+2$, we have that 
$H^6(P_{24})=(2c-2|H^3(GL_3(\Z),(2a+2,2b+2,2c+2)))=(2c-2|\overline{2b+1,2c+1}|2a+4)+(2c-2|2c|\overline{2a+3,2b+3})+(2c-2|2c|2b+2|2a+4)$
Thereofre
\[0\rightarrow H^6(P_{24})\rightarrow H^6(P_{23})+H^6(P_{34})\rightarrow H^6(P_0)\rightarrow 0\] is a short exact sequence. 
Therefore, we obtain that
$E_2^{0,6}$ is isomorphic to $H^3_!(GL_3(\Z),(2a+2,2b+2,2c+2))+(H^3(2b,2c,2c)|2a+4)=H^3_!(GL_3(\Z),(2a+2,2b+2,2c+2))+(2c-2|\overline{2b+1,2c+1}|2a+4)$,
and
$E_2^{1,6}=E_2^{2,6}=0$.

This is written in the following table. The empty boxes denote that the corresponding $E_2^{p,q}$ term vanishes.

{\begin{center}
\begin{small}
\scriptsize\renewcommand{\arraystretch}{2.2}
\begin{longtable}{|c|c|c|c|c|c|c|c|c|}
\hline
& $p=0$
& $p=1$
& $p=2$\\
\hline
$q=0$
&  
&  
& \\
\hline
$q=1$
&
& 
&
\\
\hline
$q=2$
& 
& $E_2^{1,2}=(\overline{2a+1,2b+1}|2c|2c+2)$
& 
\\
\hline
$q=3$
& 
& $E_2^{1,3}=(2b|2a+2,2c|2c+2)$ 
&\\
\hline
$q=4$
& $E_2^{0,4}=
	\left\{
	\begin{tabular}{lll}
	$(\overline{2a+1,2c-1}|2b+2,2c+2)$\\
	$(2b,2c|\overline{2a+3,2c+1})$
	\end{tabular}
	\right.$
& 
& \\
\hline
$q=5$
& $H^2_!(GL_3(\Z),(2a+2,2b+2,2c+2))$
&
&\\
\hline
$q=6$
&
$E_2^{0,6}=
	\left\{
	\begin{tabular}{lll}
	$H^3_!(GL_3(\Z),(2a+2,2b+2,2c+2))$\\
	$(2c-2|\overline{2b+1,2c+1}|2a+4)$
	\end{tabular}
	\right.$
& 
& \\
\hline
\end{longtable}
\end{small}
\end{center}


\subsubsection{Boundary cohomology of $GL_4(\Z)$ with coefficients in case $A2'$, $(2a+1,2b+1,2c+1,2c+1)$}

From the $E_2$-page it is clear that the spectral sequence has trivial $d_2$ maps. Therefore $E_\infty^{p,q}=E_2^{p,q}$. Up to semisiplicity, we have 
$H^q_\partial(A2)=\bigoplus_{i+j=q} E_2{i,j}$. From the previous subsection we obtain

\begin{thm}
\[H^q_\partial(A2')
=
\left\{
\begin{tabular}{lll}
$0$
	& $q=2$\\
$E_2^{1,2}=(\overline{2a+1,2b+1}|2c|2c+2)$
		& $q=3$\\	
$\left\{
\begin{tabular}{llll}
$E_2^{0,4}=\left\{
	\begin{tabular}{lll}
	$(\overline{2a+1,2c-1}|2b+2,2c+2)$\\
	$(2b,2c|\overline{2a+3,2c+1})$
	\end{tabular}
	\right.$\\
$E_2^{1,3}=(2b|2a+2,2c|2c+2)$
\end{tabular}
\right\}$	
		& $q=4$\\
$0$
		& $q=5$\\ 
$E_2^{0,6}=(2c-2|\overline{2b+1,2c+1}|2a+4)$
		& $q=6$
\end{tabular}
\right.
\]
\end{thm}

\begin{thm}
\begin{eqnarray*}
H^2_\partial(GL_4(\Z),(2a+1,2b+1,2c+1,2c+1))
=
&&0\\
H^3_\partial(GL_4(\Z),(2a+1,2b+1,2c+1,2c+1))
=
&&
	S_{2a-2b+2}\\
H^4_\partial(GL_4(\Z),(2a+1,2b+1,2c+1,2c+1))
=
&&
	S_{2a-2c+4}S_{2b-2c+2}
	+
	S_{2b-2c+2}S_{2a-2c+4}	
	+
	S_{2a-2c+4}
	\\
H^5_\partial(GL_4(\Z),(2a+1,2b+1,2c+1,2c+1))
=
&&
	H^2_!(GL_3(\Z),(2a,2b,2c))\\
H^6_\partial(GL_4(\Z),(2a+1,2b+1,2c+1,2c+1))
=
&&
	S_{2b-2c+2}
	+
	H^3_!(GL_3(\Z),(2a,2b,2c))
\end{eqnarray*}

\end{thm}




\subsection{Cohomology of $GL_4(\Z)$. Highest weight $(2a,2b,2b,2d)$.  Case $A3$}


\subsubsection{Cohomology of the parabolic subgroups. Highest weight $(2a,2b,2b,2d)$. Case $A3$}
Let $\lambda=(2a,2b,2b,2d)$.

The choice of Weyl elements depend on the parity of the highest weight. Therefore, we can use the weights from the case $A3$. For them we have

$P_0$: 
$\begin{tabular}{lllll}
$w$ 		& $l$	& 	& $w(\lambda+\rho)-\rho$\\
$1234$ 	& $0$&	& $(2a|2b|2b|2d)$\\
$1432$ 	& $3$&	& $(2a|2d-2|2b|2b+2)$\\
$3214$	& $3$&	& $(2b-2|2b|2a+2|2d)$\\
$3412$	& $4$&	& $(2b-2|2d-2|2a+2|2b+2)$\\
\end{tabular}
$

$P_{12}$: 
$\begin{tabular}{lllll}
$w$ 		& $l$	& 	& $w(\lambda+\rho)-\rho$\\
$1234$ 	& $0$&	& $(2a,2b|2b|2d)$\\
$1432$ 	& $3$&	& $(2a,2d-2|2b|2b+2)$\\
$2314$	& $2$&	& $(2b-1,2b-1|2a+2|2d)$\\
$3412$	& $4$&	& $(2b-2,2d-2|2a+2|2b+2)$\\
\end{tabular}
$

$P_{23}$: 
$\begin{tabular}{lllll}
$w$ 		& $l$	& 	& $w(\lambda+\rho)-\rho$\\
$1234$ 	& $0$&	& $(2a|2b,2b|2d)$\\
$1342$ 	& $2$&	& $(2a|2b-1,2d-1|2b+2)$\\
$3124$	& $2$&	& $(2b-2|2a+1, 2b+1|2d)$\\
$3142$	& $3$&	& $(2b-2|2a+1,2d-1|2b+2)$\\
\end{tabular}
$

$P_{34}$: 
$\begin{tabular}{lllll}
$w$ 		& $l$	& 	& $w(\lambda+\rho)-\rho$\\
$1234$ 	& $0$&	& $(2a|2b|2b,2d)$\\
$1423$ 	& $2$&	& $(2a|2d-2|2b+1,2b+1)$\\
$3214$	& $3$&	& $(2b-2|2b|2a+2,2d)$\\
$3412$	& $4$&	& $(2b-2|2d-2|2a+2,2b+2)$\\
\end{tabular}
$

$P_{13}$: 
$\begin{tabular}{lllll}
$w$ 		& $l$	& 	& $w(\lambda+\rho)-\rho$\\
$1234$ 	& $0$&	& $(2a,2b,2b|2d)$\\
$1342$ 	& $2$&	& $(2a,2b-1,2d-1|2b+2)$\\
\end{tabular}
$

$P_{12,34}$: 
$\begin{tabular}{lllll}
$w$ 		& $l$	& 	& $w(\lambda+\rho)-\rho$\\
$1234$ 	& $0$&	& $(2a,2b|2b,2d)$\\
$1423$ 	& $2$&	& $(2a, 2d-2|2b+1,2b+1)$\\
$2314$ 	& $2$&	& $(2b-1,2b-1|2a+2,2d)$\\
$3412$	& $4$&	& $(2b-2,2d-2|2a+2,2b+2)$\\
\end{tabular}
$

$P_{24}$: 
$\begin{tabular}{lllll}
$w$ 		& $l$	& 	& $w(\lambda+\rho)-\rho$\\
$1234$ 	& $0$&	& $(2a|2b,2b,2d)$\\
$3124$ 	& $2$&	& $(2b-2|2a+1,2b+1,2d)$\\
\end{tabular}
$


\subsubsection{$E_1$-page. Highest weight $(2a,2b,2b,2d)$.  Case $A3$.}

{\begin{center}
\begin{small}
\scriptsize\renewcommand{\arraystretch}{2.2}
\begin{longtable}{|c|c|c|c|c|c|c|c|c|}
\hline
$E_1^{0,0}=0$ 
& $E_1^{1,0}=(2a|2b,2b|2d)$ & 
$E_1^{2,0}=(2a|2b|2b|2d)$
\\
\hline
$E_1^{0,1}=0$
	& $E_1^{1,1}=
	\left\{
	\begin{tabular}{lll}
	$(2a,2b|2b|2d)$\\
	$(2a|2b|2b,2d)$
	\end{tabular}
	\right.
	$
& $E_1^{2,1}=0$
\\
\hline
$E_1^{0,2}=(2a,2b|2b,2d)$	
& $E_1^{1,2}=0$ 
& $E_1^{2,2}=0$
\\
\hline
$E_1^{0,3}=
	\left\{
	\begin{tabular}{ll}
	$(H^3(2a,2b,2b)|2d)$\\
	$(2a|H^3(2b,2b,2d))$
	\end{tabular}
	\right.
	$
& $E_1^{1,3}=
	\left\{
	\begin{tabular}{ll}
	$(2b-1,2b-1|2a+2|2d)=0$\\
	$(2b-2|2a+1,2b+1|2d)$\\
	$(2a|2d-2|2b+1,2b+1)=0$\\
	$(2a|2b-1,2d-1|2b+2)$
	\end{tabular}
	\right.
	$
& $E_1^{2,3}=
	\left\{
	\begin{tabular}{ll}
	$(2b-2|2b|2a+2|2d)$\\
	$(2a|2d-2|2b|2b+2)$
	\end{tabular}
	\right.
	$\\
\hline
$E_1^{0,4}=
	\left\{
	\begin{tabular}{ll}
	$(H^2(2a,2b-1,2d-1)|2b+2)$\\
	$(2a,2d-2|2b+1,2b+1)=0$\\
	$(2b-1,2b-1|2a+2,2d)=0$\\
	$(2b-2|H^2(2a+1,2b+1,2d))$
	\end{tabular}
	\right.$
& $E_1^{1,4}=
	\left\{
	\begin{tabular}{ll}
	$(2a,2d-2|2b|2b+2)$\\
	$(2b-2|2a+1,2d-1|2b+2)$\\
	$(2b-2|2b|2a+2,2d)$
	\end{tabular}
	\right.$
& $E_1^{2,4}=(2b-2|2d-2|2a+2|2b+2)$\\
\hline
$E_1^{0,5}=\left\{
	\begin{tabular}{ll}
	$(H^3(2a,2b-1,2d-1)|2b+2)$\\
	$(2b-2|H^3(2a+1,2b+1,2d))$
	\end{tabular}
	\right.$
& $E_1^{1,5}=\left\{
	\begin{tabular}{ll}
	$(2b-2,2d-2|2a+2|2b+2)$\\
	$(2b-2|2d-2|2a+2,2b+2)$
	\end{tabular}
	\right.$
& $E_1^{2,5}=0$\\
\hline
$E_1^{0,6}=(2b-2,2d-2|2a+2,2b+2)$
& $E_1^{1,6}=0$
& $E_1^{1,6}=0$\\
\hline
\end{longtable}
\end{small}
\end{center}


\subsubsection{Cohomology groups of $GL_3(\Z)$ that appear on the $E_1$-page in the case $(GL_4,A3)$.}
From the computations of the cohomology of $GL_3(\Z)$ from the previous section, we have
{\begin{center}
\scriptsize\renewcommand{\arraystretch}{2.2}
\begin{longtable}{|c|c|c|}
\hline
$H^2(GL_3(\Z),(2b,2b,2d))=0$
& $H^3(GL_3(\Z),(2b,2b,2d))=(\overline{2b-1,2d-1}|2b+2)$\\
	\hline
$H^2(GL_3(\Z),(2a,2b,2b))
	=0$
& $H^3(GL_3(\Z),(2a,2b,2b))
	=(2b-2|\overline{2a+1,2b+1})$
	\\
\hline
$H^2(GL_3(\Z),(2a,2b-1,2d-1))
	=
	\Delta_1
	$
	& $H^3(GL_3(\Z),(2a,2b-1,2d-1))=(2b-2,2d-2|2a+2)$
	\\
\hline
$H^2(GL_3(\Z),(2a+1,2b+1,2d))
	=
	\Delta_2
	$
&$H^3(GL_3(\Z),(2a+1,2b+1,2d))=(2d-2|2a+2,2b+2)$\\
\hline
\end{longtable}
\end{center}
In the above Table we denote by
$\Delta_1
	\subset
	(2a,2d-2|2b)+(2b-2|\overline{2a+1,2d-1})
	$
$\Delta_2
	\subset
	(\overline{2a+1,2d-1}|2b+2)+(2b|2a+2,2d)
	$	
 diagonal embeddings.


\subsubsection{$E_2$-page. Highest weight $(2a,2b,2b,2d)$. Case $A3$}
First, we notice that $E_1^{1,0}=(2a|2b,2b|2d)$ is isomorphic to  $E_1^{2,0}=(2a|2b|2b|2d)$, because $(2b,2b)=H^0_{Eis}(GL_2(\Z),(2b,2b))=H^0(GL_1(\Z),(2b))\otimes H^0(GL_1(\Z),(2b))$ is an isomorphism. Therefore, $E_2^{p,0}=0$ for $p=0,1,2$.
For $q=1$ and $q=2$ the $d_1$ maps are trivial. Therefore, $E_2^{p,q}=E_1^{p,q}$

For $q=3$, we have a short exact sequence
\[0\rightarrow E_1^{0,3}\rightarrow E_1^{1,3}\rightarrow E_1^{2,3}\rightarrow 0.\]
It is the direct sum of two exact sequences. 
Each of them is the degree $3$ part of the long exact sequence for the boundary cohomology of $(H^3(2a,2b,2b)|2d)$ and of
	$(2a|H^3(2b,2b,2d))$.
	Therefore, $E_2^{p,3}=0$ for $p=0,1,2$

Each of the two exact sequences holds for the following reason.
Consider the $E_1^{0,3}$ term that comes from the maximal parabolic subgroup $P_{13}$. 
The corresponding summand is $(H^3(2a,2b,2b)|2d)$. 
It maps naturally to the cohomology of $P_{12}$ and $P_{23}$ which in turn map to the minimal parabolic subgroup $P_0$. 
For it we have, $H^3(GL_3(\Z),(2a,2b,2b))=H^3_{Eis}(GL_3(\Z),(2a,2b,2b))=H^3_\partial(GL_3(\Z),(2a,2b,2b))$. 
From the above computation of  $H^3_\partial(GL_3(\Z),(2a,2b,2d)$, 
we have that it maps to 
$H^3(Q_{12},(2a,2b,2d))+H^3(Q_{23},(2a,2b,2d))=(2b-1,2d-1|2a+2) + (2b-2|2a+1,2b+1)$. 
The sum maps surjectively to $H^3(Q_0,(2a,2b,2d))=(2d-2|2b|2a+2)$. Thus,
\[0\rightarrow H^3_\partial(GL_3(\Z),(2a,2b,2b)\rightarrow (2b-1,2b-1|2a+2) + (2b-2|2a+1,2b+1)\rightarrow (2b-2|2b|2a+2)\rightarrow 0\]

The cohomology of $GL_3(\Z)$ induces cohomology of $P_{13}$. From the last short exact sequence, we obtain the short exact sequence for $H^3_(P_{13},(2a,2b,2b,2d)=(H^3(2a,2b,2b)|2d)$. We have that it maps to $H^3(P_{12},(2a,2b,2b,2d))+H^3(P_{23},(2a,2b,2b,2d))=(2b-1,2b-1|2a+2|2d) + (2b-2|2a+1,2b+1|2d)$, in turn, both terms map to $H^3(P_0,(2a,2b,2b,2d))=(2b-2|2b|2a+2|2d)$. Thus the short exact sequence is 
\begin{eqnarray}
\label{P_{13},2a,2b,2b,2d}
0\rightarrow (H^3(2a,2b,2b)|2d)\rightarrow \\
\rightarrow (2b-1,2b-1|2a+2|2d) + (2b-2|2a+1,2b+1|2d)\rightarrow (2b-2|2b|2a+2|2d)\rightarrow 0.
\end{eqnarray}
Similarly, we obtain the short exact sequence 
\begin{eqnarray}
\label{P_{13},2a,2b,2b,2d}
0\rightarrow (2a|H^3(2b,2b,2d))\rightarrow\\
	\rightarrow 
	(2a|2b-1,2d-1|2b+2)+
	(2a|2d-2|2b+1,2b+1)
	\rightarrow
	(2a|2d-2|2b|2b+2)
	\rightarrow 0.
\end{eqnarray}	

The direct sum of the last two short exact sequences is exactly.
\[0\rightarrow E_1^{0,3}\rightarrow E_1^{1,3}\rightarrow E_1^{2,3}\rightarrow 0.\]
Therefore, $E_2^{0,3}=E_2^{1,3}=E_2^{2,3}=0$.

For $q=4$, we have that $E_1^{0,4}\rightarrow E_1^{1,4}$ is injective, since $\Delta_1$ and $\Delta_2$ are diagonal embeddings. 
Also, $E_1^{1,4}\rightarrow E_1^{2,4}$  is surjective. (Explain!)
Therefore, $E_2^{0,4}=E_1^{2,4}$ and $E_1^{1,4}=(2b-2|\overline{2a+1,2d-1}|2b+2)$

The argument that $E_2^{p,5}$ also vanish is similar.
 We have that $H^3(GL_3(\Z),(2a,2b-1,2d-1)=(2b-2,2d-2|2a+2)$. 
 We have also that 
 $H^3(GL_3(\Z),(2a,2b-1,2d-1)=H^3+{Eis}(GL_3(\Z),(2a,2b-1,2d-1)=H^3_\partial(GL_3(\Z),(2a,2b-1,2d-1)$, and that $H^3(Q_{12},(2a,2b-1,2d-1))=(2b-2,2d-2|2a+2)$. Therefore, the map
$H^3_\partial(GL_3(\Z),(2a,2b-1,2d-1))\rightarrow H^3(Q_{12},(2a,2b-1,2d-1))$ is an isomorphism. This isomorphism induces and isomorphism  on cohomology of parabolic subgroups of $GL_4$, namely.
$(H^3(GL_3(\Z),(2a,2b-1,2d-1)|2b+2)\rightarrow (2b-2,2d-2|2a+2|2b+2)$.
From the table with $E_1$ terms, we have that this is exactly the isomorphism
$H^5_\partial(P_{13},(2a,2b,2b,2d))\rightarrow H^5(P_{12},(2a,2b,2b,2d))$.

Similarly, one obtains that 
$H^5_\partial(P_{24},(2a,2b,2b,2d))\rightarrow H^5(P_{34},(2a,2b,2n,2d))$ is an isomorphism. 

The direct sum of the last two isomorphisms gives is
\[0\rightarrow E_1^{0,5}\rightarrow E_1^{1,5}\rightarrow E_1^{2,5}\rightarrow 0,\]
where the first nontrivial arrow is an isomorphism and $E_1^{2,5}=0$.
Therefore, $E_2^{p,5}=0$

The $E_1^{p,6}$ has only one term. Therefore, $E_2{0,6}=E_1^{0,6}=(2b-2,2d-2|2a+2,2b+2)$.



{\begin{center}
\begin{small}
\scriptsize\renewcommand{\arraystretch}{2.2}
\begin{longtable}{|c|c|c|c|c|c|c|c|c|}
\hline
& $p=0$
& $p=1$
& $p=2$\\
\hline
$q=0$
&  
&  
& $E_2^{2,0}=(2a|2b|2b|2d)$
\\
\hline
$q=1$
&
& $E_2^{1,1}=
	\left\{
	\begin{tabular}{lll}
	$(2a,2b|2b|2d)$\\
	$(2a|2b|2b,2d)$
	\end{tabular}
	\right.
	$
&
\\
\hline
$q=2$
&
$E_2^{0,2}=(2a,2b|2b,2d)
	$
&  
& 
\\
\hline
$q=3$
& 
& 
& \\
\hline
$q=4$
&
& $E_2^{1,4}=(2b-2|\overline{2a+1,2d-1}|2b+2)$
& \\
\hline
$q=5$
&
& 
&\\
\hline
$q=6$
&
$E_2^{0,6}=(2b-2,2d-2|2a+2,2b+2)$
& 
& \\
\hline
\end{longtable}
\end{small}
\end{center}


\subsubsection{Boundary cohomology. Highest weight $(2a,2b,2b,2d)$.  Case $A3$.}

By definition, the boundary cohomology is the one to which the spectral sequence converges.
That is, $\bigoplus_{prk(P)=p+1}H^q(P,V_\lambda)=>H^{p+q}_\partial(GL_4(\Z),V_\lambda)$.
From the previous subsection, it is clear that in the case 
$A1$
the spectral sequence degenerates at the $E_2$-page.
Therefore, up to semisimplicity, 
$H^{n}_\partial(GL_4(\Z),V_\lambda)=\bigoplus_{p+q=n} E_2^{p,q}$.
Since we know the $E_2$ terms which are of the form, we can find the  boundary cohomology using our computation of the $E_2$ terms
Let us denote temporarily $H^q_\partial(GL_4(\Z),A3)=H^q_\partial(GL_4(\Z),(2a,2b,2b,2d))$.
Then, we can summarize the result for the boundary cohomology in the following table.

\begin{thm}

$H^q_\partial(A3)=
\left\{
\begin{tabular}{llll}
	$0$ & $q=0$\\
	$0$ & $q=1$\\
	$\left\{
	\begin{tabular}{lll}
	$E_2^{0,2}=(2a,2b|2b,2d)$\\
	$E_2^{1,1}=
		\left\{
		\begin{tabular}{lll}
		$(2a,2b|2b|2d)$\\
		$(2a|2b|2b,2d)$
		\end{tabular}
		\right\}$\\
	$E_2^{2,0}=(2a|2b|2b|2d)$
	\end{tabular}
		\right\}$
		& $q=2$\\
	$0$ & $q=3$\\
	$0$ & $q=4$\\
	$E_2^{1,4}=(2b-2|\overline{2a+1,2d-1}|2b+2)$
	& $q=5$\\
	$E_2^{0,6}=(2b-2,2d-2|2a+2,2b+2)$ & $q=6$\\
	$0$ & $q=7$\\
	$0$ & $q=8$
\end{tabular}
\right.
$
\end{thm}
 
\begin{thm}
\begin{eqnarray*}
H^2_\partial(GL_4(\Z),(2a,2b,2b,2d))
=
&&
	S_{2a-2b+2}S_{2b-2d+2}
	+
	S_{2a-2b+2}
	+
	S_{2b-2d+2}
	+
	\C\\
H^3_\partial(GL_4(\Z),(2a,2b,2b,2d))
=
&&0\\
H^4_\partial(GL_4(\Z),(2a,2b,2b,2d))
=
&&0\\
H^5_\partial(GL_4(\Z),(2a,2b,2b,2d))=
&&
	S_{2a-2d+4}\\
H^6_\partial(GL_4(\Z),(2a,2b,2b,2d))
=
&&
	S_{2b-2d+2}S_{2a-2b+2}
\end{eqnarray*}
\end{thm}







\subsection{Cohomology of $GL_4(\Z)$. Highest weight $(2a+1,2b+1,2b+1,2d+1)$.  Case $A3'$}


\subsubsection{Cohomology of the parabolic subgroups. Highest weight $(2a+1,2b+1,2b+1,2d+1)$. Case $A3'$}

$P_0$: 
$\begin{tabular}{lllll}
$w$ 		& $l$	& 	& $w(\lambda+\rho)-\rho$\\
$2143$ 	& $2$&	& $(2b|2a+2|2d|2b+2)$\\
$2341$ 	& $3$&	& $(2b|2b|2d|2a+4)$\\
$4123$	& $3$&	& $(2d-2|2a+2|2b+2|2b+2)$\\
$4321$	& $6$&	& $(2d-2|2b|2b+2|2a+4)$\\
\end{tabular}
$

$P_{12}$: 
$\begin{tabular}{lllll}
$w$ 		& $l$	& 	& $w(\lambda+\rho)-\rho$\\
$1243$ 	& $1$&	& $(2a+1,2b+1|2d|2b+2)$\\
$1423$ 	& $2$&	& $(2a+1,2d-1|2b+2|2b+2)$\\
$2341$	& $3$&	& $(2b,2b|2d|2a+4)$\\
$3421$	& $5$&	& $(2b-1,2d-1|2b+2|2a+4)$\\
\end{tabular}
$

$P_{23}$: 
$\begin{tabular}{lllll}
$w$ 		& $l$	& 	& $w(\lambda+\rho)-\rho$\\
$2143$ 	& $2$&	& $(2b|2a+2,2d|2b+2)$\\
$2341$	& $3$&	& $(2b|2b,2d|2a+4)$\\
$4123$	& $3$&	& $(2d-2|2a+2, 2b+2|2b+2)$\\
$4231$	& $5$&	& $(2d-2|2b+1,2b+1|2a+4)$\\
\end{tabular}
$

$P_{34}$: 
$\begin{tabular}{lllll}
$w$ 		& $l$	& 	& $w(\lambda+\rho)-\rho$\\
$2134$ 	& $1$&	& $(2b|2a+2|2b+1,2d+1)$\\
$2314$	& $2$&	& $(2b|2b|2a+3,2d+1)$\\
$4123$	& $3$&	& $(2d-2|2a+2|2b+2,2b+2)$\\
$4312$ 	& $5$&	& $(2d-2|2b|2a+3,2b+3)$\\
\end{tabular}
$
$P_{13}$: 
$\begin{tabular}{lllll}
$w$ 		& $l$	& 	& $w(\lambda+\rho)-\rho$\\
$1243$ 	& $1$&	& $(2a+1,2b+1,2d|2b+2)$\\
$2341$ 	& $3$&	& $(2b,2b,2d|2a+4)$\\
\end{tabular}
$

$P_{12,34}$: 
$\begin{tabular}{lllll}
$w$ 		& $l$	& 	& $w(\lambda+\rho)-\rho$\\
$1234$ 	& $0$&	& $(2a+1,2b+1|2b+1,2d+1)$\\
$1423$ 	& $2$&	& $(2a+1, 2d-1|2b+2,2b+2)$\\
$2314$ 	& $2$&	& $(2b,2b|2a+3,2d+1)$\\
$3412$	& $4$&	& $(2b-1,2d-1|2a+3,2b+3)$\\
\end{tabular}
$

$P_{24}$: 
$\begin{tabular}{lllll}
$w$ 		& $l$	& 	& $w(\lambda+\rho)-\rho$\\
$2134$ 	& $1$&	& $(2b|2a+2,2b+1,2d+1)$\\
$4123$ 	& $3$&	& $(2d-2|2a+2,2b+2,2b+2)$\\
\end{tabular}
$


\subsubsection{$E_1$-page. Highest weight $(2a+1,2b+1,2b+1,2d+1)$.  Case $A3'$}

{\begin{center}
\begin{small}
\scriptsize\renewcommand{\arraystretch}{2.2}
\begin{longtable}{|c|c|c|c|c|c|c|c|c|}
\hline
$E_1^{0,0}=0$ 
& $E_1^{1,0}=0$ & 
$E_1^{2,0}=0$
\\
\hline
$E_1^{0,1}=0$
& $E_1^{1,1}=0$
& $E_1^{2,1}=0$
\\
\hline
$E_1^{0,2}=(2a+1,2b+1|2b+1,2d+1)$ 
& $E_1^{1,2}=
	\left\{
	\begin{tabular}{ll}
	$(2a+1,2b+1|2d|2b+2)$\\
	$(2b|2a+2|2b+1,2d+1)$
	\end{tabular}
	\right.$ 
& $E_1^{2,2}=(2b|2a+2|2d|2b+2)$
\\
\hline
$E_1^{0,3}=
	\left\{
	\begin{tabular}{ll}
	$(H^2(2a+1,2b+1,2d)|2b+2)$\\
	$(2a+1,2d-1|2b+2,2b+2)$\\
	$(2b,2b|2a+3,2d+1)$\\
	$(2b|H^2(2a+2,2b+1,2d+1))$
	\end{tabular}
	\right.
	$
& $E_1^{1,3}=
	\left\{
	\begin{tabular}{ll}
	$(2a+1,2d-1|2b+2|2b+2)$\\
	$(2b,2b|2d|2a+4)$\\
	$(2b|2a+2,2d|2b+2)$\\
	$(2b|2b|2a+3,2d+1)$\\
	$(2d-2|2a+2|2b+2,2b+2)$
	\end{tabular}
	\right.
	$
& $E_1^{2,3}=
	\left\{
	\begin{tabular}{ll}
	$(2b|2b|2d|2a+4)$\\
	$(2d-2|2a+2|2b+2|2b+2)$
	\end{tabular}
	\right.
	$\\
\hline
$E_1^{0,4}=
	\left\{
	\begin{tabular}{ll}
	$(H^3(2a+1,2b+1,2d)|2b+2)$\\
	$(2b|H^3(2a+2,2b+1,2d+1))$
	\end{tabular}
	\right.$
& $E_1^{1,4}=
	\left\{
	\begin{tabular}{ll}
	$(2b|2b,2d|2a+4)$\\
	$(2d-2|2a+2,2b+2|2b+2)$
	\end{tabular}
	\right.$
& $E_1^{2,4}=0$\\
\hline
$E_1^{0,5}=\left\{
	\begin{tabular}{ll}
	$(H^2(2b,2b,2d)|2a+4)$\\
	$(2d-2|H^2(2a+2,2b+2,2b+2))$
	\end{tabular}
	\right.$
& $E_1^{1,5}=0$
& $E_1^{2,5}=0$\\
\hline
$E_1^{0,6}=
	\left\{
	\begin{tabular}{ll}
	$(H^3(2b,2b,2d)|2a+4)$\\
	$(2b-1,2d-1|2a+3,2b+3)$\\
	$(2d-2|H^3(2a+2,2b+2,2b+2))$
	\end{tabular}
	\right.$
& $E_1^{1,6}=
	\left\{
	\begin{tabular}{ll}
	$(2b-1,2d-1|2b+2|2a+4)$\\
	$(2d-2|2b+1,2b+1|2a+4)=0$\\
	$(2d-2|2b|2a+3,2b+3)$
	\end{tabular}
	\right.$
& $E_1^{2,6}=(2d-2|2b|2b+2|2a+4)$\\
\hline
\end{longtable}
\end{small}
\end{center}


\subsubsection{Certain cohomology groups of $GL_3(\Z)$ that appear on the $E_1$-page in the case $A3'$.}
From the computations of the cohomology of $GL_3(\Z)$ from the previous section, we have
{\begin{center}
\scriptsize\renewcommand{\arraystretch}{2.2}
\begin{longtable}{|c|c|c|}
\hline
$H^2(GL_3(\Z),(2a+1,2b+1,2d))
	=
	\Delta_1
	$
&$H^3(GL_3(\Z),(2a+1,2b+1,2d))=(2d-2|2a+2,2b+2)$\\
\hline
$H^2(GL_3(\Z),(2a+2,2b+1,2d+1))
	=
	\Delta_2
	$
& $H^3(GL_3(\Z),(2a+2,2b+1,2d+1))=(2c,2d|2a+4)$
\\
\hline
$H^2(GL_3(\Z),(2b,2b,2d))=0$
& $H^3(GL_3(\Z),(2b,2b,2d)) = (\overline{2b-1,2d-1}|2b+2)$
	\\
\hline
$H^2(GL_3(\Z),(2a+2,2b+2,2b+2))=0$
& $H^3(GL_3(\Z),(2a+2,2b+2,2b+2)) = (2b|\overline{2a+3,2b+3})$\\
	\hline
\end{longtable}
\end{center}
where
$\Delta_1 \subset (\overline{2a+1,2d-1}|2b+2) + (2b|2a+2,2d)$
	and
$\Delta_2 \subset (2a+2,2d|2b+2) + (2b|\overline{2a+3,2d+1})$
are diagonal embeddings.


\subsubsection{$E_2$ page. Highest weight $(2a+1,2b+1,2b+1,2d+1)$.  Case $A3'$}

From consideration of Euler characteristics we can conclude that the map $(2a+1,2b+1)\rightarrow (2b|2a+2)$ is surjective for $a>b$, which corresponds to the Eisenstein series $E_{2a-2b+2}$ for the classical modular group $SL_2(\Z)$.

From the surjectivity of the above map, used several times, we have that the following map
$E_1^{0,2}\rightarrow ker[E_1^{1,2} \rightarrow E_1^{1,2}] $ 
is surjective.
Therefore, 
$E_2^{0,2}=(\overline{2a+1,2b+1}|\overline{2b+1,2d+1})$

Now let us examine the third line $E_1^{p,3}$ of the $E_1$-page. 
We have the following isomorphisms  coming from  the $d_1$-map from $E_1^{1,3}$ to $E_1^{2,3}$
$(2d-2|2a+2|2b+2,2b+2)\rightarrow (2d-2|2a+2|2b+2|2b+2)$ 
and
$(2b,2b|2d|2a+4)\rightarrow (2b|2b|2d|2a+4)$.
We also have the following isomorphisms coming from  the $d_1$-map from $E_1^{0,3}$ to $E_1^{1,3}$
$(2a+1,2d-1|2b+2,2b+2)\rightarrow (2a+1,2d-1|2b+2|2b+2)$ 
and
$(2b,2b|2a+3,2d+1)\rightarrow (2b|2b|2a+3,2d+1)$. 
Therefore, $E_2^{1,3}=E_2^{2,3}=0$ and 
\begin{eqnarray*}
E_2^{0,3}	=	&&ker[E_1^{0,3}\rightarrow E_1^{1,3}]=\\
		=	&&ker[(H^2(2a+1,2b+1,2d)|2b+2)+(2b|H^2(2a+2,2b+1,2d+1))\rightarrow\\
			&& \rightarrow (2b|2a+2,2d|2b+2)]=\\
		=	&& ker[\Delta_1+\Delta_2\rightarrow (2b|2a+2,2d|2b+2)]
\end{eqnarray*}
Since both $\Delta_1$ and $\Delta_2$ are isomorphic to $(2b|2a+2,2d|2b+2)$,
we obtain that $E_2^{0,3}$ is also isomorphic to $(2b|2a+2,2d|2b+2)$

From cohomology of $GL_3(\Z)$ with coefficients in $(2a+1,2b+1,2d)$, we have that 
$H^3(2a+1,2b+1,2d)=(2d-2|2a+2,2b+2)$.
Therefore,
$(H^3(2a+1,2b+1,2d)|2b+2)=(2d-2|2a+2,2b+2|2c+2)$.
Similarly, from the cohomology of $GL_3(\Z)$ with coefficients in $(2a+2,2b+1,2d+1)$l, we have that
$H^3(2a+2,2b+1,2d+1)=(2b,2d|2a+4)$. Therefore,
$(2b|H^3(2a+2,2b+1,2d+1))=(2b|2b,2d|2a+4)$.
Therefore,  $E_2^{0,4}=E_2^{1,4}=E_2^{2,4}=0$.

Since $H^2(GL_3(\Z),(2b,2b,2d))=0$
and $H^2(GL_3(\Z),(2a+2,2b+2,2b+2))=0$, we obtain that $E_1^{p,5}=0$ for $p=0,1,2$. Therefore, $E_2^{p,5}=0$ for $p=0,1,2$.

For $E_1^{p,6}$ we have that 
\[H^6(P_{12,34})\rightarrow ker\left[H^6(P_{12})+H^6(P_{34})\rightarrow H^6(P_0)\right]\] maps surjectively 
with kernel $(\overline{2b-1,2d-1}|\overline{2a+3,2b+3})$.
$E_2^{0,6}$ is isomorphic to $(\overline{2b-1,2d-1}|\overline{2a+3,2b+3}+(H^3(2b,2b,2d)|2a+4)+(2d-2|H^3(2a+2,2b+2,2b+2))
= (\overline{2b-1,2d-1}|\overline{2a+3,2b+3})
+
 (\overline{2b-1,2d-1}|2b+2|2a+4)
+
(2d-2|2b|\overline{2a+3,2b+3})$

{\begin{center}
\begin{small}
\scriptsize\renewcommand{\arraystretch}{2.2}
\begin{longtable}{|c|c|c|c|c|c|c|c|c|}
\hline
& $p=0$
& $p=1$
& $p=2$\\
\hline
$q=0$
&  
&  
& \\
\hline
$q=1$
&
& 
&
\\
\hline
$q=2$
& $E_2^{0,2}=(\overline{2a+1,2b+1}|\overline{2b+1,2d+1})$
&  
& 
\\
\hline
$q=3$
& $E_2^{0,3}=(2b|2a+2,2d|2b+2)$
&
&\\
\hline
$q=4$
& 
& 
& \\
\hline
$q=5$
& 
& 
&\\
\hline
$q=6$
&
$E_2^{0,6}=\left\{
	\begin{tabular}{lll}
	$(\overline{2b-1,2d-1}|\overline{2a+3,2b+3})$\\
	$(\overline{2b-1,2d-1}|2b+2|2a+4)$\\
	$(2d-2|2b|\overline{2a+3,2b+3})$
	\end{tabular}
	\right.$
& 
& \\
\hline
\end{longtable}
\end{small}
\end{center}


\subsubsection{Boundary cohomology. Highest weight $(2a+1,2b+1,2c+1,2d+1)$.  Case $A3'$.}
Again the boundary cohomology degenerates at the $E_2$-page. Therefore, up to semisimplicity we have
$H^n_\partial(GL_4(\Z),V_\lambda)=\bigoplus_{p+q=n}E_2^{p,q}$.
More systematically, we have

Let $H^q=H^q_\partial(GL_4(\Z),V_\lambda)$ be the cohomology of $GL_4(\Z)$ with coefficients in the highest weight representation with weight 
$\lambda=(2a+1,2b+1,2c+1,2d+1)$ written as a character of the split maximal torus. Then,

\begin{thm}
\[H^q(A3')
=
\left\{
\begin{tabular}{lll}
$E_2^{0,2}=(\overline{2a+1,2b+1}|\overline{2b+1,2d+1})$ 
		& $q=2$\\
$E_2^{0,3}=(2b|2a+2,2d|2b+2)$
		& $q=3$\\	
$0$
		& $q=4$\\
$0$
		& $q=5$\\ 
$E_2^{0,6}=\left\{
	\begin{tabular}{lll}
	$(\overline{2b-1,2d-1}|\overline{2a+3,2b+3})$\\
	$(\overline{2b-1,2d-1}|2b+2|2a+4)$\\
	$(2d-2|2b|\overline{2a+3,2b+3})$
	\end{tabular}
	\right\}$
		& $q=6$
\end{tabular}
\right.
\]
\end{thm}

\begin{thm}
\begin{eqnarray*}
H^2_\partial(GL_4(\Z),(2a+1,2b+1,2b+1,2d+1))
=
&&
	S_{2a-2b+2}S_{2b-2d+2}\\
H^3_\partial(GL_4(\Z),(2a+1,2b+1,2b+1,2d+1))
=
&&
	S_{2a-2d+4}\\
H^4_\partial(GL_4(\Z),(2a+1,2b+1,2b+1,2d+1))
=
&&
	0\\
H^5_\partial(GL_4(\Z),(2a+1,2b+1,2b+1,2d+1))
=
&&
	0\\
H^6_\partial(GL_4(\Z),(2a+1,2b+1,2b+1,2d+1))
=
&&
	S_{2b-2d+2}S_{2a-2b+2}
	+
	S_{2b-2d+2}	
	+
	S_{2a-2b+2}
\end{eqnarray*}
\end{thm}




\subsection{Cohomology of $GL_4(\Z)$. Highest weight $(2a,2a,2c,2d)$.  Case $A4$}


\subsubsection{Cohomology of the parabolic subgroups. Highest weight $(2a,2a,2c,2d)$. Case $A4$}
Let $\lambda=(2a,2a,2c,2d)$.

$P_0$: 
$\begin{tabular}{lllll}
$w$ 		& $l$	& 	& $w(\lambda+\rho)-\rho$\\
$1234$ 	& $0$&	& $(2a|2a|2c|2d)$\\
$1432$ 	& $3$&	& $(2a|2d-2|2c|2a+2)$\\
$3214$	& $3$&	& $(2c-2|2a|2a+2|2d)$\\
$3412$	& $4$&	& $(2c-2|2d-2|2a+2|2a+2)$\\
\end{tabular}
$

$P_{12}$: 
$\begin{tabular}{lllll}
$w$ 		& $l$	& 	& $w(\lambda+\rho)-\rho$\\
$1234$ 	& $0$&	& $(2a,2a|2c|2d)$\\
$1432$ 	& $3$&	& $(2a,2d-2|2c|2a+2)$\\
$2314$	& $2$&	& $(2a-1,2c-1|2a+2|2d)$\\
$3412$	& $4$&	& $(2c-2,2d-2|2a+2|2a+2)$\\
\end{tabular}
$

$P_{23}$: 
$\begin{tabular}{lllll}
$w$ 		& $l$	& 	& $w(\lambda+\rho)-\rho$\\
$1234$ 	& $0$&	& $(2a|2a,2c|2d)$\\
$1342$ 	& $2$&	& $(2a|2c-1,2d-1|2a+2)$\\
$3124$	& $2$&	& $(2c-2|2a+1, 2a+1|2d)$\\
$3142$	& $3$&	& $(2c-2|2a+1,2d-1|2a+2)$\\
\end{tabular}
$

$P_{34}$: 
$\begin{tabular}{lllll}
$w$ 		& $l$	& 	& $w(\lambda+\rho)-\rho$\\
$1234$ 	& $0$&	& $(2a|2a|2c,2d)$\\
$1423$ 	& $2$&	& $(2a|2d-2|2a+1,2c+1)$\\
$3214$	& $3$&	& $(2c-2|2a|2a+2,2d)$\\
$3412$	& $4$&	& $(2c-2|2d-2|2a+2,2a+2)$\\
\end{tabular}
$

$P_{13}$: 
$\begin{tabular}{lllll}
$w$ 		& $l$	& 	& $w(\lambda+\rho)-\rho$\\
$1234$ 	& $0$&	& $(2a,2a,2c|2d)$\\
$1342$ 	& $2$&	& $(2a,2c-1,2d-1|2a+2)$\\
\end{tabular}
$

$P_{12,34}$: 
$\begin{tabular}{lllll}
$w$ 		& $l$	& 	& $w(\lambda+\rho)-\rho$\\
$1234$ 	& $0$&	& $(2a,2a|2c,2d)$\\
$1423$ 	& $2$&	& $(2a, 2d-2|2a+1,2c+1)$\\
$2314$ 	& $2$&	& $(2a-1,2c-1|2a+2,2d)$\\
$3412$	& $4$&	& $(2c-2,2d-2|2a+2,2a+2)$\\
\end{tabular}
$

$P_{24}$: 
$\begin{tabular}{lllll}
$w$ 		& $l$	& 	& $w(\lambda+\rho)-\rho$\\
$1234$ 	& $0$&	& $(2a|2a,2c,2d)$\\
$3124$ 	& $2$&	& $(2c-2|2a+1,2a+1,2d)$\\
\end{tabular}
$


\subsubsection{$E_1$-page. Highest weight $(2a,2a,2c,2d)$.  Case $A4$.}

{\begin{center}
\begin{small}
\scriptsize\renewcommand{\arraystretch}{2.2}
\begin{longtable}{|c|c|c|c|c|c|c|c|c|}
\hline
$E_1^{0,0}=0$ 
& $E_1^{1,0}=(2a,2a|2c|2d)$ & 
$E_1^{2,0}=(2a|2a|2c|2d)$
\\
\hline
$E_1^{0,1}=(2a,2a|2c,2d)$
	& $E_1^{1,1}=
	\left\{
	\begin{tabular}{lll}
	$(2a|2a,2c|2d)$\\
	$(2a|2a|2c,2d)$
	\end{tabular}
	\right.
	$
& $E_1^{2,1}=0$
\\
\hline
$E_1^{0,2}=0$	
& $E_1^{1,2}=0$ 
& $E_1^{2,2}=0$
\\
\hline
$E_1^{0,3}=
	\left\{
	\begin{tabular}{ll}
	$(H^3(2a,2a,2c)|2d)$\\
	$(2a|H^3(2a,2c,2d))$
	\end{tabular}
	\right.
	$
& $E_1^{1,3}=
	\left\{
	\begin{tabular}{ll}
	$(2a-1,2c-1|2a+2|2d)$\\
	$(2c-2|2a+1,2a+1|2d)=0$\\
	$(2a|2d-2|2a+1,2c+1)$\\
	$(2a|2c-1,2d-1|2a+2)$
	\end{tabular}
	\right.
	$
& $E_1^{2,3}=
	\left\{
	\begin{tabular}{ll}
	$(2c-2|2a|2a+2|2d)$\\
	$(2a|2d-2|2c|2a+2)$
	\end{tabular}
	\right.
	$\\
\hline
$E_1^{0,4}=
	\left\{
	\begin{tabular}{ll}
	$(H^2(2a,2c-1,2d-1)|2a+2)$\\
	$(2a,2d-2|2a+1,2c+1)$\\
	$(2a-1,2c-1|2a+2,2d)$\\
	$(2c-2|H^2(2a+1,2a+1,2d))$
	\end{tabular}
	\right.$
& $E_1^{1,4}=
	\left\{
	\begin{tabular}{ll}
	$(2a,2d-2|2c|2a+2)$\\
	$(2c-2|2a+1,2d-1|2a+2)$\\
	$(2c-2|2a|2a+2,2d)$\\
	$(2c-2|2d-2|2a+2,2a+2)$
	\end{tabular}
	\right.$
& $E_1^{2,4}=(2c-2|2d-2|2a+2|2a+2)$\\
\hline
$E_1^{0,5}=\left\{
	\begin{tabular}{ll}
	$(H^3(2a,2c-1,2d-1)|2a+2)$\\
	$(2c-2|H^3(2a+1,2a+1,2d))$\\
	$(2c-2,2d-2|2a+2,2a+2)$
	\end{tabular}
	\right.$
& $E_1^{1,5}=(2c-2,2d-2|2a+2|2a+2)$
& $E_1^{2,5}=0$\\
\hline
$E_1^{0,6}=0$
& $E_1^{1,6}=0$
& $E_1^{1,6}=0$\\
\hline
\end{longtable}
\end{small}
\end{center}


\subsubsection{Certain cohomology groups of $GL_3(\Z)$ that appear on the $E_1$-page in the case $A4$.}
From the computations of the cohomology of $GL_3(\Z)$ from the previous section, we have
{\begin{center}
\scriptsize\renewcommand{\arraystretch}{2.2}
\begin{longtable}{|c|c|c|}
\hline
$H^2(GL_3(\Z),(2a,2a,2c))=0$
& $H^3(GL_3(\Z),(2a,2a,2c))
	=(\overline{2a-1,2c-1}|2a+2)$\\
	\hline
$H^2(GL_3(\Z),(2a,2c,2d))=0$
& $H^3(GL_3(\Z),(2a,2c,2d))
	=
	\left\{	\begin{tabular}{llll}
	$(\overline{2c-1,2d-1}|2c+2)$\\
	$(2d-2|\overline{2a+1,2c+1})$\\
	$(2d-2|2c|2a+2)$
	\end{tabular}
	\right.$\\
\hline
$H^2(GL_3(\Z),(2a,2c-1,2d-1))
	=
	\Delta$
	& $H^3(GL_3(\Z),(2a,2c-1,2d-1))=(2c-2,2d-2|2a+2)$
\\
\hline
$H^2(GL_3(\Z),(2a+1,2a+1,2d))
	=
	\Delta_0+\Delta_1$
&$H^3(GL_3(\Z),(2a+1,2a+1,2d))=0$\\
\hline

\end{longtable}
\end{center}
where
$\Delta
	\subset (2a,2d-2|2c)+(2c-2|\overline{2a+1,2d-1})$,	
$\Delta_1
	\subset 
	(\overline{2a+1,2d-1}|2a+2) + (2a|2a+2,2d)$
	and
$\Delta_0\subset (2a|2a+2|2d)+(2d-2|2a+2|2a+2)$
are diagonal embeddings. 


\subsubsection{$E_2$-page. Highest weight $(2a,2a,2c,2d)$. Case $A4$}

For $q=0$, we have that $d_1:E_1^{1,0}\rightarrow  E_1^{2,0}$ is an isomorphism. Therefore $E_2^{p,0}=0$ for $p=0,1,2$.

For $q=1$, we have an isomorphism $E_1^{0,1}=(2a,2a|2c,2d)\rightarrow (2a|2a|2c,2d)\subset E_1^{1,1}$. 
Thus, $E_2^{0,1}=(2a|2a,2c|2d)$ and $E_2^{p,1}=0$ for $p=0,2$

For $q=2$ all $E_1^{p,2}$ are zero, therefore $E_2^{p,2}=0$.

For $q=3$, we have a short exact sequence
$0\rightarrow (H^3(2a,2a,2c)|2d)\rightarrow (2a-1,2c-1|2a+2|2d) \rightarrow (2c-2|2a|2a+2|2d)\rightarrow 0$.
We also have another short exact sequence
$0\rightarrow (2a|H^3(2a,2c,2d)) \rightarrow (2a|2d-2|2a+1,2c+1) + (2a|2c-1,2d-1|2a+2) \rightarrow (2a|2d-2|2c|2a+2) \rightarrow 0$
Therefore $E_2^{p,3}=0$ for $p=0,1,2$.

For $q=4$
The map $(2a,2d-2|2a+1,2c+1)\rightarrow (2a,2d-2|2c|2a+2)$ is surjective and the map 
$H^2(2a,2c-1,2d-1)|2a+2)\rightarrow ker[(2a,2d-2|2c|2a+2)+(2c-2|\overline{2a+1,2d-1}|2a+2)$ is a diagonal embedding.
Therefore
$(H^2(2a,2c-1,2d-1)|2a+2)+(2a,2d-2|2a+1,2c+1)\rightarrow ker[(2a,2d-2|2c|2a+2)+(2c-2|2a+1,2d-1|2a+2)\rightarrow  (2c-2|2d-2|2a+2|2a+2)]$
is surjective and the kernel is $E_2^{0,4}=(2a,2d-2|\overline{2a+1,2c+1})$.
the other $E_2^{p,4}=0$ for $p=1,2$.

For $q=5$, we have $E_2^{0,5}=ker[E_1^{0,5}\rightarrow E_1^{1,5}]$. 
We have $E_1^{0,5}=(H3(GL3(Z),(2a,2c - 1,2d - 1))|2a+2) + (2c-2|H^3(G_3(\Z),(2a+1,2a+1,2d))) + (2c-2,2d-2|2a+2,2a+2)$ and
$E_1^{1,5}=(2c - 2,2d - 2|2a + 2|2a+2)$.
Using that
$(H^3(GL_3(\Z),(2a,2c - 1,2d -1))|2a+2) = (2c - 2,2d - 2|2a + 2|2a+2)$ and $(2c-2|H^3(GL_3(\Z),(2a+1,2a+1,2d))=0$,
we obtain that 
$E_2^{0,5}=(2c-2,2d-2|2a+2,2a+2)$. Combining all the computations for the $E_2$ page, we obtain the following table.

{\begin{center}
\begin{small}
\scriptsize\renewcommand{\arraystretch}{2.2}
\begin{longtable}{|c|c|c|c|c|c|c|c|c|}
\hline
& $p=0$
& $p=1$
& $p=2$\\
\hline
$q=0$
&  
&  
& 
\\
\hline
$q=1$
&
& $E_2^{1,1}=(2a|2a,2c|2d)$
&
\\
\hline
$q=2$
&
&  
& 
\\
\hline
$q=3$
&
& 
& \\
\hline
$q=4$
&$E_2^{0,4}=(2a,2d-2|\overline{2a+1,2c+1})$	
&
& \\
\hline
$q=5$
& $E_2^{0,5}=(2c-2,2d-2|2a+2,2a+2)$
& 
&\\
\hline
$q=6$
&
& 
& \\
\hline
\end{longtable}
\end{small}
\end{center}


\subsubsection{Boundary cohomology. Highest weight $(2a,2a,2c,2d)$.  Case $A4$.}

By definition, the boundary cohomology is the one to which the spectral sequence converges.
That is, $\bigoplus_{prk(P)=p+1}H^q(P,V_\lambda)=>H^{p+q}_\partial(GL_4(\Z),V_\lambda)$.
From the previous subsection, it is clear that in the case 
$A4$
the spectral sequence degenerates at the $E_2$-page.
Therefore, up to semisimplicity, 
$H^{n}_\partial(GL_4(\Z),V_\lambda)=\bigoplus_{p+q=n} E_2^{p,q}$.
Since we know the $E_2$ terms which are of the form, we can find the  boundary cohomology using our computation of the $E_2$ terms
Let us denote temporarily $H^q_\partial(GL_4(\Z),A4)=H^q_\partial(GL_4(\Z),(2a,2a,2c,2d))$.
Then, we can summarize the result for the boundary cohomology in the following table.

\begin{thm}

$H^q_\partial(GL_4(\Z),A4)=
\left\{
\begin{tabular}{llll}
	$0$ & $q=0$\\
	$0$ & $q=1$\\
	$E_2^{1,1}=(2a|2a,2c|2d)$ 
		& $q=2$\\
	$0$ & $q=3$\\
	$E_2^{0,4}=(2a,2d-2|\overline{2a+1,2c+1})+(2a,2c|\overline{2a+3,2d+1})+(2a|2a+2,2d|2c+2)$
		& $q=4$\\
	$E_2^{0,5}=(2c-2,2d-2|2a+2,2a+2)$ & $q=5$\\
	$0$ & $q=6$\\
	$0$ & $q=7$\\
	$0$ & $q=8$
\end{tabular}
\right.
$
\end{thm}

\begin{thm}
\begin{eqnarray*}
H^2_\partial(GL_4(\Z),(2a,2a,2c,2d))
=
&&
	S_{2a-2c+2}\\
H^3_\partial(GL_4(\Z),(2a,2a,2c,2d))
=
&&
	0\\
H^4_\partial(GL_4(\Z),(2a,2a,2c,2d))
=
&&
	S_{2a-2d+4}S_{2a-2c+2}
	+S_{2a-2d+4}S_{2a-2c+2}
	+S_{2a-2d+4}\\
H^5_\partial(GL_4(\Z),(2a,2a,2c,2d))
=
&&
	S_{2c-2d+2}\\
H^6_\partial(GL_4(\Z),(2a,2a,2c,2d))
=
&&
0
\end{eqnarray*}

\end{thm}




\subsection{Cohomology of $GL_4(\Z)$. Highest weight $(2a+1,2a+1,2c+1,2d+1)$.  Case $A4'$}


\subsubsection{Cohomology of the parabolic subgroups. Highest weight $(2a+1,2a+1,2c+1,2d+1)$. Case $A4'$}
For that case the weights are $(2a+1,2a+1,2c+1,2d+1)$.

$P_0$: 
$\begin{tabular}{lllll}
$w$ 		& $l$	& 	& $w(\lambda+\rho)-\rho$\\
$2143$ 	& $2$&	& $(2a|2a+2|2d|2c+2)$\\
$2341$ 	& $3$&	& $(2a|2c|2d|2a+4)$\\
$4123$	& $3$&	& $(2d-2|2a+2|2a+2|2c+2)$\\
$4321$	& $6$&	& $(2d-2|2c|2a+2|2a+4)$\\
\end{tabular}
$

$P_{12}$: 
$\begin{tabular}{lllll}
$w$ 		& $l$	& 	& $w(\lambda+\rho)-\rho$\\
$1243$ 	& $1$&	& $(2a+1,2a+1|2d|2c+2)$\\
$1423$ 	& $2$&	& $(2a+1,2d-1|2a+2|2c+2)$\\
$2341$	& $3$&	& $(2a,2c|2d|2a+4)$\\
$3421$	& $5$&	& $(2c-1,2d-1|2a+2|2a+4)$\\
\end{tabular}
$

$P_{23}$: 
$\begin{tabular}{lllll}
$w$ 		& $l$	& 	& $w(\lambda+\rho)-\rho$\\
$2143$ 	& $2$&	& $(2a|2a+2,2d|2c+2)$\\
$2341$	& $3$&	& $(2a|2c,2d|2a+4)$\\
$4123$	& $3$&	& $(2d-2|2a+2, 2a+2|2c+2)$\\
$4231$	& $5$&	& $(2d-2|2a+1,2c+1|2a+4)$\\
\end{tabular}
$

$P_{34}$: 
$\begin{tabular}{lllll}
$w$ 		& $l$	& 	& $w(\lambda+\rho)-\rho$\\
$2134$ 	& $1$&	& $(2a|2a+2|2c+1,2d+1)$\\
$2314$	& $2$&	& $(2a|2c|2a+3,2d+1)$\\
$4123$	& $3$&	& $(2d-2|2a+2|2a+2,2c+2)$\\
$4312$ 	& $5$&	& $(2d-2|2c|2a+3,2a+3)$\\
\end{tabular}
$

$P_{13}$: 
$\begin{tabular}{lllll}
$w$ 		& $l$	& 	& $w(\lambda+\rho)-\rho$\\
$1243$ 	& $1$&	& $(2a+1,2a+1,2d|2c+2)$\\
$2341$ 	& $3$&	& $(2a,2c,2d|2a+4)$\\
\end{tabular}
$

$P_{12,34}$: 
$\begin{tabular}{lllll}
$w$ 		& $l$	& 	& $w(\lambda+\rho)-\rho$\\
$1234$ 	& $0$&	& $(2a+1,2a+1|2c+1,2d+1)$\\
$1423$ 	& $2$&	& $(2a+1, 2d-1|2a+2,2c+2)$\\
$2314$ 	& $2$&	& $(2a,2c|2a+3,2d+1)$\\
$3412$	& $4$&	& $(2c-1,2d-1|2a+3,2a+3)$\\
\end{tabular}
$

$P_{24}$: 
$\begin{tabular}{lllll}
$w$ 		& $l$	& 	& $w(\lambda+\rho)-\rho$\\
$2134$ 	& $1$&	& $(2a|2a+2,2c+1,2d+1)$\\
$4123$ 	& $3$&	& $(2d-2|2a+2,2a+2,2c+2)$\\
\end{tabular}
$


\subsubsection{$E_1$-page. Highest weight $(2a+1,2a+1,2c+1,2d+1)$.  Case $A1'$}

{\begin{center}
\begin{small}
\scriptsize\renewcommand{\arraystretch}{2.2}
\begin{longtable}{|c|c|c|c|c|c|c|c|c|}
\hline
$E_1^{0,0}=0$ 
& $E_1^{1,0}=0$ & 
$E_1^{2,0}=0$
\\
\hline
$E_1^{0,1}=0$
& $E_1^{1,1}=0$
& $E_1^{2,1}=0$
\\
\hline
$E_1^{0,2}=(2a+1,2a+1|2c+1,2d+1)=0$ 
& $E_1^{1,2}=
	\left\{
	\begin{tabular}{ll}
	$(2a+1,2a+1|2d|2c+2)=0$\\
	$(2a|2a+2|2c+1,2d+1)$
	\end{tabular}
	\right.$ 
& $E_1^{2,2}=(2a|2a+2|2d|2c+2)$
\\
\hline
$E_1^{0,3}=
	\left\{
	\begin{tabular}{ll}
	$(H^2(2a+1,2a+1,2d)|2c+2)$\\
	$(2a|H^2(2a+2,2c+1,2d+1))$
	\end{tabular}
	\right.
	$
& $E_1^{1,3}=
	\left\{
	\begin{tabular}{ll}
	$(2a+1,2d-1|2a+2|2c+2)$\\
	$(2a|2a+2,2d|2c+2)$\\
	$(2d-2|2a+2,2a+2|2c+2)$\\
	$(2a|2c|2a+3,2d+1)$
	\end{tabular}
	\right.
	$
& $E_1^{2,3}=
	\left\{
	\begin{tabular}{ll}
	$(2a|2c|2d|2a+4)$\\
	$(2d-2|2a+2|2a+2|2c+2)$
	\end{tabular}
	\right.
	$\\
\hline
$E_1^{0,4}=
	\left\{
	\begin{tabular}{ll}
	$(H^3(2a+1,2a+1,2d)|2c+2)$\\
	$(2a+1,2d-1|2a+2,2c+2)$\\
	$(2a,2c|2a+3,2d+1)$\\
	$(2a|H^3(2a+2,2c+1,2d+1))$
	\end{tabular}
	\right.$
& $E_1^{1,4}=
	\left\{
	\begin{tabular}{ll}
	$(2a,2c|2d|2a+4)$\\
	$(2a|2c,2d|2a+4)$\\	
	$(2d-2|2a+2|2a+2,2c+2)$
	\end{tabular}
	\right.$
& $E_1^{2,4}=0$\\
\hline
$E_1^{0,5}=\left\{
	\begin{tabular}{ll}
	$(H^2(2a,2c,2d)|2a+4)$\\
	$(2d-2|H^2(2a+2,2a+2,2c+2))$
	\end{tabular}
	\right.$
& $E_1^{1,5}=0$
& $E_1^{2,5}=0$\\
\hline
$E_1^{0,6}=
	\left\{
	\begin{tabular}{ll}
	$(H^3(2a,2c,2d)|2a+4)$\\
	$(2c-1,2d-1|2a+3,2a+3)=0$\\
	$(2d-2|H^3(2a+2,2a+2,2c+2))$
	\end{tabular}
	\right.$
& $E_1^{1,6}=
	\left\{
	\begin{tabular}{ll}
	$(2c-1,2d-1|2a+2|2a+4)$\\
	$(2d-2|2a+1,2c+1|2a+4)$\\
	$(2d-2|2c|2a+3,2a+3)=0$
	\end{tabular}
	\right.$
& $E_1^{2,6}=(2d-2|2c|2a+2|2a+4)$\\
\hline
\end{longtable}
\end{small}
\end{center}


\subsubsection{Certain cohomology groups of $GL_3(\Z)$ that appear on the $E_1$-page in the case $A4'$.}
From the computations of the cohomology of $GL_3(\Z)$ from the previous section, we have
{\begin{center}
\scriptsize\renewcommand{\arraystretch}{2.2}
\begin{longtable}{|c|c|c|}
\hline
$H^2(GL_3(\Z),(2a+1,2a+1,2d))
	=
	\Delta_0+\Delta_1
	$
&$H^3(GL_3(\Z),(2a+1,2a+1,2d))=0$\\
\hline
$H^2(GL_3(\Z),(2a+2,2c+1,2d+1))
	=
	\Delta_2
	$
& $H^3(GL_3(\Z),(2a+2,2c+1,2d+1))=(2c,2d|2a+4)$
\\
\hline
$H^2(GL_3(\Z),(2a,2c,2d))=0$
& $H^3(GL_3(\Z),(2a,2c,2d))
	=
	\left\{	\begin{tabular}{llll}
	$(\overline{2c-1,2d-1}|2c+2)$\\
	$(2d-2|\overline{2a+1,2c+1})$\\
	$(2d-2|2c|2a+2)$
	\end{tabular}
	\right.$\\
\hline
$H^2(GL_3(\Z),(2a+2,2a+2,2c+2))=0$
& $H^3(GL_3(\Z),(2a+2,2a+2,2c+2))
	=
	(\overline{2a+1,2c+1}|2a+4)$\\
	\hline
\end{longtable}
\end{center}
where
$\Delta_0=(2a|2a+2|2d)+(2d-2|2a+2|2a+2)$
$\Delta_1 = (\overline{2a+1,2d-1}|2a+2) + (2a|2a+2,2d)$.
	and
$\Delta_2 = (2a+2,2d|2c+2) + (2c|\overline{2a+3,2d+1})$.


\subsubsection{$E_2$ page. Highest weight $(2a+1,2a+1,2c+1,2d+1)$.  Case $A4'$}

For $q=2$ we compute in the following way.
From the surjectivity of the  map $E_1^{1,2} \rightarrow E_1^{1,2}$,
we find that the kernel is $E_2^{1,2}=(2a|2a+2|\overline{2c+1,2d+1})$.
Also, $E_2^{0,2}=E_1^{2,2} =0$.

For $q=3$,we obtain that $E_1^{1,3}\rightarrow E_1^{2,3}$ is surjective because 
$(2a+1,2d-1|2a+2|2c+2)\rightarrow (2d-2|2a+2|2a+2|2c+2)$
and	$(2a|2c|2a+3,2d+1)\rightarrow (2a|2c|2d|2a+4)$ are both surjective.
Also $\Delta_0$ maps isomorphically to $(2d-2|2a+2,2a+2|2c+2)$.
Finally both $\Delta_1$ and $\Delta_2$ embed diagonally in 
$(\overline{2a+1,2d-1}|2a+2|2c+2)+ (2a|2a+2,2d|2c+2)$ and 
$(2a|2a+2,2d|2c+2)+(2a|2c|\overline{2a+3,2d+1})$, respectively. Therefore $E_2^{1,3}$ up to semisimplicity is isomorphic to $(2a|2a+2,2d|2c+2)$.
Also $E_2^{0,3}=E_2^{2,3}=0$.

For $q=4$,
we have that the maps 
$(2a+1,2d-1|2a+2,2c+2)\rightarrow (2d-2|2a+2|2a+2,2c-2)$ 
and
$(2a,2c|2a+3,2d+1)\rightarrow (2a,2c|2d,2a+4)$ 
are surjective.
From cohomology of $GL_3(\Z)$ with coefficients in $(2a+1,2a+1,2d)$, we have that 
$H^3(2a+1,2a+1,2d)=0$.
Therefore,
$(H^3(2a+1,2a+1,2d)|2c+2)=0$.
Similarly, from the cohomology of $GL_3(\Z)$ with coefficients in $(2a+2,2c+1,2d+1)$l, we have that
$H^3(2a+2,2c+1,2d+1)=(2c,2d|2a+4)$. Therefore,
$(2a|H^3(2a+2,2c+1,2d+1))=(2a|2c,2d|2a+4)$.
Then, $E_2^{0,4}=(\overline{2a+1,2d-1}|2a+2,2c+2) + (2a,2c|\overline{2a+3,2d+1})$
and $E_2^{1,4}=E_2^{2,4}=0$.

For  $q=6$, we have that the following.
From the cohomology of $GL_3(\Z)$ with coefficients in $(2a,2c,2d)$, we have that 
$H^3(2a,2c,2d)=(\overline{2c-1,2d-1}|2a+2)+(2d-2|\overline{2a+1,2c+1})+(2d-2|2c|2a+2)$.
Therefore, 
$(H^3(2a,2c,2d)|2a+4)=(\overline{2c-1,2d-1}|2a+2|2a+4)+(2d-2|\overline{2a+1,2c+1}|2a+4)+(2d-2|2c|2a+2|2a+4)$.
Then
$E_2^{0,6}$ is isomorphic to $H^6(P_{24})=(2d-2|H^3(2a+2,2a+2,2c+2))=(2d-2|\overline{2a+1,2c+1}|2a+4)$. 
Also,
$E_2^{1,6}=E_2^{2,6}=0$.

This is summrized in the following table. The empty boxes denote that the corresponding $E_2^{p,q}$ term vanishes.

{\begin{center}
\begin{small}
\scriptsize\renewcommand{\arraystretch}{2.2}
\begin{longtable}{|c|c|c|c|c|c|c|c|c|}
\hline
& $p=0$
& $p=1$
& $p=2$\\
\hline
$q=0$
&  
&  
& \\
\hline
$q=1$
&
& 
&
\\
\hline
$q=2$
& 
& $E_2^{1,2}=(2a|2a+2|\overline{2c+1,2d+1})$
& 
\\
\hline
$q=3$
& 
& $E_2^{1,3}=(2a|2a+2,2d|2c+2)$
&\\
\hline
$q=4$
& $E_2^{0,4}=
	\left\{
	\begin{tabular}{lll}
	$(\overline{2a+1,2d-1}|2a+2,2c+2)$\\
	$(2a,2c|\overline{2a+3,2d+1})$
	\end{tabular}
	\right.$
& 
& \\
\hline
$q=5$
& 
& 
&\\
\hline
$q=6$
&
$E_2^{0,6}=(2d-2|\overline{2a+1,2c+1}|2a+4)$
& 
& \\
\hline
\end{longtable}
\end{small}
\end{center}


\subsubsection{Boundary cohomology. Highest weight $(2a+1,2a+1,2c+1,2d+1)$.  Case $A4'$.}
Again the boundary cohomology degenerates at the $E_2$-page. Therefore, up to semisimplicity we have
$H^n_\partial(GL_4(\Z),V_\lambda)=\bigoplus_{p+q=n}E_2^{p,q}$.
More systematically, we have

Let $H^q=H^q_\partial(GL_4(\Z),V_\lambda)$ be the cohomology of $GL_4(\Z)$ with coefficients in the highest weight representation with weight 
$\lambda=(2a+1,2a+1,2c+1,2d+1)$ written as a character of the split maximal torus. Then,

\begin{thm}
\[H^q
=
\left\{
\begin{tabular}{lll}
$0$
		& $q=2$\\
$E_2^{1,2}=(2a|2a+2|\overline{2c+1,2d+1})$
		& $q=3$\\	
$E_2^{0,4}=
	\left\{
	\begin{tabular}{lll}
	$(\overline{2a+1,2d-1}|2a+2,2c+2)$\\
	$(2a,2c|\overline{2a+3,2d+1})$
	\end{tabular}
	\right\}$
		& $q=4$\\
$0$
		& $q=5$\\ 
$E_2^{0,6}=(2d-2|\overline{2a+1,2c+1}|2a+4)$
		& $q=6$
\end{tabular}
\right.
\]
\end{thm}

Note that $H^3$ is a potentially ghost space, since it comes from $E_{p,q}$ with $p>0$. This will we important for computation of cohomology of $GL_5(\Z)$ when one considers a maximal parabolic subgroup $P_{14}$ or $P_{25}$ that contain $GL_4$. For those parabolic subgroup we might need to consider cohomology of $P_{14}$ which up to semisimplicity is cohomology of $GL_4(\Z)$. When the representation of $GL_4$ is with highest weight $(2a+1,2a+1,2c+1,2d+1)$, then the existence of potentially ghost class leads to existence of nontrivial $d_2$-map from $E_2^{0,q}$ to $E^{2,q-1}$.

\begin{thm}
\begin{eqnarray*}
H^2_\partial(GL_4(\Z),(2a+1,2a+1,2c+1,2d+1))=
&&
	0\\
H^3_\partial(GL_4(\Z),(2a+1,2a+1,2c+1,2d+1))=
&&
	S_{2c-2d+2}\\
H^4_\partial(GL_4(\Z),(2a+1,2a+1,2c+1,2d+1))=
&&
	S_{2a-2d+4}S_{2a-2c+2}
	+
	S_{2a-2c+2}S_{2a-2d+4}\\
H^5_\partial(GL_4(\Z),(2a+1,2a+1,2c+1,2d+1))=
&&
	0\\
H^6_\partial(GL_4(\Z),(2a+1,2a+1,2c+1,2d+1))=
&&
	S_{2a-2c+2}
\end{eqnarray*}
\end{thm}




\subsection{Cohomology of $GL_4(\Z)$. Highest weight $(2a,2b,2b,2b)$.  Case $A5$}


\subsubsection{Cohomology of the parabolic subgroups. Highest weight $(2a,2b,2b,2b)$. Case $A5$}
Let $\lambda=(2a,2b,2b,2b)$. This representation is the $(2a-2b)$-symmetric power of the standard representation of $GL_4$.

The choice of Weyl elements depend on the parity of the highest weight. Therefore, we can use the weights from the case $A3$. For them we have

$P_0$: 
$\begin{tabular}{lllll}
$w$ 		& $l$	& 	& $w(\lambda+\rho)-\rho$\\
$1234$ 	& $0$&	& $(2a|2b|2b|2b)$\\
$1432$ 	& $3$&	& $(2a|2b-2|2b|2b+2)$\\
$3214$	& $3$&	& $(2b-2|2b|2a+2|2b)$\\
$3412$	& $4$&	& $(2b-2|2b-2|2a+2|2b+2)$\\
\end{tabular}
$

$P_{12}$: 
$\begin{tabular}{lllll}
$w$ 		& $l$	& 	& $w(\lambda+\rho)-\rho$\\
$1234$ 	& $0$&	& $(2a,2b|2b|2b)$\\
$1432$ 	& $3$&	& $(2a,2b-2|2b|2b+2)$\\
$2314$	& $2$&	& $(2b-1,2b-1|2a+2|2b)$\\
$3412$	& $4$&	& $(2b-2,2b-2|2a+2|2b+2)$\\
\end{tabular}
$

$P_{23}$: 
$\begin{tabular}{lllll}
$w$ 		& $l$	& 	& $w(\lambda+\rho)-\rho$\\
$1234$ 	& $0$&	& $(2a|2b,2b|2b)$\\
$1342$ 	& $2$&	& $(2a|2b-1,2b-1|2b+2)$\\
$3124$	& $2$&	& $(2b-2|2a+1, 2b+1|2b)$\\
$3142$	& $3$&	& $(2b-2|2a+1,2b-1|2b+2)$\\
\end{tabular}
$

$P_{34}$: 
$\begin{tabular}{lllll}
$w$ 		& $l$	& 	& $w(\lambda+\rho)-\rho$\\
$1234$ 	& $0$&	& $(2a|2b|2b,2b)$\\
$1423$ 	& $2$&	& $(2a|2b-2|2b+1,2b+1)$\\
$3214$	& $3$&	& $(2b-2|2b|2a+2,2b)$\\
$3412$	& $4$&	& $(2b-2|2b-2|2a+2,2b+2)$\\
\end{tabular}
$

$P_{13}$: 
$\begin{tabular}{lllll}
$w$ 		& $l$	& 	& $w(\lambda+\rho)-\rho$\\
$1234$ 	& $0$&	& $(2a,2b,2b|2b)$\\
$1342$ 	& $2$&	& $(2a,2b-1,2b-1|2b+2)$\\
\end{tabular}
$

$P_{12,34}$: 
$\begin{tabular}{lllll}
$w$ 		& $l$	& 	& $w(\lambda+\rho)-\rho$\\
$1234$ 	& $0$&	& $(2a,2b|2b,2b)$\\
$1423$ 	& $2$&	& $(2a, 2b-2|2b+1,2b+1)$\\
$2314$ 	& $2$&	& $(2b-1,2b-1|2a+2,2b)$\\
$3412$	& $4$&	& $(2b-2,2b-2|2a+2,2b+2)$\\
\end{tabular}
$

$P_{24}$: 
$\begin{tabular}{lllll}
$w$ 		& $l$	& 	& $w(\lambda+\rho)-\rho$\\
$1234$ 	& $0$&	& $(2a|2b,2b,2b)$\\
$3124$ 	& $2$&	& $(2b-2|2a+1,2b+1,2b)$\\
\end{tabular}
$


\subsubsection{$E_1$-page. Highest weight $(2a,2b,2b,2b)$.  Case $A5$.}

{\begin{center}
\begin{small}
\scriptsize\renewcommand{\arraystretch}{2.2}
\begin{longtable}{|c|c|c|c|c|c|c|c|c|}
\hline
$E_1^{0,0}=(2a|2b,2b,2b)$ 
& $E_1^{1,0}=(2a|2b,2b|2b) +(2a|2b|2b,2b)$ 
& $E_1^{2,0}=(2a|2b|2b|2b)$
\\
\hline
$E_1^{0,1}=(2a,2b|2b,2b)$
& $E_1^{1,1}=(2a,2b|2b|2b)$
& $E_1^{2,1}=0$
\\
\hline
$E_1^{0,2}=0$	
& $E_1^{1,2}=0$ 
& $E_1^{2,2}=0$
\\
\hline
$E_1^{0,3}=(H^3(2a,2b,2b)|2b)$
& $E_1^{1,3}=
	\left\{
	\begin{tabular}{ll}
	$(2b-1,2b-1|2a+2|2b)=0$\\
	$(2b-2|2a+1,2b+1|2b)$\\
	$(2a|2b-2|2b+1,2b+1)=0$\\
	$(2a|2b-1,2b-1|2b+2)=0$
	\end{tabular}
	\right.
	$
& $E_1^{2,3}=
	\left\{
	\begin{tabular}{ll}
	$(2b-2|2b|2a+2|2b)$\\
	$(2a|2b-2|2b|2b+2)$
	\end{tabular}
	\right.
	$\\
\hline
$E_1^{0,4}=
	\left\{
	\begin{tabular}{ll}
	$(H^2(2a,2b-1,2b-1)|2b+2)$\\
	$(2a,2b-2|2b+1,2b+1)=0$\\
	$(2b-1,2b-1|2a+2,2b)=0$\\
	$(2b-2|H^2(2a+1,2b+1,2b))$
	\end{tabular}
	\right.$
& $E_1^{1,4}=
	\left\{
	\begin{tabular}{ll}
	$(2a,2b-2|2b|2b+2)$\\
	$(2b-2,2b-2|2a+2|2b+2)$\\
	$(2b-2|2a+1,2b-1|2b+2)$\\
	$(2b-2|2b|2a+2,2b)$
	\end{tabular}
	\right.$
& $E_1^{2,4}=(2b-2|2b-2|2a+2|2b+2)$\\
\hline
$E_1^{0,5}=\left\{
	\begin{tabular}{ll}
	$(2b-2,2b-2|2a+2,2b+2)$\\
	$(H^3(2a,2b-1,2b-1)|2b+2)$\\
	$(2b-2|H^3(2a+1,2b+1,2b))$
	\end{tabular}
	\right.$
& $E_1^{1,5}=(2b-2|2b-2|2a+2,2b+2)$
& $E_1^{2,5}=0$\\
\hline
$E_1^{0,6}=0$
& $E_1^{1,6}=0$
& $E_1^{1,6}=0$\\
\hline
\end{longtable}
\end{small}
\end{center}


\subsubsection{Cohomology groups of $GL_3(\Z)$ that appear on the $E_1$-page in the case $(GL_4,A5)$.}
From the computations of the cohomology of $GL_3(\Z)$ from the previous section, we have
{\begin{center}
\scriptsize\renewcommand{\arraystretch}{2.2}
\begin{longtable}{|c|c|c|}
\hline
$H^2(GL_3(\Z),(2b,2b,2b))=0$
& $H^3(GL_3(\Z),(2b,2b,2b))=0$\\
	\hline
$H^2(GL_3(\Z),(2a,2b,2b))
	=0$
& $H^3(GL_3(\Z),(2a,2b,2b))
	=(2b-2|\overline{2a+1,2b+1})$
	\\
\hline
$H^2(GL_3(\Z),(2a,2b-1,2b-1))
	=
	\Delta_1 
	+ \Delta_0
	$
	& $H^3(GL_3(\Z),(2a,2b-1,2b-1))=0$
	\\
\hline
$H^2(GL_3(\Z),(2a+1,2b+1,2b))
	=
	\Delta_2
	$
&$H^3(GL_3(\Z),(2a+1,2b+1,2b))=(2b-2|2a+2,2b+2)$\\
\hline
\end{longtable}
\end{center}
In the above Table we denote by
$\Delta_0\subset (2a|2b-2|2b)+(2b-2|2b-2|2a+2)$
$\Delta_1
	\subset
	(2a,2b-2|2b)+(2b-2|\overline{2a+1,2b-1})
	$
$\Delta_2
	\subset
	(\overline{2a+1,2b-1}|2b+2)+(2b|2a+2,2b)
	$	
 diagonal embeddings.

Also, $(2b,2b,2b)=\C$. Therefore, $H^0(GL_3(\Z),(2b,2b,2b))=H^0(GL_3(\Z),\C)=\C$


\subsubsection{$E_2$-page. Highest weight $(2a,2b,2b,2b)$. Case $A5$}
For $q=0$, we notice that
$0\rightarrow (2a|2b,2b,2b)\rightarrow  (2a|2b,2b|2b) (2a|2b|2b,2b) \rightarrow  (2a|2b|2b|2b)\rightarrow  0$ is a short exact sequence.
 Therefore, $E_2^{p,0}=0$ for $p=0,1,2$.
For $q=1$, we have that $d_1:E_1^{0,1}\rightarrow E_1^{1,1}$ is an isomorphism. Therefore,
 $E_2^{p,1}=0$ for $p=0,1,2$.
For $q=2$ all $E_1^{p,2}=0$ Therefore, $E_2^{p,q}=E_1^{p,q}=0$

For $q=3$, we have a short exact sequence
\[0\rightarrow E_1^{0,3}\rightarrow E_1^{1,3}\rightarrow (2b - 2|2b|2a + 2|2b) \rightarrow 0.\]
Therefore, $E_2^{p,3}=0$ for $p=0,1$ and $E_2^{2,3}=(2a|2b - 2|2b|2b + 2)$.

For $q=4$, we have that $E_1^{0,4}\rightarrow E_1^{1,4}+E_2^{2,3}$ is injective, since $\Delta_0$, $\Delta_1$ and $\Delta_2$ are diagonal embeddings. 
Also, $E_1^{1,4}\rightarrow E_1^{2,4}$  is surjective. 
If $\Delta_0$ is a ghost class then there is a non-trivial $d_2$-map. However, from computation of the cohomology of $GL_4(\Z)$ with coefficients in $det$ and from the cohomology of $GL_5(\Z)$ with trivial coefficients we obtain that $\Delta_0$ embeds diagonally in the boundary cohomology. This is only verified in the case of $V_3\otimes  det$ and $V_3^*\otimes det$. In case $A5$ we need to consider higher degree of odd symmetric powers. Here we assume that again $D_0$ embeds diagonally.
Using this assumption, we obtain that $d_2=0$. Therefore
Therefore, $E_2^{0,4}=E_1^{2,4}$ and $E_1^{1,4}=(2b-2|\overline{2a+1,2b-1}|2b+2)$

For $q=5$, we have that $H^3(GL_3(\Z),(2a,2b-1,2b-1)=0$. 
We also have that $(2b-2,2b-2|2a+2,2b+2)= (2b-2|2b-2|2a+2,2b+2)$
Therefore, $E_2^{0,5}=H^3(GL_3(\Z),(2a+1,2b+1,2b))=(2b-2|2a+2,2b+2)$ and $E_2^{p,5}=0$ for $p=1,2$.



{\begin{center}
\begin{small}
\scriptsize\renewcommand{\arraystretch}{2.2}
\begin{longtable}{|c|c|c|c|c|c|c|c|c|}
\hline
& $p=0$
& $p=1$
& $p=2$\\
\hline
$q=0$
&  
&  
& 
\\
\hline
$q=1$
&
& 
&
\\
\hline
$q=2$
&
&  
& 
\\
\hline
$q=3$
& 
& 
& $E_2^{2,3}=(2a|2b - 2|2b|2b + 2)$
\\
\hline
$q=4$
&
& $E_2^{1,4}=(2b-2|\overline{2a+1,2b-1}|2b+2)$
& \\
\hline
$q=5$
& $E_2^{0,5}=(2b-2,2b-2|2a+2,2b+2)$
& 
&\\
\hline
$q=6$
&
& 
& \\
\hline
\end{longtable}
\end{small}
\end{center}


\subsubsection{Boundary cohomology. Highest weight $(2a,2b,2b,2b)$.  Case $A5$.}

By definition, the boundary cohomology is the one to which the spectral sequence converges.
That is, $\bigoplus_{prk(P)=p+1}H^q(P,V_\lambda)=>H^{p+q}_\partial(GL_4(\Z),V_\lambda)$.
From the previous subsection, it is clear that in the case 
$A1$
the spectral sequence degenerates at the $E_2$-page.
Therefore, up to semisimplicity, 
$H^{n}_\partial(GL_4(\Z),V_\lambda)=\bigoplus_{p+q=n} E_2^{p,q}$.
Since we know the $E_2$ terms which are of the form, we can find the  boundary cohomology using our computation of the $E_2$ terms
Let us denote temporarily $H^q_\partial(GL_4(\Z),A5)=H^q_\partial(GL_4(\Z),(2a,2b,2b,2b))$.
Then, we can summarize the result for the boundary cohomology in the following table.

\begin{thm}

$H^q_\partial(GL_4(\Z),A5)=
\left\{
\begin{tabular}{llll}
	$0$ & $q=0$\\
	$0$ & $q=1$\\
	$0$ & $q=2$\\
	$0$ & $q=3$\\
	$0$ & $q=4$\\
		$\left\{
		\begin{tabular}{lll}
		$E_2^{0,5}=(2b-2,2b-2|2a+2,2b+2)$\\
		$E_2^{1,4}=(2b-2|\overline{2a+1,2b-1}|2b+2)$\\
		$E_2^{2,3}=(2a|2b - 2|2b|2b + 2)$
		\end{tabular}	
		\right\}$
	 		& $q=5$\\
	$0$ & $q=6$\\
	$0$ & $q=7$\\
	$0$ & $q=8$
\end{tabular}
\right.
$
\end{thm}

\begin{thm}

\begin{eqnarray*}
H^2_\partial(GL_4(\Z),(2a,2b,2b,2b))
=
&&
	0\\
H^3_\partial(GL_4(\Z),(2a,2b,2b,2b))
=
&&
	0\\
H^4_\partial(GL_4(\Z),(2a,2b,2b,2b))
=
&&
	0\\
H^5_\partial(GL_4(\Z),(2a,2b,2b,2b))
=
&&
	S_{2a-2b+2}
	+
	S_{2a-2b+4}
	+
	\C\\
H^6_\partial(GL_4(\Z),(2a,2b,2b,2b))
=
&&
	0\\
\end{eqnarray*}
\end{thm}




\subsection{Cohomology of $GL_4(\Z)$. Highest weight $(2a+1,2b+1,2b+1,2b+1)$.  Case $A5'$}


\subsubsection{Cohomology of the parabolic subgroups. Highest weight $(2a+1,2b+1,2b+1,2b+1)$. Case $A5'$}

$P_0$: 
$\begin{tabular}{lllll}
$w$ 		& $l$	& 	& $w(\lambda+\rho)-\rho$\\
$2143$ 	& $2$&	& $(2b|2a+2|2b|2b+2)$\\
$2341$ 	& $3$&	& $(2b|2b|2b|2a+4)$\\
$4123$	& $3$&	& $(2b-2|2a+2|2b+2|2b+2)$\\
$4321$	& $6$&	& $(2b-2|2b|2b+2|2a+4)$\\
\end{tabular}
$

$P_{12}$: 
$\begin{tabular}{lllll}
$w$ 		& $l$	& 	& $w(\lambda+\rho)-\rho$\\
$1243$ 	& $1$&	& $(2a+1,2b+1|2b|2b+2)$\\
$1423$ 	& $2$&	& $(2a+1,2b-1|2b+2|2b+2)$\\
$2341$	& $3$&	& $(2b,2b|2b|2a+4)$\\
$3421$	& $5$&	& $(2b-1,2b-1|2b+2|2a+4)$\\
\end{tabular}
$

$P_{23}$: 
$\begin{tabular}{lllll}
$w$ 		& $l$	& 	& $w(\lambda+\rho)-\rho$\\
$2143$ 	& $2$&	& $(2b|2a+2,2b|2b+2)$\\
$2341$	& $3$&	& $(2b|2b,2b|2a+4)$\\
$4123$	& $3$&	& $(2b-2|2a+2, 2b+2|2b+2)$\\
$4231$	& $5$&	& $(2b-2|2b+1,2b+1|2a+4)$\\
\end{tabular}
$

$P_{34}$: 
$\begin{tabular}{lllll}
$w$ 		& $l$	& 	& $w(\lambda+\rho)-\rho$\\
$2134$ 	& $1$&	& $(2b|2a+2|2b+1,2b+1)$\\
$2314$	& $2$&	& $(2b|2b|2a+3,2b+1)$\\
$4123$	& $3$&	& $(2b-2|2a+2|2b+2,2b+2)$\\
$4312$ 	& $5$&	& $(2b-2|2b|2a+3,2b+3)$\\
\end{tabular}
$
$P_{13}$: 
$\begin{tabular}{lllll}
$w$ 		& $l$	& 	& $w(\lambda+\rho)-\rho$\\
$1243$ 	& $1$&	& $(2a+1,2b+1,2b|2b+2)$\\
$2341$ 	& $3$&	& $(2b,2b,2b|2a+4)$\\
\end{tabular}
$

$P_{12,34}$: 
$\begin{tabular}{lllll}
$w$ 		& $l$	& 	& $w(\lambda+\rho)-\rho$\\
$1234$ 	& $0$&	& $(2a+1,2b+1|2b+1,2b+1)$\\
$1423$ 	& $2$&	& $(2a+1, 2b-1|2b+2,2b+2)$\\
$2314$ 	& $2$&	& $(2b,2b|2a+3,2b+1)$\\
$3412$	& $4$&	& $(2b-1,2b-1|2a+3,2b+3)$\\
\end{tabular}
$

$P_{24}$: 
$\begin{tabular}{lllll}
$w$ 		& $l$	& 	& $w(\lambda+\rho)-\rho$\\
$2134$ 	& $1$&	& $(2b|2a+2,2b+1,2b+1)$\\
$4123$ 	& $3$&	& $(2b-2|2a+2,2b+2,2b+2)$\\
\end{tabular}
$


\subsubsection{$E_1$-page. Highest weight $(2a+1,2b+1,2b+1,2b+1)$.  Case $A5'$}

{\begin{center}
\begin{small}
\scriptsize\renewcommand{\arraystretch}{2.2}
\begin{longtable}{|c|c|c|c|c|c|c|c|c|}
\hline
$E_1^{0,0}=0$ 
& $E_1^{1,0}=0$ & 
$E_1^{2,0}=0$
\\
\hline
$E_1^{0,1}=0$
& $E_1^{1,1}=0$
& $E_1^{2,1}=0$
\\
\hline
$E_1^{0,2}=(2a+1,2b+1|2b+1,2b+1)=0$ 
& $E_1^{1,2}=
	\left\{
	\begin{tabular}{ll}
	$(2a+1,2b+1|2b|2b+2)$\\
	$(2b|2a+2|2b+1,2b+1)=0$
	\end{tabular}
	\right.$ 
& $E_1^{2,2}=(2b|2a+2|2b|2b+2)$
\\
\hline
$E_1^{0,3}=
	\left\{
	\begin{tabular}{ll}
	$(H^2(2a+1,2b+1,2b)|2b+2)$\\
	$H^0((2b,2b,2b)|2a+4)$\\
	$(2a+1,2b-1|2b+2,2b+2)$\\
	$(2b,2b|2a+3,2b+1)$\\
	$(2b|H^2(2a+2,2b+1,2b+1))$
	\end{tabular}
	\right.
	$
& $E_1^{1,3}=
	\left\{
	\begin{tabular}{ll}
	$(2a+1,2b-1|2b+2|2b+2)$\\
	$(2b,2b|2b|2a+4)$\\
	$(2b|2b,2b|2a+4)$\\
	$(2b|2a+2,2b|2b+2)$\\
	$(2b|2b|2a+3,2b+1)$\\
	$(2b-2|2a+2|2b+2,2b+2)$
	\end{tabular}
	\right.
	$
& $E_1^{2,3}=
	\left\{
	\begin{tabular}{ll}
	$(2b|2b|2b|2a+4)$\\
	$(2b-2|2a+2|2b+2|2b+2)$
	\end{tabular}
	\right.
	$\\
\hline
$E_1^{0,4}=
	\left\{
	\begin{tabular}{ll}
	$(H^3(2a+1,2b+1,2b)|2b+2)$\\
	$(2b|H^3(2a+2,2b+1,2b+1))$
	\end{tabular}
	\right.$
& $E_1^{1,4}=(2b-2|2a+2,2b+2|2b+2)$
& $E_1^{2,4}=0$\\
\hline
$E_1^{0,5}=0$
& $E_1^{1,5}=0$
& $E_1^{2,5}=0$\\
\hline
$E_1^{0,6}=
	\left\{
	\begin{tabular}{ll}
	$(H^3(2b,2b,2b)|2a+4)=0$\\
	$(2b-1,2b-1|2a+3,2b+3)=0$\\
	$(2b-2|H^3(2a+2,2b+2,2b+2))$
	\end{tabular}
	\right.$
& $E_1^{1,6}=
	\left\{
	\begin{tabular}{ll}
	$(2b-1,2b-1|2b+2|2a+4)=0$\\
	$(2b-2|2b+1,2b+1|2a+4)=0$\\
	$(2b-2|2b|2a+3,2b+3)$
	\end{tabular}
	\right.$
& $E_1^{2,6}=(2b-2|2b|2b+2|2a+4)$\\
\hline
\end{longtable}
\end{small}
\end{center}


\subsubsection{Certain cohomology groups of $GL_3(\Z)$ that appear on the $E_1$-page in the case $A5'$.}
From the computations of the cohomology of $GL_3(\Z)$ from the previous section, we have
{\begin{center}
\scriptsize\renewcommand{\arraystretch}{2.2}
\begin{longtable}{|c|c|c|}
\hline
$H^2(GL_3(\Z),(2a+1,2b+1,2b))
	=
	\Delta_2
	$
&$H^3(GL_3(\Z),(2a+1,2b+1,2b))=(2b-2|2a+2,2b+2)$\\
\hline
$H^2(GL_3(\Z),(2a+2,2b+1,2b+1))
	=
	\Delta_1+\Delta_0
	$
& $H^3(GL_3(\Z),(2a+2,2b+1,2b+1))=0$
\\
\hline
$H^2(GL_3(\Z),(2b,2b,2b))=0$
& $H^3(GL_3(\Z),(2b,2b,2b)) = (\overline{2b-1,2b-1}|2b+2)$
	\\
\hline
$H^2(GL_3(\Z),(2a+2,2b+2,2b+2))=0$
& $H^3(GL_3(\Z),(2a+2,2b+2,2b+2)) = (2b|\overline{2a+3,2b+3})$\\
	\hline
\end{longtable}
\end{center}
where
$\Delta_2 \subset (\overline{2a+1,2b-1}|2b+2) + (2b|2a+2,2b)$,
$\Delta_1 \subset (2a+2,2b|2b+2) + (2b|\overline{2a+3,2b+1})$
and 
$\Delta_0 \subset (2a+2|2b|2b+2) + (2b|2b|2a+4)$
are diagonal embeddings.


\subsubsection{$E_2$ page. Highest weight $(2a+1,2b+1,2b+1,2b+1)$.  Case $A5'$}

From consideration of Euler characteristics we can conclude that the map $(2a+1,2b+1)\rightarrow (2b|2a+2)$ is surjective for $a>b$, which corresponds to the Eisenstein series $E_{2a-2b+2}$ for the classical modular group $SL_2(\Z)$.

For $q=2$, we have the surjective map $E_1^{1,2} \rightarrow E_1^{1,2}$
is surjective.
Therefore, 
$E_2^{1,2}=(\overline{2a+1,2b+1}|2b|2b+2)$.

Now let us examine the third line $E_1^{p,3}$ of the $E_1$-page. 
We have the following isomorphisms coming from  the $d_1$-map from $E_1^{0,3}$ to $E_1^{1,3}$
\[(2a+1,2b-1|2b+2,2b+2)\rightarrow (2a+1,2b-1|2b+2|2b+2)\]
and
\[(2b,2b|2a+3,2b+1)\rightarrow (2b|2b|2a+3,2b+1).\]
We also have the following isomorphism coming from  the $d_1$-map from $E_1^{1,3}$ to $E_1^{2,3}$.
\[(2b-2|2a+2|2b+2,2b+2)\rightarrow (2b-2|2a+2|2b+2|2b+2).\]
We an exact sequence
\[0\rightarrow (2b, 2b,2b|2a + 4)\rightarrow (2b,2b|2b|2a + 4)+ (2b|2b,2b|2a + 4)\rightarrow (2b|2b|2b|2a + 4)\rightarrow 0.\]

The remaining part of the third row $E_1^{p,3}$ is
\[(H^2(2a + 1,2b + 1,2b)|2b + 2)+(2b|H^2(2a + 2,2b + 1,2b + 1))\rightarrow (2b|2a + 2, 2b|2b + 2)\]

Therefore, $E_2^{1,3}=E_2^{2,3}=0$ and 
\begin{eqnarray*}
E_2^{0,3}	=	&&ker[E_1^{0,3}\rightarrow E_1^{1,3}]=\\
		=	&&ker[(H^2(2a + 1,2b + 1,2b)|2b + 2)+(2b|H^2(2a + 2,2b + 1,2b + 1))\rightarrow (2b|2a + 2, 2b|2b + 2)=\\
			&& \rightarrow (2b|2a+2,2b|2b+2)]=\\
		=	&& ker[\Delta_0\Delta_1+\Delta_2\rightarrow (2b|2a+2,2b|2b+2)]
\end{eqnarray*}
Since both $\Delta_1$ and $\Delta_2$ are isomorphic to $(2b|2a+2,2b|2b+2)$,
and $\Delta_0$ is $\C$.
We obtain that $E_2^{0,3}$ is also isomorphic to $(2b|2a+2,2b|2b+2)+ (2b|2b|2b|2a+4)$

For $q=4$, we have $H^3(GL_3(\Z),(2a+2,2b+1,2b+1))=0$. Then $E_1^{0,4}$ is isomorphic to $E_1^{1,4}$. Therefore $E_2^{p,4}=0$ for $p=0,1,2$. 

For $q=5$, we have $E_1^{p,5}=0$. Therefore, $E_2^{p,5}=0$.
For $q=6$, the sixth row is a short exact sequence \[0\rightarrow E_1^{0,6}\rightarrow E_1^{1,6}\rightarrow E_1^{2,6}\rightarrow 0.\]
Therefore, $E_2^{p,6}=0$ for $p=0,1,2$.

{\begin{center}
\begin{small}
\scriptsize\renewcommand{\arraystretch}{2.2}
\begin{longtable}{|c|c|c|c|c|c|c|c|c|}
\hline
& $p=0$
& $p=1$
& $p=2$\\
\hline
$q=0$
&  
&  
& \\
\hline
$q=1$
&
& 
&
\\
\hline
$q=2$
&
&  $E_2^{1,2}=(\overline{2a+1,2b+1}|2b|2b+2)$
& 
\\
\hline
$q=3$
& $E_2^{0,3}=
	\left\{
	\begin{tabular}{lll}
	$(2b|2a+2,2b|2b+2)$\\
	$(2b|2b|2b|2a+4)$
	\end{tabular}
	\right\}$
&
&\\
\hline
$q=4$
& 
& 
& \\
\hline
$q=5$
& 
& 
&\\
\hline
$q=6$
&
& 
& \\
\hline
\end{longtable}
\end{small}
\end{center}


\subsubsection{Boundary cohomology. Highest weight $(2a+1,2b+1,2b+1,2b+1)$.  Case $A5'$.}
Again the boundary cohomology degenerates at the $E_2$-page. Therefore, up to semisimplicity we have
$H^n_\partial(GL_4(\Z),V_\lambda)=\bigoplus_{p+q=n}E_2^{p,q}$.
More systematically, we have

\begin{thm}
Let $H^q=H^q_\partial(GL_4(\Z),V_\lambda)$ be the cohomology of $GL_4(\Z)$ with coefficients in the highest weight representation with weight 
$\lambda=(2a+1,2b+1,2b+1,2b+1)$ written as a character of the split maximal torus. Then,
\[H^q
=
\left\{
\begin{tabular}{lll}
$0$ 
		& $q=2$\\
$\left\{\begin{tabular}{lll}
$E_2^{0,3}=
	\left\{
	\begin{tabular}{lll}
	$(2b|2a+2,2b|2b+2)$\\
	$(2b|2b|2b|2a+4)$
	\end{tabular}
	\right\}$\\
$E_2^{1,2}=(\overline{2a+1,2b+1}|2b|2b+2)$
\end{tabular}
\right\}$
		& $q=3$\\	
$0$
		& $q=4$\\
$0$
		& $q=5$\\ 
$0$
		& $q=6$
\end{tabular}
\right.
\]
\end{thm}

\begin{thm}
\[H^3_\partial(GL_4(\Z),(2a+1,2b+1,2b+1,2b+1))
=
	S_{2a-2b+4}
	+
	\C
	+S_{2a-2b+2}
	\]
and
\[H^q_\partial(GL_4(\Z),(2a+1,2b+1,2b+1,2b+1))
=
	0,\mbox{ for }q \neq 3.\]
\end{thm}




\subsection{Cohomology of $GL_4(\Z)$. Highest weight $(2a,2a,2c,2c)$.  Case $A6$}


\subsubsection{Cohomology of the parabolic subgroups. Highest weight $(2a,2a,2c,2c)$. Case $A6$}
Let $\lambda=(2a,2a,2c,2c)$.

$P_0$: 
$\begin{tabular}{lllll}
$w$ 		& $l$	& 	& $w(\lambda+\rho)-\rho$\\
$1234$ 	& $0$&	& $(2a|2a|2c|2c)$\\
$1432$ 	& $3$&	& $(2a|2c-2|2c|2a+2)$\\
$3214$	& $3$&	& $(2c-2|2a|2a+2|2c)$\\
$3412$	& $4$&	& $(2c-2|2c-2|2a+2|2a+2)$\\
\end{tabular}
$

$P_{12}$: 
$\begin{tabular}{lllll}
$w$ 		& $l$	& 	& $w(\lambda+\rho)-\rho$\\
$1234$ 	& $0$&	& $(2a,2a|2c|2c)$\\
$1432$ 	& $3$&	& $(2a,2c-2|2c|2a+2)$\\
$2314$	& $2$&	& $(2a-1,2c-1|2a+2|2c)$\\
$3412$	& $4$&	& $(2c-2,2c-2|2a+2|2a+2)$\\
\end{tabular}
$

$P_{23}$: 
$\begin{tabular}{lllll}
$w$ 		& $l$	& 	& $w(\lambda+\rho)-\rho$\\
$1234$ 	& $0$&	& $(2a|2a,2c|2c)$\\
$1342$ 	& $2$&	& $(2a|2c-1,2c-1|2a+2)$\\
$3124$	& $2$&	& $(2c-2|2a+1, 2a+1|2c)$\\
$3142$	& $3$&	& $(2c-2|2a+1,2c-1|2a+2)$\\
\end{tabular}
$

$P_{34}$: 
$\begin{tabular}{lllll}
$w$ 		& $l$	& 	& $w(\lambda+\rho)-\rho$\\
$1234$ 	& $0$&	& $(2a|2a|2c,2c)$\\
$1423$ 	& $2$&	& $(2a|2c-2|2a+1,2c+1)$\\
$3214$	& $3$&	& $(2c-2|2a|2a+2,2c)$\\
$3412$	& $4$&	& $(2c-2|2c-2|2a+2,2a+2)$\\
\end{tabular}
$

$P_{13}$: 
$\begin{tabular}{lllll}
$w$ 		& $l$	& 	& $w(\lambda+\rho)-\rho$\\
$1234$ 	& $0$&	& $(2a,2a,2c|2c)$\\
$1342$ 	& $2$&	& $(2a,2c-1,2c-1|2a+2)$\\
\end{tabular}
$

$P_{12,34}$: 
$\begin{tabular}{lllll}
$w$ 		& $l$	& 	& $w(\lambda+\rho)-\rho$\\
$1234$ 	& $0$&	& $(2a,2a|2c,2c)$\\
$1423$ 	& $2$&	& $(2a, 2c-2|2a+1,2c+1)$\\
$2314$ 	& $2$&	& $(2a-1,2c-1|2a+2,2c)$\\
$3412$	& $4$&	& $(2c-2,2c-2|2a+2,2a+2)$\\
\end{tabular}
$

$P_{24}$: 
$\begin{tabular}{lllll}
$w$ 		& $l$	& 	& $w(\lambda+\rho)-\rho$\\
$1234$ 	& $0$&	& $(2a|2a,2c,2c)$\\
$3124$ 	& $2$&	& $(2c-2|2a+1,2a+1,2c)$\\
\end{tabular}
$


\subsubsection{$E_1$-page. Highest weight $(2a,2a,2c,2c)$.  Case $A6$.}

{\begin{center}
\begin{small}
\scriptsize\renewcommand{\arraystretch}{2.2}
\begin{longtable}{|c|c|c|c|c|c|c|c|c|}
\hline
$E_1^{0,0}=(2a,2a|2c,2c)$
& $E_1^{1,0}=
	\left\{
	\begin{tabular}{lll}
	$(2a,2a|2c|2c)$\\
	$(2a|2a|2c,2c)$
	\end{tabular}
	\right.
	$
 & 
$E_1^{2,0}=(2a|2a|2c|2c)$
\\
\hline
$E_1^{0,1}=0$
	& $E_1^{1,1}=
	\left\{
	\begin{tabular}{lll}
	$(2a|2a,2c|2c)$\\
	\end{tabular}
	\right.
	$
& $E_1^{2,1}=0$
\\
\hline
$E_1^{0,2}=0$	
& $E_1^{1,2}=0$ 
& $E_1^{2,2}=0$
\\
\hline
$E_1^{0,3}=
	\left\{
	\begin{tabular}{ll}
	$(H^3(2a,2a,2c)|2c)$\\
	$(2a|H^3(2a,2c,2c))$
	\end{tabular}
	\right.
	$
& $E_1^{1,3}=
	\left\{
	\begin{tabular}{ll}
	$(2a-1,2c-1|2a+2|2c)$\\
	$(2c-2|2a+1,2a+1|2c)=0$\\
	$(2a|2c-2|2a+1,2c+1)$\\
	$(2a|2c-1,2c-1|2a+2)=0$
	\end{tabular}
	\right.
	$
& $E_1^{2,3}=
	\left\{
	\begin{tabular}{ll}
	$(2c-2|2a|2a+2|2c)$\\
	$(2a|2c-2|2c|2a+2)$
	\end{tabular}
	\right.
	$\\
\hline
$E_1^{0,4}=
	\left\{
	\begin{tabular}{ll}
	$(H^2(2a,2c-1,2c-1)|2a+2)$\\
	$(2a,2c-2|2a+1,2c+1)$\\
	$(2a-1,2c-1|2a+2,2c)$\\
	$(2c-2,2c-2|2a+2,2a+2)$\\
	$(2c-2|H^2(2a+1,2a+1,2c))$
	\end{tabular}
	\right.$
& $E_1^{1,4}=
	\left\{
	\begin{tabular}{ll}
	$(2a,2c-2|2c|2a+2)$\\
	$(2c-2,2c-2|2a+2|2a+2)$\\
	$(2c-2|2a+1,2c-1|2a+2)$\\
	$(2c-2|2a|2a+2,2c)$\\
	$(2c-2|2c-2|2a+2,2a+2)$
	\end{tabular}
	\right.$
& $E_1^{2,4}=(2c-2|2c-2|2a+2|2a+2)$\\
\hline
$E_1^{0,5}=0$
& $E_1^{1,5}=0$
& $E_1^{2,5}=0$\\
\hline
$E_1^{0,6}=0$
& $E_1^{1,6}=0$
& $E_1^{1,6}=0$\\
\hline
\end{longtable}
\end{small}
\end{center}


\subsubsection{Certain cohomology groups of $GL_3(\Z)$ that appear on the $E_1$-page in the case $A6$.}
From the computations of the cohomology of $GL_3(\Z)$ from the previous section, we have
{\begin{center}
\scriptsize\renewcommand{\arraystretch}{2.2}
\begin{longtable}{|c|c|c|}
\hline
$H^2(GL_3(\Z),(2a,2a,2c))=0$
& $H^3(GL_3(\Z),(2a,2a,2c))
	=(\overline{2a-1,2c-1}|2a+2)$\\
	\hline
$H^2(GL_3(\Z),(2a,2c,2c))=0$
& $H^3(GL_3(\Z),(2a,2c,2c))
	=(2c-2|\overline{2a+1,2c+1})$\\
	\hline
$H^2(GL_3(\Z),(2a,2c-1,2c-1))
	=
	\Delta_0+\Delta_1$
	& $H^3(GL_3(\Z),(2a,2c-1,2c-1))=0$
\\
\hline
$H^2(GL_3(\Z),(2a+1,2a+1,2c))
	=
	\Delta'_0+\Delta'_1$
&$H^3(GL_3(\Z),(2a+1,2a+1,2c))=0$\\
\hline

\end{longtable}
\end{center}
where
$\Delta_0=(2a|2c-2|2c)+(2c-2|2c-2|2a+2)$,
$\Delta_1
	\subset (2a,2c-2|2c)+(2c-2|\overline{2a+1,2c-1})$,	
$\Delta'_0\subset (2a|2a+2|2c)+(2c-2|2a+2|2a+2)$
and 
$\Delta'_1
	\subset 
	(\overline{2a+1,2c-1}|2a+2) + (2a|2a+2,2c)$
are diagonal embeddings. 


\subsubsection{$E_2$-page. Highest weight $(2a,2a,2c,2c)$. Case $A6$}

For $q=0$, we have a short exact sequence \[0 \rightarrow  E_1^{0,0} \rightarrow  E_1^{1,0}\rightarrow  E_1^{2,0}\rightarrow  0.\]  Therefore $E_2^{p,0}=0$ for $p=0,1,2$.

For $q=1$, we have an isomorphism $E_2^{p,1}=E_1^{p,1}$ for $p=0,1,2$.
Thus, $E_2^{0,1}=(2a|2a,2c|2c)$ and $E_2^{p,1}=0$ for $p=0,2$

For $q=2$ all $E_1^{p,2}$ are zero, therefore $E_2^{p,2}=0$.

For $q=3$, we have a short exact sequence
$0\rightarrow (H^3(2a,2a,2c)|2c)\rightarrow (2a-1,2c-1|2a+2|2c) \rightarrow (2c-2|2a|2a+2|2c)\rightarrow 0$.
We also have another short exact sequence
$0\rightarrow (2a|H^3(2a,2c,2c)) \rightarrow (2a|2c-2|2a+1,2c+1) + (2a|2c-1,2c-1|2a+2) \rightarrow (2a|2c-2|2c|2a+2) \rightarrow 0$
Therefore $E_2^{p,3}=0$ for $p=0,1,2$.

For $q=4$ 

The map $(2a,2c-2|2a+1,2c+1)\rightarrow (2a,2c-2|2c|2a+2)$ is surjective and the map 
 is a diagonal embedding.
Therefore
$(H^2(2a,2c-1,2c-1)|2a+2)+(2a,2c-2|2a+1,2c+1)\rightarrow ker[(2a,2c-2|2c|2a+2)+(2c-2|2a+1,2c-1|2a+2)\rightarrow  (2c-2|2c-2|2a+2|2a+2)]$
is surjective and the kernel is $E_2^{0,4}=(2a,2c-2|\overline{2a+1,2c+1})$.
the other $E_2^{p,4}=0$ for $p=1,2$.

{\begin{center}
\begin{small}
\scriptsize\renewcommand{\arraystretch}{2.2}
\begin{longtable}{|c|c|c|c|c|c|c|c|c|}
\hline
& $p=0$
& $p=1$
& $p=2$\\
\hline
$q=0$
&  
&  
& 
\\
\hline
$q=1$
&
& $E_2^{1,1}=(2a|2a,2c|2c)$
&
\\
\hline
$q=2$
&
&  
& 
\\
\hline
$q=3$
&
& 
& \\
\hline
$q=4$
&$E_2^{0,4}=
	\left\{
	\begin{tabular}{lll}
	$(\overline{2a-1,2c-1}|2a+2,2c)$\\
	$(2c-2|2a+1,2c-1|2a+2)$\\
	$(2a,2c-2|\overline{2a+1,2c+1})$	
	\end{tabular}
	\right.
	$
&
& \\
\hline
$q=5$
& 
& 
&\\
\hline
$q=6$
&
& 
& \\
\hline
\end{longtable}
\end{small}
\end{center}


\subsubsection{Boundary cohomology. Highest weight $(2a,2a,2c,2c)$.  Case $A6$.}

By definition, the boundary cohomology is the one to which the spectral sequence converges.
That is, $\bigoplus_{prk(P)=p+1}H^q(P,V_\lambda)=>H^{p+q}_\partial(GL_4(\Z),V_\lambda)$.
From the previous subsection, it is clear that in the case 
$A6$
the spectral sequence degenerates at the $E_2$-page.
Therefore, up to semisimplicity, 
$H^{n}_\partial(GL_4(\Z),V_\lambda)=\bigoplus_{p+q=n} E_2^{p,q}$.
Since we know the $E_2$ terms which are of the form, we can find the  boundary cohomology using our computation of the $E_2$ terms
Let us denote temporarily $H^q_\partial(GL_4(\Z),A6)=H^q_\partial(GL_4(\Z),(2a,2a,2c,2c))$.
Then, we can summarize the result for the boundary cohomology in the following table.

$H^q_\partial(GL_4(\Z),A6)=
\left\{
\begin{tabular}{llll}
	$0$ & $q=0$\\
	$0$ & $q=1$\\
	$E_2^{1,1}=(2a|2a,2c|2c)$ 
		& $q=2$\\
	$0$ & $q=3$\\
	$E_2^{0,4}
	=
	\left\{
	\begin{tabular}{lll}
	$(\overline{2a-1,2c-1}|2a+2,2c)$\\
	$(2c-2|2a+1,2c-1|2a+2)$\\
	$(2a,2c-2|\overline{2a+1,2c+1})$	
	\end{tabular}
	\right\}
	$
		& $q=4$\\
	$0$ & $q=5$\\
	$0$ & $q=6$\\
	$0$ & $q=7$\\
	$0$ & $q=8$
\end{tabular}
\right.
$

\begin{thm}
$H^q_\partial(GL_4(\Z),A6)=
\left\{
\begin{tabular}{llll}
	$S_{2a-2c+2}$ 
		& $q=2$\\
	$
	\C^2\otimes S_{2a-2c+2}\otimes S_{2a-2c+4}
	+
	S_{2a-2c+4}
	$
		& $q=4$\\
	$0$ & $q\neq 2,4$\\
\end{tabular}
\right.
$

\end{thm}




\subsection{Cohomology of $GL_4(\Z)$. Highest weight $(2a+1,2a+1,2c+1,2c+1)$.  Case $A6'$}


\subsubsection{Cohomology of the parabolic subgroups. Highest weight $(2a+1,2a+1,2c+1,2c+1)$. Case $A6'$}
For that case the weights are $(2a+1,2a+1,2c+1,2c+1)$.

$P_0$: 
$\begin{tabular}{lllll}
$w$ 		& $l$	& 	& $w(\lambda+\rho)-\rho$\\
$2143$ 	& $2$&	& $(2a|2a+2|2c|2c+2)$\\
$2341$ 	& $3$&	& $(2a|2c|2c|2a+4)$\\
$4123$	& $3$&	& $(2c-2|2a+2|2a+2|2c+2)$\\
$4321$	& $6$&	& $(2c-2|2c|2a+2|2a+4)$\\
\end{tabular}
$

$P_{12}$: 
$\begin{tabular}{lllll}
$w$ 		& $l$	& 	& $w(\lambda+\rho)-\rho$\\
$1243$ 	& $1$&	& $(2a+1,2a+1|2c|2c+2)$\\
$1423$ 	& $2$&	& $(2a+1,2c-1|2a+2|2c+2)$\\
$2341$	& $3$&	& $(2a,2c|2c|2a+4)$\\
$3421$	& $5$&	& $(2c-1,2c-1|2a+2|2a+4)$\\
\end{tabular}
$

$P_{23}$: 
$\begin{tabular}{lllll}
$w$ 		& $l$	& 	& $w(\lambda+\rho)-\rho$\\
$2143$ 	& $2$&	& $(2a|2a+2,2c|2c+2)$\\
$2341$	& $3$&	& $(2a|2c,2c|2a+4)$\\
$4123$	& $3$&	& $(2c-2|2a+2, 2a+2|2c+2)$\\
$4231$	& $5$&	& $(2c-2|2a+1,2c+1|2a+4)$\\
\end{tabular}
$

$P_{34}$: 
$\begin{tabular}{lllll}
$w$ 		& $l$	& 	& $w(\lambda+\rho)-\rho$\\
$2134$ 	& $1$&	& $(2a|2a+2|2c+1,2c+1)$\\
$2314$	& $2$&	& $(2a|2c|2a+3,2c+1)$\\
$4123$	& $3$&	& $(2c-2|2a+2|2a+2,2c+2)$\\
$4312$ 	& $5$&	& $(2c-2|2c|2a+3,2a+3)$\\
\end{tabular}
$

$P_{13}$: 
$\begin{tabular}{lllll}
$w$ 		& $l$	& 	& $w(\lambda+\rho)-\rho$\\
$1243$ 	& $1$&	& $(2a+1,2a+1,2c|2c+2)$\\
$2341$ 	& $3$&	& $(2a,2c,2c|2a+4)$\\
\end{tabular}
$

$P_{12,34}$: 
$\begin{tabular}{lllll}
$w$ 		& $l$	& 	& $w(\lambda+\rho)-\rho$\\
$1234$ 	& $0$&	& $(2a+1,2a+1|2c+1,2c+1)$\\
$1423$ 	& $2$&	& $(2a+1, 2c-1|2a+2,2c+2)$\\
$2314$ 	& $2$&	& $(2a,2c|2a+3,2c+1)$\\
$3412$	& $4$&	& $(2c-1,2c-1|2a+3,2a+3)$\\
\end{tabular}
$

$P_{24}$: 
$\begin{tabular}{lllll}
$w$ 		& $l$	& 	& $w(\lambda+\rho)-\rho$\\
$2134$ 	& $1$&	& $(2a|2a+2,2c+1,2c+1)$\\
$4123$ 	& $3$&	& $(2c-2|2a+2,2a+2,2c+2)$\\
\end{tabular}
$


\subsubsection{$E_1$-page. Highest weight $(2a+1,2a+1,2c+1,2c+1)$.  Case $A6'$}

{\begin{center}
\begin{small}
\scriptsize\renewcommand{\arraystretch}{2.2}
\begin{longtable}{|c|c|c|c|c|c|c|c|c|}
\hline
$E_1^{0,0}=0$ 
& $E_1^{1,0}=0$ & 
$E_1^{2,0}=0$
\\
\hline
$E_1^{0,1}=0$
& $E_1^{1,1}=0$
& $E_1^{2,1}=0$
\\
\hline
$E_1^{0,2}=(2a+1,2a+1|2c+1,2c+1)=0$ 
& $E_1^{1,2}=
	\left\{
	\begin{tabular}{ll}
	$(2a+1,2a+1|2c|2c+2)=0$\\
	$(2a|2a+2|2c+1,2c+1)=0$
	\end{tabular}
	\right.$ 
& $E_1^{2,2}=(2a|2a+2|2c|2c+2)$
\\
\hline
$E_1^{0,3}=
	\left\{
	\begin{tabular}{ll}
	$(H^2(2a+1,2a+1,2c)|2c+2)$\\
	$(2a|H^2(2a+2,2c+1,2c+1))$
	\end{tabular}
	\right.
	$
& $E_1^{1,3}=
	\left\{
	\begin{tabular}{ll}
	$(2a+1,2c-1|2a+2|2c+2)$\\
	$(2a|2a+2,2c|2c+2)$\\
	$(2c-2|2a+2,2a+2|2c+2)$\\
	$(2a|2c,2c|2a+4)$\\	
	$(2a|2c|2a+3,2c+1)$
	\end{tabular}
	\right.
	$
& $E_1^{2,3}=
	\left\{
	\begin{tabular}{ll}
	$(2a|2c|2c|2a+4)$\\
	$(2c-2|2a+2|2a+2|2c+2)$
	\end{tabular}
	\right.
	$\\
\hline
$E_1^{0,4}=
	\left\{
	\begin{tabular}{ll}
	$(H^3(2a+1,2a+1,2c)|2c+2)$\\
	$(2a+1,2c-1|2a+2,2c+2)$\\
	$(2a,2c|2a+3,2c+1)$\\
	$(2a|H^3(2a+2,2c+1,2c+1))$
	\end{tabular}
	\right.$
& $E_1^{1,4}=
	\left\{
	\begin{tabular}{ll}
	$(2a,2c|2c|2a+4)$\\
	$(2c-2|2a+2|2a+2,2c+2)$
	\end{tabular}
	\right.$
& $E_1^{2,4}=0$\\
\hline
$E_1^{0,5}=0$
& $E_1^{1,5}=0$
& $E_1^{2,5}=0$\\
\hline
$E_1^{0,6}=
	\left\{
	\begin{tabular}{ll}
	$(H^3(2a,2c,2c)|2a+4)$\\
	$(2c-1,2c-1|2a+3,2a+3)=0$\\
	$(2c-2|H^3(2a+2,2a+2,2c+2))$
	\end{tabular}
	\right.$
& $E_1^{1,6}=
	\left\{
	\begin{tabular}{ll}
	$(2c-1,2c-1|2a+2|2a+4)=0$\\
	$(2c-2|2a+1,2c+1|2a+4)$\\
	$(2c-2|2c|2a+3,2a+3)=0$
	\end{tabular}
	\right.$
& $E_1^{2,6}=(2c-2|2c|2a+2|2a+4)$\\
\hline
\end{longtable}
\end{small}
\end{center}


\subsubsection{Certain cohomology groups of $GL_3(\Z)$ that appear on the $E_1$-page in the case $A6'$.}
From the computations of the cohomology of $GL_3(\Z)$ from the previous section, we have
{\begin{center}
\scriptsize\renewcommand{\arraystretch}{2.2}
\begin{longtable}{|c|c|c|}
\hline
$H^2(GL_3(\Z),(2a+1,2a+1,2c))
	=
	\Delta_0+\Delta_1
	$
&$H^3(GL_3(\Z),(2a+1,2a+1,2c))=0$\\
\hline
$H^2(GL_3(\Z),(2a+2,2c+1,2c+1))
	=
	\Delta'_0+\Delta'_1$
& $H^3(GL_3(\Z),(2a+2,2c+1,2c+1))=0$
\\
\hline
$H^2(GL_3(\Z),(2a,2c,2c))=0$
& $H^3(GL_3(\Z),(2a,2c,2c))
	=(2c-2|\overline{2a+1,2c+1})$\\
	\hline
$H^2(GL_3(\Z),(2a+2,2a+2,2c+2))=0$
& $H^3(GL_3(\Z),(2a+2,2a+2,2c+2))
	=0$\\
	\hline
\end{longtable}
\end{center}
where
$\Delta_0=(2a|2a+2|2c)+(2c-2|2a+2|2a+2)$
$\Delta_1 = (\overline{2a+1,2c-1}|2a+2) + (2a|2a+2,2c)$.
	and
$\Delta'_0=(2a+2|2c|2c+2)+(2c|2c|2a+4)$
$\Delta'_1 = (2a+2,2c|2c+2) + (2c|\overline{2a+3,2c+1})$.


\subsubsection{$E_2$ page. Highest weight $(2a+1,2a+1,2c+1,2c+1)$.  Case $A6'$}

For $q=0,1,2$ the spectral sequence stabilizes at the $E_1$-page.
Among them, the only non-zero term is $E_1^{2,2}=(2a|2a+2|2c|2c+2)$

For $q=3$,
we obtain that $E_1^{1,3}\rightarrow E_1^{2,3}$ is surjective because 
$(2a+1,2c-1|2a+2|2c+2)\rightarrow (2c-2|2a+2|2a+2|2c+2)$
and	$(2a|2c|2a+3,2c+1)\rightarrow (2a|2c|2c|2a+4)$ are both surjective.
Also $\Delta_0$ maps isomorphically to $(2c-2|2a+2,2a+2|2c+2)$.
Finally both $\Delta_1$ and $\Delta_2$ embed diagonally in 
$(\overline{2a+1,2c-1}|2a+2|2c+2)+ (2a|2a+2,2c|2c+2)$ and 
$(2a|2a+2,2c|2c+2)+(2a|2c|\overline{2a+3,2c+1})$, respectively. Therefore $E_2^{1,3}$ up to semisimplicity is isomorphic to $(2a|2a+2,2c|2c+2)$.
Also $E_2^{0,3}=E_2^{2,3}=0$.

For $q=4$,
we have that the maps 
$(2a+1,2c-1|2a+2,2c+2)\rightarrow (2c-2|2a+2|2a+2,2c-2)$ 
and
$(2a,2c|2a+3,2c+1)\rightarrow (2a,2c|2c,2a+4)$ 
are surjective.
From cohomology of $GL_3(\Z)$ with coefficients in $(2a+1,2a+1,2c)$, we have that 
$H^3(2a+1,2a+1,2c)=0$.
Therefore,
$(H^3(2a+1,2a+1,2c)|2c+2)=0$.
Similarly, from the cohomology of $GL_3(\Z)$ with coefficients in $(2a+2,2c+1,2c+1)$l, we have that
$H^3(2a+2,2c+1,2c+1)=0$. Therefore,
$(2a|H^3(2a+2,2c+1,2c+1))=0$.
Then, $E_2^{0,4}=(\overline{2a+1,2c-1}|2a+2,2c+2) + (2a,2c|\overline{2a+3,2c+1})$
and $E_2^{1,4}=E_2^{2,4}=0$.

For  $q=6$, we have that the following.
From the cohomology of $GL_3(\Z)$ with coefficients in $(2a,2c,2c)$, we have that 
$H^3(2a,2c,2c)=(2c-2|\overline{2a+1,2c+1})$.
Therefore, 
$(H^3(2a,2c,2c)|2a+4)=(2c-2|\overline{2a+1,2c+1}|2a+4)$.
Then
$E_2^{0,6}$ is isomorphic to $H^6(P_{24})=(2c-2|H^3(2a+2,2a+2,2c+2))=(2c-2|\overline{2a+1,2c+1}|2a+4)$. 
Also,
$E_2^{1,6}=E_2^{2,6}=0$.

This is summrized in the following table. The empty boxes denote that the corresponding $E_2^{p,q}$ term vanishes.

{\begin{center}
\begin{small}
\scriptsize\renewcommand{\arraystretch}{2.2}
\begin{longtable}{|c|c|c|c|c|c|c|c|c|}
\hline
& $p=0$
& $p=1$
& $p=2$\\
\hline
$q=0$
&  
&  
& \\
\hline
$q=1$
&
& 
&
\\
\hline
$q=2$
& 
& 
& $E_1^{2,2}=(2a|2a+2|2c|2c+2)$
\\
\hline
$q=3$
& 
& $E_2^{1,3}=(2a|2a+2,2c|2c+2)$
&\\
\hline
$q=4$
& $E_2^{0,4}=
	\left\{
	\begin{tabular}{lll}
	$(\overline{2a+1,2c-1}|2a+2,2c+2)$\\
	$(2a,2c|\overline{2a+3,2c+1})$
	\end{tabular}
	\right.$
& 
& \\
\hline
$q=5$
& 
& 
&\\
\hline
$q=6$
&
$E_2^{0,6}=(2c-2|\overline{2a+1,2c+1}|2a+4)$
& 
& \\
\hline
\end{longtable}
\end{small}
\end{center}


\subsubsection{Boundary cohomology. Highest weight $(2a+1,2a+1,2c+1,2c+1)$.  Case $A6'$.}
Again the boundary cohomology degenerates at the $E_2$-page. Therefore, up to semisimplicity we have
$H^n_\partial(GL_4(\Z),V_\lambda)=\bigoplus_{p+q=n}E_2^{p,q}$.
More systematically, we have

Let $H^q=H^q_\partial(GL_4(\Z),V_\lambda)$ be the cohomology of $GL_4(\Z)$ with coefficients in the highest weight representation with weight 
$\lambda=(2a+1,2a+1,2c+1,2c+1)$ written as a character of the split maximal torus. Then,
\[H^q
=
\left\{
\begin{tabular}{lll}
$0$
		& $q=2$\\
$0$
		& $q=3$\\	
$\left\{
\begin{tabular}{lll}
$E_2^{0,4}=
	\left\{
	\begin{tabular}{lll}
	$(\overline{2a+1,2c-1}|2a+2,2c+2)$\\
	$(2a,2c|\overline{2a+3,2c+1})$
	\end{tabular}
	\right\}$\\
$E_2^{1,3}=(2a|2a+2,2c|2c+2)$\\	
$E_1^{2,2}=(2a|2a+2|2c|2c+2)$
\end{tabular}
\right\}$
		& $q=4$\\
$0$
		& $q=5$\\ 
$E_2^{0,6}=(2c-2|\overline{2a+1,2c+1}|2a+4)$
		& $q=6$
\end{tabular}
\right.
\]

\begin{thm}
\[H^q
=
\left\{
\begin{tabular}{lll}
$\C^2\otimes S_{2a-2c+4}\otimes S_{2a-2c}
+
S_{2a-2c+4}
+
\C
$		& $q=4$\\
$S_{2a-2c+2}$
		& $q=6$\\

$0$
		& $q\neq 4,6$\\	
\end{tabular}
\right.
\]
\end{thm}

Note that $H^3$ is a potentially ghost space, since it comes from $E_{p,q}$ with $p>0$. This will we important for computation of cohomology of $GL_5(\Z)$ when one considers a maximal parabolic subgroup $P_{14}$ or $P_{25}$ that contain $GL_4$. For those parabolic subgroup we might need to consider cohomology of $P_{14}$ which up to semisimplicity is cohomology of $GL_4(\Z)$. When the representation of $GL_4$ is with highest weight $(2a+1,2a+1,2c+1,2c+1)$, then the existence of potentially ghost class leads to existence of nontrivial $d_2$-map from $E_2^{0,q}$ to $E^{2,q-1}$.




\subsection{Cohomology of $GL_4(\Z)$. Highest weight $(2a,2a,2a,2d)$.  Case $A7$}


\subsubsection{Cohomology of the parabolic subgroups. Highest weight $(2a,2a,2a,2d)$. Case $A7$}
Let $\lambda=(2a,2a,2a,2d)$.

$P_0$: 
$\begin{tabular}{lllll}
$w$ 		& $l$	& 	& $w(\lambda+\rho)-\rho$\\
$1234$ 	& $0$&	& $(2a|2a|2a|2d)$\\
$1432$ 	& $3$&	& $(2a|2d-2|2a|2a+2)$\\
$3214$	& $3$&	& $(2a-2|2a|2a+2|2d)$\\
$3412$	& $4$&	& $(2a-2|2d-2|2a+2|2a+2)$\\
\end{tabular}
$

$P_{12}$: 
$\begin{tabular}{lllll}
$w$ 		& $l$	& 	& $w(\lambda+\rho)-\rho$\\
$1234$ 	& $0$&	& $(2a,2a|2a|2d)$\\
$1432$ 	& $3$&	& $(2a,2d-2|2a|2a+2)$\\
$2314$	& $2$&	& $(2a-1,2a-1|2a+2|2d)$\\
$3412$	& $4$&	& $(2a-2,2d-2|2a+2|2a+2)$\\
\end{tabular}
$

$P_{23}$: 
$\begin{tabular}{lllll}
$w$ 		& $l$	& 	& $w(\lambda+\rho)-\rho$\\
$1234$ 	& $0$&	& $(2a|2a,2a|2d)$\\
$1342$ 	& $2$&	& $(2a|2a-1,2d-1|2a+2)$\\
$3124$	& $2$&	& $(2a-2|2a+1, 2a+1|2d)$\\
$3142$	& $3$&	& $(2a-2|2a+1,2d-1|2a+2)$\\
\end{tabular}
$

$P_{34}$: 
$\begin{tabular}{lllll}
$w$ 		& $l$	& 	& $w(\lambda+\rho)-\rho$\\
$1234$ 	& $0$&	& $(2a|2a|2a,2d)$\\
$1423$ 	& $2$&	& $(2a|2d-2|2a+1,2a+1)$\\
$3214$	& $3$&	& $(2a-2|2a|2a+2,2d)$\\
$3412$	& $4$&	& $(2a-2|2d-2|2a+2,2a+2)$\\
\end{tabular}
$

$P_{13}$: 
$\begin{tabular}{lllll}
$w$ 		& $l$	& 	& $w(\lambda+\rho)-\rho$\\
$1234$ 	& $0$&	& $(2a,2a,2a|2d)$\\
$1342$ 	& $2$&	& $(2a,2a-1,2d-1|2a+2)$\\
\end{tabular}
$

$P_{12,34}$: 
$\begin{tabular}{lllll}
$w$ 		& $l$	& 	& $w(\lambda+\rho)-\rho$\\
$1234$ 	& $0$&	& $(2a,2a|2a,2d)$\\
$1423$ 	& $2$&	& $(2a, 2d-2|2a+1,2a+1)$\\
$2314$ 	& $2$&	& $(2a-1,2a-1|2a+2,2d)$\\
$3412$	& $4$&	& $(2a-2,2d-2|2a+2,2a+2)$\\
\end{tabular}
$

$P_{24}$: 
$\begin{tabular}{lllll}
$w$ 		& $l$	& 	& $w(\lambda+\rho)-\rho$\\
$1234$ 	& $0$&	& $(2a|2a,2a,2d)$\\
$3124$ 	& $2$&	& $(2a-2|2a+1,2a+1,2d)$\\
\end{tabular}
$


\subsubsection{$E_1$-page. Highest weight $(2a,2a,2a,2d)$.  Case $A7$.}

{\begin{center}
\begin{small}
\scriptsize\renewcommand{\arraystretch}{2.2}
\begin{longtable}{|c|c|c|c|c|c|c|c|c|}
\hline
$E_1^{0,0}=(H^0(2a,2a,2a)|2d)$ 
& $E_1^{1,0}=
	\left\{
	\begin{tabular}{lll}
		$(2a,2a|2a|2d)$\\
		$(2a|2a,2a|2d)$
	\end{tabular}
	\right.
	$
& $E_1^{2,0}=(2a|2a|2a|2d)$
\\
\hline
$E_1^{0,1}=(2a,2a|2a,2d)$
& $E_1^{1,1}=(2a|2a|2a,2d)$
& $E_1^{2,1}=0$
\\
\hline
$E_1^{0,2}=0$	
& $E_1^{1,2}=0$ 
& $E_1^{2,2}=0$
\\
\hline
$E_1^{0,3}=(2a|H^3(2a,2a,2d))$
& $E_1^{1,3}=
	\left\{
	\begin{tabular}{ll}
	$(2a-1,2a-1|2a+2|2d)=0$\\
	$(2a-2|2a+1,2a+1|2d)=0$\\
	$(2a|2d-2|2a+1,2a+1)=0$\\
	$(2a|2a-1,2d-1|2a+2)$
	\end{tabular}
	\right.
	$
& $E_1^{2,3}=
	\left\{
	\begin{tabular}{ll}
	$(2a-2|2a|2a+2|2d)$\\
	$(2a|2d-2|2a|2a+2)$
	\end{tabular}
	\right.
	$\\
\hline
$E_1^{0,4}=
	\left\{
	\begin{tabular}{ll}
	$(H^2(2a,2a-1,2d-1)|2a+2)$\\
	$(2a,2d-2|2a+1,2a+1)=0$\\
	$(2a-1,2a-1|2a+2,2d)=0$\\
	$(2a-2|H^2(2a+1,2a+1,2d))$
	\end{tabular}
	\right.$
& $E_1^{1,4}=
	\left\{
	\begin{tabular}{ll}
	$(2a,2d-2|2a|2a+2)$\\
	$(2a-2|2a+1,2d-1|2a+2)$\\
	$(2a-2|2a|2a+2,2d)$\\
	$(2a-2|2d-2|2a+2,2a+2)$
	\end{tabular}
	\right.$
& $E_1^{2,4}=(2a-2|2d-2|2a+2|2a+2)$\\
\hline
$E_1^{0,5}=\left\{
	\begin{tabular}{ll}
	$(H^3(2a,2a-1,2d-1)|2a+2)$\\
	$(2a-2|H^3(2a+1,2a+1,2d))$\\
	$(2a-2,2d-2|2a+2,2a+2)$
	\end{tabular}
	\right.$
& $E_1^{1,5}=(2a-2,2d-2|2a+2|2a+2)$
& $E_1^{2,5}=0$\\
\hline
$E_1^{0,6}=0$
& $E_1^{1,6}=0$
& $E_1^{1,6}=0$\\
\hline
\end{longtable}
\end{small}
\end{center}


\subsubsection{Certain cohomology groups of $GL_3(\Z)$ that appear on the $E_1$-page in the case $A7$.}
From the computations of the cohomology of $GL_3(\Z)$ from the previous section, we have
{\begin{center}
\scriptsize\renewcommand{\arraystretch}{2.2}
\begin{longtable}{|c|c|c|}
\hline
$H^2(GL_3(\Z),(2a,2a,2a))=0$
& $H^3(GL_3(\Z),(2a,2a,2a))=0$\\
	\hline
$H^2(GL_3(\Z),(2a,2a,2d))=0$
& $H^3(GL_3(\Z),(2a,2a,2d))
	=
	(\overline{2a-1,2d-1}|2a+2)$\\
\hline
$H^2(GL_3(\Z),(2a,2a-1,2d-1))
	=
	\Delta_2$
	& $H^3(GL_3(\Z),(2a,2a-1,2d-1))=(2a-2,2d-2|2a+2)$
\\
\hline
$H^2(GL_3(\Z),(2a+1,2a+1,2d))
	=
	\Delta_0+\Delta_1$
&$H^3(GL_3(\Z),(2a+1,2a+1,2d))=0$\\
\hline

\end{longtable}
\end{center}
where
$\Delta_2
	\subset (2a,2d-2|2a)+(2a-2|\overline{2a+1,2d-1})$,	
$\Delta_1
	\subset 
	(\overline{2a+1,2d-1}|2a+2) + (2a|2a+2,2d)$
	and
$\Delta_0\subset (2a|2a+2|2d)+(2d-2|2a+2|2a+2)$
are diagonal embeddings. 

Also $H^0(GL_3(\Z),(2a,2a,2a))=H^0(GL_3(\Z),\C)=\C$.


\subsubsection{$E_2$-page. Highest weight $(2a,2a,2a,2d)$. Case $A7$}

For $q=0$, we have a short exact sequence \[0\rightarrow E_1^{0,0}\rightarrow E_1^{1,0}\rightarrow  E_1^{2,0}\rightarrow 0.\] Therefore $E_2^{p,0}=0$ for $p=0,1,2$.

For $q=1$, we have an isomorphism $E_1^{0,1}=(2a,2a|2a,2d)\rightarrow (2a|2a|2a,2d)\subset E_1^{1,1}$. 
Thus, $E_2^{p,1}=0$ for $p=0,1,2$

For $q=2$ all $E_1^{p,2}$ are zero, therefore $E_2^{p,2}=0$.

For $q=3$, we have a short exact sequence
$0\rightarrow (2a|H^3(2a,2a,2d)) \rightarrow (2a|2a-1,2d-1|2a+2) \rightarrow (2a|2d-2|2a|2a+2) \rightarrow 0$
Therefore $E_2^{p,3}=0$ for $p=0,1$ and $E_2^{2,3}=(2a-2|2a|2a+2|2d)$.

For $q=4$,
 we have that $E_1^{0,4}\rightarrow E_1^{1,4}+E_2^{2,3}$ is injective, since $\Delta_0$, $\Delta_1$ and $\Delta_2$ are diagonal embeddings. 
Also, $E_1^{1,4}\rightarrow E_1^{2,4}$  is surjective. 
If $\Delta_0$ is a ghost class then there is a non-trivial $d_2$-map. However, from computation of the cohomology of $GL_4(\Z)$ with coefficients in $det$ and from the cohomology of $GL_5(\Z)$ with trivial coefficients we obtain that $\Delta_0$ embeds diagonally in the boundary cohomology. This is only verified in the case of $V_3\otimes  det$ and $V_3^*\otimes det$. In case $A5$ we need to consider higher degree of odd symmetric powers. Here we assume that again $D_0$ embeds diagonally.
Using this assumption, we obtain that $d_2=0$. Therefore
Therefore, $E_2^{0,4}=E_2^{2,4}=0$ and $E_2^{1,4}=(2a-2|\overline{2a+1,2d-1}|2a+2)$

For $q=5$, we have $E_2^{0,5}=ker[E_1^{0,5}\rightarrow E_1^{1,5}]$. 
We have $E_1^{0,5}=(H3(GL_3(Z),(2a,2a - 1,2d - 1))|2a+2) + (2a-2|H^3(G_3(\Z),(2a+1,2a+1,2d))) + (2a-2,2d-2|2a+2,2a+2)$ and
$E_1^{1,5}=(2a - 2,2d - 2|2a + 2|2a+2)$.
Using that
$(H^3(GL_3(\Z),(2a,2a - 1,2d -1))|2a+2) = (2a - 2,2d - 2|2a + 2|2a+2)$ and $(2a-2|H^3(GL_3(\Z),(2a+1,2a+1,2d))=0$,
we obtain that 
$E_2^{0,5}=(2a-2,2d-2|2a+2,2a+2)$. Combining all the computations for the $E_2$ page, we obtain the following table.

{\begin{center}
\begin{small}
\scriptsize\renewcommand{\arraystretch}{2.2}
\begin{longtable}{|c|c|c|c|c|c|c|c|c|}
\hline
& $p=0$
& $p=1$
& $p=2$\\
\hline
$q=0$
&  
&  
& 
\\
\hline
$q=1$
&
& 
&
\\
\hline
$q=2$
&
&  
& 
\\
\hline
$q=3$
&
& 
& $E_2^{2,3} = (2a-2|2a|2a+2|2d)$\\
\hline
$q=4$
&
& $E_2^{1,4}=(2a-2|\overline{2a+1,2d-1}|2a+2)$
& \\
\hline
$q=5$
& $E_2^{0,5}=(2a-2,2d-2|2a+2,2a+2)$
& 
&\\
\hline
$q=6$
&
& 
& \\
\hline
\end{longtable}
\end{small}
\end{center}


\subsubsection{Boundary cohomology. Highest weight $(2a,2a,2a,2d)$.  Case $A7$.}

By definition, the boundary cohomology is the one to which the spectral sequence converges.
That is, $\bigoplus_{prk(P)=p+1}H^q(P,V_\lambda)=>H^{p+q}_\partial(GL_4(\Z),V_\lambda)$.
From the previous subsection, it is clear that in the case 
$A7$
the spectral sequence degenerates at the $E_2$-page.
Therefore, up to semisimplicity, 
$H^{n}_\partial(GL_4(\Z),V_\lambda)=\bigoplus_{p+q=n} E_2^{p,q}$.
Since we know the $E_2$ terms which are of the form, we can find the  boundary cohomology using our computation of the $E_2$ terms
Let us denote temporarily $H^q_\partial(GL_4(\Z),A7)=H^q_\partial(GL_4(\Z),(2a,2a,2a,2d))$.
Then, we can summarize the result for the boundary cohomology in the following table.

$H^q_\partial(GL_4(\Z),A7)=
\left\{
\begin{tabular}{llll}
	$0$ & $q=0$\\
	$0$ & $q=1$\\
	$0$ & $q=2$\\
	$0$ & $q=3$\\
	$0$ & $q=4$\\
	$
	\left\{
	\begin{tabular}{lll}
	$E_2^{0,5}=(2a-2,2d-2|2a+2,2a+2)$\\
	$E_2^{1,4}=(2a-2|\overline{2a+1,2d-1}|2a+2)$\\
	$E_2^{2,3} = (2a-2|2a|2a+2|2d)$
	\end{tabular}
	\right\}
	$
	 & $q=5$\\
	$0$ & $q=6$\\
	$0$ & $q=7$\\
	$0$ & $q=8$
\end{tabular}
\right.
$

\begin{thm}
$H^q_\partial(GL_4(\Z),(2a,2b,2a,2d))=
\left\{
\begin{tabular}{llll}
	$
	S_{2a-2d+2}
	+
	S_{2a-2d+4}
	+
	\C
	$
		& $q=5$\\
	$0$ 
		& $q\neq 5$
\end{tabular}
\right.
$

\end{thm}




\subsection{Cohomology of $GL_4(\Z)$. Highest weight $(2a+1,2a+1,2a+1,2d+1)$.  Case $A7'$}


\subsubsection{Cohomology of the parabolic subgroups. Highest weight $(2a+1,2a+1,2a+1,2d+1)$. Case $A7'$}
For that case the weights are $(2a+1,2a+1,2a+1,2d+1)$.

$P_0$: 
$\begin{tabular}{lllll}
$w$ 		& $l$	& 	& $w(\lambda+\rho)-\rho$\\
$2143$ 	& $2$&	& $(2a|2a+2|2d|2a+2)$\\
$2341$ 	& $3$&	& $(2a|2a|2d|2a+4)$\\
$4123$	& $3$&	& $(2d-2|2a+2|2a+2|2a+2)$\\
$4321$	& $6$&	& $(2d-2|2a|2a+2|2a+4)$\\
\end{tabular}
$

$P_{12}$: 
$\begin{tabular}{lllll}
$w$ 		& $l$	& 	& $w(\lambda+\rho)-\rho$\\
$1243$ 	& $1$&	& $(2a+1,2a+1|2d|2a+2)$\\
$1423$ 	& $2$&	& $(2a+1,2d-1|2a+2|2a+2)$\\
$2341$	& $3$&	& $(2a,2a|2d|2a+4)$\\
$3421$	& $5$&	& $(2a-1,2d-1|2a+2|2a+4)$\\
\end{tabular}
$

$P_{23}$: 
$\begin{tabular}{lllll}
$w$ 		& $l$	& 	& $w(\lambda+\rho)-\rho$\\
$2143$ 	& $2$&	& $(2a|2a+2,2d|2a+2)$\\
$2341$	& $3$&	& $(2a|2a,2d|2a+4)$\\
$4123$	& $3$&	& $(2d-2|2a+2, 2a+2|2a+2)$\\
$4231$	& $5$&	& $(2d-2|2a+1,2a+1|2a+4)$\\
\end{tabular}
$

$P_{34}$: 
$\begin{tabular}{lllll}
$w$ 		& $l$	& 	& $w(\lambda+\rho)-\rho$\\
$2134$ 	& $1$&	& $(2a|2a+2|2a+1,2d+1)$\\
$2314$	& $2$&	& $(2a|2a|2a+3,2d+1)$\\
$4123$	& $3$&	& $(2d-2|2a+2|2a+2,2a+2)$\\
$4312$ 	& $5$&	& $(2d-2|2a|2a+3,2a+3)$\\
\end{tabular}
$

$P_{13}$: 
$\begin{tabular}{lllll}
$w$ 		& $l$	& 	& $w(\lambda+\rho)-\rho$\\
$1243$ 	& $1$&	& $(2a+1,2a+1,2d|2a+2)$\\
$2341$ 	& $3$&	& $(2a,2a,2d|2a+4)$\\
\end{tabular}
$

$P_{12,34}$: 
$\begin{tabular}{lllll}
$w$ 		& $l$	& 	& $w(\lambda+\rho)-\rho$\\
$1234$ 	& $0$&	& $(2a+1,2a+1|2a+1,2d+1)$\\
$1423$ 	& $2$&	& $(2a+1, 2d-1|2a+2,2a+2)$\\
$2314$ 	& $2$&	& $(2a,2a|2a+3,2d+1)$\\
$3412$	& $4$&	& $(2a-1,2d-1|2a+3,2a+3)$\\
\end{tabular}
$

$P_{24}$: 
$\begin{tabular}{lllll}
$w$ 		& $l$	& 	& $w(\lambda+\rho)-\rho$\\
$2134$ 	& $1$&	& $(2a|2a+2,2a+1,2d+1)$\\
$4123$ 	& $3$&	& $(2d-2|2a+2,2a+2,2a+2)$\\
\end{tabular}
$


\subsubsection{$E_1$-page. Highest weight $(2a+1,2a+1,2a+1,2d+1)$.  Case $A7'$}

{\begin{center}
\begin{small}
\scriptsize\renewcommand{\arraystretch}{2.2}
\begin{longtable}{|c|c|c|c|c|c|c|c|c|}
\hline
$E_1^{0,0}=0$ 
& $E_1^{1,0}=0$ & 
$E_1^{2,0}=0$
\\
\hline
$E_1^{0,1}=0$
& $E_1^{1,1}=0$
& $E_1^{2,1}=0$
\\
\hline
$E_1^{0,2}=(2a+1,2a+1|2a+1,2d+1)=0$ 
& $E_1^{1,2}=
	\left\{
	\begin{tabular}{ll}
	$(2a+1,2a+1|2d|2a+2)=0$\\
	$(2a|2a+2|2a+1,2d+1)$
	\end{tabular}
	\right.$ 
& $E_1^{2,2}=(2a|2a+2|2d|2a+2)$
\\
\hline
$E_1^{0,3}=
	\left\{
	\begin{tabular}{ll}
	$(H^2(2a+1,2a+1,2d)|2a+2)$\\
	$(2a+1,2d-1|2a+2,2a+2)$\\
	$(2a,2a|2a+3,2d+1)$\\
	$(2a|H^2(2a+2,2a+1,2d+1))$\\
	$(2d-2,H^0(2a+2,2a+2,2a+2))$
	\end{tabular}
	\right.
	$
& $E_1^{1,3}=
	\left\{
	\begin{tabular}{ll}
	$(2a+1,2d-1|2a+2|2a+2)$\\
	$(2a,2a|2d|2a+4)$\\
	$(2a|2a+2,2d|2a+2)$\\
	$(2d-2|2a+2,2a+2|2a+2)$\\
	$(2a|2a|2a+3,2d+1)$\\
	$(2d-2|2a+2|2a+2,2a+2)$
	\end{tabular}
	\right.
	$
& $E_1^{2,3}=
	\left\{
	\begin{tabular}{ll}
	$(2a|2a|2d|2a+4)$\\
	$(2d-2|2a+2|2a+2|2a+2)$
	\end{tabular}
	\right.
	$\\
\hline
$E_1^{0,4}=
	\left\{
	\begin{tabular}{ll}
	$(H^3(2a+1,2a+1,2d)|2a+2)$\\
	$(2a|H^3(2a+2,2a+1,2d+1))$
	\end{tabular}
	\right.$
& $E_1^{1,4}=(2a|2a,2d|2a+4)$
& $E_1^{2,4}=0$\\
\hline
$E_1^{0,5}=\left\{
	\begin{tabular}{ll}
	$(H^2(2a,2a,2d)|2a+4)$\\
	$(2d-2|H^2(2a+2,2a+2,2a+2))$
	\end{tabular}
	\right.$
& $E_1^{1,5}=0$
& $E_1^{2,5}=0$\\
\hline
$E_1^{0,6}=
	\left\{
	\begin{tabular}{ll}
	$(H^3(2a,2a,2d)|2a+4)$\\
	$(2a-1,2d-1|2a+3,2a+3)=0$\\
	$(2d-2|H^3(2a+2,2a+2,2a+2))$
	\end{tabular}
	\right.$
& $E_1^{1,6}=
	\left\{
	\begin{tabular}{ll}
	$(2a-1,2d-1|2a+2|2a+4)$\\
	$(2d-2|2a+1,2a+1|2a+4)=0$\\
	$(2d-2|2a|2a+3,2a+3)=0$
	\end{tabular}
	\right.$
& $E_1^{2,6}=(2d-2|2a|2a+2|2a+4)$\\
\hline
\end{longtable}
\end{small}
\end{center}


\subsubsection{Certain cohomology groups of $GL_3(\Z)$ that appear on the $E_1$-page in the case $A7'$.}
From the computations of the cohomology of $GL_3(\Z)$ from the previous section, we have
{\begin{center}
\scriptsize\renewcommand{\arraystretch}{2.2}
\begin{longtable}{|c|c|c|}
\hline
$H^2(GL_3(\Z),(2a+1,2a+1,2d))
	=
	\Delta_0+\Delta_1
	$
&$H^3(GL_3(\Z),(2a+1,2a+1,2d))=0$\\
\hline
$H^2(GL_3(\Z),(2a+2,2a+1,2d+1))
	=
	\Delta_2
	$
& $H^3(GL_3(\Z),(2a+2,2a+1,2d+1))=(2a,2d|2a+4)$
\\
\hline
$H^2(GL_3(\Z),(2a,2a,2d))=0$
& $H^3(GL_3(\Z),(2a,2a,2d))
	=(\overline{2a-1,2d-1}|2a+2)$\\
	\hline
$H^2(GL_3(\Z),(2a+2,2a+2,2a+2))=0$
& $H^3(GL_3(\Z),(2a+2,2a+2,2a+2))=0$\\
	\hline
\end{longtable}
\end{center}
where
$\Delta_0=(2a|2a+2|2d)+(2d-2|2a+2|2a+2)$
$\Delta_1 = (\overline{2a+1,2d-1}|2a+2) + (2a|2a+2,2d)$.
	and
$\Delta_2 = (2a+2,2d|2a+2) + (2a|\overline{2a+3,2d+1})$.
Also, $H^0(GL_3(\Z),(2a+2,2a+2,2a+2))=H^0(GL_3(\Z),\C)=\C$.


\subsubsection{$E_2$ page. Highest weight $(2a+1,2a+1,2a+1,2d+1)$.  Case $A7'$}

For $q=2$ we have a surjective map $E_1^{1,2} \rightarrow E_1^{2,2}$
Therefore $E_2^{0,2}=E_1^{2,2} =0$ and $E_2^{1,2}=(2a|2a+2|\overline{2a+1,2d+1})$

For $q=3$, 

EXPLAIN !!! (See A5')

Now let us examine the third line $E_1^{p,3}$ of the $E_1$-page. 
We have the following isomorphisms coming from  the $d_1$-map from $E_1^{1,3}$ to $E_1^{2,3}$
\[(2a, 2a|2d|2a + 4)\rightarrow (2a, 2a|2d|2a + 4)\]
and
\[((2a,2a|2a+3,2d+1)\rightarrow (2a,2a|2a+3,2d+1).\]
We also have the following isomorphism coming from  the $d_1$-map from $E_1^{0,3}$ to $E_1^{1,3}$.
\[(2b-2|2a+2|2b+2,2b+2)\rightarrow (2b-2|2a+2|2b+2|2b+2).\]
We an exact sequence
\[0\rightarrow (2d -2|2a+2,2a+2,2a+2)\rightarrow (2d -2|2a+2,2a+2|2a+2)(2d -2|2a+2|2a+2,2a+2) \rightarrow (2d -2|2a+2|2a+2|2a+2)\rightarrow 0.\]

The remaining part of the third row $E_1^{p,3}$ is
\[(H^2(2a+1,2a+1,2d)|2a+2)+(2a|H2(2a + 2,2a + 1,2d + 1))\rightarrow  (2a|2a + 2, 2d|2a + 2).\]

Therefore, $E_2^{1,3}=E_2^{2,3}=0$ and 
\begin{eqnarray*}
E_2^{0,3}	=	&&ker[E_1^{0,3}\rightarrow E_1^{1,3}]=\\
		=	&&ker[(H^2(2a+1,2a+1,2d)|2a+2)+(2a|H2(2a + 2,2a + 1,2d + 1))\rightarrow\\
			&&\rightarrow  (2a|2a + 2, 2d|2a + 2)]=\\
			&& ker[\Delta_0+\Delta_1+\Delta_2\rightarrow (2a|2a + 2, 2d|2a + 2)]
\end{eqnarray*}
Since both $\Delta_1$ and $\Delta_2$ are isomorphic to $((2a|2a + 2, 2d|2a + 2)$,
and $\Delta_0$ is $\C$.
We obtain that $E_2^{0,3}$ is also isomorphic to $(2a|2a + 2, 2d|2a + 2)+ (2d-2|2a+2|2a+2|2a+2)$

For $q=4$,
we have an isomorphism $E_1^{0,4}\rightarrow E_1^{1,4}$. Therefore $E_2^{p,4}=0$ for $p=0,1,2$.

For  $q=6$, we have a short exact sequence
\[0 \rightarrow E_1^{0,6}\rightarrow E_1^{1,6}\rightarrow E_1^{2,6}\rightarrow 0.\]
Therefore,
$E_2^{p,6}=0$ for $p=0,1,2$.

This is summarized in the following table. The empty boxes denote that the corresponding $E_2^{p,q}$ term vanishes.

{\begin{center}
\begin{small}
\scriptsize\renewcommand{\arraystretch}{2.2}
\begin{longtable}{|c|c|c|c|c|c|c|c|c|}
\hline
& $p=0$
& $p=1$
& $p=2$\\
\hline
$q=0$
&  
&  
& \\
\hline
$q=1$
&
& 
&
\\
\hline
$q=2$
& 
& $E_2^{1,2}=(2a|2a+2|\overline{2a+1,2d+1})$
& 
\\
\hline
$q=3$
&  $E_2^{0,3}= (2d-2|2a+2|2a+2|2a+2)+\Delta$
&
&\\
\hline
$q=4$
& 
& 
& \\
\hline
$q=5$
& 
& 
&\\
\hline
$q=6$
&
& 
& \\
\hline
\end{longtable}
\end{small}
\end{center}
where 
$\Delta= ker[\Delta_1+\Delta_2\rightarrow (2a|2a + 2, 2d|2a + 2)$, which is isomorphic to $(2a|2a + 2, 2d|2a + 2)$.
]


\subsubsection{Boundary cohomology. Highest weight $(2a+1,2a+1,2a+1,2d+1)$.  Case $A7'$.}
Again the boundary cohomology degenerates at the $E_2$-page. Therefore, up to semisimplicity we have
$H^n_\partial(GL_4(\Z),V_\lambda)=\bigoplus_{p+q=n}E_2^{p,q}$.
More systematically, we have

Let $H^q=H^q_\partial(GL_4(\Z),V_\lambda)$ be the cohomology of $GL_4(\Z)$ with coefficients in the highest weight representation with weight 
$\lambda=(2a+1,2a+1,2a+1,2d+1)$ written as a character of the split maximal torus. Then,
\[H^q
=
\left\{
\begin{tabular}{lll}
$0$
		& $q=2$\\
$\left\{
\begin{tabular}{lll}
	$E_2^{0,3}= (2d-2|2a+2|2a+2|2a+2)+\Delta$\\
	$E_2^{1,2}=(2a|2a+2|\overline{2a+1,2d+1})$
\end{tabular}
\right\}$
		& $q=3$\\	

		& $q=4$\\
$0$
		& $q=5$\\ 
$0$
		& $q=6$
\end{tabular}
\right.
\]
where $\Delta= ker[\Delta_1+\Delta_2\rightarrow (2a|2a + 2, 2d|2a + 2)$, which is isomorphic to $(2a|2a + 2, 2d|2a + 2)$.

\begin{thm}
\[H^q
=
\left\{
\begin{tabular}{lll}
$S_{2a-2d+4}
+
S_{2a-2d+2}
+
\C$
		& $q=3$\\	
$0$
		& $q\neq 3$\\
\end{tabular}
\right.
\]

\end{thm}




\subsection{Cohomology of $GL_4(\Z)$. Highest weight $(0,0,0,0)$.  Case $A8$}


\subsubsection{Cohomology of the parabolic subgroups. Highest weight $(0,0,0,0)$. Case $A8$}
Let $\lambda=(0,0,0,0)$.

$P_0$: 
$\begin{tabular}{lllll}
$w$ 		& $l$	& 	& $w(\lambda+\rho)-\rho$\\
$1234$ 	& $0$&	& $(0|0|0|0)$\\
$1432$ 	& $3$&	& $(0|-2|0|2)$\\
$3214$	& $3$&	& $(-2|0|2|0)$\\
$3412$	& $4$&	& $(-2|-2|2|2)$\\
\end{tabular}
$

$P_{12}$: 
$\begin{tabular}{lllll}
$w$ 		& $l$	& 	& $w(\lambda+\rho)-\rho$\\
$1234$ 	& $0$&	& $(0,0|0|0)$\\
$1432$ 	& $3$&	& $(0,-2|0|2)$\\
$2314$	& $2$&	& $(-1,-1|2|0)$\\
$3412$	& $4$&	& $(-2,-2|2|2)$\\
\end{tabular}
$

$P_{23}$: 
$\begin{tabular}{lllll}
$w$ 		& $l$	& 	& $w(\lambda+\rho)-\rho$\\
$1234$ 	& $0$&	& $(0|0,0|0)$\\
$1342$ 	& $2$&	& $(0|-1,-1|2)$\\
$3124$	& $2$&	& $(-2|1, 1|0)$\\
$3142$	& $3$&	& $(-2|1,-1|2)$\\
\end{tabular}
$

$P_{34}$: 
$\begin{tabular}{lllll}
$w$ 		& $l$	& 	& $w(\lambda+\rho)-\rho$\\
$1234$ 	& $0$&	& $(0|0|0,0)$\\
$1423$ 	& $2$&	& $(0|-2|1,1)$\\
$3214$	& $3$&	& $(-2|0|2,0)$\\
$3412$	& $4$&	& $(-2|-2|2,2)$\\
\end{tabular}
$

$P_{13}$: 
$\begin{tabular}{lllll}
$w$ 		& $l$	& 	& $w(\lambda+\rho)-\rho$\\
$1234$ 	& $0$&	& $(0,0,0|0)$\\
$1342$ 	& $2$&	& $(0,-1,-1|2)$\\
\end{tabular}
$

$P_{12,34}$: 
$\begin{tabular}{lllll}
$w$ 		& $l$	& 	& $w(\lambda+\rho)-\rho$\\
$1234$ 	& $0$&	& $(0,0|0,0)$\\
$1423$ 	& $2$&	& $(0, -2|1,1)$\\
$2314$ 	& $2$&	& $(-1,-1|2,0)$\\
$3412$	& $4$&	& $(-2,-2|2,2)$\\
\end{tabular}
$

$P_{24}$: 
$\begin{tabular}{lllll}
$w$ 		& $l$	& 	& $w(\lambda+\rho)-\rho$\\
$1234$ 	& $0$&	& $(0|0,0,0)$\\
$3124$ 	& $2$&	& $(-2|1,1,0)$\\
\end{tabular}
$


\subsubsection{$E_1$-page. Highest weight $(0,0,0,0)$.  Case $A8$.}

{\begin{center}
\begin{small}
\scriptsize\renewcommand{\arraystretch}{2.2}
\begin{longtable}{|c|c|c|c|c|c|c|c|c|}
\hline
$E_1^{0,0}=\left\{
	\begin{tabular}{lll}
	$(H^0(0,0,0)|0)$\\
	$(0,0|0,0)$\\
	$(0|H^0(0,0,0))$ 
	\end{tabular}
	\right.
	$
& $E_1^{1,0}=
	\left\{
	\begin{tabular}{lll}
		$(0,0|0|0)$\\
		$(0|0,0|0)$\\
		$(0|0|0,0)$
	\end{tabular}
	\right.
	$
& $E_1^{2,0}=(0|0|0|0)$
\\
\hline
$E_1^{0,1}=0$
& $E_1^{1,1}=0$
& $E_1^{2,1}=0$
\\
\hline
$E_1^{0,2}=0$	
& $E_1^{1,2}=0$ 
& $E_1^{2,2}=0$
\\
\hline
$E_1^{0,3}=0$
& $E_1^{1,3}=0$
& $E_1^{2,3}=
	\left\{
	\begin{tabular}{ll}
	$(-2|0|2|0)$\\
	$(0|-2|0|2)$
	\end{tabular}
	\right.
	$\\
\hline
$E_1^{0,4}=
	\left\{
	\begin{tabular}{ll}
	$(H^2(0,-1,-1)|2)$\\
	$(-2,-2|2,2)$\\
	$(-2|H^2(1,1,0))$
	\end{tabular}
	\right.$
& $E_1^{1,4}=
	\left\{
	\begin{tabular}{ll}
	$(0,-2|0|2)=0$\\
	$(-2,-2|2|2)$\\
	$(-2|1,-1|2)$\\
	$(-2|0|2,0)=0$\\
	$(-2|-2|2,2)$
	\end{tabular}
	\right.$
& $E_1^{2,4}=(-2|-2|2|2)$\\
\hline
$E_1^{0,5}=\left\{
	\begin{tabular}{ll}
	$(H^3(0,-1,-1)|2)=0$\\
	$(-2|H^3(1,1,0))=0$
	\end{tabular}
	\right.$
& $E_1^{1,5}=0$
& $E_1^{2,5}=0$\\
\hline
$E_1^{0,6}=0$
& $E_1^{1,6}=0$
& $E_1^{2,6}=0$\\
\hline
\end{longtable}
\end{small}
\end{center}


\subsubsection{Certain cohomology groups of $GL_3(\Z)$ that appear on the $E_1$-page in the case $A8$.}
From the computations of the cohomology of $GL_3(\Z)$ from the previous section, we have
{\begin{center}
\scriptsize\renewcommand{\arraystretch}{2.2}
\begin{longtable}{|c|c|c|c|}
\hline
$H^0(GL_3(\Z),(0,0,0))
	=(0|0|0)$
&
$H^2(GL_3(\Z),(0,0,0))
	=0$
&
$H^3(GL_3(\Z),(0,0,0))
	=0$
\\
\hline
$H^0(GL_3(\Z),(0,-1,-1))
	=0$
&
$H^2(GL_3(\Z),(0,-1,-1))
	=\Delta_0$
& 
$H^3(GL_3(\Z),(0,-1,-1))
	=0$
\\
\hline
$H^2(GL_3(\Z),(1,1,0))
	=0$
&
$H^2(GL_3(\Z),(1,1,0))
	=\Delta'_0$
&
$H^3(GL_3(\Z),(1,1,0))=0$\\
\hline
\end{longtable}
\end{center}
where
$\Delta_0
	\subset (0|-2|0)+(-2|-2|2)$,	
	and
$\Delta'_0\subset (0|2|0)+(-2|2|2)$
are diagonal embeddings.


\subsubsection{$E_2$-page. Highest weight $(0,0,0,0)$. Case $A8$}

For $q=4$,

(EXPLAIN !!!)

$\Delta_0+\Delta'_0\rightarrow (-2|\overline{1,-1}|2)$ is surjective Because of GL5.

 we have that $E_1^{0,4}\rightarrow E_1^{1,4}+E_2^{2,3}$ is injective, since $\Delta_0$, $\Delta_1$ and $\Delta_2$ are diagonal embeddings. 
Also, $E_1^{1,4}\rightarrow E_1^{2,4}$  is surjective. 
If $\Delta_0$ is a ghost class then there is a non-trivial $d_2$-map. However, from computation of the cohomology of $GL_4(\Z)$ with coefficients in $det$ and from the cohomology of $GL_5(\Z)$ with trivial coefficients we obtain that $\Delta_0$ embeds diagonally in the boundary cohomology. This is only verified in the case of $V_3\otimes  det$ and $V_3^*\otimes det$. In case $A5$ we need to consider higher degree of odd symmetric powers. Here we assume that again $D_0$ embeds diagonally.
Using this assumption, we obtain that $d_2=0$. Therefore
Therefore, $E_2^{0,4}=E_2^{2,4}=0$ and $E_2^{1,4}=(-2|\overline{1,-1}|2)$

{\begin{center}
\begin{small}
\scriptsize\renewcommand{\arraystretch}{2.2}
\begin{longtable}{|c|c|c|c|c|c|c|c|c|}
\hline
& $p=0$
& $p=1$
& $p=2$\\
\hline
$q=0$
&  
&  
& 
\\
\hline
$q=1$
&
& 
&
\\
\hline
$q=2$
&
&  
& 
\\
\hline
$q=3$
&
& 
& $E_1^{2,3}=
	\left\{
	\begin{tabular}{ll}
	$(-2|0|2|0)$\\
	$(0|-2|0|2)$
	\end{tabular}
	\right.
	$
\\
\hline
$q=4$
& $E_2^{0,4}=\Delta''_0$
& 
& \\
\hline
$q=5$
& 
& 
&\\
\hline
$q=6$
&
& 
& \\
\hline
\end{longtable}
\end{small}
\end{center}

There is a nontrivial $d_2$-map, sending $d_2:\Delta''\rightarrow (-2|0|2|0)+(0|-2|0|2)$ injectively. The cokernel of this map is $E_3^{2,3}$

\subsubsection{Boundary cohomology. Highest weight $(0,0,0,0)$.  Case $A8$.}

By definition, the boundary cohomology is the one to which the spectral sequence converges.
That is, $\bigoplus_{prk(P)=p+1}H^q(P,V_\lambda)=>H^{p+q}_\partial(GL_4(\Z),V_\lambda)$.
From the previous subsection, it is clear that in the case 
$A8$
the spectral sequence degenerates at the $E_2$-page.
Therefore, up to semisimplicity, 
$H^{n}_\partial(GL_4(\Z),V_\lambda)=\bigoplus_{p+q=n} E_2^{p,q}$.
Since we know the $E_2$ terms which are of the form, we can find the  boundary cohomology using our computation of the $E_2$ terms
Let us denote temporarily $H^q_\partial(GL_4(\Z),A8)=H^q_\partial(GL_4(\Z),(0,0,0,0))$.
Then, we can summarize the result for the boundary cohomology in the following table.

$H^q_\partial(GL_4(\Z),A8)=
\left\{
\begin{tabular}{llll}
	$(0|0|0|0)$ & $q=0$\\
	$0$ & $q=1$\\
	$0$ & $q=2$\\
	$0$ & $q=3$\\
	$0$ & $q=4$\\
	$E_2^{2,3} = (-2|0|2|0)$
	& $q=5$\\
	$0$ & $q=6$\\
	$0$ & $q=7$\\
	$0$ & $q=8$
\end{tabular}
\right.
$

\begin{thm}
$H^q_\partial(GL_4(\Z),\C)=
\left\{
\begin{tabular}{llll}
	$\C$ & $q=0$\\
	$\C$	& $q=5$\\
	$0$ & $q\neq 0, 5$\\
\end{tabular}
\right.
$
\end{thm}




\subsection{Cohomology of $GL_4(\Z)$. Highest weight $(1,1,1,1)$.  Case $A8'$}


\subsubsection{Cohomology of the parabolic subgroups. Highest weight $(1,1,1,1)$. Case $A8'$}
For that case the weights are $(1,1,1,1)$.

$P_0$: 
$\begin{tabular}{lllll}
$w$ 		& $l$	& 	& $w(\lambda+\rho)-\rho$\\
$2143$ 	& $2$&	& $(0|2|0|2)$\\
$2341$ 	& $3$&	& $(0|0|0|4)$\\
$4123$	& $3$&	& $(-2|2|2|2)$\\
$4321$	& $6$&	& $(-2|0|2|4)$\\
\end{tabular}
$

$P_{12}$: 
$\begin{tabular}{lllll}
$w$ 		& $l$	& 	& $w(\lambda+\rho)-\rho$\\
$1243$ 	& $1$&	& $(1,1|0|2)$\\
$1423$ 	& $2$&	& $(1,-1|2|2)$\\
$2341$	& $3$&	& $(0,0|0|4)$\\
$3421$	& $5$&	& $(-1,-1|2|4)$\\
\end{tabular}
$

$P_{23}$: 
$\begin{tabular}{lllll}
$w$ 		& $l$	& 	& $w(\lambda+\rho)-\rho$\\
$2143$ 	& $2$&	& $(0|2,0|2)$\\
$2341$	& $3$&	& $(0|0,0|4)$\\
$4123$	& $3$&	& $(-2|2, 2|2)$\\
$4231$	& $5$&	& $(-2|1,1|4)$\\
\end{tabular}
$

$P_{34}$: 
$\begin{tabular}{lllll}
$w$ 		& $l$	& 	& $w(\lambda+\rho)-\rho$\\
$2134$ 	& $1$&	& $(0|2|1,1)$\\
$2314$	& $2$&	& $(0|0|3,1)$\\
$4123$	& $3$&	& $(-2|2|2,2)$\\
$4312$ 	& $5$&	& $(-2|0|3,3)$\\
\end{tabular}
$

$P_{13}$: 
$\begin{tabular}{lllll}
$w$ 		& $l$	& 	& $w(\lambda+\rho)-\rho$\\
$1243$ 	& $1$&	& $(1,1,0|2)$\\
$2341$ 	& $3$&	& $(0,0,0|4)$\\
\end{tabular}
$

$P_{12,34}$: 
$\begin{tabular}{lllll}
$w$ 		& $l$	& 	& $w(\lambda+\rho)-\rho$\\
$1234$ 	& $0$&	& $(1,1|1,1)$\\
$1423$ 	& $2$&	& $(1, -1|2,2)$\\
$2314$ 	& $2$&	& $(0,0|3,1)$\\
$3412$	& $4$&	& $(-1,-1|3,3)$\\
\end{tabular}
$

$P_{24}$: 
$\begin{tabular}{lllll}
$w$ 		& $l$	& 	& $w(\lambda+\rho)-\rho$\\
$2134$ 	& $1$&	& $(0|2,1,1)$\\
$4123$ 	& $3$&	& $(-2|2,2,2)$\\
\end{tabular}
$


\subsubsection{$E_1$-page. Highest weight $(1,1,1,1)$.  Case $A8'$}

{\begin{center}
\begin{small}
\scriptsize\renewcommand{\arraystretch}{2.2}
\begin{longtable}{|c|c|c|c|c|c|c|c|c|}
\hline
$E_1^{0,0}=0$ 
& $E_1^{1,0}=0$ & 
$E_1^{2,0}=0$
\\
\hline
$E_1^{0,1}=0$
& $E_1^{1,1}=0$
& $E_1^{2,1}=0$
\\
\hline
$E_1^{0,2}=0$ 
& $E_1^{1,2}=0$ 
& $E_1^{2,2}=(0|2|0|2)$
\\
\hline
$E_1^{0,3}=
	\left\{
	\begin{tabular}{ll}
	$(H^2(1,1,0)|2)$\\
	$(H^0(0,0,0)|4)$\\
	$(1,-1|2,2)$\\
	$(0,0|3,1)$\\
	$(0|H^2(2,1,1))$\\
	$(-2,H^0(2,2,2))$
	\end{tabular}
	\right.
	$
& $E_1^{1,3}=
	\left\{
	\begin{tabular}{ll}
	$(1,-1|2|2)$\\
	$(0,0|0|4)$\\
	$(0|2,0|2)=0$\\
	$(0|0,0|4)$\\
	$(-2|2,2|2)$\\
	$(0|0|3,1)$\\
	$(-2|2|2,2)$
	\end{tabular}
	\right.
	$
& $E_1^{2,3}=
	\left\{
	\begin{tabular}{ll}
	$(0|0|0|4)$\\
	$(-2|2|2|2)$
	\end{tabular}
	\right.
	$\\
\hline
$E_1^{0,4}=0$
& $E_1^{1,4}=$
& $E_1^{2,4}=0$\\
\hline
$E_1^{0,5}=0$
& $E_1^{1,5}=0$
& $E_1^{2,5}=0$\\
\hline
$E_1^{0,6}=0$
& $E_1^{1,6}=0$
& $E_1^{2,6}=(-2|0|2|4)$\\
\hline
\end{longtable}
\end{small}
\end{center}


\subsubsection{Certain cohomology groups of $GL_3(\Z)$ that appear on the $E_1$-page in the case $A8'$.}
From the computations of the cohomology of $GL_3(\Z)$ from the previous section, we have
{\begin{center}
\scriptsize\renewcommand{\arraystretch}{2.2}
\begin{longtable}{|c|c|c|c|}
\hline
$H^0(GL_3(\Z),(0,0,0))
	=
	(0|0|0)
	$
& 
$H^2(GL_3(\Z),(0,0,0))
	=
	0
	$
&$H^3(GL_3(\Z),(0,0,0))=0$\\
\hline
$H^0(GL_3(\Z),(1,1,0))
	=
	0
	$
& 
$H^2(GL_3(\Z),(1,1,0))
	=
	\Delta_0
	$
&
$H^3(GL_3(\Z),(1,1,0))=0$\\
\hline
$H^0(GL_3(\Z),(2,1,1))
	=
	0
	$
&
$H^2(GL_3(\Z),(2,1,1))
	=
	\Delta'_1
	$
& $H^3(GL_3(\Z),(2,1,1))=(0|0|4)$
\\
\hline
\end{longtable}
\end{center}
where
$\Delta_0=(0|2|0)+(-2|2|2)$
$\Delta_1 = (\overline{1,-1}|2) + (0|2,0)$.4
	and
$\Delta_2 = (2,0|2) + (0|\overline{3,1})$.


\subsubsection{$E_2$ page. Highest weight $(1,1,1,1)$.  Case $A8'$}

For $q=2$ we have a surjective map $E_1^{1,2} \rightarrow E_1^{2,2}$
Therefore $E_2^{0,2}=E_1^{2,2} =0$ and $E_2^{1,2}=(0|2|\overline{1,1})$

For $q=3$, 

EXPLAIN !!! 

2
Now let us examine the third line $E_1^{p,3}$ of the $E_1$-page. 
We have the following isomorphisms coming from  the $d_1$-map from $E_1^{1,3}$ to $E_1^{2,3}$
\[(0, 0|0|4)\rightarrow (0, 0|0|4)\]
and
\[((0,0|3,1)\rightarrow (0,0|3,1).\]
We also have the following isomorphism coming from  the $d_1$-map from $E_1^{0,3}$ to $E_1^{1,3}$.
\[(-2|2|2,2)\rightarrow (-2|2|2|2).\]
We an exact sequence
\[0\rightarrow (0 -2|2,2,2)\rightarrow (0 -2|2,2|2)(0 -2|2|2,2) \rightarrow (0 -2|2|2|2)\rightarrow 0.\]

The remaining part of the third row $E_1^{p,3}$ is
\[(H^2(1,1,0)|2)+(0|H2(2,1,1))\rightarrow  (0|2, 0|2).\]

Therefore, $E_2^{1,3}=E_2^{2,3}=0$ and 
\begin{eqnarray*}
E_2^{0,3}	=	&&ker[E_1^{0,3}\rightarrow E_1^{1,3}]=\\
		=	&&ker[(H^2(1,1,0)|2)+(0|H2(2,1,1))\rightarrow\\
			&&\rightarrow  (0|2, 0|2)]=\\
			&& ker[\Delta_0+\Delta_1+\Delta_2\rightarrow (0|2, 0|2)]
\end{eqnarray*}
Since both $\Delta_1$ and $\Delta_2$ are isomorphic to $((0|2, 0|2)$,
and $\Delta_0$ is $\C$.
We obtain that $E_2^{0,3}$ is also isomorphic to $(0|2, 0|2)+ (-2|2|2|2)$

For $q=4$,
we have an isomorphism $E_1^{0,4}\rightarrow E_1^{1,4}$. Therefore $E_2^{p,4}=0$ for $p=0,1,2$.

For  $q=6$, we have a short exact sequence
\[0 \rightarrow E_1^{0,6}\rightarrow E_1^{1,6}\rightarrow E_1^{2,6}\rightarrow 0.\]
Therefore,
$E_2^{p,6}=0$ for $p=0,1,2$.

This is summarized in the following table. The empty boxes denote that the corresponding $E_2^{p,q}$ term vanishes.

{\begin{center}
\begin{small}
\scriptsize\renewcommand{\arraystretch}{2.2}
\begin{longtable}{|c|c|c|c|c|c|c|c|c|}
\hline
& $p=0$
& $p=1$
& $p=2$\\
\hline
$q=0$
&  
&  
& \\
\hline
$q=1$
&
& 
&
\\
\hline
$q=2$
& 
&
&  $E_2^{2,2}=(0|2|0|2)$ 
\\
\hline
$q=3$
&  $E_2^{0,3}= 
	\left\{
	\begin{tabular}{lll}
	$(H^2(1,1,0)|2)$\\
	$(0|H^2(2,1,1))$
	\end{tabular}
	\right.
	$
&
&\\
\hline
$q=4$
& 
& 
& \\
\hline
$q=5$
&
& 
&\\
\hline
$q=6$
&
& 
& $E_2^{2,6}=(-2|0|2|4)$ \\
\hline
\end{longtable}
\end{small}
\end{center}
where 
$\Delta= ker[\Delta_1+\Delta_2\rightarrow (0|2, 0|2)$, which is isomorphic to $(0|2, 0|2)$.
]


The map $d_2:(H^2(1,1,0)|2)+(0|H^2(2,1,1))\rightarrow (0|2|0|2)$
is nontrivial. The kernel is $E_3^{0,3}=(0|2|0|2).$

This is summarized in the following table. The empty boxes denote that the corresponding $E_3^{p,q}$ term vanishes.

{\begin{center}
\begin{small}
\scriptsize\renewcommand{\arraystretch}{2.2}
\begin{longtable}{|c|c|c|c|c|c|c|c|c|}
\hline
& $p=0$
& $p=1$
& $p=2$\\
\hline
$q=0$
&  
&  
& \\
\hline
$q=1$
&
& 
&
\\
\hline
$q=2$
& 
& 
& 
\\
\hline
$q=3$
&  $E_3^{0,3}=(0|2|0|2)$
&
&\\
\hline
$q=4$
& 
& 
& \\
\hline
$q=5$
& 
& 
&\\
\hline
$q=6$
&
& 
& $E_3^{2,6}=(-2|0|2|4)$ \\
\hline
\end{longtable}
\end{small}
\end{center}


\subsubsection{Boundary cohomology. Highest weight $(1,1,1,1)$.  Case $A8'$.}
Again the boundary cohomology degenerates at the $E_2$-page. Therefore, up to semisimplicity we have
$H^n_\partial(GL_4(\Z),V_\lambda)=\bigoplus_{p+q=n}E_2^{p,q}$.
More systematically, we have

Let $H^q=H^q_\partial(GL_4(\Z),V_\lambda)$ be the cohomology of $GL_4(\Z)$ with coefficients in the highest weight representation with weight 
$\lambda=(1,1,1,1)$ written as a character of the split maximal torus. Then,
\[H^q
=
\left\{
\begin{tabular}{lll}
$0$
		& $q=2$\\
$E_3^{0,3}= (0|2|0|2)$
		& $q=3$\\	
$0$
		& $q=4$\\
$0$
		& $q=5$\\ 
$0$
		& $q=6$\\
$0$
		& $q=7$\\
$E_3^{2,6}=(-2|0|2|4)$
		& $q=8$
\end{tabular}
\right.
\]

\begin{thm}
$H^q_\partial(GL_4(\Z),det)=
\left\{
\begin{tabular}{llll}
	$\C$ & $q=3$\\
	$\C$	& $q=8$\\
	$0$ & $q\neq 3, 8$\\
\end{tabular}
\right.
$
\end{thm}




\subsection{Cohomology of $GL_4(\Z)$. Highest weight $(2a+1,2b+1,2c,2d)$.  Case $B1$}


\subsubsection{Cohomology of the parabolic subgroups. Highest weight $(2a+1,2b+1,2c,2d)$. Case $B1$}
Let $\lambda=(2a+1,2b+1,2c,2d)$.

For the minimal parabolic subgroup $P_0$ we have to consider only the permutations that send $1$ and $4$ to $2$ or $4$, $2$ and $3$ to $1$ or $3$.
All possible such permutations are: $2134$, $2431$, $3124$ and $3421$ . For them we have the following weights and lengths.

$P_0$: 
$\begin{tabular}{lllll}
$w$ 		& $l$	& 	& $w(\lambda+\rho)-\rho$\\
$2134$ 	& $1$&	& $(2b|2a+2|2c|2d)$\\
$2431$	& $4$&	& $(2b|2d-2|2c|2a+4)$\\
$3124$	& $2$&	& $(2c-2|2a+2|2b+2|2d)$\\
$3421$	& $5$&	& $(2c-2|2d-2|2b+2|2a+4)$\\
\end{tabular}
$

For the intermediate parabolic subgroup $P_{12}$ we have to consider only the permutations that have $2$ or $3$ at the third place and $1$ or $4$ at the fourth place. Together with that the first and the second entry should be in increasing order.
For the last two elements of the permutation, we have $34$,  $24$, $31$ and $21$. Since the first two elements of the permutation have to be in increasing order, the permutations are $(1234)$, $1324$, $2431$ and $3421$.
For them we have the following weights and lengths.

$P_{12}$: 
$\begin{tabular}{lllll}
$w$ 		& $l$	& 	& $w(\lambda+\rho)-\rho$\\
$1234$ 	& $0$&	& $(2a+1,2b+1|2c|2d)$\\
$1324$ 	& $1$&	& $(2a+1,2c-1|2b+2|2d)$\\
$2431$	& $4$&	& $(2b,2d-2|2c|2a+4)$\\
$3421$	& $5$&	& $(2c-2,2d-2|2b+2|2a+4)$\\
\end{tabular}
$

For the intermediate parabolic subgroup $P_{23}$ we have to consider only the permutations that have $2$ or $3$ at the first place and $1$ or $4$ at the fourth place. Together with that the second and the third entry should be in increasing order.
The permutations are $2134$, $2341$, $3124$ and $3241$.
For them we have the following weights and lengths.

$P_{23}$: 
$\begin{tabular}{lllll}
$w$ 		& $l$	& 	& $w(\lambda+\rho)-\rho$\\
$2134$	& $1$&	& $(2b|2a+2,2c|2d)$\\
$2341$	& $3$&	& $(2b|2c-1,2d-1|2a+4)$\\
$3124$	& $2$&	& $(2c-2|2a+2, 2b+2|2d)$\\
$3241$	& $4$&	& $(2c-2|2b+1,2d-1|2a+4)$\\
\end{tabular}
$

For the intermediate parabolic subgroup $P_{34}$ we have to consider only the permutations that have $2$ or $3$ at the first place and $1$ or $4$ at the second place. Together with that the third and the fourth entry should be in increasing order.
For the first two elements of the permutation, we have $12$, $14$, $32$ and $34$. Since the last two elements of the permutation have to be in increasing order, the permutations are $2134$, $2413$, $3124$ and $3412$
For them we have the following weights and lengths.

$P_{34}$: 
$\begin{tabular}{lllll}
$w$ 		& $l$	& 	& $w(\lambda+\rho)-\rho$\\
$2134$	& $1$&	& $(2b|2a+2|2c,2d)$\\
$2413$ 	& $3$&	& $(2b|2d-2|2a+3,2c+1)$\\
$3124$	& $2$&	& $(2c-2|2a+2|2b+2,2d)$\\
$3412$	& $4$&	& $(2c-2|2d-2|2a+3,2b+3)$\\
\end{tabular}
$

For the maximal parabolic subgroup $P_{13}$ we have to consider only the permutations that have $1$ or $4$ at the fourth place. Together with that the first, the second  and the third entry should be in increasing order.
For the last element of the permutation, we have $1$ or $4$ Since the first three elements of the permutation have to be in increasing order, the permutations are $1234$, $2341$
For them we have the following weights and lengths.

$P_{13}$: 
$\begin{tabular}{lllll}
$w$ 		& $l$	& 	& $w(\lambda+\rho)-\rho$\\
$1234$ 	& $0$&	& $(2a+1,2b+1,2c|2d)$\\
$2341$ 	& $3$&	& $(2b,2c-1,2d-1|2a+4)$\\
\end{tabular}
$

For the parabolic subgroup $P_{12,34}$, we have to pick permutations $w$ so that $w(\lambda+\rho)-\rho$ has the same parity for the first two elements. All possibilities for the first two elements of the permutation are $odd,odd$ which can be achieved with $12$ and $13$ and $even,even$ which can be achieved with $24$ and $34$. 
The corresponding permutations are $1234$, $1324$, $2413$ and $3412$.

$P_{12,34}$: 
$\begin{tabular}{lllll}
$w$ 		& $l$	& 	& $w(\lambda+\rho)-\rho$\\
$1234$ 	& $0$&	& $(2a+1,2b+1|2c,2d)$\\
$1324$ 	& $1$&	& $(2a+1, 2c-1|2b+2,2d)$\\
$2413$	& $3$&	& $(2b,2d-2|2a+3,2c+1)$\\
$3412$	& $4$&	& $(2c-2,2d-2|2a+3,2b+3)$\\
\end{tabular}
$

For the maximal parabolic subgroup $P_{24}$ we have to consider only the permutations that have $2$ or $3$ at the first place. Together with that the second, the third and the fourth entry should be in increasing order.
For the first element of the permutation, we have $1$ or $3$ Since the first three elements of the permutation have to be in increasing order, the permutations are $(1234)$, $(3124)$
For them we have the following weights and lengths.

$P_{24}$: 
$\begin{tabular}{lllll}
$w$ 		& $l$	& 	& $w(\lambda+\rho)-\rho$\\
$2134$ 	& $1$&	& $(2b|2a+2,2c,2d)$\\
$3124$ 	& $2$&	& $(2c-2|2a+2,2b+2,2d)$\\
\end{tabular}
$


\subsubsection{$E_1$-page. Highest weight $(2a+1,2b+1,2c,2d)$.  Case $B1$.}

{\begin{center}
\begin{small}
\scriptsize\renewcommand{\arraystretch}{2.2}
\begin{longtable}{|c|c|c|c|c|c|c|c|c|}
\hline
$E_1^{0,0}=0$ 
& $E_1^{1,0}=0$ & 
$E_1^{2,0}=0$
\\
\hline
$E_1^{0,1}=0$
& $E_1^{1,1}=(2a+1,2b+1|2c|2d)$
& $E_1^{2,1}=(2b|2a+2|2c|2d)$
\\
\hline
$E_1^{0,2}=\left\{
	\begin{tabular}{ll}
	$(H^2(2a+1,2b+1,2c)|2d)$\\
	$(2a+1,2b+1|2c,2d)$	
	\end{tabular}
	\right.
	$
& $E_1^{1,2}=\left\{
	\begin{tabular}{ll}
	$(2a+1,2c-1|2b+2|2d)$\\
	$(2b|2a+2,2c|2d)$\\
	$(2b|2a+2|2c,2d)$
	\end{tabular}
	\right.
	$
& $E_1^{2,2}=(2c-2|2a+2|2b+2|2d)$
\\
\hline
$E_1^{0,3}=
	\left\{
	\begin{tabular}{ll}
	$(H^3(2a+1,2b+1,2c)|2d)$\\
	$(2a+1, 2c-1|2b+2,2d)$
	\end{tabular}
	\right.
	$
& $E_1^{1,3}=
	\left\{
	\begin{tabular}{ll}
	$(2c-2|2a+2, 2b+2|2d)$\\
	$(2c-2|2a+2|2b+2,2d)$
	\end{tabular}
	\right.
	$
& $E_1^{2,3}=0
	$\\
\hline
$E_1^{0,4}=(2b|H^3(2a+2,2c,2d))$
& $E_1^{1,4}=
	\left\{
	\begin{tabular}{ll}
	$(2b|2c-1,2d-1|2a+4)$\\
	$(2b|2d-2|2a+3,2c+1)$
	\end{tabular}
	\right.$
& $E_1^{2,4}=(2b|2d-2|2c|2a+4)$\\
\hline
$E_1^{0,5}=\left\{
	\begin{tabular}{ll}
	$(H^2(2b,2c-1,2d-1)|2a+4)$\\
	$(2b,2d-2|2a+3,2c+1)$\\
	 $(2c-2|H^3(2a+2,2b+2,2d))$
	\end{tabular}
	\right.$
& $E_1^{1,5}=\left\{
	\begin{tabular}{ll}
	$(2b,2d-2|2c|2a+4)$\\
	$(2c-2|2b+1,2d-1|2a+4)$\\
	$(2c-2|2d-2|2a+3,2b+3)$\\
	\end{tabular}
	\right.$
& $E_1^{2,5}=(2c-2|2d-2|2b+2|2a+4)$\\
\hline
$E_1^{0,6}=
	\left\{
	\begin{tabular}{ll}
	$(H^3(2b,2c-1,2d-1)|2a+4)$\\
	$(2c-2,2d-2|2a+3,2b+3)$
	\end{tabular}
	\right.$
& $E_1^{1,6}=(2c-2,2d-2|2b+2|2a+4)$
& $E_1^{1,6}=0$\\
\hline
\end{longtable}
\end{small}
\end{center}


\subsubsection{Certain cohomology groups of $GL_3(\Z)$ that appear on the $E_1$-page in the case $B1$.}
From the computations of the cohomology of $GL_3(\Z)$ from the previous section, we have
{\begin{center}
\scriptsize\renewcommand{\arraystretch}{2.2}
\begin{longtable}{|c|c|c|}
\hline
$H^2(2a+1,2b+1,2c)=\Delta_1$
& $H^3(2a+1,2b+1,2c)=(2c-2|2a+2,2b+2)$\\
\hline
$H^2(2a+2,2c,2d)=0$
& $H^3(2a+2,2c,2d)
	=
	\left\{	\begin{tabular}{llll}
	$(\overline{2c-1,2d-1}|2a+4)$\\
	$(2d-2|\overline{2a+3,2c+1})$\\
	$(2d-2|2c|2a+4)$
	\end{tabular}
	\right.$\\
	\hline
$H^2(GL_3(\Z),(2a+2,2b+2,2d))=0$
& $H^3(GL_3(\Z),(2a+2,2b+2,2d))
	=
	\left\{	\begin{tabular}{llll}
	$(\overline{2b+1,2d-1}|2a+4)$\\
	$(2d-2|\overline{2a+3,2b+3})$\\
	$(2d-2|2b+2|2a+4)$
	\end{tabular}
	\right.$\\
	\hline
$H^2(GL_3(\Z),(2b,2c-1,2d-1))
	=
	\Delta_2
	\subset
	\left\{	\begin{tabular}{llll}
	$(2b,2d-2|2c)$\\
	$(2c-2|\overline{2b+1,2d-1})$	
	\end{tabular}
	\right.$
	& $H^3(GL_3(\Z),(2b,2c-1,2d-1))=(2c-2,2d-2|2b+2)$
\\
\hline
\end{longtable}
\end{center}
where 
$\Delta_1\subset (2b|2a+2,2c)+(\overline{2a+1,2c-1}|2b+2)$ 
and
$\Delta_2
	\subset
	(2b,2d-2|2c)+(2c-2|\overline{2b+1,2d-1})$.
	are diagonal embeddings.


\subsubsection{$E_2$-page. Highest weight $(2a,2b,2c,2d)$. Case $B1$}

For $q=1$, we have that $E_1^{1,1}\rightarrow E_1^{2,1}$, explicitly,
$(2a + 1, 2b + 1|2c|2d)\rightarrow (2b|2a + 2|2c|2d)$
is surjective. And the kernel is $E_2^{1,1}=(\overline{2a + 1, 2b + 1}|2c|2d)$.

For $q=2$, we have a surjective map
$ (2a + 1, 2b + 1|2c, 2d)\rightarrow (2b|2a + 2|2c, 2d)$
and an injective map $(\Delta_1|2d)\subset (2b|2a+2,2c|2d)+(2a+1,2c-1|2b+2|2d)$.
For $E_1^{2,1}\rightarrow E_1^{2,2}$, we have a surjective map
$(2a+1,2c-1|2b+2|2d)\rightarrow (2c-2|2a+2|2b+2|2d)$.
Therefore,
$E_2^{0,2}= (\overline{2a + 1, 2b + 1}|2c, 2d)$,
$E_2^{1,2}=coker[(\Delta_1|2d)\subset (2b|2a+2,2c|2d)+(\overline{2a+1,2c-1}|2b+2|2d]$, which is isomorphic to $(2b|2a+2,2c|2d)$ and
$E_2^{2,2}=0$.

For $q=3$, we have a surjective map
$(2a + 1,2c - 1|2b + 2,2d)\rightarrow (2c - 2|2a + 2|2b + 2,2d)$
and an isomorphism
$(H^3(2a+1,2b+1,2c)|2d)\rightarrow (2c - 2|2a + 2,2b + 2|2d)$.
Therefore
$E_2^{0,3}=(\overline{2a + 1,2c - 1}|2b + 2,2d)$
and 
$E_2^{p,3}=0$ for $p=1,2$.

For $q=4$, the fourth row is a short exact sequence.
Therefore, $E_2^{p,4}=0$ for $p=0,1,2$.

For $q=5$, similarly to the fourth row, we have a short exact sequence
\begin{eqnarray*}
0\rightarrow (2c-2|H (2a+2,2b+2,2d)) \rightarrow \\
\rightarrow (2c-2|2b+1,2d-1|2a+4)+ (2c-2|2d-2|2a+3,2b+3) \rightarrow \\
\rightarrow(2c-2|2d-2|2b+2|2a+4)\rightarrow 0.
\end{eqnarray*}
We aslo have a surjective map
$(2b,2d - 2|2a + 3,2c + 1)\rightarrow (2b, 2d - 2|2c|2a + 4)$.
Therefore,
\begin{eqnarray*}
E_2^{0,5}
&&=(H^2(2b,2c - 1,2d - 1)|2a + 4)+(2b,2d - 2|\overline{2a + 3,2c + 1})=\\
&&=(\Delta_2|2a+4)+(2b,2d - 2|\overline{2a + 3,2c + 1})
\end{eqnarray*}

For $q=6$, we have that
$(H^3(2b,2c-1,2d-1)|2a+4)\rightarrow (2c-2,2d-2|2b+2|2a+4)$ is an isomorphism. Therefore,
$E_2^{0,6}=(2c-2,2d-2|2a+3,2b+3)$ and
$E_2^{p,6}=0$ for $p=1,2$.

{\begin{center}
\begin{small}
\scriptsize\renewcommand{\arraystretch}{2.2}
\begin{longtable}{|c|c|c|c|c|c|c|c|c|}
\hline
& $p=0$
& $p=1$
& $p=2$\\
\hline
$q=0$
&  
&  
& 
\\
\hline
$q=1$
&
& $E_2^{1,1}=(\overline{2a + 1, 2b + 1}|2c|2d)$
&
\\
\hline
$q=2$
&
$E_2^{0,2}=E_2^{0,2}= (\overline{2a + 1, 2b + 1}|2c, 2d)$
& $E_2^{1,2}=\overline{\Delta_1}$
&\\
\hline
$q=3$
& $E_2^{0,3}=(\overline{2a + 1,2c - 1}|2b + 2,2d)$
& 
& \\
\hline
$q=4$
&
&
& \\
\hline
$q=5$
& $E_2^{0,5}=\Delta_2+(2b,2d - 2|\overline{2a + 3,2c + 1})$
& 
&\\
\hline
$q=6$
& $E_2^{0,6}=(2c-2,2d-2|2a+3,2b+3)$
& 
& \\
\hline
\end{longtable}
\end{small}
\end{center}

$\overline{\Delta_1}=coker\left[(\Delta_1|2d)\rightarrow (2b|2a+2,2c|2d)+(\overline{2a+1,2c-1}|2b+2|2d)\right]$


\subsubsection{Boundary cohomology. Highest weight $(2a+1,2b+1,2c,2d)$.  Case $B1$.}
Again the boundary cohomology degenerates at the $E_2$-page. Therefore, up to semisimplicity we have
$H^n_\partial(GL_4(\Z),V_\lambda)=\bigoplus_{p+q=n}E_2^{p,q}$.
More systematically, we have

Let $H^q=H^q_\partial(GL_4(\Z),V_\lambda)$ be the cohomology of $GL_4(\Z)$ with coefficients in the highest weight representation with weight 
$\lambda=(2a+1,2b+1,2c,2d)$ written as a character of the split maximal torus. Then,
\[H^q
=
\left\{
\begin{tabular}{lll}
$\left\{
\begin{tabular}{lll}
$E_2^{0,2}= (\overline{2a + 1, 2b + 1}|2c, 2d)$\\
$E_2^{1,1}=(\overline{2a + 1, 2b + 1}|2c|2d)$
\end{tabular}
	\right\}$
		& $q=2$\\
		$\left\{
\begin{tabular}{lll}
$E_2^{0,3}=(\overline{2a + 1,2c - 1}|2b + 2,2d)$\\
$E_2^{1,2}=\overline{\Delta_1}$
\end{tabular}
	\right\}$
		& $q=3$\\	
$0$
		& $q=4$\\
$E_2^{0,5}=\Delta_2+(2b,2d - 2|\overline{2a + 3,2c + 1})$
		& $q=5$\\ 
$E_2^{0,6}=(2c-2,2d-2|2a+3,2b+3)$
		& $q=6$
\end{tabular}
\right.
\]
where
\[\Delta_1\subset (2b|2a+2,2c)+(\overline{2a+1,2c-1}|2b+2)\]
and
\[\overline{\Delta_1}=coker\left[(\Delta_1|2d)\rightarrow (2b|2a+2,2c|2d)+(\overline{2a+1,2c-1}|2b+2|2d)\right],\]
which is isomorphic to 
$(2b|2a+2,2c|2d)$.
Also
$\Delta_2
	\subset
	(2b,2d-2|2c)+(2c-2|\overline{2b+1,2d-1})$ and
$(\Delta_2|2a+4)$ is isomorphic to $(2b, 2d - 2|2c|2a+4)$.

\begin{thm}
\[H^q
=
\left\{
\begin{tabular}{lll}
$S_{2a-2b+2}\times S_{2c-2d+2}
+
S_{2a-2b+2}$
			& $q=2$\\
$S_{2a-2c+4}\otimes S_{2b-2d+4}
+
S_{2a-2c+4}$
			& $q=3$\\
$0$
			& $q=4$\\
$S_{2a-2c+4}\otimes S_{2b-2d+4}
+
S_{2b-2d+4}$
			& $q=5$\\
$S_{2a-2b+2}\otimes S_{2c-2d+2}
+
S_{2c-2d+2}$
			& $q=6$
\end{tabular}
	\right.
\]
\end{thm}




\subsection{Cohomology of $GL_4(\Z)$. Highest weight $(2a,2b,2c-1,2d-1)$.  Case $B1'$}


\subsubsection{Cohomology of the parabolic subgroups. Highest weight $(2a,2b,2c-1,2d-1)$. Case $B1'$}
For that case the weights are $(2a,2b,2c-1,2d-1)$.
We have to find the elements $w$ of the Weyl group that give a non-trivial cohomology of the minimal parabolic subgroup. The first entry has to go to the first or the third place in order to be even. Similarly, second entry should go to the second or fourth place. Also, the third entry should go to the second or the fourth and the four elements have to go the first and the third elements. 
Thus, all permutations are 
$1243$,
$1342$,
$4213$ and
$4312$.

$P_0$: 
$\begin{tabular}{lllll}
$w$ 		& $l$	& 	& $w(\lambda+\rho)-\rho$\\
$1243$ 	& $1$&	& $(2a|2b|2d-2|2c)$\\
$1342$	& $2$&	& $(2a|2c-2|2d-2|2b+2)$\\
$4213$	& $4$&	& $(2d-4|2b|2a+2|2c)$\\
$4312$	& $5$&	& $(2d-4|2c-2|2a+2|2b+2)$\\
\end{tabular}
$

For the intermediate parabolic subgroup $P_{12}$ we have to consider only the permutations that have $1$ or $4$ at the third place and $2$ or $3$ at the fourth place. Together with that the first and the second entry should be in increasing order.
For the last two elements of the permutation, we have $43$,  $42$, $13$ and $12$. Since the first two elements of the permutation have to be in increasing order, the permutations are 
$1243$,
$1342$,
$2413$ and
$3412$.
For them we have the following weights and lengths.

$P_{12}$: 
$\begin{tabular}{lllll}
$w$ 		& $l$	& 	& $w(\lambda+\rho)-\rho$\\
$1243$ 	& $1$&	& $(2a,2b|2d-2|2c)$\\
$1342$	& $2$&	& $(2a,2c-2|2d-2|2b+2)$\\
$2413$	& $3$&	& $(2b-1,2d-3|2a+2|2c)$\\
$3412$	& $4$&	& $(2c-3,2d-3|2a+2|2b+2)$\\
\end{tabular}
$

For the intermediate parabolic subgroup $P_{23}$ we have to consider only the permutations that have $2$ or $3$ at the fourth place and $1$ or $4$ at the first place. Together with that the second and the third entry should be in increasing order.
The permutations are 
$1243$,
$1342$,
$4123$ and
$4132$.
For them we have the following weights and lengths.

$P_{23}$: 
$\begin{tabular}{lllll}
$w$ 		& $l$	& 	& $w(\lambda+\rho)-\rho$\\
$1243$	& $1$&	& $(2a|2b,2d-2|2c)$\\
$1342$	& $2$&	& $(2a|2c-2,2d-2|2b+2)$\\
$4123$	& $3$&	& $(2d-4|2a+1, 2b+1|2c)$\\
$4132$	& $4$&	& $(2d-4|2a+1,2c-1|2b+2)$\\
\end{tabular}
$

For the intermediate parabolic subgroup $P_{34}$ we have to consider only the permutations that have $1$ or $4$ at the first place and $2$ or $3$ at the second place. Together with that the third and the fourth entry should be in increasing order.
For the first two elements of the permutation, we have 
$12$, $13$, $42$ and $43$. 
Since the last two elements of the permutation have to be in increasing order, the permutations are 
$1234$,
$1324$,
$4213$ and
$4312$.
For them we have the following weights and lengths.

$P_{34}$: 
$\begin{tabular}{lllll}
$w$ 		& $l$	& 	& $w(\lambda+\rho)-\rho$\\
$1234$	& $0$&	& $(2a|2b|2c-1,2d-1)$\\
$1324$	& $1$&	& $(2a|2c-2|2b+1,2d-1)$\\
$4213$	& $4$&	& $(2d-4|2b|2a+2,2c)$\\
$4312$	& $5$&	& $(2d-4|2c-2|2a+2,2b+2)$\\
\end{tabular}
$

For the maximal parabolic subgroup $P_{13}$ we have to consider only the permutations that have $2$ or $3$ at the fourth place. Together with that the first, the second  and the third entry should be in increasing order.
For the last element of the permutation, we have $1$ or $3$ Since the first three elements of the permutation have to be in increasing order, the permutations are$1243$, $2341$.
For them we have the following weights and lengths.

$P_{13}$: 
$\begin{tabular}{lllll}
$w$ 		& $l$	& 	& $w(\lambda+\rho)-\rho$\\
$1243$ 	& $1$&	& $(2a,2b,2d-2|2c)$\\
$1342$ 	& $2$&	& $(2a,2c-2,2d-2|2b+2)$\\
\end{tabular}
$

For the maximal parabolic subgroup $P_{12,34}$ we have to consider only elements that have the same parity in the first and the second entry. Together with that the first and the second, entry should be in increasing order. Also the third and the fourth entry should be in increasing order. Also the first and the second entry should have opposite parity. All the possibilities for the first two entries are: $12$, $13$, $24$, $34$. The corresponding permutations are $(1234)$, $(1423)$, $(2314)$ and $(3412)$

$P_{12,34}$: 
$\begin{tabular}{lllll}
$w$ 		& $l$	& 	& $w(\lambda+\rho)-\rho$\\
$1234$ 	& $0$&	& $(2a,2b|2c-1,2d-1)$\\
$1324$ 	& $1$&	& $(2a, 2c-2|2b+1,2d-1)$\\
$2413$ 	& $3$&	& $(2b-1,2d-3|2a+2,2c)$\\
$3412$	& $4$&	& $(2c-3,2d-3|2a+2,2b+2)$\\
\end{tabular}
$

For the maximal parabolic subgroup $P_{24}$ we have to consider only the permutations that have $1$ or $4$ at the first place. Together with that the second, the third and the fourth entry should be in increasing order.
For the first element of the permutation, we have $1$ or $4$ Since the first three elements of the permutation have to be in increasing order, the permutations are  $1234$, $4123$
For them we have the following weights and lengths.

$P_{24}$: 
$\begin{tabular}{lllll}
$w$ 		& $l$	& 	& $w(\lambda+\rho)-\rho$\\
$1234$ 	& $0$&	& $(2a|2b,2c-1,2d-1)$\\
$4123$ 	& $3$&	& $(2d-4|2a+1,2b+1,2c)$\\
\end{tabular}
$


\subsubsection{$E_1$-page. Highest weight $(2a,2b,2c-1,2d-1)$.  Case $B1'$}

{\begin{center}
\begin{small}
\scriptsize\renewcommand{\arraystretch}{2.2}
\begin{longtable}{|c|c|c|c|c|c|c|c|c|}
\hline
$E_1^{0,0}=0$ 
& $E_1^{1,0}=0$ & 
$E_1^{2,0}=0$
\\
\hline
$E_1^{0,1}=0$
& $E_1^{1,1}=(2a|2b|2c - 1, 2d - 1)$
& $E_1^{2,1}=(2a|2b|2d - 2|2c)$
\\
\hline
$E_1^{0,2}=
	\left\{
	\begin{tabular}{ll}
	$(2a,2b|2c-1,2d-1)$\\
	$(2a|H^2(2b,2c-1,2d-1))$
	\end{tabular}
	\right.$ 
& $E_1^{1,2}=
	\left\{
	\begin{tabular}{ll}
	$(2a, 2b|2d - 2|2c)$\\
	$(2a|2b, 2d - 2|2c)$\\
	$(2a|2c-2|2b+1,2d-1)$
	\end{tabular}
	\right.$ 
& $E_1^{2,2}=(2a|2c - 2|2d - 2|2b + 2)$
\\
\hline
$E_1^{0,3}=
	\left\{
	\begin{tabular}{ll}
	$(2a,2c-2|2b+1,2d-1)$\\
	$(2a|H^3(2b,2c-1,2d-1))$
	\end{tabular}
	\right.
	$
& $E_1^{1,3}=
	\left\{
	\begin{tabular}{ll}
	 $(2a,2c-2|2d-2|2b+2)$\\
	$(2a|2c-2,2d-2|2b+2)$
	\end{tabular}
	\right.
	$
& $E_1^{2,3}=0$\\
\hline
$E_1^{0,4}=(H^3(2a,2b,2d-2)|2c)$
& $E_1^{1,4}=
	\left\{
	\begin{tabular}{ll}
	$(2b-1,2d-3|2a+2|2c)$\\
	$(2d-4|2a+1,2b+1|2c)$
	\end{tabular}
	\right.$
& $E_1^{2,4}=(2d-4|2b|2a+2|2c)$\\
\hline
$E_1^{0,5}=\left\{
	\begin{tabular}{ll}
	$(H^3(2a,2c-2,2d-2)|2b+2)$\\
	$(2b-1,2d-3|2a+2,2c)$\\
	$(2d-4|H^2(2a+1,2b+1,2c))$
	\end{tabular}
	\right.$
& $E_1^{1,5}=\left\{
	\begin{tabular}{ll}
	$(2c-3,2d-3|2a+2|2b+2)$\\
	$(2d-4|2a+1,2c-1|2b+2)$\\
	$(2d - 4|2b|2a + 2, 2c)$
	\end{tabular}
	\right.$
& $E_1^{2,5}=(2d-4|2c-2|2a+2|2b+2)$\\
\hline
$E_1^{0,6}=
	\left\{
	\begin{tabular}{ll}
	$(2d-4|H^3(2a+1,2b+1,2c))$\\
	$(2c-3,2d-3|2a+2,2b+2)$
	\end{tabular}
	\right.$
& $E_1^{1,6}=(2d-4|2c-2|2a+2,2b+2)$
& $E_1^{2,6}=0$\\
\hline
\end{longtable}
\end{small}
\end{center}


\subsubsection{Certain cohomology groups of $GL_3(\Z)$ that appear on the $E_1$-page in the case $B1'$.}
From the computations of the cohomology of $GL_3(\Z)$ from the previous section, we have
{\begin{center}
\scriptsize\renewcommand{\arraystretch}{2.2}
\begin{longtable}{|c|c|c|}
\hline
$H^2(2b, 2c - 1, 2d - 1)
	=
	\Delta_1$
&$H^3(GL_3(\Z),(2b, 2c - 1, 2d - 1))=(2c-2,2d-2|2b+2)$\\
\hline
$H^2(2a, 2b, 2d - 2)=0$
& $H^3(2a, 2b, 2d - 2)=
	\left\{	\begin{tabular}{llll}
	$(\overline{2b-1,2d-3}|2a+2)$\\
	$(2d-4|\overline{2a+1,2b+1})$\\
	$(2d-4|2b|2a+2)$
	\end{tabular}
	\right.$\\
\hline
$H^2(2a,2c - 2,2d - 2)=0$
& $H^3(2a,2c - 2,2d - 2)=
	=
	\left\{	\begin{tabular}{llll}
	$(\overline{2c-3,2d-3}|2a+2)$\\
	$(2d-4|\overline{2a+1,2c-1})$\\
	$(2d-4|2c-2|2a+2)$
	\end{tabular}
	\right.$\\
\hline
$H^2(2a+1,2b+1,2c)=\Delta_2$
&$H^3(2a+1,2b+1,2c)=(2c-2|2a+2,2b+2)$\\
\hline
\end{longtable}
\end{center}
where
$\Delta_1\subset (2b,2d-2|2c)+(2c-2|\overline{2b+1,2d-1})$
and
$\Delta_2\subset (\overline{2a+1,2c-1}|2b+2)+(2b|2a+2,2c)$.


\subsubsection{$E_2$ page. Highest weight $(2a,2b,2c-1,2d+1)$.  Case $B1'$}

From consideration of Euler characteristics we can conclude that the map $(2a+1,2b+1)\rightarrow (2b|2a+2)$ is surjective for $a>b$, which corresponds to the Eisenstein series $E_{2a-2b+2}$ for the classical modular group $SL_2(\Z)$.

For $q=1$, we have a surjective map $E_1^{1,1}rightarrow E_1^{2,1}$, which is given explicitly by
$(2a|2b|2c - 1, 2d - 1) \rightarrow (2a|2b|2d - 2|2c)$. Therefore, $E_2^{1,1}=(2a|2b|\overline{2c - 1, 2d - 1})$ and $E_2^{p,2}=0$ for $p=0$ and $p=2$.

For $q=2$, we have a surjective maps $(2a|2c - 2|2b + 1,2d - 1)\rightarrow (2a|2c-2|2d-2|2b+2)$
and
$(2a, 2b|2c - 1, 2d - 1) \rightarrow (2a, 2b|2d - 2|2c)$. We also have an injective map $(2a|\Delta_1)\rightarrow (2a|2b, 2d - 2|2c) (2a|2c - 2|2b + 1,2d -1)$.
Therefore $E_2^{0,2}=(2a, 2b|\overline{2c - 1, 2d - 1})$. Also, 
\[E_2^{1,2}=\overline{\Delta}_1=coker\left[(2a|\Delta_1)\rightarrow (2a|2b, 2d - 2|2c) (2a|2c - 2|2b + 1,2d -1)\right]\] 
is isomorphic to 
$(2a|2b, 2d - 2|2c)$. Finally, $E_2^{2,2}=0$

For $q=3$, we have a surjective map
$(2a,2c - 2|2b + 1,2d - 1)
\rightarrow 
(2a,2c - 2|2d - 2|2b + 2)$ and an isomorphism 
$(2a|H^3(2b, 2c -1, 2d - 1))
\rightarrow
(2a|2c - 2,2d - 2|2b + 2)$
Therefore, $E_2^{0,3}=(2a,2c - 2|\overline{2b + 1,2d - 1})$.

For $q=4$, we have a short exact sequence
$0\rightarrow E_1^{0,4}\rightarrow E_1^{1,4}\rightarrow E_1^{2,4}\rightarrow 0$.
Therefore, $E_2^{p,4}=0$ for $p=0,1,2$.

For $q=5$, we have a surjective map
$(2b - 1,2d - 3|2a + 2,2c)\rightarrow
(2d - 4|2b|2a + 2, 2c)$.
We also have a short exact sequence
\begin{eqnarray*}
&0\rightarrow (H^3(2a,2c -2,2d - 2)|2b + 2)rightarrow\\
&\rightarrow (2c-3,2d-3|2a+2|2b+2) + (2d - 4|2a + 1, 2c - 1|2b + 2)\rightarrow\\
&\rightarrow (2d-4|2c-2|2a+2|2b+2)\rightarrow 0
\end{eqnarray*}
Therefore,
$E_2^{0,5}=(\overline{2b - 1,2d - 3}|2a + 2,2c)+(2d-4|H^2 (2a+1,2b+1,2c))$ and $E_2^{p,5}=0$ for $p=1,2$.
Note that $(2d-4|H (2a+1,2b+1,2c))=\Delta_2$, which is isomorphic to $(2d - 4|2b|2a + 2, 2c)$

For $q=6$, we have a surjective map
$(2c-3,2d-3|2a+2,2b+2)
\rightarrow
(2d-4|2c-2|2a+2,2b+2)$
Therefore
$E_2^{0,6}=(\overline{2c-3,2d-3}|2a+2,2b+2)+(2d - 4|2b|2a + 2, 2c)$
and $E_2^{p,6}=0$ for $p=1,2$.

{\begin{center}
\begin{small}
\scriptsize\renewcommand{\arraystretch}{2.2}
\begin{longtable}{|c|c|c|c|c|c|c|c|c|}
\hline
& $p=0$
& $p=1$
& $p=2$\\
\hline
$q=0$
&  
&  
& \\
\hline
$q=1$
&
& $E_2^{1,1}=(2a|2b|\overline{2c - 1, 2d - 1})$
& 
\\
\hline
$q=2$
&  $E_2^{0,2}=(2a, 2b|\overline{2c - 1, 2d - 1})$
&  $E_2^{1,2}=(2a|2b, 2d - 2|2c)$
& 
\\
\hline
$q=3$
& $E_2^{0,3}=(2a,2c - 2|\overline{2b + 1,2d - 1})$
& 
&\\
\hline
$q=4$
&
& 
& \\
\hline
$q=5$
& $E_2^{0,5}=\left\{
	\begin{tabular}{lll}
	$(\overline{2b - 1,2d - 3}|2a + 2,2c)$\\
	$(2d - 4|2b|2a + 2, 2c)$
	\end{tabular}
	\right.$
& 
&\\
\hline
$q=6$
&
$E_2^{0,6}=\left\{
	\begin{tabular}{lll}
	$(\overline{2c-3,2d-3}|2a+2,2b+2)$\\
	$(2d - 4|2b|2a + 2, 2c)$
	\end{tabular}
	\right.$
& 
& \\
\hline
\end{longtable}
\end{small}
\end{center}


\subsubsection{Boundary cohomology. Highest weight $(2a,2b,2c-1,2d-1)$.  Case $B1'$.}
Again the boundary cohomology degenerates at the $E_2$-page. Therefore, up to semisimplicity we have
$H^n_\partial(GL_4(\Z),V_\lambda)=\bigoplus_{p+q=n}E_2^{p,q}$.
More systematically, we have

Let $H^q=H^q_\partial(GL_4(\Z),V_\lambda)$ be the cohomology of $GL_4(\Z)$ with coefficients in the highest weight representation with weight 
$\lambda=(2a,2b,2c-1,2d-1)$ written as a character of the split maximal torus. Then,
\[H^q
=
\left\{
\begin{tabular}{lll}
	$\left\{\begin{tabular}{lll}
	$E_2^{0,2}=(2a, 2b|\overline{2c - 1, 2d - 1})$\\
	$E_2^{1,1}=(2a|2b|\overline{2c - 1, 2d - 1})$
	\end{tabular}
	\right\}$
		& $q=2$\\
	$\left\{\begin{tabular}{lll}
	$E_2^{0,3}=(2a,2c - 2|\overline{2b + 1,2d - 1})$
	$E_2^{1,2}=(2a|2b, 2d - 2|2c)$
	\end{tabular}
	\right\}$
		& $q=3$\\	
$0$
		& $q=4$\\
$E_2^{0,5}=\left\{
	\begin{tabular}{lll}
	$(\overline{2b - 1,2d - 3}|2a + 2,2c)$\\
	$(2d - 4|2b|2a + 2, 2c)$
	\end{tabular}
	\right\}$
		& $q=5$\\ 
$E_2^{0,6}=\left\{
	\begin{tabular}{lll}
	$(\overline{2c-3,2d-3}|2a+2,2b+2)$\\
	$(2d - 4|2b|2a + 2, 2c)$
	\end{tabular}
	\right\}$
		& $q=6$
\end{tabular}
\right.
\]

Since the above $H^4_\partial(B1')$ contains $E_2^{p,q}$ with $p>0$, we have that $H^4_\partial(B1')$ contains a potentially ghost class, namely,
\[pGh^4(GL_4(\Z),(2a+1,2b+1,2c+1,2d+1))=E_2^{1,3}=(2b|2a+2,2d|2c+2).\]

This is important not for the cohomology of $GL_4(\Z)$ but for the cohomology of $GL_5(\Z)$. In a similar way as the potentially ghost classes in $GL_3(\Z)$ have importance for the cohomology of $GL(\Z)$. The next representation of $GL_4$, namely $(2a,2b,2c,2c)$, or case $A2$, exhibits exactly the use of potentially ghost classes in $GL_3(\Z)$ and their effect on the cohomology of $GL_4(\Z)$.

\begin{thm}
\[H^q
=
\left\{
\begin{tabular}{lll}
$S_{2a-2b+2}\otimes S_{2c-2d+2}
+
S_{2c-2d+2}$
				& $q=2$\\
$S_{2a-2c+4}\otimes S_{2b-2d+4}
+
S_{2b-2d+4}$
				& $q=3$\\
$0$
		& $q=4$\\
$S_{2a-2c+4}\otimes S_{2b-2d+4}
+
S_{2a-2c+4}$
				& $q=5$\\	
$S_{2a-2b+2}\otimes S_{2c-2d+2}
+
S_{2a-2b+2}$
				& $q=6$
\end{tabular}
\right.
\]

\end{thm}




\subsection{Cohomology of $GL_4(\Z)$. Highest weight $(2a+1,2b+1,2c,2c)$.  Case $B2$}


\subsubsection{Cohomology of the parabolic subgroups. Highest weight $(2a+1,2b+1,2c,2c)$. Case $B2$}
Let $\lambda=(2a+1,2b+1,2c,2c)$.

For the minimal parabolic subgroup $P_0$ we have to consider only the permutations that send $1$ and $4$ to $2$ or $4$, $2$ and $3$ to $1$ or $3$.
All possible such permutations are: $2134$, $2431$, $3124$ and $3421$ . For them we have the following weights and lengths.

$P_0$: 
$\begin{tabular}{lllll}
$w$ 		& $l$	& 	& $w(\lambda+\rho)-\rho$\\
$2134$ 	& $1$&	& $(2b|2a+2|2c|2c)$\\
$2431$	& $4$&	& $(2b|2c-2|2c|2a+4)$\\
$3124$	& $2$&	& $(2c-2|2a+2|2b+2|2c)$\\
$3421$	& $5$&	& $(2c-2|2c-2|2b+2|2a+4)$\\
\end{tabular}
$

For the intermediate parabolic subgroup $P_{12}$ we have to consider only the permutations that have $2$ or $3$ at the third place and $1$ or $4$ at the fourth place. Together with that the first and the second entry should be in increasing order.
For the last two elements of the permutation, we have $34$,  $24$, $31$ and $21$. Since the first two elements of the permutation have to be in increasing order, the permutations are $(1234)$, $1324$, $2431$ and $3421$.
For them we have the following weights and lengths.

$P_{12}$: 
$\begin{tabular}{lllll}
$w$ 		& $l$	& 	& $w(\lambda+\rho)-\rho$\\
$1234$ 	& $0$&	& $(2a+1,2b+1|2c|2c)$\\
$1324$ 	& $1$&	& $(2a+1,2c-1|2b+2|2c)$\\
$2431$	& $4$&	& $(2b,2c-2|2c|2a+4)$\\
$3421$	& $5$&	& $(2c-2,2c-2|2b+2|2a+4)$\\
\end{tabular}
$

For the intermediate parabolic subgroup $P_{23}$ we have to consider only the permutations that have $2$ or $3$ at the first place and $1$ or $4$ at the fourth place. Together with that the second and the third entry should be in increasing order.
The permutations are $2134$, $2341$, $3124$ and $3241$.
For them we have the following weights and lengths.

$P_{23}$: 
$\begin{tabular}{lllll}
$w$ 		& $l$	& 	& $w(\lambda+\rho)-\rho$\\
$2134$	& $1$&	& $(2b|2a+2,2c|2c)$\\
$2341$	& $3$&	& $(2b|2c-1,2c-1|2a+4)$\\
$3124$	& $2$&	& $(2c-2|2a+2, 2b+2|2c)$\\
$3241$	& $4$&	& $(2c-2|2b+1,2c-1|2a+4)$\\
\end{tabular}
$

For the intermediate parabolic subgroup $P_{34}$ we have to consider only the permutations that have $2$ or $3$ at the first place and $1$ or $4$ at the second place. Together with that the third and the fourth entry should be in increasing order.
For the first two elements of the permutation, we have $12$, $14$, $32$ and $34$. Since the last two elements of the permutation have to be in increasing order, the permutations are $2134$, $2413$, $3124$ and $3412$
For them we have the following weights and lengths.

$P_{34}$: 
$\begin{tabular}{lllll}
$w$ 		& $l$	& 	& $w(\lambda+\rho)-\rho$\\
$2134$	& $1$&	& $(2b|2a+2|2c,2c)$\\
$2413$ 	& $3$&	& $(2b|2c-2|2a+3,2c+1)$\\
$3124$	& $2$&	& $(2c-2|2a+2|2b+2,2c)$\\
$3412$	& $4$&	& $(2c-2|2c-2|2a+3,2b+3)$\\
\end{tabular}
$

For the maximal parabolic subgroup $P_{13}$ we have to consider only the permutations that have $1$ or $4$ at the fourth place. Together with that the first, the second  and the third entry should be in increasing order.
For the last element of the permutation, we have $1$ or $4$ Since the first three elements of the permutation have to be in increasing order, the permutations are $1234$, $2341$
For them we have the following weights and lengths.

$P_{13}$: 
$\begin{tabular}{lllll}
$w$ 		& $l$	& 	& $w(\lambda+\rho)-\rho$\\
$1234$ 	& $0$&	& $(2a+1,2b+1,2c|2c)$\\
$2341$ 	& $3$&	& $(2b,2c-1,2c-1|2a+4)$\\
\end{tabular}
$

For the parabolic subgroup $P_{12,34}$, we have to pick permutations $w$ so that $w(\lambda+\rho)-\rho$ has the same parity for the first two elements. All possibilities for the first two elements of the permutation are $odd,odd$ which can be achieved with $12$ and $13$ and $even,even$ which can be achieved with $24$ and $34$. 
The corresponding permutations are $1234$, $1324$, $2413$ and $3412$.

$P_{12,34}$: 
$\begin{tabular}{lllll}
$w$ 		& $l$	& 	& $w(\lambda+\rho)-\rho$\\
$1234$ 	& $0$&	& $(2a+1,2b+1|2c,2c)$\\
$1324$ 	& $1$&	& $(2a+1, 2c-1|2b+2,2c)$\\
$2413$	& $3$&	& $(2b,2c-2|2a+3,2c+1)$\\
$3412$	& $4$&	& $(2c-2,2c-2|2a+3,2b+3)$\\
\end{tabular}
$

For the maximal parabolic subgroup $P_{24}$ we have to consider only the permutations that have $2$ or $3$ at the first place. Together with that the second, the third and the fourth entry should be in increasing order.
For the first element of the permutation, we have $1$ or $3$ Since the first three elements of the permutation have to be in increasing order, the permutations are $(1234)$, $(3124)$
For them we have the following weights and lengths.

$P_{24}$: 
$\begin{tabular}{lllll}
$w$ 		& $l$	& 	& $w(\lambda+\rho)-\rho$\\
$2134$ 	& $1$&	& $(2b|2a+2,2c,2c)$\\
$3124$ 	& $2$&	& $(2c-2|2a+2,2b+2,2c)$\\
\end{tabular}
$


\subsubsection{$E_1$-page. Highest weight $(2a+1,2b+1,2c,2c)$.  Case $B2$.}

{\begin{center}
\begin{small}
\scriptsize\renewcommand{\arraystretch}{2.2}
\begin{longtable}{|c|c|c|c|c|c|c|c|c|}
\hline
$E_1^{0,0}=0$ 
& $E_1^{1,0}=0$ & 
$E_1^{2,0}=0$
\\
\hline
$E_1^{0,1}=(2a+1,2b+1|2c,2c)$
& $E_1^{1,1}=\left\{
	\begin{tabular}{ll}
	$(2a+1,2b+1|2c|2c)$\\
	$(2b|2a+2|2c,2c)$
	\end{tabular}
	\right.
	$
& $E_1^{2,1}=(2b|2a+2|2c|2c)$
\\
\hline
$E_1^{0,2}=(H^2(2a+1,2b+1,2c)|2c)$
& $E_1^{1,2}=\left\{
	\begin{tabular}{ll}
	$(2a+1,2c-1|2b+2|2c)$\\
	$(2b|2a+2,2c|2c)$
	\end{tabular}
	\right.
	$
& $E_1^{2,2}=(2c-2|2a+2|2b+2|2c)$
\\
\hline
$E_1^{0,3}=
	\left\{
	\begin{tabular}{ll}
	$(H^3(2a+1,2b+1,2c)|2c)$\\
	$(2a+1, 2c-1|2b+2,2c)$
	\end{tabular}
	\right.
	$
& $E_1^{1,3}=
	\left\{
	\begin{tabular}{ll}
	$(2c-2|2a+2, 2b+2|2c)$\\
	$(2c-2|2a+2|2b+2,2c)$
	\end{tabular}
	\right.
	$
& $E_1^{2,3}=0
	$\\
\hline
$E_1^{0,4}=(2b|H^3(2a+2,2c,2c))$
& $E_1^{1,4}=
	\left\{
	\begin{tabular}{ll}
	$(2b|2c-1,2c-1|2a+4)=0$\\
	$(2b|2c-2|2a+3,2c+1)$
	\end{tabular}
	\right.$
& $E_1^{2,4}=(2b|2c-2|2c|2a+4)$\\
\hline
$E_1^{0,5}=\left\{
	\begin{tabular}{ll}
	$(H^2(2b,2c-1,2c-1)|2a+4)$\\
	$(2b,2c-2|2a+3,2c+1)$\\
	$(2c-2,2c-2|2a+3,2b+3)$\\
	 $(2c-2|H^3(2a+2,2b+2,2c))$
	\end{tabular}
	\right.$
& $E_1^{1,5}=\left\{
	\begin{tabular}{ll}
	$(2b,2c-2|2c|2a+4)$\\
	$(2c-2,2c-2|2b+2|2a+4)$\\
	$(2c-2|2b+1,2c-1|2a+4)$\\
	$(2c-2|2c-2|2a+3,2b+3)$\\
	\end{tabular}
	\right.$
& $E_1^{2,5}=(2c-2|2c-2|2b+2|2a+4)$\\
\hline
$E_1^{0,6}=(H^3(2b,2c-1,2c-1)|2a+4)=0$
& $E_1^{1,6}=0$
& $E_1^{1,6}=0$\\
\hline
\end{longtable}
\end{small}
\end{center}


\subsubsection{Certain cohomology groups of $GL_3(\Z)$ that appear on the $E_1$-page in the case $B2$.}
From the computations of the cohomology of $GL_3(\Z)$ from the previous section, we have
{\begin{center}
\scriptsize\renewcommand{\arraystretch}{2.2}
\begin{longtable}{|c|c|c|}
\hline
$H^2(2a+1,2b+1,2c)=\Delta_2$
& $H^3(2a+1,2b+1,2c)=(2c-2|2a+2,2b+2)$\\
\hline
$H^2(2a+2,2c,2c)=0$
& $H^3(2a+2,2c,2c)=(2c-2|\overline{2a+3,2c+1})$\\
	\hline
$H^2(GL_3(\Z),(2a+2,2b+2,2c))=0$
& $H^3(GL_3(\Z),(2a+2,2b+2,2c))
	=
	\left\{	\begin{tabular}{llll}
	$(\overline{2b+1,2c-1}|2a+4)$\\
	$(2c-2|\overline{2a+3,2b+3})$\\
	$(2c-2|2b+2|2a+4)$
	\end{tabular}
	\right.$\\
	\hline
$H^2(GL_3(\Z),(2b,2c-1,2c-1))
	=
	\Delta_1 +\Delta_0$
	& $H^3(GL_3(\Z),(2b,2c-1,2c-1))=(2c-2,2c-2|2b+2)$
\\
\hline
\end{longtable}
\end{center}
where 
$\Delta_2\subset (2b|2a+2,2c)+(\overline{2a+1,2c-1}|2b+2)$,
$\Delta_1
	\subset
	(2b,2c-2|2c)+(2c-2|\overline{2b+1,2c-1})$
and
$\Delta_0
	\subset
	(2b|2c-2|2c)+(2c-2|2c-2|2b+2)$.
are diagonal embeddings.


\subsubsection{$E_2$-page. Highest weight $(2a,2b,2c,2c)$. Case $B2$}

For $q=1$, we have that $ (2a + 1, 2b + 1|2c,2c)\rightarrow (2a + 1, 2b + 1|2c|2c)$
 and
$ (2b|2a + 2|2c, 2c)\rightarrow (2b|2a + 2|2c,|2c)$ are isomorphisms.
Therefore, $E_2^{p,1}=0$ for $p=0,1,2$

For $q=2$, we have a surjective map
$ (2a + 1,2c - 1|2b + 2|2c)\rightarrow (2c-2|2a+2|2b+2|2c) 1$
and an injective map $(\Delta_2|2c)\subset (2b|2a+2,2c|2c)+(2a+1,2c-1|2b+2|2c)$.
For $E_1^{2,1}\rightarrow E_1^{2,2}$, we have a surjective map
$(2a+1,2c-1|2b+2|2c)\rightarrow (2c-2|2a+2|2b+2|2c)$.
Therefore,
$E_2^{0,2}=0$,
$E_2^{1,2}=coker[(\Delta_1|2c)\subset (2b|2a+2,2c|2c)+(\overline{2a+1,2c-1}|2b+2|2c]$, which is isomorphic to $(2b|2a+2,2c|2c)$ and
$E_2^{2,2}=0$.

For $q=3$, we have a surjective map
$(2a + 1,2c - 1|2b + 2,2c)\rightarrow (2c - 2|2a + 2|2b + 2,2c)$
and an isomorphism
$(H^3(2a+1,2b+1,2c)|2c)\rightarrow (2c - 2|2a + 2,2b + 2|2c)$.
Therefore
$E_2^{0,3}=(\overline{2a + 1,2c - 1}|2b + 2,2c)$
and 
$E_2^{p,3}=0$ for $p=1,2$.

For $q=4$, the fourth row is a short exact sequence.
Therefore, $E_2^{p,4}=0$ for $p=0,1,2$.

For $q=5$, similarly to the fourth row, we have a short exact sequence
\begin{eqnarray*}
0\rightarrow (2c-2|H^3 (2a+2,2b+2,2c)) \rightarrow \\
\rightarrow (2c-2|2b+1,2c-1|2a+4)+ (2c-2|2c-2|2a+3,2b+3) \rightarrow \\
\rightarrow(2c-2|2c-2|2b+2|2a+4)\rightarrow 0.
\end{eqnarray*}
We also have the surjective maps
$(2b,2c - 2|2a + 3,2c + 1)\rightarrow (2b, 2c - 2|2c|2a + 4)$
and
$(2c-2,2c-2|2a+3,2b+3)\rightarrow (2c-2,2c-2|2b+2|2a+4)$.
Therefore,
$E_2^{0,5}
(2b,2c - 2|\overline{2a + 3,2c + 1})+(2c-2,2c-2|\overline{2a+3,2b+3})$.

{\begin{center}
\begin{small}
\scriptsize\renewcommand{\arraystretch}{2.2}
\begin{longtable}{|c|c|c|c|c|c|c|c|c|}
\hline
& $p=0$
& $p=1$
& $p=2$\\
\hline
$q=0$
&  
&  
& 
\\
\hline
$q=1$
&
& 
&
\\
\hline
$q=2$
&
& $E_2^{1,2}=\overline{\Delta_2}$
&\\
\hline
$q=3$
& $E_2^{0,3}=(\overline{2a + 1,2c - 1}|2b + 2,2c)$
& 
& \\
\hline
$q=4$
&
&
& \\
\hline
$q=5$
& $E_2^{0,5}=
	\left\{
	\begin{tabular}{ll}
	$\Delta_1+\Delta_0$\\
	$(2b,2c - 2|\overline{2a + 3,2c + 1})$\\
	$(2c-2,2c-2|\overline{2a+3,2b+3})$
	\end{tabular}
	\right.$
& 
&\\
\hline
$q=6$
& 
& 
& \\
\hline
\end{longtable}
\end{small}
\end{center}

$\overline{\Delta_2}=coker\left[(\Delta_2|2c)\rightarrow (2b|2a+2,2c|2c)+(\overline{2a+1,2c-1}|2b+2|2c)\right]$, which is isomorphic to $(2b|2a+2,2c|2c)$.


\subsubsection{Boundary cohomology. Highest weight $(2a+1,2b+1,2c,2c)$.  Case $B2$.}
Again the boundary cohomology degenerates at the $E_2$-page. Therefore, up to semisimplicity we have
$H^n_\partial(GL_4(\Z),V_\lambda)=\bigoplus_{p+q=n}E_2^{p,q}$.
More systematically, we have

Let $H^q=H^q_\partial(GL_4(\Z),V_\lambda)$ be the cohomology of $GL_4(\Z)$ with coefficients in the highest weight representation with weight 
$\lambda=(2a+1,2b+1,2c,2c)$ written as a character of the split maximal torus. Then,
\[H^q
=
\left\{
\begin{tabular}{lll}
$0$
		& $q=2$\\
		$\left\{
\begin{tabular}{lll}
$E_2^{0,3}=(\overline{2a + 1,2c - 1}|2b + 2,2c)$\\
$E_2^{1,2}=\overline{\Delta_2}$
\end{tabular}
	\right\}$
		& $q=3$\\	
$0$
		& $q=4$\\
$E_2^{0,5}=
	\left\{
	\begin{tabular}{lll}
	$(\Delta_1|2a+4)$\\
	$(\Delta_0|2a+4)$\\
	$(2b,2c - 2|\overline{2a + 3,2c + 1})$\\
	$(2c-2,2c-2|\overline{2a+3,2b+3})$
	\end{tabular}
	\right\}$
		& $q=5$\\ 
		& $q=6$
\end{tabular}
\right.
\]
where
\[\Delta_2\subset (2b|2a+2,2c)+(\overline{2a+1,2c-1}|2b+2)\]
and
\[\overline{\Delta_2}=coker\left[(\Delta_2|2c)\rightarrow (2b|2a+2,2c|2c)+(\overline{2a+1,2c-1}|2b+2|2c)\right],\]
which is isomorphic to 
$(2b|2a+2,2c|2c)$.
We have
$\Delta_1
	\subset
	(2b,2c-2|2c)+(2c-2|\overline{2b+1,2c-1})$ and
$(\Delta_1|2a+4)$ is isomorphic to $(2b, 2c - 2|2c|2a+4)$.
And also
$\Delta_0
	\subset
	(2b|2c-2|2c)+(2c-2|2c-2|2b+2)$.
Thus,
$(\Delta_0|2a+4)
	\subset
	(2b|2c-2|2c|2a+4)+(2c-2|2c-2|2b+2|2a+4)$.

\begin{thm}
\[H^q
=
\left\{
\begin{tabular}{lll}
$0$
		& $q=2$\\
$S_{2a-2c+4}\otimes S_{2b-2c+4}
+
S_{2a-2c+4}$
		& $q=3$\\
$0$
		& $q=4$\\
$S_{2b-2c+4}
+
\C
+
S_{2b-2c+4}\otimes S_{2a_2c+4}
+
S_{2a-2b+2}$
			& $q=5$\\
$0$ 			& $q=6$
\end{tabular}
\right.
\]

\end{thm}




\subsection{Cohomology of $GL_4(\Z)$. Highest weight $(2a,2b,2c-1,2c-1)$.  Case $B2'$}


\subsubsection{Cohomology of the parabolic subgroups. Highest weight $(2a,2b,2c-1,2c-1)$. Case $B2'$}
For that case the weights are $(2a,2b,2c-1,2c-1)$.
We have to find the elements $w$ of the Weyl group that give a non-trivial cohomology of the minimal parabolic subgroup. The first entry has to go to the first or the third place in order to be even. Similarly, second entry should go to the second or fourth place. Also, the third entry should go to the second or the fourth and the four elements have to go the first and the third elements. 
Thus, all permutations are 
$1243$,
$1342$,
$4213$ and
$4312$.

$P_0$: 
$\begin{tabular}{lllll}
$w$ 		& $l$	& 	& $w(\lambda+\rho)-\rho$\\
$1243$ 	& $1$&	& $(2a|2b|2c-2|2c)$\\
$1342$	& $2$&	& $(2a|2c-2|2c-2|2b+2)$\\
$4213$	& $4$&	& $(2c-4|2b|2a+2|2c)$\\
$4312$	& $5$&	& $(2c-4|2c-2|2a+2|2b+2)$\\
\end{tabular}
$

For the intermediate parabolic subgroup $P_{12}$ we have to consider only the permutations that have $1$ or $4$ at the third place and $2$ or $3$ at the fourth place. Together with that the first and the second entry should be in increasing order.
For the last two elements of the permutation, we have $43$,  $42$, $13$ and $12$. Since the first two elements of the permutation have to be in increasing order, the permutations are 
$1243$,
$1342$,
$2413$ and
$3412$.
For them we have the following weights and lengths.

$P_{12}$: 
$\begin{tabular}{lllll}
$w$ 		& $l$	& 	& $w(\lambda+\rho)-\rho$\\
$1243$ 	& $1$&	& $(2a,2b|2c-2|2c)$\\
$1342$	& $2$&	& $(2a,2c-2|2c-2|2b+2)$\\
$2413$	& $3$&	& $(2b-1,2c-3|2a+2|2c)$\\
$3412$	& $4$&	& $(2c-3,2c-3|2a+2|2b+2)$\\
\end{tabular}
$

For the intermediate parabolic subgroup $P_{23}$ we have to consider only the permutations that have $2$ or $3$ at the fourth place and $1$ or $4$ at the first place. Together with that the second and the third entry should be in increasing order.
The permutations are 
$1243$,
$1342$,
$4123$ and
$4132$.
For them we have the following weights and lengths.

$P_{23}$: 
$\begin{tabular}{lllll}
$w$ 		& $l$	& 	& $w(\lambda+\rho)-\rho$\\
$1243$	& $1$&	& $(2a|2b,2c-2|2c)$\\
$1342$	& $2$&	& $(2a|2c-2,2c-2|2b+2)$\\
$4123$	& $3$&	& $(2c-4|2a+1, 2b+1|2c)$\\
$4132$	& $4$&	& $(2c-4|2a+1,2c-1|2b+2)$\\
\end{tabular}
$

For the intermediate parabolic subgroup $P_{34}$ we have to consider only the permutations that have $1$ or $4$ at the first place and $2$ or $3$ at the second place. Together with that the third and the fourth entry should be in increasing order.
For the first two elements of the permutation, we have 
$12$, $13$, $42$ and $43$. 
Since the last two elements of the permutation have to be in increasing order, the permutations are 
$1234$,
$1324$,
$4213$ and
$4312$.
For them we have the following weights and lengths.

$P_{34}$: 
$\begin{tabular}{lllll}
$w$ 		& $l$	& 	& $w(\lambda+\rho)-\rho$\\
$1234$	& $0$&	& $(2a|2b|2c-1,2c-1)$\\
$1324$	& $1$&	& $(2a|2c-2|2b+1,2c-1)$\\
$4213$	& $4$&	& $(2c-4|2b|2a+2,2c)$\\
$4312$	& $5$&	& $(2c-4|2c-2|2a+2,2b+2)$\\
\end{tabular}
$

For the maximal parabolic subgroup $P_{13}$ we have to consider only the permutations that have $2$ or $3$ at the fourth place. Together with that the first, the second  and the third entry should be in increasing order.
For the last element of the permutation, we have $1$ or $3$ Since the first three elements of the permutation have to be in increasing order, the permutations are$1243$, $2341$.
For them we have the following weights and lengths.

$P_{13}$: 
$\begin{tabular}{lllll}
$w$ 		& $l$	& 	& $w(\lambda+\rho)-\rho$\\
$1243$ 	& $1$&	& $(2a,2b,2c-2|2c)$\\
$1342$ 	& $2$&	& $(2a,2c-2,2c-2|2b+2)$\\
\end{tabular}
$

For the maximal parabolic subgroup $P_{12,34}$ we have to consider only elements that have the same parity in the first and the second entry. Together with that the first and the second, entry should be in increasing order. Also the third and the fourth entry should be in increasing order. Also the first and the second entry should have opposite parity. All the possibilities for the first two entries are: $12$, $13$, $24$, $34$. The corresponding permutations are $(1234)$, $(1423)$, $(2314)$ and $(3412)$

$P_{12,34}$: 
$\begin{tabular}{lllll}
$w$ 		& $l$	& 	& $w(\lambda+\rho)-\rho$\\
$1234$ 	& $0$&	& $(2a,2b|2c-1,2c-1)$\\
$1324$ 	& $1$&	& $(2a, 2c-2|2b+1,2c-1)$\\
$2413$ 	& $3$&	& $(2b-1,2c-3|2a+2,2c)$\\
$3412$	& $4$&	& $(2c-3,2c-3|2a+2,2b+2)$\\
\end{tabular}
$

For the maximal parabolic subgroup $P_{24}$ we have to consider only the permutations that have $1$ or $4$ at the first place. Together with that the second, the third and the fourth entry should be in increasing order.
For the first element of the permutation, we have $1$ or $4$ Since the first three elements of the permutation have to be in increasing order, the permutations are  $1234$, $4123$
For them we have the following weights and lengths.

$P_{24}$: 
$\begin{tabular}{lllll}
$w$ 		& $l$	& 	& $w(\lambda+\rho)-\rho$\\
$1234$ 	& $0$&	& $(2a|2b,2c-1,2c-1)$\\
$4123$ 	& $3$&	& $(2c-4|2a+1,2b+1,2c)$\\
\end{tabular}
$


\subsubsection{$E_1$-page. Highest weight $(2a,2b,2c-1,2c-1)$.  Case $B2'$}

{\begin{center}
\begin{small}
\scriptsize\renewcommand{\arraystretch}{2.2}
\begin{longtable}{|c|c|c|c|c|c|c|c|c|}
\hline
$E_1^{0,0}=0$ 
& $E_1^{1,0}=0$ & 
$E_1^{2,0}=0$
\\
\hline
$E_1^{0,1}=0$
& $E_1^{1,1}=0$
& $E_1^{2,1}=(2a|2b|2c - 2|2c)$
\\
\hline
$E_1^{0,2}=
	\left\{
	\begin{tabular}{ll}
		$(2a|H^2(2b,2c-1,2c-1))$
	\end{tabular}
	\right.$ 
& $E_1^{1,2}=
	\left\{
	\begin{tabular}{ll}
	$(2a, 2b|2c - 2|2c)$\\
	$(2a|2b, 2c - 2|2c)$\\
	$(2a|2c-2,2c-2|2b+2)$\\
	$(2a|2c-2|2b+1,2c-1)$
	\end{tabular}
	\right.$ 
& $E_1^{2,2}=(2a|2c - 2|2c - 2|2b + 2)$
\\
\hline
$E_1^{0,3}=
	\left\{
	\begin{tabular}{ll}
	$(2a,2c-2|2b+1,2c-1)$\\
	$(2a|H^3(2b,2c-1,2c-1))$
	\end{tabular}
	\right.
	$
& $E_1^{1,3}=
	\left\{
	\begin{tabular}{ll}
	 $(2a,2c-2|2c-2|2b+2)$\\
	\end{tabular}
	\right.
	$
& $E_1^{2,3}=0$\\
\hline
$E_1^{0,4}=(H^3(2a,2b,2c-2)|2c)$
& $E_1^{1,4}=
	\left\{
	\begin{tabular}{ll}
	$(2b-1,2c-3|2a+2|2c)$\\
	$(2c-4|2a+1,2b+1|2c)$
	\end{tabular}
	\right.$
& $E_1^{2,4}=(2c-4|2b|2a+2|2c)$\\
\hline
$E_1^{0,5}=\left\{
	\begin{tabular}{ll}
	$(H^3(2a,2c-2,2c-2)|2b+2)$\\
	$(2b-1,2c-3|2a+2,2c)$\\
	$(2c-4|H^2(2a+1,2b+1,2c))$
	\end{tabular}
	\right.$
& $E_1^{1,5}=\left\{
	\begin{tabular}{ll}
	$(2c-4|2a+1,2c-1|2b+2)$\\
	$(2c - 4|2b|2a + 2, 2c)$
	\end{tabular}
	\right.$
& $E_1^{2,5}=(2c-4|2c-2|2a+2|2b+2)$\\
\hline
$E_1^{0,6}=(2c-4|H^3(2a+1,2b+1,2c))$
& $E_1^{1,6}=(2c-4|2c-2|2a+2,2b+2)$
& $E_1^{2,6}=0$\\
\hline
\end{longtable}
\end{small}
\end{center}


\subsubsection{Certain cohomology groups of $GL_3(\Z)$ that appear on the $E_1$-page in the case $B2'$.}
From the computations of the cohomology of $GL_3(\Z)$ from the previous section, we have
{\begin{center}
\scriptsize\renewcommand{\arraystretch}{2.2}
\begin{longtable}{|c|c|c|}
\hline
$H^2(2b, 2c - 1, 2c - 1)
	=
	\Delta_0+\Delta_1$
&$H^3(GL_3(\Z),(2b, 2c - 1, 2c - 1))=0$\\
\hline
$H^2(2a, 2b, 2c - 2)=0$
& $H^3(2a, 2b, 2c - 2)=
	\left\{	\begin{tabular}{llll}
	$(\overline{2b-1,2c-3}|2a+2)$\\
	$(2c-4|\overline{2a+1,2b+1})$\\
	$(2c-4|2b|2a+2)$
	\end{tabular}
	\right.$\\
\hline
$H^2(2a,2c - 2,2c - 2)=0$
& $H^3(2a,2c - 2,2c - 2)=(2c-4|\overline{2a+1,2c-1})$\\
\hline
$H^2(2a+1,2b+1,2c)=\Delta_2$
&$H^3(2a+1,2b+1,2c)=(2c-2|2a+2,2b+2)$\\
\hline
\end{longtable}
\end{center}
where
$\Delta_1\subset (2b|2c-2|2c)+(2c-2|2c-2|2b+2)$,
$\Delta_1\subset (2b,2c-2|2c)+(2c-2|\overline{2b+1,2c-1})$
and
$\Delta_2\subset (\overline{2a+1,2c-1}|2b+2)+(2b|2a+2,2c)$.


\subsubsection{$E_2$ page. Highest weight $(2a,2b,2c-1,2c-1)$.  Case $B2'$}

For $q=1$, we have $E_2^{2,1}=(2a|2b|2c - 2|2c)$

For $q=2$,  we have $(2a|\Delta_1)\rightarrow (2a|2b, 2c - 2|2c) + (2a|2c - 2|2b + 1,2c - 1)\rightarrow (2a|2c-2|2c-2|2b+2)$, where the first map is injective and the second surjective. The homology of the above sequence is concentrated in the middle term and isomorphic to $(2a|2b, 2c - 2|2c)$.
For the term $(2a|\Delta_0)$ we have an $(2a,|\Delta_0)\rightarrow E_2^{2,1}+ (2a|2c - 2,2c - 2|2b + 2)$ an injective map. Since we assume that this embedding is diagonal, we obtain that the map $(2a,|\Delta_0)\rightarrow (2a|2c - 2,2c - 2|2b + 2)$  is an isomorphism. Therefore,
$E_2^{1,2}=(2a, 2b|2c + 2|2c)+(2a|2b, 2c - 2|2c)$ and $E_2^{p,2}=0$ for $p=0$ and $p=2$.

For $q=3$, we have a surjective map
$(2a,2c - 2|2b + 1,2c - 1)
\rightarrow 
(2a,2c - 2|2c - 2|2b + 2)$ and 
$(2a|H^3(2b, 2c -1, 2c - 1))=0$
Therefore, $E_2^{0,3}=(2a,2c - 2|\overline{2b + 1,2c - 1})$.

For $q=4$, we have a short exact sequence
$0\rightarrow E_1^{0,4}\rightarrow E_1^{1,4}\rightarrow E_1^{2,4}\rightarrow 0$.
Therefore, $E_2^{p,4}=0$ for $p=0,1,2$.

For $q=5$, we have a short exact sequence
\begin{eqnarray*}
&&0\rightarrow(H3(2a,2c - 2,2c - 2)|2b + 2)\rightarrow \\
&&\rightarrow (2c-4|2a+1,2c-1|2b+2)\rightarrow\\
&&\rightarrow  (2c-4|2c-2|2a+2|2b+2) \rightarrow 0
\end{eqnarray*}
Also, we have a surjective map
$(2b - 1,2c - 3|2a + 2,2c)\rightarrow
(2c - 4|2b|2a + 2, 2c)$.

Then,
$E_2^{0,5}=(\overline{2b - 1,2c - 3}|2a + 2,2c)+(2c-4|H^2 (2a+1,2b+1,2c))$ and $E_2^{p,5}=0$ for $p=1,2$.
Note that $(2c-4|H^2(2a+1,2b+1,2c))=(2c-4|\Delta_2)$, which is isomorphic to $(2c - 4|2b|2a + 2, 2c)$
Therefore, 
$E_2^{0,5}=(\overline{2b - 1,2c - 3}|2a + 2,2c)+(2c - 4|2b|2a + 2, 2c)$

For $q=6$, we have an isomorphism $E_1^{0,6}=(2c-4|H3(2a+1,2b+1,2c))\rightarrow E_1^{1,6}=(2c-4|2c-2|2a+2,2b+2)$.
Therefore, $E_2^{p,6}=0$ for $q=0,1,2$.

{\begin{center}
\begin{small}
\scriptsize\renewcommand{\arraystretch}{2.2}
\begin{longtable}{|c|c|c|c|c|c|c|c|c|}
\hline
& $p=0$
& $p=1$
& $p=2$\\
\hline
$q=0$
&  
&  
& \\
\hline
$q=1$
&
& 
& $E_2^{2,1}=(2a|2b|2c - 2|2c)$
\\
\hline
$q=2$
&  
&  $E_2^{1,2}=
	\left\{
	\begin{tabular}{ll}
	$(2a, 2b|2c + 2|2c)$\\
	$(2a|2b, 2c - 2|2c)$
	\end{tabular}
	\right.
	$
& 
\\
\hline
$q=3$
& $E_2^{0,3}=(2a,2c - 2|\overline{2b + 1,2c - 1})$
& 
&\\
\hline
$q=4$
&
& 
& \\
\hline
$q=5$
& $E_2^{0,5}=\left\{
	\begin{tabular}{lll}
	$(\overline{2b - 1,2c - 3}|2a + 2,2c)$\\
	$(2c - 4|2b|2a + 2, 2c)$
	\end{tabular}
	\right.$
& 
&\\
\hline
$q=6$
&
& 
& \\
\hline
\end{longtable}
\end{small}
\end{center}


\subsubsection{Boundary cohomology. Highest weight $(2a,2b,2c-1,2c-1)$.  Case $B2'$.}
Again the boundary cohomology degenerates at the $E_2$-page. Therefore, up to semisimplicity we have
$H^n_\partial(GL_4(\Z),V_\lambda)=\bigoplus_{p+q=n}E_2^{p,q}$.
More systematically, we have

Let $H^q=H^q_\partial(GL_4(\Z),V_\lambda)$ be the cohomology of $GL_4(\Z)$ with coefficients in the highest weight representation with weight 
$\lambda=(2a,2b,2c-1,2c-1)$ written as a character of the split maximal torus. Then,
\[H^q
=
\left\{
\begin{tabular}{lll}
	$0$		
		& $q=2$\\
	$\left\{\begin{tabular}{lll}
	$E_2^{0,3}=(2a,2c - 2|\overline{2b + 1,2c - 1})$\\
	$E_2^{1,2}=
	\left\{
	\begin{tabular}{ll}
	$(2a, 2b|2c + 2|2c)$\\
	$(2a|2b, 2c - 2|2c)$
	\end{tabular}
	\right.
	$\\
	$E_2^{2,1}=(2a|2b|2c - 2|2c)$
	\end{tabular}
	\right\}$
		& $q=3$\\	
$0$
		& $q=4$\\
$E_2^{0,5}=\left\{
	\begin{tabular}{lll}
	$(\overline{2b - 1,2c - 3}|2a + 2,2c)$\\
	$(2c - 4|2b|2a + 2, 2c)$
	\end{tabular}
	\right\}$
		& $q=5$\\ 
$0$
		& $q=6$
\end{tabular}
\right.
\]

\begin{thm}

\[H^q
=
\left\{
\begin{tabular}{lll}
	$S_{2a-2c+4}\otimes S_{2c-2c+4}
	+
	S_{2a-2b+2}
	+
	S_{2b-2c+4}
	+
	\C
	$
				& $q=3$\\
	$S_{2b-2c+4}\otimes S_{2a-2c+4}
	+
	S_{2a-2c+4}
	$
				& $q=5$\\
	$0$
				& $q\neq 3, 5$			
	\end{tabular}
\right.
\]

\end{thm}




\subsection{Cohomology of $GL_4(\Z)$. Highest weight $(2a+1,2a+1,2c,2d)$.  Case $B3$}


\subsubsection{Cohomology of the parabolic subgroups. Highest weight $(2a+1,2a+1,2c,2d)$. Case $B3$}
Let $\lambda=(2a+1,2a+1,2c,2d)$.

For the minimal parabolic subgroup $P_0$ we have to consider only the permutations that send $1$ and $4$ to $2$ or $4$, $2$ and $3$ to $1$ or $3$.
All possible such permutations are: $2134$, $2431$, $3124$ and $3421$ . For them we have the following weights and lengths.

$P_0$: 
$\begin{tabular}{lllll}
$w$ 		& $l$	& 	& $w(\lambda+\rho)-\rho$\\
$2134$ 	& $1$&	& $(2a|2a+2|2c|2d)$\\
$2431$	& $4$&	& $(2a|2d-2|2c|2a+4)$\\
$3124$	& $2$&	& $(2c-2|2a+2|2a+2|2d)$\\
$3421$	& $5$&	& $(2c-2|2d-2|2a+2|2a+4)$\\
\end{tabular}
$

For the intermediate parabolic subgroup $P_{12}$ we have to consider only the permutations that have $2$ or $3$ at the third place and $1$ or $4$ at the fourth place. Together with that the first and the second entry should be in increasing order.
For the last two elements of the permutation, we have $34$,  $24$, $31$ and $21$. Since the first two elements of the permutation have to be in increasing order, the permutations are $(1234)$, $1324$, $2431$ and $3421$.
For them we have the following weights and lengths.

$P_{12}$: 
$\begin{tabular}{lllll}
$w$ 		& $l$	& 	& $w(\lambda+\rho)-\rho$\\
$1234$ 	& $0$&	& $(2a+1,2a+1|2c|2d)$\\
$1324$ 	& $1$&	& $(2a+1,2c-1|2a+2|2d)$\\
$2431$	& $4$&	& $(2a,2d-2|2c|2a+4)$\\
$3421$	& $5$&	& $(2c-2,2d-2|2a+2|2a+4)$\\
\end{tabular}
$

For the intermediate parabolic subgroup $P_{23}$ we have to consider only the permutations that have $2$ or $3$ at the first place and $1$ or $4$ at the fourth place. Together with that the second and the third entry should be in increasing order.
The permutations are $2134$, $2341$, $3124$ and $3241$.
For them we have the following weights and lengths.

$P_{23}$: 
$\begin{tabular}{lllll}
$w$ 		& $l$	& 	& $w(\lambda+\rho)-\rho$\\
$2134$	& $1$&	& $(2a|2a+2,2c|2d)$\\
$2341$	& $3$&	& $(2a|2c-1,2d-1|2a+4)$\\
$3124$	& $2$&	& $(2c-2|2a+2, 2a+2|2d)$\\
$3241$	& $4$&	& $(2c-2|2a+1,2d-1|2a+4)$\\
\end{tabular}
$

For the intermediate parabolic subgroup $P_{34}$ we have to consider only the permutations that have $2$ or $3$ at the first place and $1$ or $4$ at the second place. Together with that the third and the fourth entry should be in increasing order.
For the first two elements of the permutation, we have $12$, $14$, $32$ and $34$. Since the last two elements of the permutation have to be in increasing order, the permutations are $2134$, $2413$, $3124$ and $3412$
For them we have the following weights and lengths.

$P_{34}$: 
$\begin{tabular}{lllll}
$w$ 		& $l$	& 	& $w(\lambda+\rho)-\rho$\\
$2134$	& $1$&	& $(2a|2a+2|2c,2d)$\\
$2413$ 	& $3$&	& $(2a|2d-2|2a+3,2c+1)$\\
$3124$	& $2$&	& $(2c-2|2a+2|2a+2,2d)$\\
$3412$	& $4$&	& $(2c-2|2d-2|2a+3,2a+3)$\\
\end{tabular}
$

For the maximal parabolic subgroup $P_{13}$ we have to consider only the permutations that have $1$ or $4$ at the fourth place. Together with that the first, the second  and the third entry should be in increasing order.
For the last element of the permutation, we have $1$ or $4$ Since the first three elements of the permutation have to be in increasing order, the permutations are $1234$, $2341$
For them we have the following weights and lengths.

$P_{13}$: 
$\begin{tabular}{lllll}
$w$ 		& $l$	& 	& $w(\lambda+\rho)-\rho$\\
$1234$ 	& $0$&	& $(2a+1,2a+1,2c|2d)$\\
$2341$ 	& $3$&	& $(2a,2c-1,2d-1|2a+4)$\\
\end{tabular}
$

For the parabolic subgroup $P_{12,34}$, we have to pick permutations $w$ so that $w(\lambda+\rho)-\rho$ has the same parity for the first two elements. All possibilities for the first two elements of the permutation are $odd,odd$ which can be achieved with $12$ and $13$ and $even,even$ which can be achieved with $24$ and $34$. 
The corresponding permutations are $1234$, $1324$, $2413$ and $3412$.

$P_{12,34}$: 
$\begin{tabular}{lllll}
$w$ 		& $l$	& 	& $w(\lambda+\rho)-\rho$\\
$1234$ 	& $0$&	& $(2a+1,2a+1|2c,2d)$\\
$1324$ 	& $1$&	& $(2a+1, 2c-1|2a+2,2d)$\\
$2413$	& $3$&	& $(2a,2d-2|2a+3,2c+1)$\\
$3412$	& $4$&	& $(2c-2,2d-2|2a+3,2a+3)$\\
\end{tabular}
$

For the maximal parabolic subgroup $P_{24}$ we have to consider only the permutations that have $2$ or $3$ at the first place. Together with that the second, the third and the fourth entry should be in increasing order.
For the first element of the permutation, we have $1$ or $3$ Since the first three elements of the permutation have to be in increasing order, the permutations are $(1234)$, $(3124)$
For them we have the following weights and lengths.

$P_{24}$: 
$\begin{tabular}{lllll}
$w$ 		& $l$	& 	& $w(\lambda+\rho)-\rho$\\
$2134$ 	& $1$&	& $(2a|2a+2,2c,2d)$\\
$3124$ 	& $2$&	& $(2c-2|2a+2,2a+2,2d)$\\
\end{tabular}
$


\subsubsection{$E_1$-page. Highest weight $(2a+1,2a+1,2c,2d)$.  Case $B3$.}

{\begin{center}
\begin{small}
\scriptsize\renewcommand{\arraystretch}{2.2}
\begin{longtable}{|c|c|c|c|c|c|c|c|c|}
\hline
$E_1^{0,0}=0$ 
& $E_1^{1,0}=0$ & 
$E_1^{2,0}=0$
\\
\hline
$E_1^{0,1}=0$
& $E_1^{1,1}=0$
& $E_1^{2,1}=(2a|2a+2|2c|2d)$
\\
\hline
$E_1^{0,2}=(H^2(2a+1,2a+1,2c)|2d)$
& $E_1^{1,2}=\left\{
	\begin{tabular}{ll}
	$(2a+1,2c-1|2a+2|2d)$\\
	$(2a|2a+2,2c|2d)$\\
	$(2c-2|2a+2, 2a+2|2d)$\\
	$(2a|2a+2|2c,2d)$
	\end{tabular}
	\right.
	$
& $E_1^{2,2}=(2c-2|2a+2|2a+2|2d)$
\\
\hline
$E_1^{0,3}=
	\left\{
	\begin{tabular}{ll}
	$(H^3(2a+1,2a+1,2c)|2d)$\\
	$(2a+1, 2c-1|2a+2,2d)$
	\end{tabular}
	\right.
	$
& $E_1^{1,3}=(2c-2|2a+2|2a+2,2d)$
& $E_1^{2,3}=0
	$\\
\hline
$E_1^{0,4}=(2a|H^3(2a+2,2c,2d))$
& $E_1^{1,4}=
	\left\{
	\begin{tabular}{ll}
	$(2a|2c-1,2d-1|2a+4)$\\
	$(2a|2d-2|2a+3,2c+1)$
	\end{tabular}
	\right.$
& $E_1^{2,4}=(2a|2d-2|2c|2a+4)$\\
\hline
$E_1^{0,5}=\left\{
	\begin{tabular}{ll}
	$(H^2(2a,2c-1,2d-1)|2a+4)$\\
	$(2a,2d-2|2a+3,2c+1)$\\
	 $(2c-2|H^3(2a+2,2a+2,2d))$
	\end{tabular}
	\right.$
& $E_1^{1,5}=\left\{
	\begin{tabular}{ll}
	$(2a,2d-2|2c|2a+4)$\\
	$(2c-2|2a+1,2d-1|2a+4)$
	\end{tabular}
	\right.$
& $E_1^{2,5}=(2c-2|2d-2|2a+2|2a+4)$\\
\hline
$E_1^{0,6}=(H^3(2a,2c-1,2d-1)|2a+4)$
& $E_1^{1,6}=(2c-2,2d-2|2a+2|2a+4)$
& $E_1^{1,6}=0$\\
\hline
\end{longtable}
\end{small}
\end{center}


\subsubsection{Certain cohomology groups of $GL_3(\Z)$ that appear on the $E_1$-page in the case $B3$.}
From the computations of the cohomology of $GL_3(\Z)$ from the previous section, we have
{\begin{center}
\scriptsize\renewcommand{\arraystretch}{2.2}
\begin{longtable}{|c|c|c|}
\hline
$H^2(2a+1,2a+1,2c)=\Delta_0+\Delta_1$
& $H^3(2a+1,2a+1,2c)=0$\\
\hline
$H^2(2a+2,2c,2d)=0$
& $H^3(2a+2,2c,2d)
	=
	\left\{	\begin{tabular}{llll}
	$(\overline{2c-1,2d-1}|2a+4)$\\
	$(2d-2|\overline{2a+3,2c+1})$\\
	$(2d-2|2c|2a+4)$
	\end{tabular}
	\right.$\\
	\hline
$H^2(GL_3(\Z),(2a+2,2a+2,2d))=0$
& $H^3(GL_3(\Z),(2a+2,2a+2,2d))
	=(\overline{2a+1,2d-1}|2a+4)$\\
	\hline
$H^2(GL_3(\Z),(2a,2c-1,2d-1))
	=
	\Delta_2
	\subset
	\left\{	\begin{tabular}{llll}
	$(2a,2d-2|2c)$\\
	$(2c-2|\overline{2a+1,2d-1})$	
	\end{tabular}
	\right.$
	& $H^3(GL_3(\Z),(2a,2c-1,2d-1))=(2c-2,2d-2|2a+2)$
\\
\hline
\end{longtable}
\end{center}
where 
$\Delta_0\subset (2a|2a+2|2c)+(2c-2|2a+2|2a+2)$,
$\Delta_1\subset (2a|2a+2,2c)+(\overline{2a+1,2c-1}|2a+2)$ 
and
$\Delta_2
	\subset
	(2a,2d-2|2c)+(2c-2|\overline{2a+1,2d-1})$.
	are diagonal embeddings.


\subsubsection{$E_2$-page. Highest weight $(2a,2a,2c,2d)$. Case $B3$}

For $q=1$, $E_2^{2,1}=E_1^{2,1}= (2a|2a + 2|2c|2d)$.

For $q=2$, we have 
the sequence
\begin{eqnarray*}
&&(\Delta_1|2d) \rightarrow\\
&&(2a + 1,2c - 1|2a + 2|2d) + (2a|2a + 2, 2c|2d)\rightarrow\\
&&(2c-2|2a+2|2a+2|2d)
\end{eqnarray*}
The first map is injective and the second surjective. Its homology is concentrated in the middle degree and isomorphic to $(2a|2a + 2, 2c|2d)$.
Also, $(\Delta_0|2d)$ is maps isomorphically to $(2c - 2|2a + 2,2a + 2|2d)$.
Therefore,
$E_2^{1,2}=(2a|2a + 2, 2c|2d)+(2a|2a + 2|2c, 2d)$
and $E_2^{p,2}=0$ for $p=0,2$.

For $q=3$, we have a surjective map
$(2a + 1,2c - 1|2a + 2,2d)\rightarrow (2c - 2|2a + 2|2a + 2,2d)$
 and
$(H^3(2a+1,2a+1,2c)|2d)=0$.
Therefore,
$E_2^{0,3}=(\overline{2a + 1,2c - 1}|2a + 2,2d)$
$E_2^{p,3}=0$ for $p=1,2$.

For $q=4$, the fourth row is a short exact sequence.
Therefore, $E_2^{p,4}=0$ for $p=0,1,2$.

For $q=5$, similarly to the fourth row, we have a short exact sequence
\begin{eqnarray*}
0\rightarrow (2c-2|H^3(2a+2,2a+2,2d)) \rightarrow \\
\rightarrow (2c-2|2a+1,2d-1|2a+4)+ (2c-2|2d-2|2a+3,2a+3) \rightarrow \\
\rightarrow(2c-2|2d-2|2a+2|2a+4)\rightarrow 0.
\end{eqnarray*}
We also have a surjective map
$(2a,2d - 2|2a + 3,2c + 1)\rightarrow (2a, 2d - 2|2c|2a + 4)$.
Therefore,
\begin{eqnarray*}
E_2^{0,5}
&&=(H^2(2a,2c - 1,2d - 1)|2a + 4)+(2a,2d - 2|\overline{2a + 3,2c + 1})=\\
&&=(\Delta_2|2a+4)+(2a,2d - 2|\overline{2a + 3,2c + 1})=\\
&&=(2a,2d-2|2c|2a+4)+(2a,2d - 2|\overline{2a + 3,2c + 1}).
\end{eqnarray*}

For $q=6$, we have that
$(H^3(2a,2c-1,2d-1)|2a+4)\rightarrow (2c-2,2d-2|2a+2|2a+4)$ is an isomorphism. Therefore,
$E_2^{p,6}=0$ for $p=0,1,2$.

{\begin{center}
\begin{small}
\scriptsize\renewcommand{\arraystretch}{2.2}
\begin{longtable}{|c|c|c|c|c|c|c|c|c|}
\hline
& $p=0$
& $p=1$
& $p=2$\\
\hline
$q=0$
&  
&  
& 
\\
\hline
$q=1$
&
& 
& $E_2^{2,1} = (2a|2a + 2|2c|2d)$.

\\
\hline
$q=2$
&
& $E_2^{1,2}=
	\left\{
	\begin{tabular}{ll}
	$(2a|2a + 2, 2c|2d)$\\
	$(2a|2a + 2|2c, 2d)$
	\end{tabular}
	\right.
	$
&\\
\hline
$q=3$
& $E_2^{0,3}=(\overline{2a + 1,2c - 1}|2a + 2,2d)$
& 
& \\
\hline
$q=4$
&
&
& \\
\hline
$q=5$
& $E_2^{0,5}=
	\left\{
	\begin{tabular}{ll}
	$(2a,2d-2|2c|2a+4)$\\
	$(2a,2d - 2|\overline{2a + 3,2c + 1})$
	\end{tabular}
	\right.
	$
& 
&\\
\hline
$q=6$
& 
& 
& \\
\hline
\end{longtable}
\end{small}
\end{center}

$\overline{\Delta_1}=coker\left[(\Delta_1|2d)\rightarrow (2a|2a+2,2c|2d)+(\overline{2a+1,2c-1}|2a+2|2d)\right]$


\subsubsection{Boundary cohomology. Highest weight $(2a+1,2a+1,2c,2d)$.  Case $B3$.}
Again the boundary cohomology degenerates at the $E_2$-page. Therefore, up to semisimplicity we have
$H^n_\partial(GL_4(\Z),V_\lambda)=\bigoplus_{p+q=n}E_2^{p,q}$.
More systematically, we have

Let $H^q=H^q_\partial(GL_4(\Z),V_\lambda)$ be the cohomology of $GL_4(\Z)$ with coefficients in the highest weight representation with weight 
$\lambda=(2a+1,2a+1,2c,2d)$ written as a character of the split maximal torus. Then,
\[H^q
=
\left\{
\begin{tabular}{lll}
	$0$
		& $q=2$\\
	$\left\{
	\begin{tabular}{lll}
	$E_2^{0,3}=(\overline{2a + 1,2c - 1}|2a + 2,2d)$\\
	$E_2^{1,2}=
	\left\{
	\begin{tabular}{ll}
	$(2a|2a + 2, 2c|2d)$\\
	$(2a|2a + 2|2c, 2d)$
	\end{tabular}
	\right\}
	$\\
	$E_2^{2,1} = (2a|2a + 2|2c|2d)$
	\end{tabular}
	\right\}$
		& $q=3$\\	
	$0$
		& $q=4$\\
	$E_2^{0,5}=
	\left\{
	\begin{tabular}{ll}
	$(2a,2d-2|2c|2a+4)$\\
	$(2a,2d - 2|\overline{2a + 3,2c + 1})$
	\end{tabular}
	\right\}
	$
		& $q=5$\\ 
	$E_2^{0,6}=0$
		& $q=6$
\end{tabular}
\right.
\]

\begin{thm}
\[H^q
=
\left\{
\begin{tabular}{lll}
	$S_{2a-2c+4}\otimes S_{2a-2d+4}
	+
	S_{2a-2c+4}
	+
	S_{2c-2d+2}
	+
	\C
	$
		& $q=3$\\
	$S_{2a-2d+4}
	+
	S_{2a-2d+4}\otimes S_{2a-2c+4}
	$
		& $q=5$\\
	$0$ 
		& $q\neq 3, 5$
\end{tabular}
\right.
\]
\end{thm}




\subsection{Cohomology of $GL_4(\Z)$. Highest weight $(2a,2a,2c-1,2d-1)$.  Case $B3'$}


\subsubsection{Cohomology of the parabolic subgroups. Highest weight $(2a,2a,2c-1,2d-1)$. Case $B3'$}
For that case the weights are $(2a,2a,2c-1,2d-1)$.
We have to find the elements $w$ of the Weyl group that give a non-trivial cohomology of the minimal parabolic subgroup. The first entry has to go to the first or the third place in order to be even. Similarly, second entry should go to the second or fourth place. Also, the third entry should go to the second or the fourth and the four elements have to go the first and the third elements. 
Thus, all permutations are 
$1243$,
$1342$,
$4213$ and
$4312$.

$P_0$: 
$\begin{tabular}{lllll}
$w$ 		& $l$	& 	& $w(\lambda+\rho)-\rho$\\
$1243$ 	& $1$&	& $(2a|2a|2d-2|2c)$\\
$1342$	& $2$&	& $(2a|2c-2|2d-2|2a+2)$\\
$4213$	& $4$&	& $(2d-4|2a|2a+2|2c)$\\
$4312$	& $5$&	& $(2d-4|2c-2|2a+2|2a+2)$\\
\end{tabular}
$

For the intermediate parabolic subgroup $P_{12}$ we have to consider only the permutations that have $1$ or $4$ at the third place and $2$ or $3$ at the fourth place. Together with that the first and the second entry should be in increasing order.
For the last two elements of the permutation, we have $43$,  $42$, $13$ and $12$. Since the first two elements of the permutation have to be in increasing order, the permutations are 
$1243$,
$1342$,
$2413$ and
$3412$.
For them we have the following weights and lengths.

$P_{12}$: 
$\begin{tabular}{lllll}
$w$ 		& $l$	& 	& $w(\lambda+\rho)-\rho$\\
$1243$ 	& $1$&	& $(2a,2a|2d-2|2c)$\\
$1342$	& $2$&	& $(2a,2c-2|2d-2|2a+2)$\\
$2413$	& $3$&	& $(2a-1,2d-3|2a+2|2c)$\\
$3412$	& $4$&	& $(2c-3,2d-3|2a+2|2a+2)$\\
\end{tabular}
$

For the intermediate parabolic subgroup $P_{23}$ we have to consider only the permutations that have $2$ or $3$ at the fourth place and $1$ or $4$ at the first place. Together with that the second and the third entry should be in increasing order.
The permutations are 
$1243$,
$1342$,
$4123$ and
$4132$.
For them we have the following weights and lengths.

$P_{23}$: 
$\begin{tabular}{lllll}
$w$ 		& $l$	& 	& $w(\lambda+\rho)-\rho$\\
$1243$	& $1$&	& $(2a|2a,2d-2|2c)$\\
$1342$	& $2$&	& $(2a|2c-2,2d-2|2a+2)$\\
$4123$	& $3$&	& $(2d-4|2a+1, 2a+1|2c)$\\
$4132$	& $4$&	& $(2d-4|2a+1,2c-1|2a+2)$\\
\end{tabular}
$

For the intermediate parabolic subgroup $P_{34}$ we have to consider only the permutations that have $1$ or $4$ at the first place and $2$ or $3$ at the second place. Together with that the third and the fourth entry should be in increasing order.
For the first two elements of the permutation, we have 
$12$, $13$, $42$ and $43$. 
Since the last two elements of the permutation have to be in increasing order, the permutations are 
$1234$,
$1324$,
$4213$ and
$4312$.
For them we have the following weights and lengths.

$P_{34}$: 
$\begin{tabular}{lllll}
$w$ 		& $l$	& 	& $w(\lambda+\rho)-\rho$\\
$1234$	& $0$&	& $(2a|2a|2c-1,2d-1)$\\
$1324$	& $1$&	& $(2a|2c-2|2a+1,2d-1)$\\
$4213$	& $4$&	& $(2d-4|2a|2a+2,2c)$\\
$4312$	& $5$&	& $(2d-4|2c-2|2a+2,2a+2)$\\
\end{tabular}
$

For the maximal parabolic subgroup $P_{13}$ we have to consider only the permutations that have $2$ or $3$ at the fourth place. Together with that the first, the second  and the third entry should be in increasing order.
For the last element of the permutation, we have $1$ or $3$ Since the first three elements of the permutation have to be in increasing order, the permutations are$1243$, $2341$.
For them we have the following weights and lengths.

$P_{13}$: 
$\begin{tabular}{lllll}
$w$ 		& $l$	& 	& $w(\lambda+\rho)-\rho$\\
$1243$ 	& $1$&	& $(2a,2a,2d-2|2c)$\\
$1342$ 	& $2$&	& $(2a,2c-2,2d-2|2a+2)$\\
\end{tabular}
$

For the maximal parabolic subgroup $P_{12,34}$ we have to consider only elements that have the same parity in the first and the second entry. Together with that the first and the second, entry should be in increasing order. Also the third and the fourth entry should be in increasing order. Also the first and the second entry should have opposite parity. All the possibilities for the first two entries are: $12$, $13$, $24$, $34$. The corresponding permutations are $(1234)$, $(1423)$, $(2314)$ and $(3412)$

$P_{12,34}$: 
$\begin{tabular}{lllll}
$w$ 		& $l$	& 	& $w(\lambda+\rho)-\rho$\\
$1234$ 	& $0$&	& $(2a,2a|2c-1,2d-1)$\\
$1324$ 	& $1$&	& $(2a, 2c-2|2a+1,2d-1)$\\
$2413$ 	& $3$&	& $(2a-1,2d-3|2a+2,2c)$\\
$3412$	& $4$&	& $(2c-3,2d-3|2a+2,2a+2)$\\
\end{tabular}
$

For the maximal parabolic subgroup $P_{24}$ we have to consider only the permutations that have $1$ or $4$ at the first place. Together with that the second, the third and the fourth entry should be in increasing order.
For the first element of the permutation, we have $1$ or $4$ Since the first three elements of the permutation have to be in increasing order, the permutations are  $1234$, $4123$
For them we have the following weights and lengths.

$P_{24}$: 
$\begin{tabular}{lllll}
$w$ 		& $l$	& 	& $w(\lambda+\rho)-\rho$\\
$1234$ 	& $0$&	& $(2a|2a,2c-1,2d-1)$\\
$4123$ 	& $3$&	& $(2d-4|2a+1,2a+1,2c)$\\
\end{tabular}
$


\subsubsection{$E_1$-page. Highest weight $(2a,2a,2c-1,2d-1)$.  Case $B3'$}

{\begin{center}
\begin{small}
\scriptsize\renewcommand{\arraystretch}{2.2}
\begin{longtable}{|c|c|c|c|c|c|c|c|c|}
\hline
$E_1^{0,0}=0$ 
& $E_1^{1,0}=0$ & 
$E_1^{2,0}=0$
\\
\hline
$E_1^{0,1}=(2a,2a|2c-1,2d-1)$
& $E_1^{1,1}=\left\{
	\begin{tabular}{ll}
	$(2a|2a|2c - 1, 2d - 1)$\\
	$(2a, 2a|2d - 2|2c)$
	\end{tabular}
	\right.$
& $E_1^{2,1}=(2a|2a|2d - 2|2c)$
\\
\hline
$E_1^{0,2}=(2a|H^2(2a,2c-1,2d-1))$
& $E_1^{1,2}=
	\left\{
	\begin{tabular}{ll}
	$(2a|2a, 2d - 2|2c)$\\
	$(2a|2c-2|2a+1,2d-1)$
	\end{tabular}
	\right.$ 
& $E_1^{2,2}=(2a|2c - 2|2d - 2|2a + 2)$
\\
\hline
$E_1^{0,3}=
	\left\{
	\begin{tabular}{ll}
	$(2a,2c-2|2a+1,2d-1)$\\
	$(2a|H^3(2a,2c-1,2d-1))$
	\end{tabular}
	\right.
	$
& $E_1^{1,3}=
	\left\{
	\begin{tabular}{ll}
	 $(2a,2c-2|2d-2|2a+2)$\\
	$(2a|2c-2,2d-2|2a+2)$
	\end{tabular}
	\right.
	$
& $E_1^{2,3}=0$\\
\hline
$E_1^{0,4}=(H^3(2a,2a,2d-2)|2c)$
& $E_1^{1,4}=(2a-1,2d-3|2a+2|2c)$
& $E_1^{2,4}=(2d-4|2a|2a+2|2c)$\\
\hline
$E_1^{0,5}=\left\{
	\begin{tabular}{ll}
	$(H^3(2a,2c-2,2d-2)|2a+2)$\\
	$(2a-1,2d-3|2a+2,2c)$\\
	$(2c-3,2d-3|2a+2,2a+2)$\\
	$(2d-4|H^2(2a+1,2a+1,2c))$
	\end{tabular}
	\right.$
& $E_1^{1,5}=\left\{
	\begin{tabular}{ll}
	$(2c-3,2d-3|2a+2|2a+2)$\\
	$(2d-4|2a+1,2c-1|2a+2)$\\
	$(2d - 4|2a|2a + 2, 2c)$\\
	$(2d-4|2c-2|2a+2,2a+2)$
	\end{tabular}
	\right.$
& $E_1^{2,5}=(2d-4|2c-2|2a+2|2a+2)$\\
\hline
$E_1^{0,6}=(2d-4|H^3(2a+1,2a+1,2c))=0$
& $E_1^{1,6}=0$
& $E_1^{2,6}=0$\\
\hline
\end{longtable}
\end{small}
\end{center}


\subsubsection{Certain cohomology groups of $GL_3(\Z)$ that appear on the $E_1$-page in the case $B3'$.}
From the computations of the cohomology of $GL_3(\Z)$ from the previous section, we have
{\begin{center}
\scriptsize\renewcommand{\arraystretch}{2.2}
\begin{longtable}{|c|c|c|}
\hline
$H^2(2a, 2c - 1, 2d - 1)
	=
	\Delta_2$
&$H^3(GL_3(\Z),(2a, 2c - 1, 2d - 1))=(2c-2,2d-2|2a+2)$\\
\hline
$H^2(2a, 2a, 2d - 2)=0$
& $H^3(2a, 2a, 2d - 2)=(\overline{2a-1,2d-3}|2a+2)$\\
\hline
$H^2(2a,2c - 2,2d - 2)=0$
& $H^3(2a,2c - 2,2d - 2)=
	=
	\left\{	\begin{tabular}{llll}
	$(\overline{2c-3,2d-3}|2a+2)$\\
	$(2d-4|\overline{2a+1,2c-1})$\\
	$(2d-4|2c-2|2a+2)$
	\end{tabular}
	\right.$\\
\hline
$H^2(2a+1,2a+1,2c)=\Delta_1+\Delta_0$
&$H^3(2a+1,2a+1,2c)=(2c-2|2a+2,2a+2)$\\
\hline
\end{longtable}
\end{center}
where
$\Delta_2\subset (2a,2d-2|2c)+(2c-2|\overline{2a+1,2d-1})$, 
$\Delta_1\subset (\overline{2a+1,2c-1}|2a+2)+(2a|2a+2,2c)$
and
$\Delta_0\subset (2c-2|2a+2|2a+2)+(2a|2a+2|2c)$.


\subsubsection{$E_2$ page. Highest weight $(2a,2a,2c-1,2d+1)$.  Case $B3'$}

For $q=1$, we have a short exact sequence $0\rightarrow E_1^{0,1}\rightarrow E_1^{1,1}\rightarrow E_1^{2,1}\rightarrow 0$. Therefore, $E_2^{p,2}=0$ for $p=0,1,2$.

For $q=2$, we have a surjective maps $(2a|2c - 2|2a + 1,2d - 1)\rightarrow (2a|2c-2|2d-2|2a+2)$.
We also have an injective map $(2a|\Delta_1)\rightarrow (2a|2a, 2d - 2|2c) (2a|2c - 2|2a + 1,2d -1)$.
Therefore $E_2^{0,2}=E_2^{2,2}=0$. Also, 
\[E_2^{1,2}=\overline{\Delta}_2=coker\left[(2a|\Delta_2)\rightarrow (2a|2a, 2d - 2|2c) (2a|2c - 2|2a + 1,2d -1)\right]\] 
is isomorphic to 
$(2a|2a, 2d - 2|2c)$

For $q=3$, we have a surjective map
$(2a,2c - 2|2a + 1,2d - 1)
\rightarrow 
(2a,2c - 2|2d - 2|2a + 2)$ and an isomorphism 
$(2a|H^3(2a, 2c -1, 2d - 1))
\rightarrow
(2a|2c - 2,2d - 2|2a + 2)$
Therefore, $E_2^{0,3}=(2a,2c - 2|\overline{2a + 1,2d - 1})$.

For $q=4$, we have a short exact sequence
$0\rightarrow E_1^{0,4}\rightarrow E_1^{1,4}\rightarrow E_1^{2,4}\rightarrow 0$.
Therefore, $E_2^{p,4}=0$ for $p=0,1,2$.

For $q=5$, we have a surjective maps
$(2a - 1,2d - 3|2a + 2,2c)\rightarrow
(2d - 4|2a|2a + 2, 2c)$
and $(2c-3,2d-3|2a+2,2a+2)\rightarrow (2d-4|2c-2|2a+2,2a+2)$

We also have a short exact sequence
\begin{eqnarray*}
&0\rightarrow (H^3(2a,2c -2,2d - 2)|2a + 2)rightarrow\\
&\rightarrow (2c-3,2d-3|2a+2|2a+2) + (2d - 4|2a + 1, 2c - 1|2a + 2)\rightarrow\\
&\rightarrow (2d-4|2c-2|2a+2|2a+2)\rightarrow 0
\end{eqnarray*}
Therefore,
$E_2^{0,5}=(\overline{2a - 1,2d - 3}|2a + 2,2c)+(\overline{2c-3,2d-3}|2a+2,2a+2)
+(2d-4|H^2 (2a+1,2a+1,2c))$ and $E_2^{p,5}=0$ for $p=1,2$.
Note that $(2d-4|H^2(2a+1,2a+1,2c))=(2d-4|\Delta_1)+(2d-4|\Delta_0)$, which is isomorphic to $(2d - 4|2a|2a + 2, 2c)+(2d-4|2c-2|2a+2|2a+2)$

For $q=6$, we have a surjective map
$(2c-3,2d-3|2a+2,2a+2)
\rightarrow
(2d-4|2c-2|2a+2,2a+2)$
Therefore
$E_2^{0,6}=(\overline{2c-3,2d-3}|2a+2,2a+2)+(2d - 4|2a|2a + 2, 2c)$
and $E_2^{p,6}=0$ for $p=1,2$.

{\begin{center}
\begin{small}
\scriptsize\renewcommand{\arraystretch}{2.2}
\begin{longtable}{|c|c|c|c|c|c|c|c|c|}
\hline
& $p=0$
& $p=1$
& $p=2$\\
\hline
$q=0$
&  
&  
& \\
\hline
$q=1$
&
& 
& 
\\
\hline
$q=2$
&  
&  $E_2^{1,2}=(2a|2a, 2d - 2|2c)$
& 
\\
\hline
$q=3$
& $E_2^{0,3}=(2a,2c - 2|\overline{2a + 1,2d - 1})$
& 
&\\
\hline
$q=4$
&
& 
& \\
\hline
$q=5$
& $E_2^{0,5}=\left\{
	\begin{tabular}{lll}
	$(\overline{2a - 1,2d - 3}|2a + 2,2c)$\\
	$(\overline{2c-3,2d-3}|2a+2,2a+2)$\\
	$(2d - 4|2a|2a + 2, 2c)$\\
	$(2d-4|2c-2|2a+2|2a+2)$
	\end{tabular}
	\right.$
& 
&\\
\hline
$q=6$
&
& 
& \\
\hline
\end{longtable}
\end{small}
\end{center}


\subsubsection{Boundary cohomology. Highest weight $(2a,2a,2c-1,2d-1)$.  Case $B3'$.}
Again the boundary cohomology degenerates at the $E_2$-page. Therefore, up to semisimplicity we have
$H^n_\partial(GL_4(\Z),V_\lambda)=\bigoplus_{p+q=n}E_2^{p,q}$.
More systematically, we have

Let $H^q=H^q_\partial(GL_4(\Z),V_\lambda)$ be the cohomology of $GL_4(\Z)$ with coefficients in the highest weight representation with weight 
$\lambda=(2a,2a,2c-1,2d-1)$ written as a character of the split maximal torus. Then,
\[H^q
=
\left\{
\begin{tabular}{lll}
	$0$
		& $q=2$\\
	$\left\{\begin{tabular}{lll}
	$E_2^{0,3}=(2a,2c - 2|\overline{2a + 1,2d - 1})$\\
	$E_2^{1,2}=(2a|2a, 2d - 2|2c)$
	
	\end{tabular}
	\right\}$
		& $q=3$\\	
$0$
		& $q=4$\\
$E_2^{0,5}=\left\{
	\begin{tabular}{lll}
	$(\overline{2a - 1,2d - 3}|2a + 2,2c)$\\
	$(2d - 4|2a|2a + 2, 2c)$\\
	$(2d-4|2c-2|2a+2|2a+2)$
	\end{tabular}
	\right\}$
		& $q=5$\\ 
$0$
		& $q=6$
\end{tabular}
\right.
\]

Since the above $H^4_\partial(B3')$ contains $E_2^{p,q}$ with $p>0$, we have that $H^4_\partial(B3')$ contains a potentially ghost class, namely,
\[pGh^3(GL_4(\Z),(2a+1,2a+1,2c+1,2d+1))=E_2^{1,2}=(2a|2a, 2d - 2|2c).\]

\begin{thm}
\[H^q
=
\left\{
\begin{tabular}{lll}
	$S_{2a-2c+4}\otimes S_{2a-2d+4}
	+
	S_{2a-2d+4}
	$
				& $q=3$\\
	$S_{2a-2c+4}\otimes S_{2a-2d+4}
	+
	S_{2a-2c+4}
	$
				& $q=5$\\
	$0$
		& $q\neq 3, 5$
\end{tabular}
\right.
\]
\end{thm}




\subsection{Cohomology of $GL_4(\Z)$. Highest weight $(2a+1,2a+1,2c,2c)$.  Case $B4$}


\subsubsection{Cohomology of the parabolic subgroups. Highest weight $(2a+1,2a+1,2c,2c)$. Case $B4$}
Let $\lambda=(2a+1,2a+1,2c,2c)$.

For the minimal parabolic subgroup $P_0$ we have to consider only the permutations that send $1$ and $4$ to $2$ or $4$, $2$ and $3$ to $1$ or $3$.
All possible such permutations are: $2134$, $2431$, $3124$ and $3421$ . For them we have the following weights and lengths.

$P_0$: 
$\begin{tabular}{lllll}
$w$ 		& $l$	& 	& $w(\lambda+\rho)-\rho$\\
$2134$ 	& $1$&	& $(2a|2a+2|2c|2c)$\\
$2431$	& $4$&	& $(2a|2c-2|2c|2a+4)$\\
$3124$	& $2$&	& $(2c-2|2a+2|2a+2|2c)$\\
$3421$	& $5$&	& $(2c-2|2c-2|2a+2|2a+4)$\\
\end{tabular}
$

For the intermediate parabolic subgroup $P_{12}$ we have to consider only the permutations that have $2$ or $3$ at the third place and $1$ or $4$ at the fourth place. Together with that the first and the second entry should be in increasing order.
For the last two elements of the permutation, we have $34$,  $24$, $31$ and $21$. Since the first two elements of the permutation have to be in increasing order, the permutations are $(1234)$, $1324$, $2431$ and $3421$.
For them we have the following weights and lengths.

$P_{12}$: 
$\begin{tabular}{lllll}
$w$ 		& $l$	& 	& $w(\lambda+\rho)-\rho$\\
$1234$ 	& $0$&	& $(2a+1,2a+1|2c|2c)$\\
$1324$ 	& $1$&	& $(2a+1,2c-1|2a+2|2c)$\\
$2431$	& $4$&	& $(2a,2c-2|2c|2a+4)$\\
$3421$	& $5$&	& $(2c-2,2c-2|2a+2|2a+4)$\\
\end{tabular}
$

For the intermediate parabolic subgroup $P_{23}$ we have to consider only the permutations that have $2$ or $3$ at the first place and $1$ or $4$ at the fourth place. Together with that the second and the third entry should be in increasing order.
The permutations are $2134$, $2341$, $3124$ and $3241$.
For them we have the following weights and lengths.

$P_{23}$: 
$\begin{tabular}{lllll}
$w$ 		& $l$	& 	& $w(\lambda+\rho)-\rho$\\
$2134$	& $1$&	& $(2a|2a+2,2c|2c)$\\
$2341$	& $3$&	& $(2a|2c-1,2c-1|2a+4)$\\
$3124$	& $2$&	& $(2c-2|2a+2, 2a+2|2c)$\\
$3241$	& $4$&	& $(2c-2|2a+1,2c-1|2a+4)$\\
\end{tabular}
$

For the intermediate parabolic subgroup $P_{34}$ we have to consider only the permutations that have $2$ or $3$ at the first place and $1$ or $4$ at the second place. Together with that the third and the fourth entry should be in increasing order.
For the first two elements of the permutation, we have $12$, $14$, $32$ and $34$. Since the last two elements of the permutation have to be in increasing order, the permutations are $2134$, $2413$, $3124$ and $3412$
For them we have the following weights and lengths.

$P_{34}$: 
$\begin{tabular}{lllll}
$w$ 		& $l$	& 	& $w(\lambda+\rho)-\rho$\\
$2134$	& $1$&	& $(2a|2a+2|2c,2c)$\\
$2413$ 	& $3$&	& $(2a|2c-2|2a+3,2c+1)$\\
$3124$	& $2$&	& $(2c-2|2a+2|2a+2,2c)$\\
$3412$	& $4$&	& $(2c-2|2c-2|2a+3,2a+3)$\\
\end{tabular}
$

For the maximal parabolic subgroup $P_{13}$ we have to consider only the permutations that have $1$ or $4$ at the fourth place. Together with that the first, the second  and the third entry should be in increasing order.
For the last element of the permutation, we have $1$ or $4$ Since the first three elements of the permutation have to be in increasing order, the permutations are $1234$, $2341$
For them we have the following weights and lengths.

$P_{13}$: 
$\begin{tabular}{lllll}
$w$ 		& $l$	& 	& $w(\lambda+\rho)-\rho$\\
$1234$ 	& $0$&	& $(2a+1,2a+1,2c|2c)$\\
$2341$ 	& $3$&	& $(2a,2c-1,2c-1|2a+4)$\\
\end{tabular}
$

For the parabolic subgroup $P_{12,34}$, we have to pick permutations $w$ so that $w(\lambda+\rho)-\rho$ has the same parity for the first two elements. All possibilities for the first two elements of the permutation are $odd,odd$ which can be achieved with $12$ and $13$ and $even,even$ which can be achieved with $24$ and $34$. 
The corresponding permutations are $1234$, $1324$, $2413$ and $3412$.

$P_{12,34}$: 
$\begin{tabular}{lllll}
$w$ 		& $l$	& 	& $w(\lambda+\rho)-\rho$\\
$1234$ 	& $0$&	& $(2a+1,2a+1|2c,2c)$\\
$1324$ 	& $1$&	& $(2a+1, 2c-1|2a+2,2c)$\\
$2413$	& $3$&	& $(2a,2c-2|2a+3,2c+1)$\\
$3412$	& $4$&	& $(2c-2,2c-2|2a+3,2a+3)$\\
\end{tabular}
$

For the maximal parabolic subgroup $P_{24}$ we have to consider only the permutations that have $2$ or $3$ at the first place. Together with that the second, the third and the fourth entry should be in increasing order.
For the first element of the permutation, we have $1$ or $3$ Since the first three elements of the permutation have to be in increasing order, the permutations are $(1234)$, $(3124)$
For them we have the following weights and lengths.

$P_{24}$: 
$\begin{tabular}{lllll}
$w$ 		& $l$	& 	& $w(\lambda+\rho)-\rho$\\
$2134$ 	& $1$&	& $(2a|2a+2,2c,2c)$\\
$3124$ 	& $2$&	& $(2c-2|2a+2,2a+2,2c)$\\
\end{tabular}
$


\subsubsection{$E_1$-page. Highest weight $(2a+1,2a+1,2c,2c)$.  Case $B4$.}

{\begin{center}
\begin{small}
\scriptsize\renewcommand{\arraystretch}{2.2}
\begin{longtable}{|c|c|c|c|c|c|c|c|c|}
\hline
$E_1^{0,0}=0$ 
& $E_1^{1,0}=0$ & 
$E_1^{2,0}=0$
\\
\hline
$E_1^{0,1}=0$
& $E_1^{1,1}=(2a|2a+2|2c,2c)$
& $E_1^{2,1}=(2a|2a+2|2c|2c)$
\\
\hline
$E_1^{0,2}=(H^2(2a+1,2a+1,2c)|2c)$
& $E_1^{1,2}=\left\{
	\begin{tabular}{ll}
	$(2a+1,2c-1|2a+2|2c)$\\
	$(2a|2a+2,2c|2c)$\\
	$(2c-2|2a+2, 2a+2|2c)$
	\end{tabular}
	\right.
	$
& $E_1^{2,2}=(2c-2|2a+2|2a+2|2c)$
\\
\hline
$E_1^{0,3}=
	\left\{
	\begin{tabular}{ll}
	$(H^3(2a+1,2a+1,2c)|2c)$\\
	$(2a+1, 2c-1|2a+2,2c)$
	\end{tabular}
	\right.
	$
& $E_1^{1,3}=(2c-2|2a+2|2a+2,2c)$
& $E_1^{2,3}=0
	$\\
\hline
$E_1^{0,4}=(2a|H^3(2a+2,2c,2c))$
& $E_1^{1,4}=(2a|2c-2|2a+3,2c+1)$
& $E_1^{2,4}=(2a|2c-2|2c|2a+4)$\\
\hline
$E_1^{0,5}=\left\{
	\begin{tabular}{ll}
	$(H^2(2a,2c-1,2c-1)|2a+4)$\\
	$(2a,2c-2|2a+3,2c+1)$\\
	 $(2c-2|H^3(2a+2,2a+2,2c))$
	\end{tabular}
	\right.$
& $E_1^{1,5}=\left\{
	\begin{tabular}{ll}
	$(2a,2c-2|2c|2a+4)$\\
	$(2c-2,2c-2|2a+2|2a+4)$\\
	$(2c-2|2a+1,2c-1|2a+4)$
	\end{tabular}
	\right.$
& $E_1^{2,5}=(2c-2|2c-2|2a+2|2a+4)$\\
\hline
$E_1^{0,6}=(H^3(2a,2c-1,2c-1)|2a+4)$
& $E_1^{1,6}=0$
& $E_1^{1,6}=0$\\
\hline
\end{longtable}
\end{small}
\end{center}


\subsubsection{Certain cohomology groups of $GL_3(\Z)$ that appear on the $E_1$-page in the case $B4$.}
From the computations of the cohomology of $GL_3(\Z)$ from the previous section, we have
{\begin{center}
\scriptsize\renewcommand{\arraystretch}{2.2}
\begin{longtable}{|c|c|c|}
\hline
$H^2(2a+1,2a+1,2c)=\Delta_0+\Delta_1$
& $H^3(2a+1,2a+1,2c)=0$\\
\hline
$H^2(2a+2,2c,2c)=0$
& $H^3(2a+2,2c,2c)=(2c-2|\overline{2a+3,2c+1})$\\
	\hline
$H^2(GL_3(\Z),(2a+2,2a+2,2c))=0$
& $H^3(GL_3(\Z),(2a+2,2a+2,2c))
	=(\overline{2a+1,2c-1}|2a+4)$\\
	\hline
$H^2(GL_3(\Z),(2a,2c-1,2c-1))
	=
	\Delta'_0+\Delta'_1$
	& $H^3(GL_3(\Z),(2a,2c-1,2c-1))=0$
\\
\hline
\end{longtable}
\end{center}
where 
$\Delta_0\subset (2a|2a+2|2c)+(2c-2|2a+2|2a+2)$,
$\Delta_1\subset (2a|2a+2,2c)+(\overline{2a+1,2c-1}|2a+2)$,
and
$\Delta'_0
	\subset
	(2a|2c-2|2c)+(2c-2|2c-2|2a+2)$,
$\Delta'_1
	\subset
	(2a,2c-2|2c)+(2c-2|\overline{2a+1,2c-1})$.
	are diagonal embeddings.


\subsubsection{$E_2$-page. Highest weight $(2a,2a,2c,2c)$. Case $B4$}

For $q=1$, we have an isomorphism $E_1^{1,1}=(2a|2a + 2|2c, 2c)\rightarrow E_1^{2,1}=(2a|2a + 2|2c|2c)$. Therefore, $E_2^{p,1}=0$ for $p=0,1,2$.

For $q=2$, we have 
the sequence
\begin{eqnarray*}
&&(\Delta_1|2c) \rightarrow\\
&&(2a + 1,2c - 1|2a + 2|2c) + (2a|2a + 2, 2c|2c)\rightarrow\\
&&(2c-2|2a+2|2a+2|2c)
\end{eqnarray*}
The first map is injective and the second surjective. Its homology is concentrated in the middle degree and isomorphic to $(2a|2a + 2, 2c|2c)$.
Also, $(\Delta_0|2c)$ is maps isomorphically to $(2c - 2|2a + 2,2a + 2|2c)$.
Therefore,
$E_2^{1,2}=(2a|2a + 2, 2c|2c)$
and $E_2^{p,2}=0$ for $p=0,2$.

For $q=3$, we have a surjective map
$(2a + 1,2c - 1|2a + 2,2c)\rightarrow (2c - 2|2a + 2|2a + 2,2c)$
 and
$(H^3(2a+1,2a+1,2c)|2c)=0$.
Therefore,
$E_2^{0,3}=(\overline{2a + 1,2c - 1}|2a + 2,2c)$
$E_2^{p,3}=0$ for $p=1,2$.

For $q=4$, the fourth row is a short exact sequence.
Therefore, $E_2^{p,4}=0$ for $p=0,1,2$.

For $q=5$, similarly to the fourth row, we have a short exact sequence
\begin{eqnarray*}
0\rightarrow (2c-2|H^3(2a+2,2a+2,2c)) \rightarrow \\
\rightarrow (2c-2|2a+1,2c-1|2a+4) \rightarrow \\
\rightarrow(2c-2|2c-2|2a+2|2a+4)\rightarrow 0.
\end{eqnarray*}
We also have a surjective map
$(2a,2c - 2|2a + 3,2c + 1)\rightarrow (2a, 2c - 2|2c|2a + 4)$.
Therefore,
\begin{eqnarray*}
E_2^{0,5}
&&=(2a,2c - 2|\overline{2a + 3,2c + 1})+ker\left[(H^2(2a,2c - 1,2c - 1)|2a + 4)\rightarrow (2c-2,2c-2|2a+2|2a+4)\right]\\
&&=(2a,2c - 2|\overline{2a + 3,2c + 1})(\Delta'_1|2a+4)=\\
&&(2a,2c - 2|\overline{2a + 3,2c + 1})+(2a,2c-2|2c|2a+4)
\end{eqnarray*}

For $q=6$, we have that
$H^3(2a,2c-1,2c-1)=0$ Therefore,
$E_2^{p,6}=E_1^{p,6}=0$ for $p=0,1,2$.

{\begin{center}
\begin{small}
\scriptsize\renewcommand{\arraystretch}{2.2}
\begin{longtable}{|c|c|c|c|c|c|c|c|c|}
\hline
& $p=0$
& $p=1$
& $p=2$\\
\hline
$q=0$
&  
&  
& 
\\
\hline
$q=1$
&
& 
&

\\
\hline
$q=2$
&
& $E_2^{1,2}=(2a|2a + 2, 2c|2c)$
&\\
\hline
$q=3$
& $E_2^{0,3}=(\overline{2a + 1,2c - 1}|2a + 2,2c)$
& 
& \\
\hline
$q=4$
&
&
& \\
\hline
$q=5$
& $E_2^{0,5}=
	\left\{
	\begin{tabular}{ll}
	$(2a,2c-2|2c|2a+4)$\\
	$(2a,2c - 2|\overline{2a + 3,2c + 1})$
	\end{tabular}
	\right.
	$
& 
&\\
\hline
$q=6$
& 
& 
& \\
\hline
\end{longtable}
\end{small}
\end{center}

$\overline{\Delta_1}=coker\left[(\Delta_1|2c)\rightarrow (2a|2a+2,2c|2c)+(\overline{2a+1,2c-1}|2a+2|2c)\right]$


\subsubsection{Boundary cohomology. Highest weight $(2a+1,2a+1,2c,2c)$.  Case $B4$.}
Again the boundary cohomology degenerates at the $E_2$-page. Therefore, up to semisimplicity we have
$H^n_\partial(GL_4(\Z),V_\lambda)=\bigoplus_{p+q=n}E_2^{p,q}$.
More systematically, we have

Let $H^q=H^q_\partial(GL_4(\Z),V_\lambda)$ be the cohomology of $GL_4(\Z)$ with coefficients in the highest weight representation with weight 
$\lambda=(2a+1,2a+1,2c,2c)$ written as a character of the split maximal torus. Then,
\[H^q
=
\left\{
\begin{tabular}{lll}
	$0$
		& $q=2$\\
	$\left\{
	\begin{tabular}{lll}
	$E_2^{0,3}=(\overline{2a + 1,2c - 1}|2a + 2,2c)$\\
	$E_2^{1,2}=(2a|2a + 2, 2c|2c)$
	\end{tabular}
	\right\}$\\
		& $q=3$\\	
	$0$
		& $q=4$\\
	$E_2^{0,5}=
	\left\{
	\begin{tabular}{ll}
	$(2a,2c-2|2c|2a+4)$\\
	$(2a,2c - 2|\overline{2a + 3,2c + 1})$
	\end{tabular}
	\right\}
	$
		& $q=5$\\ 
	$0$
		& $q=6$
\end{tabular}
\right.
\]

\begin{thm}
\[H^q
=
\left\{
\begin{tabular}{lll}
	$S_{2a-2c+4}\otimes S_{2a-2c+4}
	+
	S_{2a-2c+4}
	$	
			& $q=3, 5$\\	
	$0$
		& $q\neq 3, 5$
\end{tabular}
\right.
\]

\end{thm}




\subsection{Cohomology of $GL_4(\Z)$. Highest weight $(2a,2a,2c-1,2c-1)$.  Case $B4'$}


\subsubsection{Cohomology of the parabolic subgroups. Highest weight $(2a,2a,2c-1,2c-1)$. Case $B4'$}
For that case the weights are $(2a,2a,2c-1,2c-1)$.
We have to find the elements $w$ of the Weyl group that give a non-trivial cohomology of the minimal parabolic subgroup. The first entry has to go to the first or the third place in order to be even. Similarly, second entry should go to the second or fourth place. Also, the third entry should go to the second or the fourth and the four elements have to go the first and the third elements. 
Thus, all permutations are 
$1243$,
$1342$,
$4213$ and
$4312$.

$P_0$: 
$\begin{tabular}{lllll}
$w$ 		& $l$	& 	& $w(\lambda+\rho)-\rho$\\
$1243$ 	& $1$&	& $(2a|2a|2c-2|2c)$\\
$1342$	& $2$&	& $(2a|2c-2|2c-2|2a+2)$\\
$4213$	& $4$&	& $(2c-4|2a|2a+2|2c)$\\
$4312$	& $5$&	& $(2c-4|2c-2|2a+2|2a+2)$\\
\end{tabular}
$

For the intermediate parabolic subgroup $P_{12}$ we have to consider only the permutations that have $1$ or $4$ at the third place and $2$ or $3$ at the fourth place. Together with that the first and the second entry should be in increasing order.
For the last two elements of the permutation, we have $43$,  $42$, $13$ and $12$. Since the first two elements of the permutation have to be in increasing order, the permutations are 
$1243$,
$1342$,
$2413$ and
$3412$.
For them we have the following weights and lengths.

$P_{12}$: 
$\begin{tabular}{lllll}
$w$ 		& $l$	& 	& $w(\lambda+\rho)-\rho$\\
$1243$ 	& $1$&	& $(2a,2a|2c-2|2c)$\\
$1342$	& $2$&	& $(2a,2c-2|2c-2|2a+2)$\\
$2413$	& $3$&	& $(2a-1,2c-3|2a+2|2c)$\\
$3412$	& $4$&	& $(2c-3,2c-3|2a+2|2a+2)$\\
\end{tabular}
$

For the intermediate parabolic subgroup $P_{23}$ we have to consider only the permutations that have $2$ or $3$ at the fourth place and $1$ or $4$ at the first place. Together with that the second and the third entry should be in increasing order.
The permutations are 
$1243$,
$1342$,
$4123$ and
$4132$.
For them we have the following weights and lengths.

$P_{23}$: 
$\begin{tabular}{lllll}
$w$ 		& $l$	& 	& $w(\lambda+\rho)-\rho$\\
$1243$	& $1$&	& $(2a|2a,2c-2|2c)$\\
$1342$	& $2$&	& $(2a|2c-2,2c-2|2a+2)$\\
$4123$	& $3$&	& $(2c-4|2a+1, 2a+1|2c)$\\
$4132$	& $4$&	& $(2c-4|2a+1,2c-1|2a+2)$\\
\end{tabular}
$

For the intermediate parabolic subgroup $P_{34}$ we have to consider only the permutations that have $1$ or $4$ at the first place and $2$ or $3$ at the second place. Together with that the third and the fourth entry should be in increasing order.
For the first two elements of the permutation, we have 
$12$, $13$, $42$ and $43$. 
Since the last two elements of the permutation have to be in increasing order, the permutations are 
$1234$,
$1324$,
$4213$ and
$4312$.
For them we have the following weights and lengths.

$P_{34}$: 
$\begin{tabular}{lllll}
$w$ 		& $l$	& 	& $w(\lambda+\rho)-\rho$\\
$1234$	& $0$&	& $(2a|2a|2c-1,2c-1)$\\
$1324$	& $1$&	& $(2a|2c-2|2a+1,2c-1)$\\
$4213$	& $4$&	& $(2c-4|2a|2a+2,2c)$\\
$4312$	& $5$&	& $(2c-4|2c-2|2a+2,2a+2)$\\
\end{tabular}
$

For the maximal parabolic subgroup $P_{13}$ we have to consider only the permutations that have $2$ or $3$ at the fourth place. Together with that the first, the second  and the third entry should be in increasing order.
For the last element of the permutation, we have $1$ or $3$ Since the first three elements of the permutation have to be in increasing order, the permutations are$1243$, $2341$.
For them we have the following weights and lengths.

$P_{13}$: 
$\begin{tabular}{lllll}
$w$ 		& $l$	& 	& $w(\lambda+\rho)-\rho$\\
$1243$ 	& $1$&	& $(2a,2a,2c-2|2c)$\\
$1342$ 	& $2$&	& $(2a,2c-2,2c-2|2a+2)$\\
\end{tabular}
$

For the maximal parabolic subgroup $P_{12,34}$ we have to consider only elements that have the same parity in the first and the second entry. Together with that the first and the second, entry should be in increasing order. Also the third and the fourth entry should be in increasing order. Also the first and the second entry should have opposite parity. All the possibilities for the first two entries are: $12$, $13$, $24$, $34$. The corresponding permutations are $(1234)$, $(1423)$, $(2314)$ and $(3412)$

$P_{12,34}$: 
$\begin{tabular}{lllll}
$w$ 		& $l$	& 	& $w(\lambda+\rho)-\rho$\\
$1234$ 	& $0$&	& $(2a,2a|2c-1,2c-1)$\\
$1324$ 	& $1$&	& $(2a, 2c-2|2a+1,2c-1)$\\
$2413$ 	& $3$&	& $(2a-1,2c-3|2a+2,2c)$\\
$3412$	& $4$&	& $(2c-3,2c-3|2a+2,2a+2)$\\
\end{tabular}
$

For the maximal parabolic subgroup $P_{24}$ we have to consider only the permutations that have $1$ or $4$ at the first place. Together with that the second, the third and the fourth entry should be in increasing order.
For the first element of the permutation, we have $1$ or $4$ Since the first three elements of the permutation have to be in increasing order, the permutations are  $1234$, $4123$
For them we have the following weights and lengths.

$P_{24}$: 
$\begin{tabular}{lllll}
$w$ 		& $l$	& 	& $w(\lambda+\rho)-\rho$\\
$1234$ 	& $0$&	& $(2a|2a,2c-1,2c-1)$\\
$4123$ 	& $3$&	& $(2c-4|2a+1,2a+1,2c)$\\
\end{tabular}
$


\subsubsection{$E_1$-page. Highest weight $(2a,2a,2c-1,2c-1)$.  Case $B4'$}

{\begin{center}
\begin{small}
\scriptsize\renewcommand{\arraystretch}{2.2}
\begin{longtable}{|c|c|c|c|c|c|c|c|c|}
\hline
$E_1^{0,0}=0$ 
& $E_1^{1,0}=0$ & 
$E_1^{2,0}=0$
\\
\hline
$E_1^{0,1}=0$
& $E_1^{1,1}=(2a, 2a|2c - 2|2c)$
& $E_1^{2,1}=(2a|2a|2c - 2|2c)$
\\
\hline
$E_1^{0,2}=(2a|H^2(2a,2c-1,2c-1))$
& $E_1^{1,2}=
	\left\{
	\begin{tabular}{ll}
	$(2a|2a, 2c - 2|2c)$\\
	$(2a|2c-2,2c-2|2a+2)$\\
	$(2a|2c-2|2a+1,2c-1)$
	\end{tabular}
	\right.$ 
& $E_1^{2,2}=(2a|2c - 2|2c - 2|2a + 2)$
\\
\hline
$E_1^{0,3}=
	\left\{
	\begin{tabular}{ll}
	$(2a,2c-2|2a+1,2c-1)$\\
	$(2a|H^3(2a,2c-1,2c-1))$
	\end{tabular}
	\right.
	$
& $E_1^{1,3}=(2a,2c-2|2c-2|2a+2)$
& $E_1^{2,3}=0$\\
\hline
$E_1^{0,4}=(H^3(2a,2a,2c-2)|2c)$
& $E_1^{1,4}=(2a-1,2c-3|2a+2|2c)$
& $E_1^{2,4}=(2c-4|2a|2a+2|2c)$\\
\hline
$E_1^{0,5}=\left\{
	\begin{tabular}{ll}
	$(H^3(2a,2c-2,2c-2)|2a+2)$\\
	$(2a-1,2c-3|2a+2,2c)$\\
	$(2c-4|H^2(2a+1,2a+1,2c))$
	\end{tabular}
	\right.$
& $E_1^{1,5}=\left\{
	\begin{tabular}{ll}
	$(2c-4|2a+1,2c-1|2a+2)$\\
	$(2c - 4|2a|2a + 2, 2c)$\\
	$(2c-4|2c-2|2a+2,2a+2)$
	\end{tabular}
	\right.$
& $E_1^{2,5}=(2c-4|2c-2|2a+2|2a+2)$\\
\hline
$E_1^{0,6}=(2c-4|H^3(2a+1,2a+1,2c))$
& $E_1^{1,6}=0$
& $E_1^{2,6}=0$\\
\hline
\end{longtable}
\end{small}
\end{center}


\subsubsection{Certain cohomology groups of $GL_3(\Z)$ that appear on the $E_1$-page in the case $B4'$.}
From the computations of the cohomology of $GL_3(\Z)$ from the previous section, we have
{\begin{center}
\scriptsize\renewcommand{\arraystretch}{2.2}
\begin{longtable}{|c|c|c|}
\hline
$H^2(2a, 2c - 1, 2c - 1)
	=
	\Delta_0+\Delta_1$
&$H^3(GL_3(\Z),(2a, 2c - 1, 2c - 1))=0$\\
\hline
$H^2(2a, 2a, 2c - 2)=0$
& $H^3(2a, 2a, 2c - 2)=(\overline{2a-1,2c-3}|2a+2)$\\
\hline
$H^2(2a,2c - 2,2c - 2)=0$
& $H^3(2a,2c - 2,2c - 2)=(2c-4|\overline{2a+1,2c-1})$\\
\hline
$H^2(2a+1,2a+1,2c)=\Delta'_0+\Delta'_1$
&$H^3(2a+1,2a+1,2c)=0$\\
\hline
\end{longtable}
\end{center}
where
$\Delta_0\subset (2a|2c-2|2c)+(2c-2|2c-2|2a+2)$, 
$\Delta_1\subset (2a,2c-2|2c)+(2c-2|\overline{2a+1,2c-1})$, 
and
$\Delta'_0\subset (2c-2|2a+2|2a+2)+(2a|2a+2|2c)$
$\Delta'_1\subset (\overline{2a+1,2c-1}|2a+2)+(2a|2a+2,2c)$


\subsubsection{$E_2$ page. Highest weight $(2a,2a,2c-1,2c+1)$.  Case $B4'$}

For $q=1$, we have an isomorphism $ E_1^{1,1}\rightarrow E_1^{2,1}$. Therefore, $E_2^{p,2}=0$ for $p=0,1,2$.

For $q=2$, we have a surjective maps $(2a|2c - 2|2a + 1,2c - 1)\rightarrow (2a|2c-2|2c-2|2a+2)$.
We also have an injective map $(2a|\Delta_1)\rightarrow (2a|2a, 2c - 2|2c) (2a|2c - 2|2a + 1,2c -1)$ and an isomorphism
$(2a|\Delta_0)\rightarrow (2a|2c - 2,2c - 2|2a + 2)$
Therefore $E_2^{0,2}=E_2^{2,2}=0$. Also, 
\[E_2^{1,2}=\overline{\Delta}_1=coker\left[(2a|\Delta_1)\rightarrow (2a|2a, 2c - 2|2c) (2a|2c - 2|\overline{2a + 1,2c -1})\right]\] 
is isomorphic to 
$(2a|2a, 2c - 2|2c)$

For $q=3$, we have a surjective map
$(2a,2c - 2|2a + 1,2c - 1)
\rightarrow 
(2a,2c - 2|2c - 2|2a + 2)$. Also, $H^3(2a, 2c -1, 2c - 1)=0$
Therefore, $E_2^{0,3}=(2a,2c - 2|\overline{2a + 1,2c - 1})$.

For $q=4$, we have a short exact sequence
$0\rightarrow E_1^{0,4}\rightarrow E_1^{1,4}\rightarrow E_1^{2,4}\rightarrow 0$.
Therefore, $E_2^{p,4}=0$ for $p=0,1,2$.

For $q=5$, we have a surjective maps
$(2a - 1,2c - 3|2a + 2,2c)\rightarrow
(2c - 4|2a|2a + 2, 2c)$
and
$(2c-4|H^2 (2a+1,2a+1,2c))\rightarrow (2c-4|2c-2|2a+2,2a+2)$, with kernel $(2c-4|\Delta'_1)$.
We also have a short exact sequence
\begin{eqnarray*}
&0\rightarrow (H^3(2a,2c -2,2c - 2)|2a + 2)rightarrow\\
&\rightarrow (2c - 4|2a + 1, 2c - 1|2a + 2)\rightarrow\\
&\rightarrow (2c-4|2c-2|2a+2|2a+2)\rightarrow 0
\end{eqnarray*}
Therefore,
$E_2^{0,5}=(\overline{2a - 1,2c - 3}|2a + 2,2c)+(2c-4|\Delta'_1)$ and $E_2^{p,5}=0$ for $p=1,2$.
Note that $(2c-4|\Delta_1)$, which is isomorphic to $(2c - 4|2a|2a + 2, 2c))$

For $q=6$, we have a surjective map
$(2c-3,2c-3|2a+2,2a+2)
\rightarrow
(2c-4|2c-2|2a+2,2a+2)$
Therefore
$E_2^{0,6}=(\overline{2c-3,2c-3}|2a+2,2a+2)+(2c - 4|2a|2a + 2, 2c)$
and $E_2^{p,6}=0$ for $p=1,2$.

{\begin{center}
\begin{small}
\scriptsize\renewcommand{\arraystretch}{2.2}
\begin{longtable}{|c|c|c|c|c|c|c|c|c|}
\hline
& $p=0$
& $p=1$
& $p=2$\\
\hline
$q=0$
&  
&  
& \\
\hline
$q=1$
&
& 
& 
\\
\hline
$q=2$
&  
&  $E_2^{1,2}=(2a|2a, 2c - 2|2c)$
& 
\\
\hline
$q=3$
& $E_2^{0,3}=(2a,2c - 2|\overline{2a + 1,2c - 1})$
& 
&\\
\hline
$q=4$
&
& 
& \\
\hline
$q=5$
& $E_2^{0,5}=\left\{
	\begin{tabular}{lll}
	$(\overline{2a - 1,2c - 3}|2a + 2,2c)$\\
	$(2c - 4|2a|2a + 2, 2c)$
	\end{tabular}
	\right.$
& 
&\\
\hline
$q=6$
&
& 
& \\
\hline
\end{longtable}
\end{small}
\end{center}


\subsubsection{Boundary cohomology. Highest weight $(2a,2a,2c-1,2c-1)$.  Case $B4'$.}
Again the boundary cohomology degenerates at the $E_2$-page. Therefore, up to semisimplicity we have
$H^n_\partial(GL_4(\Z),V_\lambda)=\bigoplus_{p+q=n}E_2^{p,q}$.
More systematically, we have

Let $H^q=H^q_\partial(GL_4(\Z),V_\lambda)$ be the cohomology of $GL_4(\Z)$ with coefficients in the highest weight representation with weight 
$\lambda=(2a,2a,2c-1,2c-1)$ written as a character of the split maximal torus. Then,
\[H^q
=
\left\{
\begin{tabular}{lll}
	$0$
		& $q=2$\\
	$\left\{\begin{tabular}{lll}
	$E_2^{0,3}=(2a,2c - 2|\overline{2a + 1,2c - 1})$\\
	$E_2^{1,2}=(2a|2a, 2c - 2|2c)$
	
	\end{tabular}
	\right\}$
		& $q=3$\\	
$0$
		& $q=4$\\
$E_2^{0,5}=\left\{
	\begin{tabular}{lll}
	$(\overline{2a - 1,2c - 3}|2a + 2,2c)$\\
	$(2c - 4|2a|2a + 2, 2c)$
	\end{tabular}
	\right\}$
		& $q=5$\\ 
$0$
		& $q=6$
\end{tabular}
\right.
\]

Since the above $H^4_\partial(B4')$ contains $E_2^{p,q}$ with $p>0$, we have that $H^4_\partial(B4')$ contains a potentially ghost class, namely,
\[pGh^3(GL_4(\Z),(2a,2a,2c-1,2c-1))=E_2^{1,2}=(2a|2a, 2c - 2|2c).\]

\begin{thm}
\[H^q
=
\left\{
\begin{tabular}{lll}
	$S_{2a-2c+4}\otimes S_{2a-2c+4}
	+
	S_{2a-2c+4}
	$	
			& $q=3, 5$\\	
	$0$
		& $q\neq 3, 5$
\end{tabular}
\right.
\]

\end{thm}




\subsection{Cohomology of $GL_4(\Z)$. Highest weight $(2a+1,2b,2c,2d-1)$.  Case $C1$}


\subsubsection{Cohomology of the parabolic subgroups. Highest weight $(2a+1,2b,2c,2d-1)$. Case $C1$}
Let $\lambda=(2a+1,2b,2c,2d-1)$.

For the minimal parabolic subgroup $P_0$ we have to consider only the permutations that send $1$ and $2$ to second or fourth place, $3$ and $4$ to first or third place
All possible such permutations are: 
$3142$,
$3241$,
$4132$,
$4231$. 
For them we have the following weights and lengths.

$P_0$: 
$\begin{tabular}{lllll}
$w$ 		& $l$	& 	& $w(\lambda+\rho)-\rho$\\
$3142$ 	& $3$&	& $(2c-2|2a+2|2d-2|2b+2)$\\
$3241$	& $4$&	& $(2c-2|2b|2d-2|2a+4)$\\
$4132$	& $4$&	& $(2d-4|2a+2|2c|2b+2)$\\
$4231$	& $5$&	& $(2d-4|2b|2c|2a+4)$\\
\end{tabular}
$

For the intermediate parabolic subgroup $P_{12}$ we have to consider only the permutations that have $3$ or $4$ at the third place and $1$ or $2$ at the fourth place. Together with that the first and the second entry should be in increasing order.
For the last two elements of the permutation, we have $42$,  $41$, $32$ and $31$. Since the first two elements of the permutation have to be in increasing order, the permutations are 
$1342$,
$2341$,
$1432$,
$2431$.

For them we have the following weights and lengths.

$P_{12}$: 
$\begin{tabular}{lllll}
$w$ 		& $l$	& 	& $w(\lambda+\rho)-\rho$\\
$1342$,	& $2$&	& $(2a+1,2c-1|2d-2|2b+2)$\\
$2341$,	& $3$&	& $(2b-1,2c-1|2d-2|2a+4)$\\
$1432$,	& $3$&	& $(2a+1,2d-3|2c|2b+2)$\\
$2431$.	& $4$&	& $(2b-1,2d-3|2c|2a+4)$\\
\end{tabular}
$

For the intermediate parabolic subgroup $P_{23}$ we have to consider only the permutations that have $3$ or $4$ at the first place and $1$ or $2$ at the fourth place. Together with that the second and the third entry should be in increasing order.
The permutations are
$3142$,
$3241$,
$4132$,
$4231$.
For them we have the following weights and lengths.

$P_{23}$: 
$\begin{tabular}{lllll}
$w$ 		& $l$	& 	& $w(\lambda+\rho)-\rho$\\
$3142$	& $3$&	& $(2c-2|2a+2,2d-2|2b+2)$\\
$3241$	& $4$&	& $(2c-2|2b,2d-2|2a+4)$\\
$4132$	& $4$&	& $(2d-4|2a+2, 2c|2b+2)$\\
$4231$	& $5$&	& $(2d-4|2b,2c|2a+4)$\\
\end{tabular}
$

For the intermediate parabolic subgroup $P_{34}$ we have to consider only the permutations that have $3$ or $4$ at the first place and $1$ or $2$ at the second place. Together with that the third and the fourth entry should be in increasing order.
For the first two elements of the permutation, we have 
$31$,
$32$,
$41$,
$42$.
Since the last two elements of the permutation have to be in increasing order, the permutations are 
$3124$,
$3214$,
$4123$,
$4213$.
For them we have the following weights and lengths.

$P_{34}$: 
$\begin{tabular}{lllll}
$w$ 		& $l$	& 	& $w(\lambda+\rho)-\rho$\\
$3124$	& $2$&	& $(2c-2|2a+2|2b+1,2d-1)$\\
$3214$	& $3$&	& $(2c-2|2b|2a+3,2d-1)$\\
$4123$	& $3$&	& $(2d-4|2a+2|2b+1,2c+1)$\\
$4213$	& $4$&	& $(2d-4|2b|2a+3,2c+1)$\\
\end{tabular}
$

For the maximal parabolic subgroup $P_{13}$ we have to consider only the permutations that have $1$ or $2$ at the fourth place. Together with that the first, the second  and the third entry should be in increasing order.
For the last element of the permutation, we have $1$ or $4$ Since the first three elements of the permutation have to be in increasing order, the permutations are 
$1342$,
$2341$.

For them we have the following weights and lengths.

$P_{13}$: 
$\begin{tabular}{lllll}
$w$ 		& $l$	& 	& $w(\lambda+\rho)-\rho$\\
$1342$ 	& $2$&	& $(2a+1,2c-1,2d-2|2b+2)$\\
$2341$ 	& $3$&	& $(2b-1,2c-1,2d-2|2a+4)$\\
\end{tabular}
$

For the parabolic subgroup $P_{12,34}$, we have to pick permutations $w$ so that $w(\lambda+\rho)-\rho$ has the same parity for the first two elements. All possibilities for the first two elements of the permutation are $odd,odd$ which can be achieved with $13$, $14$, $23$ and $24$. The corresponding permutations are 
$1324$,
$1423$,
$2314$,
$2413$.
The first two element could be $even,even$. It can be achieved with 
$32$
$32$
$41$
$41$
However, they should be in increasing order, in order to be a highest weight for $GL_2$. They are not. Therefore we should discard them.
$P_{12,34}$: 
$\begin{tabular}{lllll}
$w$ 		& $l$	& 	& $w(\lambda+\rho)-\rho$\\
$1324$ 	& $1$&	& $(2a+1,2c-1|2b+1,2d-1)$\\
$1423$ 	& $2$&	& $(2a+1, 2d-3|2b+1,2c+1)$\\
$2314$	& $2$&	& $(2b-1,2c-1|2a+3,2d-1)$\\
$2413$	& $3$&	& $(2b-1,2d-3|2a+3,2c+1)$\\
\end{tabular}
$

For the maximal parabolic subgroup $P_{24}$ we have to consider only the permutations that have $3$ or $4$ at the first place. Together with that the second, the third and the fourth entry should be in increasing order.
Since the first three elements of the permutation have to be in increasing order, the permutations are
$3124$,
$4123$.
For them we have the following weights and lengths.

$P_{24}$: 
$\begin{tabular}{lllll}
$w$ 		& $l$	& 	& $w(\lambda+\rho)-\rho$\\
$3124$ 	& $2$&	& $(2c-2|2a+2,2b+1,2d-1)$\\
$4123$ 	& $3$&	& $(2d-4|2a+2,2b+1,2c+1)$\\
\end{tabular}
$


\subsubsection{$E_1$-page. Highest weight $(2a+1,2b,2c,2d-1)$.  Case $C1$.}

{\begin{center}
\begin{small}
\scriptsize\renewcommand{\arraystretch}{2.2}
\begin{longtable}{|c|c|c|c|c|c|c|c|c|}
\hline
$E_1^{0,0}=0$ 
& $E_1^{1,0}=0$ & 
$E_1^{2,0}=0$
\\
\hline
$E_1^{0,1}=0$
& $E_1^{1,1}=0$
& $E_1^{2,1}=0$
\\
\hline
$E_1^{0,2}=0$	
& $E_1^{1,2}=0$ 
& $E_1^{2,2}=0$
\\
\hline
$E_1^{0,3}=(2a+1,2c-1|2b+1,2d-1)$
& $E_1^{1,3}=
	\left\{
	\begin{tabular}{ll}
	$(2a+1,2c-1|2d-2|2b+2)$\\
	 $(2c-2|2a+2|2b+1,2d-1)$
	\end{tabular}
	\right.
	$
& $E_1^{2,3}=(2c-2|2a+2|2d-2|2b+2)$\\
\hline
$E_1^{0,4}=
	\left\{
	\begin{tabular}{ll}
	$(H^2(2a+1,2c-1,2d-2)|2b+2)$\\
	$(2a+1, 2d-3|2b+1,2c+1)$\\\
	$(2b-1,2c-1|2a+3,2d-1)$\\
	$(2c-2|H^2(2a+2,2b+1,2d-1))$
	\end{tabular}
	\right.$
& $E_1^{1,4}=
	\left\{
	\begin{tabular}{ll}
	$(2b-1,2c-1|2d-2|2a+4)$\\
	$(2a+1,2d-3|2c|2b+2)$\\
	$(2c-2|2a+2,2d-2|2b+2)$\\
	$(2c-2|2b|2a+3,2d-1)$\\
	$(2d-4|2a+2|2b+1,2c+1)$
	\end{tabular}
	\right.$
& $E_1^{2,4}=
	\left\{
	\begin{tabular}{ll}
	$(2c-2|2b|2d-2|2a+4)$\\
	$(2d-4|2a+2|2c|2b+2)$
	\end{tabular}
	\right.$\\
\hline
$E_1^{0,5}=\left\{
	\begin{tabular}{ll}
	$(H^3(2a+1,2c-1,2d-2)|2b+2)$\\
	$(H^2(2b-1,2c-1,2d-2)|2a+4)$\\
	$(2b-1,2d-3|2a+3,2c+1)$\\
	$(2c-2|H^3(2a+2,2b+1,2d-1))$\\
	$(2d-4|H^2(2a+2,2b+1,2c+1))$
	\end{tabular}
	\right.$
& $E_1^{1,5}=\left\{
	\begin{tabular}{ll}
	$(2b-1,2d-3|2c|2a+4)$\\
	$(2c-2|2b,2d-2|2a+4)$\\
	$(2d-4|2a+2, 2c|2b+2)$\\
	$(2d-4|2b|2a+3,2c+1)$
	\end{tabular}
	\right.$
& $E_1^{2,5}=(2d-4|2b|2c|2a+4)$
\\
\hline
$E_1^{0,6}=
	\left\{
	\begin{tabular}{ll}
	$(H^3(2b-1,2c-1,2d-2)|2a+4)$\\
	$(2d-4|H^3(2a+2,2b+1,2c+1))$
	\end{tabular}
	\right.$
& $E_1^{1,6}=(2d-4|2b,2c|2a+4)$
& $E_1^{1,6}=0$\\
\hline
\end{longtable}
\end{small}
\end{center}


\subsubsection{Certain cohomology groups of $GL_3(\Z)$ that appear on the $E_1$-page in the case $C1$.}
From the computations of the cohomology of $GL_3(\Z)$ from the previous section, we have
{\begin{center}
\scriptsize\renewcommand{\arraystretch}{2.2}
\begin{longtable}{|c|c|c|}
\hline

$H^2(2a+1,2c-1,2d-2)
	=
	\Delta_1$
&$H^3(2a+1,2c-1,2d-2)=(2d-4|2a+2,2c)$\\
\hline

$(H^2(2b-1,2c-1,2d-2)|=
	\Delta_2$
& $(H^3(2b-1,2c-1,2d-2)|=(2d-4|2b,2c))$\\
\hline

$H^2(2a+2,2b+1,2d-1)=
	\Delta_3$
& $H^3(2a+2,2b+1,2d-1)=(2b,2d-2|2a+4)$\\
\hline

$H^2(2a+2,2b+1,2c+1)=
	\Delta_4$
&$H^3(2a+2,2b+1,2c+1)=(2b,2c|2a+4)$\\
\hline
\end{longtable}
\end{center}
where

$\Delta_1\subset (\overline{2a+1,2d-3}|2c)+(2c-2|2a+2,2d-2)$,

$\Delta_2\subset (\overline{2b-1,2d-3}|2c)+(2c-2|2b,2d-2)$,

$\Delta_3\subset (2a+2,2d-2|2b+2)+(2b,\overline{2a+3,2d-1})$,

$\Delta_4\subset (2a+2|2c,2b+2)+(2b|\overline{2a+3,2c+1})$.

\subsubsection{$E_2$ page. Highest weight $(2a,2a,2c-1,2c+1)$.  Case $C1$}

For $q=3$, we have 
\begin{eqnarray*}
&&E_1^{0,3}=(2a+1,2c-1|2b+1,2d-1)=\\
&&(\overline{2a+1,2c-1}|\overline{2b+1,2d-1})+(\overline{2a+1,2c-1}|2d-2|2b+2)+(2c-2|2a+2|\overline{|2b+1,2d-1})+(2c-2|2a+2|2d-2|2b+2)
\end{eqnarray*}
Also
\begin{eqnarray*}
&&E_1^{1,3}=(2a+1,2c-1|2b+1,2d-1)=\\
&&+(\overline{2a+1,2c-1}|2d-2|2b+2)+(2c-2|2a+2|2d-2|2b+2)+(2c-2|2a+2|\overline{|2b+1,2d-1})+(2c-2|2a+2|2d-2|2b+2)
\end{eqnarray*}
Therefore,
$E_2^{0,3}=(\overline{2a+1,2c-1}|\overline{2b+1,2d-1})$
and
$E_2^{p,3}=0$ for $p=1,2$.

For $q=4$, similarly to the case $q=3$, we have the two sequences
\[(2a+1, 2d-3|2b+1,2c+1)\rightarrow (2a + 1,2d - 3|2c|2b + 2)+(2d-4|2a+2|2b+1,2c+1)\rightarrow (2d - 4|2a + 2|2c|2b + 2)\]
and
\[(2b-1,2c-1|2a+3,2d-1)\rightarrow (2b-1,2c-1|2d-2|2a+4) + (2c - 2|2b|2a + 3,2d -1) \rightarrow (2c - 2|2b|2d - 2|2a + 4).\]
Both of them have homology concentrated at $p=0$, which is $(\overline{2a+1, 2d-3}|\overline{2b+1,2c+1})$ and $(\overline{2b-1,2c-1}|\overline{2a+3,2d-1})$, respectively.
Also $(H^2(2a+1,2c-1,2d-2)|2b+2)\rightarrow (2c-2|2a+2,2d-2|2b+2)$ and
$(2c - 2|H^2(2a + 2,2b + 1,2d - 1))\rightarrow (2c-2|2a+2,2d-2|2b+2)$ are isomorphisms since $\Delta_1$ and $\Delta_4$ are diagonal embeddings. (we actually need that at one of them is.) Then the kernel of
\[(H^2(2a+1,2c-1,2d-2)|2b+2)+(2c -2|H^2(2a + 2,2b + 1,2d - 1))\rightarrow (2c-2|2a+2,2d-2|2b+2)\] 
is $(2c-2|2a+2,2d-2|2b+2)$.
Therefore, 
\[E_2^{04}=(\overline{2a+1, 2d-3}|\overline{2b+1,2c+1}) + (\overline{2b-1,2c-1}|\overline{2a+3,2d-1})+(2c-2|2a+2,2d-2|2b+2),\]
and 
$E_2^{p,4}=0$ for $p=1,2$.

For $q=5$, we have 
isomorphism 
$(H^3(2a + 1,2c - 1,2d - 2)|2b + 2)\rightarrow (2d - 4|2a + 2,2c|2b + 2)$
and
$(2c-2|H^3(2a+2,2b+1,2d-1))\rightarrow (2c - 2|2b,2d - 2|2a + 4)$.
Similarly to $q=3$ we have a sequence
\begin{eqnarray*}
&&(2b-1,2d-3|2a+3,2c+1)\rightarrow (2b - 1,2d - 3|2c|2a + 4) + (2d - 4|2b|2a + 3,2c + 1)\\
&&\rightarrow (2d - 4|2b|2c|2a + 4),
\end{eqnarray*}
whose homology is concentrated at $p=0$ and equal to $(\overline{2b-1,2d-3}|\overline{2a+3,2c+1})$.
Therefore,

\begin{eqnarray*}
&&E_2^{0,5}=(H^2(2b - 1,2c - 1,2d - 2)|2a + 4)+(\overline{2b-1,2d-3}|\overline{2a+3,2c+1})+(2d - 4|H^2(2a + 2,2b + 1,2c + 1))=\\
&&=\Delta_2+(\overline{2b-1,2d-3}|\overline{2a+3,2c+1})+\Delta_3=\\
&& =(2b - 1,2d - 3|2c|2a + 4)+(\overline{2b-1,2d-3}|\overline{2a+3,2c+1})+(2d - 4|2b|2a + 3,2c + 1)
 \end{eqnarray*}

For $q=6$, we have the isomorphisms
$(H^3(2b-1,2c-1,2d-2)|2a+4)\rightarrow (2d - 4|2b, 2c|2a + 4)$
and
$(2d - 4|H3(2a + 2,2b + 1,2c + 1))\rightarrow (2d - 4|2b, 2c|2a + 4)$
Therefore
$E_2^{0,6}=(2d - 4|2b, 2c|2a + 4)$
and $E_2^{p,6}=0$ for $p=1,2$.

{\begin{center}
\begin{small}
\scriptsize\renewcommand{\arraystretch}{2.2}
\begin{longtable}{|c|c|c|c|c|c|c|c|c|}
\hline
& $p=0$
& $p=1$
& $p=2$\\
\hline
$q=0$
&  
&  
& \\
\hline
$q=1$
&
& 
& 
\\
\hline
$q=2$
&  
& 
& 
\\
\hline
$q=3$
& $E_2^{0,3}=(\overline{2a+1,2c-1}|\overline{2b+1,2d-1})$
& 
&\\
\hline
$q=4$
& $E_2^{04}=\left\{
	\begin{tabular}{lll}
	$(\overline{2a+1, 2d-3}|\overline{2b+1,2c+1})$\\
	$(\overline{2b-1,2c-1}|\overline{2a+3,2d-1})$\\
	$(2c-2|2a+2,2d-2|2b+2)$
	\end{tabular}
	\right.$
& 
& \\
\hline
$q=5$
& $E_2^{0,5}=\left\{
	\begin{tabular}{lll}
	$(2b - 1,2d - 3|2c|2a + 4)$\\
	$(\overline{2b-1,2d-3}|\overline{2a+3,2c+1})$\\
	$(2d - 4|2b|2a + 3,2c + 1)$
	\end{tabular}
	\right.$
& 
&\\
\hline
$q=6$
& $E_2^{0,6}=(2d - 4|2b, 2c|2a + 4)$
& 
& \\
\hline
\end{longtable}
\end{small}
\end{center}


\subsubsection{Boundary cohomology. Highest weight $(2a,2a,2c-1,2c-1)$.  Case $C1$.}
Again the boundary cohomology degenerates at the $E_2$-page. Therefore, up to semisimplicity we have
$H^n_\partial(GL_4(\Z),V_\lambda)=\bigoplus_{p+q=n}E_2^{p,q}$.
More systematically, we have

Let $H^q=H^q_\partial(GL_4(\Z),V_\lambda)$ be the cohomology of $GL_4(\Z)$ with coefficients in the highest weight representation with weight 
$\lambda=(2a,2a,2c-1,2c-1)$ written as a character of the split maximal torus. Then,
\[H^q
=
\left\{
\begin{tabular}{lll}
	$0$
		& $q=2$\\
	$E_2^{0,3}=(\overline{2a+1,2c-1}|\overline{2b+1,2d-1})$
		& $q=3$\\	
$E_2^{04}=\left\{
	\begin{tabular}{lll}
	$(\overline{2a+1, 2d-3}|\overline{2b+1,2c+1})$\\
	$(\overline{2b-1,2c-1}|\overline{2a+3,2d-1})$\\
	$(2c-2|2a+2,2d-2|2b+2)$
	\end{tabular}
	\right\}$
		& $q=4$\\
$E_2^{0,5}=\left\{
	\begin{tabular}{lll}
	$(2b - 1,2d - 3|2c|2a + 4)$\\
	$(\overline{2b-1,2d-3}|\overline{2a+3,2c+1})$\\
	$(2d - 4|2b|2a + 3,2c + 1)$
	\end{tabular}
	\right\}$
		& $q=5$\\ 
$E_2^{0,6}=(2d - 4|2b, 2c|2a + 4)$
		& $q=6$
\end{tabular}
\right.
\]

\begin{thm}
\[H^q
=
\left\{
\begin{tabular}{lll}
	$0$
		& $q=2$\\
	$S_{2a-2c+4}\otimes S_{2b-2d+4}$
		& $q=3$\\
	$\C^2\otimes S_{2a-2d+6}\otimes S_{2b-2c+2}
	+
	S_{2a-2d+6}
	$
		& $q=4$\\
	$S_{2b-2d+4}
	+
	S_{2b-2d+4}\otimes S_{2a-2c+4}
	+
	S_{2a-2c+4}
	$
		& $q=5$\\
	$S_{2b-2c+2}$
		& $q=6$
\end{tabular}
\right.
\]

\end{thm}




\subsection{Cohomology of $GL_4(\Z)$. Highest weight $(2a+2,2b+1,2c+1,2d)$.  Case $C1'$}


\subsubsection{Cohomology of the parabolic subgroups. Highest weight $(2a+2,2b+1,2c+1,2d)$. Case $C1'$}
For that case the weights are $(2a+2,2b+1,2c+1,2d)$.
We have to find the elements $w$ of the Weyl group that give a non-trivial cohomology of the minimal parabolic subgroup. 
The first and the second entries have to go to the  first or the third  place in order to be even. Similarly, third  and the fourth entries should go to the second or fourth place.
Thus, all permutations are 
$1324$,
$1423$,
$2314$,
$2413$.

$P_0$: 
$\begin{tabular}{lllll}
$w$ 		& $l$	& 	& $w(\lambda+\rho)-\rho$\\
$1324$	& $1$&	& $(2a+2|2c|2b+2|2d)$\\
$1423$ 	& $2$&	& $(2a+2|2d-2|2b+2|2c+2)$\\
$2314$	& $2$&	& $(2b|2c|2a+4|2d)$\\
$2413$	& $3$&	& $(2b|2d-2|2a+4|2c+2)$\\
\end{tabular}
$

For the intermediate parabolic subgroup $P_{12}$ we have to consider only the permutations that have $1$ or $2$ at the third place and $3$ or $4$ at the fourth place. Together with that the first and the second entry should be in increasing order.
For the last two elements of the permutation, we have 
$24$,
$23$,
$14$,
$13$.
Since the first two elements of the permutation have to be in increasing order, the permutations are 
$1324$,
$1423$,
$2314$,
$2413$.
For them we have the following weights and lengths.

$P_{12}$: 
$\begin{tabular}{lllll}
$w$ 		& $l$	& 	& $w(\lambda+\rho)-\rho$\\
$1324$ 	& $1$&	& $(2a+2,2c|2b+2|2d)$\\
$1423$ 	& $2$&	& $(2a+2,2d-2|2b+2|2c+2)$\\
$2314$	& $2$&	& $(2b,2c|2a+4|2d)$\\
$2413$	& $3$&	& $(2b,2d-2|2a+4|2c+2)$\\
\end{tabular}
$

For the intermediate parabolic subgroup $P_{23}$ we have to consider only the permutations that have $1$ or $2$ at the first place and $3$ or $4$ at the fourth place. Together with that the second and the third entry should be in increasing order.
The permutations are
$1234$,
$1243$,
$2134$,
$2143$.
For them we have the following weights and lengths.

$P_{23}$: 
$\begin{tabular}{lllll}
$w$ 		& $l$	& 	& $w(\lambda+\rho)-\rho$\\
$1234$ 	& $0$&	& $(2a+2|2b+1,2c+1|2d)$\\
$1243$	& $1$&	& $(2a+2|2b+1,2d-1|2c+2)$\\
$2134$	& $1$&	& $(2b|2a+3, 2c+1|2d)$\\
$2143$	& $2$&	& $(2b|2a+3,2d-1|2c+2)$\\
\end{tabular}
$

For the intermediate parabolic subgroup $P_{34}$ we have to consider only the permutations that have $1$ or $2$ at the first place and $3$ or $4$ at the second place. Together with that the third and the fourth entry should be in increasing order.
For the first two elements of the permutation, we have 
$13$,
$14$,
$23$,
$24$.
Since the last two elements of the permutation have to be in increasing order, the permutations are 
$1324$,
$1423$,
$2314$,
$2413$.
For them we have the following weights and lengths.

$P_{34}$: 
$\begin{tabular}{lllll}
$w$ 		& $l$	& 	& $w(\lambda+\rho)-\rho$\\
$1324$ 	& $1$&	& $(2a+2|2c|b+2,2d)$\\
$1423$	& $2$&	& $(2a+2|2d-2|2b+2,2c+2)$\\
$2314$	& $2$&	& $(2b|2c|2a+4,2d)$\\
$2413$ 	& $3$&	& $(2b|2d-2|2a+4,2c+2)$\\
\end{tabular}
$

For the maximal parabolic subgroup $P_{13}$ we have to consider only the permutations that have $3$ or $4$ at the fourth place. Together with that the first, the second  and the third entry should be in increasing order.
For the last element of the permutation, we have $1$ or $3$ Since the first three elements of the permutation have to be in increasing order, the permutations are $1234$ and $1243$.
For them we have the following weights and lengths.

$P_{13}$: 
$\begin{tabular}{lllll}
$w$ 		& $l$	& 	& $w(\lambda+\rho)-\rho$\\
$1234$ 	& $0$&	& $(2a+2,2b+1,2c+1|2d)$\\
$1243$ 	& $1$&	& $(2a+2,2b+1,2d-1|2c+2)$\\
\end{tabular}
$

For the maximal parabolic subgroup $P_{12,34}$ we have to consider only the permutations such that the first and the second elements of the permutation should be in increasing order. If the permutation of $(2a+2,2b+1,2c+1,2d)$ is $(even,even$ hen the possibilities for the first entry are $1$ and $2$ and for the second one are $3$ and $4$.  
Thus, the possibilities for the first two elements of the permutation are
$13$,
$14$,
$23$,
$24$.

Also the third and the fourth entry should be in increasing order. Also the first and the second entry should have opposite parity. The corresponding permutations are 
$1324$,
$1423$,
$2314$,
$2413$.

$P_{12,34}$: 
$\begin{tabular}{lllll}
$w$ 		& $l$	& 	& $w(\lambda+\rho)-\rho$\\
$1324$	& $1$&	& $(2a+2,2c|2b+2,2d)$\\
$1423$	& $2$&	& $(2a+2, 2d-2|2b+2,2c+2)$\\
$2314$	& $2$&	& $(2b,2c|2a+4,2d)$\\
$2413$	& $3$&	& $(2b,2d-2|2a+4,2c+2)$\\
\end{tabular}
$

For the maximal parabolic subgroup $P_{24}$ we have to consider only the permutations that have $1$ or $2$ at the first place. Together with that the second, the third and the fourth entry should be in increasing order.
For the first element of the permutation, we have $2$ or $4$ Since the first three elements of the permutation have to be in increasing order, the permutations are  
$1234$
and
$2134$
For them we have the following weights and lengths.

$P_{24}$: 
$\begin{tabular}{lllll}
$w$ 		& $l$	& 	& $w(\lambda+\rho)-\rho$\\
$1234$ 	& $0$&	& $(2a+2|2b+1,2c+1,2d)$\\
$2134$ 	& $1$&	& $(2b|2a+3,2c+1,2d)$\\
\end{tabular}
$


\subsubsection{$E_1$-page. Highest weight $(2a+2,2b+1,2c+1,2d)$.  Case $C1'$}

{\begin{center}
\begin{small}
\scriptsize\renewcommand{\arraystretch}{2.2}
\begin{longtable}{|c|c|c|c|c|c|c|c|c|}
\hline
$E_1^{0,0}=0$ 
& $E_1^{1,0}=0$ & 
$E_1^{2,0}=0$
\\
\hline
$E_1^{0,1}=0$
& $E_1^{1,1}=(2a+2|2b+1,2c+1|2d)$
& $E_1^{2,1}=(2a + 2|2c|2b + 2|2d)$
\\
\hline
$E_1^{0,2}=\left\{
	\begin{tabular}{ll}
	$(H^2(2a+2,2b+1,2c+1)|2d)$\\
	$(2a+2|H^2(2b+1,2c+1,2d))$
	\end{tabular}
	\right.$ 
& $E_1^{1,2}=
	\left\{
	\begin{tabular}{ll}
	$(2a + 2, 2c|2b + 2|2d)$\\
	$(2a+2|2b+1,2d-1|2c+2)$\\
	$(2b|2a + 3, 2c + 1|2d)$\\
	$(2a + 2|2c|b + 2, 2d)$
	\end{tabular}
	\right.$ 
& $E_1^{2,2}=
	\left\{
	\begin{tabular}{ll}
	$(2a+2|2d-2|2b+2|2c+2)$\\
	$(2b|2c|2a+4|2d)$
	\end{tabular}
	\right.$
\\
\hline
$E_1^{0,3}=
	\left\{
	\begin{tabular}{ll}
	$(H^3(2a+2,2b+1,2c+1)|2d)$\\
	$(H^2(2a+2,2b+1,2d-1)|2c+2)$\\
	$(2a+2,2c|2b+2,2d)$\\
	$(2a+2|H^3(2b+1,2c+1,2d))$\\
	$(2b|H^2(2a+3,2c+1,2d))$
	\end{tabular}
	\right.
	$
& $E_1^{1,3}=
	\left\{
	\begin{tabular}{ll}
	$(2a+2,2d-2|2b+2|2c+2)$\\
	$(2b, 2c|2a + 4|2d)$\\
	$(2b|2a+3,2d-1|2c+2)$\\
	$(2a+2|2d-2|2b+2,2c+2)$\\
	$(2b|2c|2a + 4, 2d)$
	\end{tabular}
	\right.
	$
& $E_1^{2,3}=(2b|2d - 2|2a + 4|2c + 2)$
	\\
\hline
$E_1^{0,4}=
	\left\{
	\begin{tabular}{ll}
	$(H^3(2a+2,2b+1,2d-1)|2c+2)$\\
	$(2a+2,2d-2|2b+2,2c+2)$\\
	$(2b, 2c|2a + 4, 2d)$\\
	$(2b|H^3(2a+3,2c+1,2d))$
	\end{tabular}
	\right.$
& $E_1^{1,4}=
	\left\{
	\begin{tabular}{ll}
	$(2b,2d-2|2a+4|2c+2)$\\
	$(2b|2d-2|2a+4,2c+2)$
	\end{tabular}
	\right.$
& $E_1^{2,4}=0$\\
\hline
$E_1^{0,5}=(2b,2d-2|2a+4,2c+2)$
& $E_1^{1,5}=0$
& $E_1^{2,5}=0$\\
\hline
$E_1^{0,6}=0$
& $E_1^{1,6}=0$
& $E_1^{2,6}=0$\\
\hline
\end{longtable}
\end{small}
\end{center}


\subsubsection{Certain cohomology groups of $GL_3(\Z)$ that appear on the $E_1$-page in the case $C1'$.}
From the computations of the cohomology of $GL_3(\Z)$ from the previous section, we have
{\begin{center}
\scriptsize\renewcommand{\arraystretch}{2.2}
\begin{longtable}{|c|c|c|}
\hline
$H^2(2a+2,2b+1,2c+1)
	=
	\Delta_1
	$
&$H^3(2a+2,2b+1,2c+1)
	=
	(2b,2c|2a+4)$\\
\hline
$H^2(2a+2,2b+1,2d-1)=
	\Delta_2
	$
& $H^3(2a+2,2b+1,2d-1)
	=
	(2b,2d-2|2a+4)$
\\
\hline
$H^2(2b+1,2c+1,2d)
	=
	\Delta_3
	$
& $H^3(2b+1,2c+1,2d)
	=
	(2d-2|2b+2,2c+2)$\\
\hline
$H^2(2a+3,2c+1,2d)
	=
	\Delta_4
	$
& $H^3(2a+3,2c+1,2d)
	=
	(2d-2|2a+4,2c+2)$\\
	\hline
\end{longtable}
\end{center}
where
$\Delta_1\subset (2+2,2c|2b+2)+(2b|\overline{2a+3,2c+1})$,
$\Delta_2\subset (2a+2,2d-2|2b+2)+(2b|\overline{2a+3,2d-1})$,
$\Delta_3\subset (2c|2b+2,2d)+(\overline{2b+1,2d-1}|2c+2)$,
$\Delta_4\subset (2c|2a+4,2d)+(\overline{2a+3,2d-1}|2c+2)$.


\subsubsection{$E_2$ page. Highest weight $(2a+2,2b+1,2c+1,2d)$.  Case $C1'$}

For $q=1$, we have a surjective map
$E_1^{1,1}=(2a+2|2b+1,2c+1|2d)\rightarrow E_1^{2,1=(2a+2|2c|2b+2|2d)}$
Its kernel is 
$E_2^{1,1}=(2a+2|\overline{2b+1,2c+1}|2d)$.

For $q=2$, we have a direct sum of two cochains
\begin{eqnarray*}
&&(H^2(2a+2,2b+1,2c+1)|2d)\rightarrow (2a + 2, 2c|2b + 2|2d) + (2b|2a + 3, 2c + 1|2d)\rightarrow \\
&&\rightarrow (2b|2c|2a + 4|2d)
\end{eqnarray*}
and
\begin{eqnarray*}
&&(2a + 2|H^2(2b + 1,2c + 1,2d))\rightarrow (2a + 2|2b + 1, 2d - 1|2c + 2) + (2a + 2|2c|b + 2, 2d) \rightarrow\\
&&\rightarrow (2a+2|2d-2|2b+2|2c+2).
\end{eqnarray*}
They have homology centered at $p=1$ which is isomorphic to $(2a + 2, 2c|2b + 2|2d) $ and  to $(2a + 2|2c|b + 2, 2d) $.
Therefore,
$E_2^{1,2}=(2a + 2, 2c|2b + 2|2d) + (2a + 2|2c|b + 2, 2d) $
and
$E_2^{p,2}=0$ for $p=0$ and $p=2$.

For $q=3$, we have a surjective map
$(2b|2a + 3,2d - 1|2c + 2)\rightarrow (2b|2d-2|2a+4|2c+2)$. Therefore, $E_2^{2,3}=0$.
We also have the isomorphisms
$(H^3(2a + 2,2b + 1,2c + 1)|2d)\rightarrow (2b, 2c|2a + 4|2d)$
and
$(2a+2|H^3(2b+1,2c+1,2d))\rightarrow (2a+2|2d-2|2b+2,2c+2)$
We also have the following embddings
$(H^2(2a + 2,2b + 1,2d - 1)|2c + 2)\rightarrow (2a+2,2d- |2b+2|2c+2) + (2b|\overline{2a + 3,2d - 1}|2c + 2)$
and
$(2b|H^2(2a + 3, 2c + 1, 2d))\rightarrow (2b|\overline{2a + 3,2d - 1}|2c + 2) + (2b|2c|2a + 4, 2d)$. However, there is a repetition of $(2b|2a + 3,2d - 1|2c + 2)$.
 \begin{eqnarray*}
 &&(H^2(2a + 2,2b + 1,2d - 1)|2c + 2)+(2b|H^2(2a + 3, 2c + 1, 2d)) \rightarrow\\
 &&(2a+2,2d- |2b+2|2c+2) + (2b|\overline{2a + 3,2d - 1}|2c + 2) + (2b|2c|2a + 4, 2d)
 \end{eqnarray*}
 We obtain that the
 homology is isomorphic $(2b|\overline{2a + 3,2d - 1}|2c + 2)$.
Therefore,
$E_2^{1.3}=(2b|2a + 3,2d - 1|2c + 2)$
and $E_2^{0,3}=(2a+2,2c|2b+2,2d)$.

For $q=4$, we have the isomorphisms
$(H^3(2a+2,2b+1,2d-1)|2c+2)\rightarrow  (2b,2d - 2|2a + 4|2c + 2)$
and
$(2b|H^3(2a + 3, 2c + 1, 2d))\rightarrow (2b|2d - 2|2a + 4,2c + 2)$
Therefore,
$E_2^{0,4}=(2a + 2, 2d - 2|2b + 2, 2c + 2) + (2b, 2c|2a + 4, 2d)$ and
$E_2^{p,4}=0$ for $p=1,2$.

For $q=5$, we have $E_2^{0,5}=E_1^{0,5}=E0,5 = (2b,2d-2|2a+4,2c+2)$
{\begin{center}
\begin{small}
\scriptsize\renewcommand{\arraystretch}{2.2}
\begin{longtable}{|c|c|c|c|c|c|c|c|c|}
\hline
& $p=0$
& $p=1$
& $p=2$\\
\hline
$q=0$
&  
&  
& \\
\hline
$q=1$
&
& $E_2^{1,1}=(2a+2|\overline{2b+1,2c+1}|2d)$
&
\\
\hline
$q=2$
& 
&  $E_2^{1,2}=\left\{
	\begin{tabular}{lll}
	$(2a + 2, 2c|2b + 2|2d)$\\
	$(2a + 2|2c|b + 2, 2d)$
	\end{tabular}
	\right.$
& 
\\
\hline
$q=3$
& $E_2^{0,3}=(2a+2,2c|2b+2,2d)$
& $E_2^{1.3}=(2b|\overline{2a + 3,2d - 1}|2c + 2)$
&\\
\hline
$q=4$
& $E_2^{0,4}=
	\left\{
	\begin{tabular}{lll}
	$(2a + 2, 2d - 2|2b + 2, 2c + 2)$\\
	$(2b, 2c|2a + 4, 2d)$ 
	\end{tabular}
	\right.$
& 
& \\
\hline
$q=5$
& $E_2^{0,5}=(2b,2d-2|2a+4,2c+2)$
& 
&\\
\hline
$q=6$
&
& 
& \\
\hline
\end{longtable}
\end{small}
\end{center}


\subsubsection{Boundary cohomology. Highest weight $(2a+2,2b+1,2c+1,2d)$.  Case $C1'$.}
Again the boundary cohomology degenerates at the $E_2$-page. Therefore, up to semisimplicity we have
$H^n_\partial(GL_4(\Z),V_\lambda)=\bigoplus_{p+q=n}E_2^{p,q}$.
More systematically, we have

Let $H^q=H^q_\partial(GL_4(\Z),V_\lambda)$ be the cohomology of $GL_4(\Z)$ with coefficients in the highest weight representation with weight 
$\lambda=(2a+2,2b+1,2c+1,2d)$ written as a character of the split maximal torus. Then,
\[H^q
=
\left\{
\begin{tabular}{lll}
$E_2^{1,1}=(2a+2|\overline{2b+1,2c+1}|2d)$
		& $q=2$\\
$\left\{\begin{tabular}{ll}
$E_2^{0,3}=(2a+2,2c|2b+2,2d)$\\
$E_2^{1,2}=\left\{
	\begin{tabular}{lll}
	$(2a + 2, 2c|2b + 2|2d)$\\
	$(2a + 2|2c|b + 2, 2d)$
	\end{tabular}
	\right\}$
\end{tabular}\right\}$
		& $q=3$\\	
$\left\{\begin{tabular}{ll}
$E_2^{0,4}=
	\left\{
	\begin{tabular}{lll}
	$(2a + 2, 2d - 2|2b + 2, 2c + 2)$\\
	$(2b, 2c|2a + 4, 2d)$ 
	\end{tabular}
	\right\}$\\
 $E_2^{1.3}=(2b|\overline{2a + 3,2d - 1}|2c + 2)$
 \end{tabular}\right\}$
		& $q=4$\\
 $E_2^{0,5}=(2b,2d-2|2a+4,2c+2)$
		& $q=5$\\ 
$0$
		& $q=6$
\end{tabular}
\right.
\]

Since the above $H^4_\partial(C1')$ contains $E_2^{p,q}$ with $p>0$, we have that $H^4_\partial(C1')$ contains a potentially ghost class, namely,
\[pGh^4(GL_4(\Z),(2a+2,2b+1,2c+1,2d))=E_2^{1,3}=(2b|2a+2,2d|2c+2).\]

This is important not for the cohomology of $GL_4(\Z)$ but for the cohomology of $GL_5(\Z)$. In a similar way as the potentially ghost classes in $GL_3(\Z)$ have importance for the cohomology of $GL(\Z)$. The next representation of $GL_4$, namely $(2a,2b,2c,2c)$, or case $A2$, exhibits exactly the use of potentially ghost classes in $GL_3(\Z)$ and their effect on the cohomology of $GL_4(\Z)$.

\begin{thm}
\[H^q
=
\left\{
\begin{tabular}{lll}
$S_{2b-2c+2}$
		& $q=2$\\
$S_{2a-2c+4}\otimes S_{2b-2d+4}
+
S_{2a-2c+4}
+
S_{2b-2d+4}
$
		& $q=3$\\
$\C^2\otimes S_{2a-2d+6}\otimes S_{2b-2c+2}
+
S_{2a-2d+6}
$
		& $q=4$\\
$S_{2b-2d+4}\otimes S_{2a-2c+4}$
		& $q=5$\\
$0$
		& $q=6$
\end{tabular}
\right.
\]
\end{thm}




\subsection{Cohomology of $GL_4(\Z)$. Highest weight $(2a+1,2b,2b,2d-1)$.  Case $C2$}


\subsubsection{Cohomology of the parabolic subgroups. Highest weight $(2a+1,2b,2b,2d-1)$. Case $C2$}
Let $\lambda=(2a+1,2b,2b,2d-1)$.

For the minimal parabolic subgroup $P_0$ we have to consider only the permutations that send $1$ and $2$ to second or fourth place, $3$ and $4$ to first or third place
All possible such permutations are: 
$3142$,
$3241$,
$4132$,
$4231$. 
For them we have the following weights and lengths.

$P_0$: 
$\begin{tabular}{lllll}
$w$ 		& $l$	& 	& $w(\lambda+\rho)-\rho$\\
$3142$ 	& $3$&	& $(2b-2|2a+2|2d-2|2b+2)$\\
$3241$	& $4$&	& $(2b-2|2b|2d-2|2a+4)$\\
$4132$	& $4$&	& $(2d-4|2a+2|2b|2b+2)$\\
$4231$	& $5$&	& $(2d-4|2b|2b|2a+4)$\\
\end{tabular}
$

For the intermediate parabolic subgroup $P_{12}$ we have to consider only the permutations that have $3$ or $4$ at the third place and $1$ or $2$ at the fourth place. Together with that the first and the second entry should be in increasing order.
For the last two elements of the permutation, we have $42$,  $41$, $32$ and $31$. Since the first two elements of the permutation have to be in increasing order, the permutations are 
$1342$,
$2341$,
$1432$,
$2431$.

For them we have the following weights and lengths.

$P_{12}$: 
$\begin{tabular}{lllll}
$w$ 		& $l$	& 	& $w(\lambda+\rho)-\rho$\\
$1342$,	& $2$&	& $(2a+1,2b-1|2d-2|2b+2)$\\
$2341$,	& $3$&	& $(2b-1,2b-1|2d-2|2a+4)$\\
$1432$,	& $3$&	& $(2a+1,2d-3|2b|2b+2)$\\
$2431$.	& $4$&	& $(2b-1,2d-3|2b|2a+4)$\\
\end{tabular}
$

For the intermediate parabolic subgroup $P_{23}$ we have to consider only the permutations that have $3$ or $4$ at the first place and $1$ or $2$ at the fourth place. Together with that the second and the third entry should be in increasing order.
The permutations are
$3142$,
$3241$,
$4132$,
$4231$.
For them we have the following weights and lengths.

$P_{23}$: 
$\begin{tabular}{lllll}
$w$ 		& $l$	& 	& $w(\lambda+\rho)-\rho$\\
$3142$	& $3$&	& $(2b-2|2a+2,2d-2|2b+2)$\\
$3241$	& $4$&	& $(2b-2|2b,2d-2|2a+4)$\\
$4132$	& $4$&	& $(2d-4|2a+2, 2b|2b+2)$\\
$4231$	& $5$&	& $(2d-4|2b,2b|2a+4)$\\
\end{tabular}
$

For the intermediate parabolic subgroup $P_{34}$ we have to consider only the permutations that have $3$ or $4$ at the first place and $1$ or $2$ at the second place. Together with that the third and the fourth entry should be in increasing order.
For the first two elements of the permutation, we have 
$31$,
$32$,
$41$,
$42$.
Since the last two elements of the permutation have to be in increasing order, the permutations are 
$3124$,
$3214$,
$4123$,
$4213$.
For them we have the following weights and lengths.

$P_{34}$: 
$\begin{tabular}{lllll}
$w$ 		& $l$	& 	& $w(\lambda+\rho)-\rho$\\
$3124$	& $2$&	& $(2b-2|2a+2|2b+1,2d-1)$\\
$3214$	& $3$&	& $(2b-2|2b|2a+3,2d-1)$\\
$4123$	& $3$&	& $(2d-4|2a+2|2b+1,2b+1)$\\
$4213$	& $4$&	& $(2d-4|2b|2a+3,2b+1)$\\
\end{tabular}
$

For the maximal parabolic subgroup $P_{13}$ we have to consider only the permutations that have $1$ or $2$ at the fourth place. Together with that the first, the second  and the third entry should be in increasing order.
For the last element of the permutation, we have $1$ or $4$ Since the first three elements of the permutation have to be in increasing order, the permutations are 
$1342$,
$2341$.

For them we have the following weights and lengths.

$P_{13}$: 
$\begin{tabular}{lllll}
$w$ 		& $l$	& 	& $w(\lambda+\rho)-\rho$\\
$1342$ 	& $2$&	& $(2a+1,2b-1,2d-2|2b+2)$\\
$2341$ 	& $3$&	& $(2b-1,2b-1,2d-2|2a+4)$\\
\end{tabular}
$

For the parabolic subgroup $P_{12,34}$, we have to pick permutations $w$ so that $w(\lambda+\rho)-\rho$ has the same parity for the first two elements. All possibilities for the first two elements of the permutation are $odd,odd$ which can be achieved with $13$, $14$, $23$ and $24$. The corresponding permutations are 
$1324$,
$1423$,
$2314$,
$2413$.
The first two element could be $even,even$. It can be achieved with 
$32$
$32$
$41$
$41$
However, they should be in increasing order, in order to be a highest weight for $GL_2$. They are not. Therefore we should discard them.
$P_{12,34}$: 
$\begin{tabular}{lllll}
$w$ 		& $l$	& 	& $w(\lambda+\rho)-\rho$\\
$1324$ 	& $1$&	& $(2a+1,2b-1|2b+1,2d-1)$\\
$1423$ 	& $2$&	& $(2a+1, 2d-3|2b+1,2b+1)$\\
$2314$	& $2$&	& $(2b-1,2b-1|2a+3,2d-1)$\\
$2413$	& $3$&	& $(2b-1,2d-3|2a+3,2b+1)$\\
\end{tabular}
$

For the maximal parabolic subgroup $P_{24}$ we have to consider only the permutations that have $3$ or $4$ at the first place. Together with that the second, the third and the fourth entry should be in increasing order.
Since the first three elements of the permutation have to be in increasing order, the permutations are
$3124$,
$4123$.
For them we have the following weights and lengths.

$P_{24}$: 
$\begin{tabular}{lllll}
$w$ 		& $l$	& 	& $w(\lambda+\rho)-\rho$\\
$3124$ 	& $2$&	& $(2b-2|2a+2,2b+1,2d-1)$\\
$4123$ 	& $3$&	& $(2d-4|2a+2,2b+1,2b+1)$\\
\end{tabular}
$


\subsubsection{$E_1$-page. Highest weight $(2a+1,2b,2b,2d-1)$.  Case $C2$.}

{\begin{center}
\begin{small}
\scriptsize\renewcommand{\arraystretch}{2.2}
\begin{longtable}{|c|c|c|c|c|c|c|c|c|}
\hline
$E_1^{0,0}=0$ 
& $E_1^{1,0}=0$ & 
$E_1^{2,0}=0$
\\
\hline
$E_1^{0,1}=0$
& $E_1^{1,1}=0$
& $E_1^{2,1}=0$
\\
\hline
$E_1^{0,2}=0$	
& $E_1^{1,2}=0$ 
& $E_1^{2,2}=0$
\\
\hline
$E_1^{0,3}=(2a+1,2b-1|2b+1,2d-1)$
& $E_1^{1,3}=
	\left\{
	\begin{tabular}{ll}
	$(2a+1,2b-1|2d-2|2b+2)$\\
	 $(2b-2|2a+2|2b+1,2d-1)$
	\end{tabular}
	\right.
	$
& $E_1^{2,3}=(2b-2|2a+2|2d-2|2b+2)$\\
\hline
$E_1^{0,4}=
	\left\{
	\begin{tabular}{ll}
	$(H^2(2a+1,2b-1,2d-2)|2b+2)$\\
	$(2b-2|H^2(2a+2,2b+1,2d-1))$
	\end{tabular}
	\right.$
& $E_1^{1,4}=
	\left\{
	\begin{tabular}{ll}
	$(2a+1,2d-3|2b|2b+2)$\\
	$(2b-2|2a+2,2d-2|2b+2)$\\
	$(2b-2|2b|2a+3,2d-1)$
	\end{tabular}
	\right.$
& $E_1^{2,4}=
	\left\{
	\begin{tabular}{ll}
	$(2b-2|2b|2d-2|2a+4)$\\
	$(2d-4|2a+2|2b|2b+2)$
	\end{tabular}
	\right.$\\
\hline
$E_1^{0,5}=\left\{
	\begin{tabular}{ll}
	$(H^3(2a+1,2b-1,2d-2)|2b+2)$\\
	$(H^2(2b-1,2b-1,2d-2)|2a+4)$\\
	$(2b-1,2d-3|2a+3,2b+1)$\\
	$(2b-2|H^3(2a+2,2b+1,2d-1))$\\
	$(2d-4|H^2(2a+2,2b+1,2b+1))$
	\end{tabular}
	\right.$
& $E_1^{1,5}=\left\{
	\begin{tabular}{ll}
	$(2b-1,2d-3|2b|2a+4)$\\
	$(2b-2|2b,2d-2|2a+4)$\\
	$(2d-4|2a+2, 2b|2b+2)$\\
	$(2d-4|2b,2b|2a+4)$\\
	$(2d-4|2b|2a+3,2b+1)$
	\end{tabular}
	\right.$
& $E_1^{2,5}=(2d-4|2b|2b|2a+4)$
\\
\hline
$E_1^{0,6}=
	\left\{
	\begin{tabular}{ll}
	$(H^3(2b-1,2b-1,2d-2)|2a+4)$\\
	$(2d-4|H^3(2a+2,2b+1,2b+1))$
	\end{tabular}
	\right.$
& $E_1^{1,6}=0$
& $E_1^{1,6}=0$\\
\hline
\end{longtable}
\end{small}
\end{center}


\subsubsection{Certain cohomology groups of $GL_3(\Z)$ that appear on the $E_1$-page in the case $C2$.}
From the computations of the cohomology of $GL_3(\Z)$ from the previous section, we have
{\begin{center}
\scriptsize\renewcommand{\arraystretch}{2.2}
\begin{longtable}{|c|c|c|}
\hline

$H^2(2a+1,2b-1,2d-2)
	=
	\Delta_1$
&$H^3(2a+1,2b-1,2d-2)=(2d-4|2a+2,2b)$\\
\hline

$(H^2(2b-1,2b-1,2d-2)|=
	\Delta_2+\Delta_0$
& $(H^3(2b-1,2b-1,2d-2)=0$\\
\hline

$H^2(2a+2,2b+1,2d-1)=
	\Delta_3$
& $H^3(2a+2,2b+1,2d-1)=(2b,2d-2|2a+4)$\\
\hline

$H^2(2a+2,2b+1,2b+1)=
	\Delta_4+\Delta'_0$
&$H^3(2a+2,2b+1,2b+1)=0$\\
\hline
\end{longtable}
\end{center}
where

$\Delta_1\subset (\overline{2a+1,2d-3}|2b)+(2b-2|2a+2,2d-2)$,

$\Delta_2\subset (\overline{2b-1,2d-3}|2b)+(2b-2|2b,2d-2)$,

$\Delta_0\subset (2b-2|2b|2d-2)+(2d-4|2b|2b)$,

$\Delta_3\subset (2a+2,2d-2|2b+2)+(2b,\overline{2a+3,2d-1})$,

$\Delta_4\subset (2a+2|2b,2b+2)+(2b|\overline{2a+3,2b+1})$

and

$\Delta'_0\subset (2a+2|2b|2b+2)+(2b|2b|2a+4)$.


\subsubsection{$E_2$ page. Highest weight $(2a,2a,2b-1,2b+1)$.  Case $C2$}

For $q=3$, we have 
\begin{eqnarray*}
&&E_1^{0,3}=(2a+1,2b-1|2b+1,2d-1)=\\
&&(\overline{2a+1,2b-1}|\overline{2b+1,2d-1})+(\overline{2a+1,2b-1}|2d-2|2b+2)+\\
&&+(2b-2|2a+2|\overline{|2b+1,2d-1})+(2b-2|2a+2|2d-2|2b+2)
\end{eqnarray*}
Also
\begin{eqnarray*}
&&E_1^{1,3}=(2a+1,2b-1|2b+1,2d-1)=\\
&&+(\overline{2a+1,2b-1}|2d-2|2b+2)+(2b-2|2a+2|2d-2|2b+2)+\\
&&+(2b-2|2a+2|\overline{|2b+1,2d-1})+(2b-2|2a+2|2d-2|2b+2)
\end{eqnarray*}
Therefore,
$E_2^{0,3}=(\overline{2a+1,2b-1}|\overline{2b+1,2d-1})$
and
$E_2^{p,3}=0$ for $p=1,2$.

For $q=4$, we have the two embeddings
 \[(H^2(2a+1,2b-1,2d-2)|2b+2)\rightarrow  (2a + 1,2d - 3|2b|2b + 2) + (2b-2|2a+2,2d-2|2b+2)\] and
\[(2b - 2|H^2(2a + 2,2b + 1,2d - 1))\rightarrow (2b-2|2a+2,2d-2|2b+2) + (2b - 2|2b|2a + 3,2d - 1).\] When we combine them we obtain the embedding
\begin{eqnarray*}
&&(H^2(2a+1,2b-1,2d-2)|2b+2) + (2b - 2|H^2(2a + 2,2b + 1,2d - 1)) \rightarrow  \\
&&\rightarrow (2a + 1,2d - 3|2b|2b + 2) +(2b-2|2a+2,2d-2|2b+2) + (2b - 2|2b|2a + 3,2d - 1).
\end{eqnarray*}
Its kernel is zero and its cokernel is
$(2b-2|2a+2,2d-2|2b+2)$
is $(2b-2|2a+2,2d-2|2b+2)$.
Therefore, 
$E_2^{1,4}=(2b-2|2a+2,2d-2|2b+2)$
and 
$E_2^{p,4}=0$ for $p=0,2$.

For $q=5$, we have 
isomorphism 
$(H^3(2a + 1,2b - 1,2d - 2)|2b + 2)\rightarrow (2d - 4|2a + 2,2b|2b + 2)$
and
$(2b-2|H^3(2a+2,2b+1,2d-1))\rightarrow (2b - 2|2b,2d - 2|2a + 4)$.
Similarly to $q=3$ we have a sequence
\begin{eqnarray*}
&&(2b-1,2d-3|2a+3,2b+1)\rightarrow (2b - 1,2d - 3|2b|2a + 4) + (2d - 4|2b|2a + 3,2b + 1)\\
&&\rightarrow (2d - 4|2b|2b|2a + 4),
\end{eqnarray*}
We also have a surjective map
$(\Delta_0|2a+4)+(2d-4|\Delta'_0)\rightarrow \rightarrow (2d - 4|2b|2b|2a + 4)$.
Its kernel is $(2d - 4|2b|2b|2a + 4)$ positioned at $p=0$.

Therefore,
\begin{eqnarray*}
&&E_2^{0,5}=\\
&&=(\Delta_2|2a+4)+(2d-4|\Delta_4)+\\
&&+(\overline{2b-1,2d-3}|\overline{2a+3,2b+1})+(2d - 4|2b|2b|2a + 4)=\\
&& =(\overline{2b - 1,2d - 3}|2b|2a + 4)+(2d - 4|2b|\overline{2a + 3,2b + 1})+\\
&&+(\overline{2b-1,2d-3}|\overline{2a+3,2b+1})+(2d - 4|2b|2b|2a + 4).
\end{eqnarray*}

For $q=6$, we have the vanishing $H^3(2b-1,2b-1,2d-2)=0$. 
Therefore, $E_2^{p,6}=0$ for $p=0,1,2$.

{\begin{center}
\begin{small}
\scriptsize\renewcommand{\arraystretch}{2.2}
\begin{longtable}{|c|c|c|c|c|c|c|c|c|}
\hline
& $p=0$
& $p=1$
& $p=2$\\
\hline
$q=0$
&  
&  
& \\
\hline
$q=1$
&
& 
& 
\\
\hline
$q=2$
&  
& 
& 
\\
\hline
$q=3$
& $E_2^{0,3}=(\overline{2a+1,2b-1}|\overline{2b+1,2d-1})$
& 
&\\
\hline
$q=4$
&
& $E_2^{1,4}=(2b-2|2a+2,2d-2|2b+2)$
& \\
\hline
$q=5$
& $E_2^{0,5}=\left\{
	\begin{tabular}{lll}
	$(\overline{2b - 1,2d - 3}|2b|2a + 4)$\\
	$(\overline{2b-1,2d-3}|\overline{2a+3,2b+1})$\\
	$(2d - 4|2b|\overline{2a + 3,2b + 1})$\\
	$(2d - 4|2b|2b|2a + 4)$
	\end{tabular}
	\right.$
& 
&\\
\hline
$q=6$
& 
& 
& \\
\hline
\end{longtable}
\end{small}
\end{center}


\subsubsection{Boundary cohomology. Highest weight $(2a,2a,2b-1,2b-1)$.  Case $C2$.}
Again the boundary cohomology degenerates at the $E_2$-page. Therefore, up to semisimplicity we have
$H^n_\partial(GL_4(\Z),V_\lambda)=\bigoplus_{p+q=n}E_2^{p,q}$.
More systematically, we have

Let $H^q=H^q_\partial(GL_4(\Z),V_\lambda)$ be the cohomology of $GL_4(\Z)$ with coefficients in the highest weight representation with weight 
$\lambda=(2a,2a,2b-1,2b-1)$ written as a character of the split maximal torus. Then,
\[H^q
=
\left\{
\begin{tabular}{lll}
	$0$
		& $q=2$\\
	$E_2^{0,3}=(\overline{2a+1,2b-1}|\overline{2b+1,2d-1})$
		& $q=3$\\	
$0$
		& $q=4$\\

$\left\{
\begin{tabular}{lll}
$E_2^{0,5}=\left\{
	\begin{tabular}{lll}
	$(\overline{2b - 1,2d - 3}|2b|2a + 4)$\\
	$(\overline{2b-1,2d-3}|\overline{2a+3,2b+1})$\\
	$(2d - 4|2b|\overline{2a + 3,2b + 1})$\\
	$(2d - 4|2b|2b|2a + 4)$
	\end{tabular}
	\right\}$\\
$E_2^{1,4}=(2b-2|2a+2,2d-2|2b+2)$
\end{tabular}
\right\}$
	
		& $q=5$\\ 
$0$
		& $q=6$
\end{tabular}
\right.
\]

\begin{thm}
\[H^q
=
\left\{
\begin{tabular}{lll}
	$S_{2a-2b+4}\otimes S_{2b-2d+4}$
		& $q=3$\\
	$S_{2b-2d+4}\otimes S_{2a-2b+4}
	+
	S_{2b-2d+4}
	+
	S_{2a-2b+4}
	+
	S_{2a-2d+6}
	+
	\C
	$
		& $q=5$\\
$0$
		& $q\neq 3, 5$
\end{tabular}
\right.
\]

\end{thm}




\subsection{Cohomology of $GL_4(\Z)$. Highest weight $(2a+2,2b+1,2b+1,2d)$.  Case $C2'$}


\subsubsection{Cohomology of the parabolic subgroups. Highest weight $(2a+2,2b+1,2b+1,2d)$. Case $C2'$}
For that case the weights are $(2a+2,2b+1,2b+1,2d)$.
We have to find the elements $w$ of the Weyl group that give a non-trivial cohomology of the minimal parabolic subgroup. 
The first and the second entries have to go to the  first or the third  place in order to be even. Similarly, third  and the fourth entries should go to the second or fourth place.
Thus, all permutations are 
$1324$,
$1423$,
$2314$,
$2413$.

$P_0$: 
$\begin{tabular}{lllll}
$w$ 		& $l$	& 	& $w(\lambda+\rho)-\rho$\\
$1324$	& $1$&	& $(2a+2|2b|2b+2|2d)$\\
$1423$ 	& $2$&	& $(2a+2|2d-2|2b+2|2b+2)$\\
$2314$	& $2$&	& $(2b|2b|2a+4|2d)$\\
$2413$	& $3$&	& $(2b|2d-2|2a+4|2b+2)$\\
\end{tabular}
$

For the intermediate parabolic subgroup $P_{12}$ we have to consider only the permutations that have $1$ or $2$ at the third place and $3$ or $4$ at the fourth place. Together with that the first and the second entry should be in increasing order.
For the last two elements of the permutation, we have 
$24$,
$23$,
$14$,
$13$.
Since the first two elements of the permutation have to be in increasing order, the permutations are 
$1324$,
$1423$,
$2314$,
$2413$.
For them we have the following weights and lengths.

$P_{12}$: 
$\begin{tabular}{lllll}
$w$ 		& $l$	& 	& $w(\lambda+\rho)-\rho$\\
$1324$ 	& $1$&	& $(2a+2,2b|2b+2|2d)$\\
$1423$ 	& $2$&	& $(2a+2,2d-2|2b+2|2b+2)$\\
$2314$	& $2$&	& $(2b,2b|2a+4|2d)$\\
$2413$	& $3$&	& $(2b,2d-2|2a+4|2b+2)$\\
\end{tabular}
$

For the intermediate parabolic subgroup $P_{23}$ we have to consider only the permutations that have $1$ or $2$ at the first place and $3$ or $4$ at the fourth place. Together with that the second and the third entry should be in increasing order.
The permutations are
$1234$,
$1243$,
$2134$,
$2143$.
For them we have the following weights and lengths.

$P_{23}$: 
$\begin{tabular}{lllll}
$w$ 		& $l$	& 	& $w(\lambda+\rho)-\rho$\\
$1234$ 	& $0$&	& $(2a+2|2b+1,2b+1|2d)$\\
$1243$	& $1$&	& $(2a+2|2b+1,2d-1|2b+2)$\\
$2134$	& $1$&	& $(2b|2a+3, 2b+1|2d)$\\
$2143$	& $2$&	& $(2b|2a+3,2d-1|2b+2)$\\
\end{tabular}
$

For the intermediate parabolic subgroup $P_{34}$ we have to consider only the permutations that have $1$ or $2$ at the first place and $3$ or $4$ at the second place. Together with that the third and the fourth entry should be in increasing order.
For the first two elements of the permutation, we have 
$13$,
$14$,
$23$,
$24$.
Since the last two elements of the permutation have to be in increasing order, the permutations are 
$1324$,
$1423$,
$2314$,
$2413$.
For them we have the following weights and lengths.

$P_{34}$: 
$\begin{tabular}{lllll}
$w$ 		& $l$	& 	& $w(\lambda+\rho)-\rho$\\
$1324$ 	& $1$&	& $(2a+2|2b|b+2,2d)$\\
$1423$	& $2$&	& $(2a+2|2d-2|2b+2,2b+2)$\\
$2314$	& $2$&	& $(2b|2b|2a+4,2d)$\\
$2413$ 	& $3$&	& $(2b|2d-2|2a+4,2b+2)$\\
\end{tabular}
$

For the maximal parabolic subgroup $P_{13}$ we have to consider only the permutations that have $3$ or $4$ at the fourth place. Together with that the first, the second  and the third entry should be in increasing order.
For the last element of the permutation, we have $1$ or $3$ Since the first three elements of the permutation have to be in increasing order, the permutations are $1234$ and $1243$.
For them we have the following weights and lengths.

$P_{13}$: 
$\begin{tabular}{lllll}
$w$ 		& $l$	& 	& $w(\lambda+\rho)-\rho$\\
$1234$ 	& $0$&	& $(2a+2,2b+1,2b+1|2d)$\\
$1243$ 	& $1$&	& $(2a+2,2b+1,2d-1|2b+2)$\\
\end{tabular}
$

For the maximal parabolic subgroup $P_{12,34}$ we have to consider only the permutations such that the first and the second elements of the permutation should be in increasing order. If the permutation of $(2a+2,2b+1,2b+1,2d)$ is $(even,even$ hen the possibilities for the first entry are $1$ and $2$ and for the second one are $3$ and $4$.  
Thus, the possibilities for the first two elements of the permutation are
$13$,
$14$,
$23$,
$24$.

Also the third and the fourth entry should be in increasing order. Also the first and the second entry should have opposite parity. The corresponding permutations are 
$1324$,
$1423$,
$2314$,
$2413$.

$P_{12,34}$: 
$\begin{tabular}{lllll}
$w$ 		& $l$	& 	& $w(\lambda+\rho)-\rho$\\
$1324$	& $1$&	& $(2a+2,2b|2b+2,2d)$\\
$1423$	& $2$&	& $(2a+2, 2d-2|2b+2,2b+2)$\\
$2314$	& $2$&	& $(2b,2b|2a+4,2d)$\\
$2413$	& $3$&	& $(2b,2d-2|2a+4,2b+2)$\\
\end{tabular}
$

For the maximal parabolic subgroup $P_{24}$ we have to consider only the permutations that have $1$ or $2$ at the first place. Together with that the second, the third and the fourth entry should be in increasing order.
For the first element of the permutation, we have $2$ or $4$ Since the first three elements of the permutation have to be in increasing order, the permutations are  
$1234$
and
$2134$
For them we have the following weights and lengths.

$P_{24}$: 
$\begin{tabular}{lllll}
$w$ 		& $l$	& 	& $w(\lambda+\rho)-\rho$\\
$1234$ 	& $0$&	& $(2a+2|2b+1,2b+1,2d)$\\
$2134$ 	& $1$&	& $(2b|2a+3,2b+1,2d)$\\
\end{tabular}
$


\subsubsection{$E_1$-page. Highest weight $(2a+2,2b+1,2b+1,2d)$.  Case $C2'$}

{\begin{center}
\begin{small}
\scriptsize\renewcommand{\arraystretch}{2.2}
\begin{longtable}{|c|c|c|c|c|c|c|c|c|}
\hline
$E_1^{0,0}=0$ 
& $E_1^{1,0}=0$ & 
$E_1^{2,0}=0$
\\
\hline
$E_1^{0,1}=0$
& $E_1^{1,1}=0$
& $E_1^{2,1}=(2a + 2|2b|2b + 2|2d)$
\\
\hline
$E_1^{0,2}=\left\{
	\begin{tabular}{ll}
	$(H^2(2a+2,2b+1,2b+1)|2d)$\\
	$(2a+2|H^2(2b+1,2b+1,2d))$
	\end{tabular}
	\right.$ 
& $E_1^{1,2}=
	\left\{
	\begin{tabular}{ll}
	$(2a + 2, 2b|2b + 2|2d)$\\
	$(2b, 2b|2a + 4|2d)$\\
	$(2a+2|2b+1,2d-1|2b+2)$\\
	$(2b|2a + 3, 2b + 1|2d)$\\
	$(2a + 2|2b|b + 2, 2d)$\\
	$(2a+2|2d-2|2b+2,2b+2)$
	\end{tabular}
	\right.$ 
& $E_1^{2,2}=
	\left\{
	\begin{tabular}{ll}
	$(2a+2|2d-2|2b+2|2b+2)$\\
	$(2b|2b|2a+4|2d)$
	\end{tabular}
	\right.$
\\
\hline
$E_1^{0,3}=
	\left\{
	\begin{tabular}{ll}
	$(H^3(2a+2,2b+1,2b+1)|2d)$\\
	$(H^2(2a+2,2b+1,2d-1)|2b+2)$\\
	$(2a+2,2b|2b+2,2d)$\\
	$(2a+2,2d-2|2b+2,2b+2)$\\
	$(2b, 2b|2a + 4, 2d)$\\
	$(2a+2|H^3(2b+1,2b+1,2d))$\\
	$(2b|H^2(2a+3,2b+1,2d))$
	\end{tabular}
	\right.
	$
& $E_1^{1,3}=
	\left\{
	\begin{tabular}{ll}
	$(2a+2,2d-2|2b+2|2b+2)$\\
	$(2b|2a+3,2d-1|2b+2)$\\
	$(2b|2b|2a + 4, 2d)$
	\end{tabular}
	\right.
	$
& $E_1^{2,3}=(2b|2d - 2|2a + 4|2b + 2)$
	\\
\hline
$E_1^{0,4}=
	\left\{
	\begin{tabular}{ll}
	$(H^3(2a+2,2b+1,2d-1)|2b+2)$\\
	$(2b|H^3(2a+3,2b+1,2d))$
	\end{tabular}
	\right.$
& $E_1^{1,4}=
	\left\{
	\begin{tabular}{ll}
	$(2b,2d-2|2a+4|2b+2)$\\
	$(2b|2d-2|2a+4,2b+2)$
	\end{tabular}
	\right.$
& $E_1^{2,4}=0$\\
\hline
$E_1^{0,5}=(2b,2d-2|2a+4,2b+2)$
& $E_1^{1,5}=0$
& $E_1^{2,5}=0$\\
\hline
$E_1^{0,6}=0$
& $E_1^{1,6}=0$
& $E_1^{2,6}=0$\\
\hline
\end{longtable}
\end{small}
\end{center}


\subsubsection{Certain cohomology groups of $GL_3(\Z)$ that appear on the $E_1$-page in the case $C2'$.}
From the computations of the cohomology of $GL_3(\Z)$ from the previous section, we have
{\begin{center}
\scriptsize\renewcommand{\arraystretch}{2.2}
\begin{longtable}{|c|c|c|}
\hline
$H^2(2a+2,2b+1,2b+1)
	=
	\Delta_1+\Delta_0
	$
&$H^3(2a+2,2b+1,2b+1)
	=
	0$\\
\hline
$H^2(2a+2,2b+1,2d-1)=
	\Delta_2
	$
& $H^3(2a+2,2b+1,2d-1)
	=
	(2b,2d-2|2a+4)$
\\
\hline
$H^2(2b+1,2b+1,2d)
	=
	\Delta_3+\Delta'_0
	$
& $H^3(2b+1,2b+1,2d)
	=
	0$\\
\hline
$H^2(2a+3,2b+1,2d)
	=
	\Delta_4
	$
& $H^3(2a+3,2b+1,2d)
	=
	(2d-2|2a+4,2b+2)$\\
	\hline
\end{longtable}
\end{center}
where
$\Delta_1\subset (2a+2,2b|2b+2)+(2b|\overline{2a+3,2b+1})$,
$\Delta_0\subset (2a+2|2b|2b+2)+(2b|2b|2a+4)$,
$\Delta_2\subset (2a+2,2d-2|2b+2)+(2b|\overline{2a+3,2d-1})$,
$\Delta_3\subset (2b|2b+2,2d)+(\overline{2b+1,2d-1}|2b+2)$,
$\Delta'_0\subset (2b|2b+2|2d)+(2d-2|2b+2|2b+2)$
and
$\Delta_4\subset (2b|2a+4,2d)+(\overline{2a+3,2d-1}|2b+2)$.


\subsubsection{$E_2$ page. Highest weight $(2a+2,2b+1,2b+1,2d)$.  Case $C2'$}

For $q=1$, we have a surjective map
$E_2^{2,1}=E_1^{2,1}=(2a+2|2b|2b+2|2d)$

For $q=2$, we will separate the diagrams in two cases: the ones coming from ${\mathbb G}_m$ and the ones coming from $GL_2$.
For the $GL_2$-type
\begin{eqnarray*}
&&(H^2(2a+2,2b+1,2b+1)|2d)=(\Delta_1|2d) \rightarrow (2a + 2, 2b|2b + 2|2d) + (2b|\overline{2a + 3, 2b + 1}|2d)\rightarrow \\
&&\rightarrow (2b|2b|2a + 4|2d)
\end{eqnarray*}
and
\begin{eqnarray*}
&&(2a + 2|H^2(2b + 1,2b + 1,2d))\rightarrow (2a + 2|\overline{2b + 1, 2d - 1}|2b + 2) + (2a + 2|2b|b + 2, 2d) \rightarrow\\
&&\rightarrow (2a+2|2d-2|2b+2|2b+2).
\end{eqnarray*}
Therefore, the homology is the above two maps are centered at $p=1$. They are equal to 
$(2b|\overline{2a + 3, 2b + 1}|2d)$ and $(2a + 2|\overline{2b + 1, 2d - 1}|2b + 2)$, respectively.

The remaining part is to consider the spaces coming from ${\mathbb G}_m$.
We have the short exact sequences 
\[0\rightarrow (\Delta_0|2d)\rightarrow (2b, 2b|2a + 4|2d) + (2b|2b|2a + 4|2d) \rightarrow(2b|2b|2a + 4|2d) \rightarrow 0\]
and
\begin{eqnarray*}
&&0\rightarrow (2a + 2|\Delta'_0)\rightarrow (2a+2|2d-2|2b+2|2b+2) + (2a+2|2d-2|2b+2,2b+2) \rightarrow\\
&&\rightarrow (2a+2|2d-2|2b+2|2b+2)\rightarrow 0
\end{eqnarray*}
Then, the homology of the above sequence is zero.
Therefore, 
$E_2^{1,2}=(2b|\overline{2a + 3, 2b + 1}|2d) + (2a + 2|\overline{2b + 1, 2d - 1}|2b + 2)$ 
and
$E_2^{p,2}=0$ for $p=0$ and $p=2$.

For $q=3$, There is surjective map
$(2b|2a + 3,2d - 1|2b + 2)\rightarrow (2b|2d-2|2a+4|2b+2)$. Therefore, $E_2^{2,3}=0$.
We also have
that 
$H^3(2a+2,2b+1,2b+1)=0$
and
$H^3(2b+1,2b+1,2d)=0$.

We also have the following isomorphisms
$(H^2(2a + 2,2b + 1,2d - 1)|2b + 2)+(2a+2,2d-2|2b+2,2b+2) \rightarrow (2a+2,2d- |2b+2|2b+2) + (2b|\overline{2a + 3,2d - 1}|2b + 2)$
and
$(2b|H^2(2a + 3, 2b + 1, 2d))+(2b, 2b|2a + 4, 2d) \rightarrow (2b|\overline{2a + 3,2d - 1}|2b + 2) + (2b|2b|2a + 4, 2d)$. However, there is a repetition of $(2b|2a + 3,2d - 1|2b + 2)$. Then the homology of 
 \begin{eqnarray*}
 &&(H^2(2a + 2,2b + 1,2d - 1)|2b + 2)+(2a+2,2d-2|2b+2,2b+2)+\\
 &&+(2b|H^2(2a + 3, 2b + 1, 2d)) + (2b, 2b|2a + 4, 2d)  \rightarrow\\
 &&(2a+2,2d- |2b+2|2b+2) + (2b|\overline{2a + 3,2d - 1}|2b + 2) + (2b|2b|2a + 4, 2d)
 \end{eqnarray*}
is isomorphic $(2b|\overline{2a + 3,2d - 1}|2b + 2)$ which is positioned at $p=0$.
Therefore,
$E_2^{0,3}=(2a+2,2b|2b+2,2d) + (2b|2a + 3,2d - 1|2b + 2)$
and
$E_2^{p,3}=0$ for $p=1,2$.

For $q=4$, we have the isomorphisms
$(H^3(2a+2,2b+1,2d-1)|2b+2)\rightarrow  (2b,2d - 2|2a + 4|2b + 2)$
and
$(2b|H^3(2a + 3, 2b + 1, 2d))\rightarrow (2b|2d - 2|2a + 4,2b + 2)$
Therefore,
$E_2^{p,4}=0$ for $p=0,1,2$.

For $q=5$, we have $E_2^{0,5}=E_1^{0,5}=(2b,2d-2|2a+4,2b+2)$
{\begin{center}
\begin{small}
\scriptsize\renewcommand{\arraystretch}{2.2}
\begin{longtable}{|c|c|c|c|c|c|c|c|c|}
\hline
& $p=0$
& $p=1$
& $p=2$\\
\hline
$q=0$
&  
&  
& \\
\hline
$q=1$
&
& 
& $E_2^{2,1}=(2a+2|2b|2b+2|2d)$
\\
\hline
$q=2$
& 
&  $E_2^{1,2}=\left\{
	\begin{tabular}{lll}
	$(2a + 2, 2b|2b + 2|2d)$\\
	$(2a + 2|2b|b + 2, 2d)$
	\end{tabular}
	\right.$
& 
\\
\hline
$q=3$
& $E_2^{0,3}=\left\{
	\begin{tabular}{lll}
	$(2a+2,2b|2b+2,2d)$\\
	$(2b|2a + 3,2d - 1|2b + 2)$
	\end{tabular}
	\right.$
& 
&\\
\hline
$q=4$
& 
& 
& \\
\hline
$q=5$
& $E_2^{0,5}=(2b,2d-2|2a+4,2b+2)$
& 
&\\
\hline
$q=6$
&
& 
& \\
\hline
\end{longtable}
\end{small}
\end{center}


\subsubsection{Boundary cohomology. Highest weight $(2a+2,2b+1,2b+1,2d)$.  Case $C2'$.}
Again the boundary cohomology degenerates at the $E_2$-page. Therefore, up to semisimplicity we have
$H^n_\partial(GL_4(\Z),V_\lambda)=\bigoplus_{p+q=n}E_2^{p,q}$.
More systematically, we have

Let $H^q=H^q_\partial(GL_4(\Z),V_\lambda)$ be the cohomology of $GL_4(\Z)$ with coefficients in the highest weight representation with weight 
$\lambda=(2a+2,2b+1,2b+1,2d)$ written as a character of the split maximal torus. Then,
\[H^q
=
\left\{
\begin{tabular}{lll}
$0$
		& $q=2$\\
$\left\{\begin{tabular}{ll}
$E_2^{0,3}=\left\{
	\begin{tabular}{lll}
	$(2a+2,2b|2b+2,2d)$\\
	$(2b|2a + 3,2d - 1|2b + 2)$
	\end{tabular}
	\right.$\\
$E_2^{1,2}=\left\{
	\begin{tabular}{lll}
	$(2a + 2, 2b|2b + 2|2d)$\\
	$(2a + 2|2b|b + 2, 2d)$
	\end{tabular}
	\right\}$\\
$E_2^{2,1}=(2a+2|\overline{2b+1,2b+1}|2d)$
\end{tabular}\right\}$
		& $q=3$\\	
$0$		
		& $q=4$\\
$E_2^{0,5}=(2b,2d-2|2a+4,2b+2)$
		& $q=5$\\ 
$0$
		& $q=6$
\end{tabular}
\right.
\]

\begin{thm}
\[H^q
=
\left\{
\begin{tabular}{lll}
$S_{2a-2b+4}\otimes S_{2b-2d+4}
+
S_{2a-2b+4}
+
S_{2a-2d+6}
+
S_{2b-2d+4}
$
	& $q=3$\\
$S_{2b-2d+4}\otimes S_{2a-2b+4}$
	& $q=5$\\
$0$
		& $q\neq 3, 5$
\end{tabular}
\right.
\]
\end{thm}




\subsection{Cohomology of $GL_4(\Z)$. Highest weight $(2a+1,2b,2c-1,2d-2)$.  Case $D1$}


\subsubsection{Cohomology of the parabolic subgroups. Highest weight $(2a+1,2b,2c-1,2d-2)$. Case $D1$}
Let $\lambda=(2a+1,2b,2c-1,2d-2)$.

All elements $w$ of the Weyl group, mapping $\lambda\mapsto w(\lambda+\rho)-\rho$, will send $(odd,even,odd,even)$ to $(odd,even,odd,even)$,
Therefore, 

$H^q(P_0,V_\lambda)=0$
$H^q(P_{12},V_\lambda)=0$
$H^q(P_{23},V_\lambda)=0$
$H^q(P_{34},V_\lambda)=0$
$H^q(P_{12,34},V_\lambda)=0$
$H^q(P_{24},V_\lambda)=0$.
The only parabolic subgroup that could have nonzero cohomology in the coefficient system $D1$ is 
$P_{13}$.
For it, we have we have the following permutations.
The last entry of the permutation could be anything. The first three should be in increasing order.
Thus, all the possibilities are
$1234$,
$1243$,
$1342$,
$2341$.

For them we have the following weights and lengths.

$P_{13}$: 
$\begin{tabular}{lllll}
$w$ 		& $l$	& 	& $w(\lambda+\rho)-\rho$\\
$1234$ 	& $0$&	& $(2a+1,2b,2c-1|2d-2)$\\
$1243$ 	& $1$&	& $(2a+1,2b,2d-3|2c)$\\
$1342$ 	& $2$&	& $(2a+1,2c-2,2d-3|2b+2)$\\
$2341$ 	& $3$&	& $(2b-1,2c-2,2d-3|2a+4)$\\
\end{tabular}
$

Note that all the the representations of $GL_3$ that appear have highest weight $(odd,even,odd)$. Therefore, it has no boundary cohomology of $GL_3(\Z)$; only inner cohomology, denoted by $H^q_!$.


\subsubsection{$E_1$-page. Highest weight $(2a+1,2b,2b,2d-1)$.  Case $D1$.}

{\begin{center}
\begin{small}
\scriptsize\renewcommand{\arraystretch}{2.2}
\begin{longtable}{|c|c|c|c|c|c|c|c|c|}
\hline
$E_1^{0,0}=0$ 
& $E_1^{1,0}=0$ & 
$E_1^{2,0}=0$
\\
\hline
$E_1^{0,1}=0$
& $E_1^{1,1}=0$
& $E_1^{2,1}=0$
\\
\hline
$E_1^{0,2}=(H^2_!(2a+1,2b,2c-1)|2d-2)$
& $E_1^{1,2}=0$ 
& $E_1^{2,2}=0$
\\
\hline
$E_1^{0,3}=\left\{
	\begin{tabular}{ll}
	$(H^3_!(2a+1,2b,2c-1)|2d-2)$\\
	$(H^2_!(2a+1,2b,2d-3)|2c)$
	\end{tabular}
	\right.$
& $E_1^{1,3}=0$
& $E_1^{2,3}=0$\\
\hline
$E_1^{0,4}=
	\left\{
	\begin{tabular}{ll}
	$(H^3_!(2a+1,2b,2d-3)|2c)$\\
	$(H^2_!(2a+1,2c-2,2d-3)|2b+2)$
	\end{tabular}
	\right.$
& $E_1^{1,4}=0$
& $E_1^{2,4}=0$\\
\hline
$E_1^{0,5}=\left\{
	\begin{tabular}{ll}
	$(H^3_!(2a+1,2c-2,2d-3)|2b+2)$\\
	$(H^2_!(2a+1,2c-2,2d-3)|2b+2)$\\
	\end{tabular}
	\right.$
& $E_1^{1,5}=0$
& $E_1^{2,5}=0$
\\
\hline
$E_1^{0,6}=(H^3_!(2a+1,2c-2,2d-3)|2b+2)$
& $E_1^{1,6}=0$
& $E_1^{1,6}=0$\\
\hline
\end{longtable}
\end{small}
\end{center}

The spectral sequence degenerated at $E_1$ level. Therefore the boundary cohomology is


\subsubsection{Boundary cohomology. Highest weight $(2a+1,2b,2c-1,2d-2)$.  Case $D1$.}
The spectral sequence degenerated at $E_1$ level. Therefore the boundary cohomology is

\[H^q
=
\left\{
\begin{tabular}{lll}
$E_1^{0,2}=(H^2_!(2a+1,2b,2c-1)|2d-2)$
		& $q=2$\\
$E_1^{0,3}=\left\{
	\begin{tabular}{ll}
	$(H^3_!(2a+1,2b,2c-1)|2d-2)$\\
	$(H^2_!(2a+1,2b,2d-3)|2c)$
	\end{tabular}
	\right\}$
		& $q=3$\\	
$E_1^{0,4}=
	\left\{
	\begin{tabular}{ll}
	$(H^3_!(2a+1,2b,2d-3)|2c)$\\
	$(H^2_!(2a+1,2c-2,2d-3)|2b+2)$
	\end{tabular}
	\right\}$	
		& $q=4$\\
$E_1^{0,5}=\left\{
	\begin{tabular}{ll}
	$(H^3_!(2a+1,2c-2,2d-3)|2b+2)$\\
	$(H^2_!(2a+1,2c-2,2d-3)|2b+2)$\\
	\end{tabular}
	\right\}$
		& $q=5$\\ 
$E_1^{0,6}=(H^3_!(2a+1,2c-2,2d-3)|2b+2)$
		& $q=6$
\end{tabular}
\right.
\]




\subsection{Cohomology of $GL_4(\Z)$. Highest weight $(2a+2,2+1,2c,2d-1)$.  Case $D1'$}


\subsubsection{Cohomology of the parabolic subgroups. Highest weight $(2a+2,2b+1,2c,2d-1)$. Case $D1'$}
Let $\lambda=(2a+2,2b+1,2c,2d-1)$.

All elements $w$ of the Weyl group, mapping $\lambda\mapsto w(\lambda+\rho)-\rho$, will send $(odd,even,odd,even)$ to $(odd,even,odd,even)$,
Therefore, 

$H^q(P_0,V_\lambda)=0$
$H^q(P_{12},V_\lambda)=0$
$H^q(P_{23},V_\lambda)=0$
$H^q(P_{34},V_\lambda)=0$
$H^q(P_{13},V_\lambda)=0$,
$H^q(P_{12,34},V_\lambda)=0$.

The only parabolic subgroup that could have nonzero cohomology in the coefficient system $D1'$ is 
$P_{24}$.
For it, we have we have the following permutations.
The first entry of the permutation could be anything. The first three should be in increasing order.
Thus, all the possibilities are
$1234$,
$2134$,
$3124$,
$4123$.

For them we have the following weights and lengths.

$P_{13}$: 
$\begin{tabular}{lllll}
$w$ 		& $l$	& 	& $w(\lambda+\rho)-\rho$\\
$1234$ 	& $0$&	& $(2a+2|2b+1,2c,2d-1)$\\
$2134$	& $1$&	& $(2b|2a+3,2c,2d-1)$\\
$3124$	& $2$&	& $(2c-2|2a+3,2b+2,2d-1)$\\
$4123$	& $3$&	& $(2d-4|2a+3,2b+2,2c+1)$\\
\end{tabular}
$

Note that all the the representations of $GL_3$ that appear have highest weight $(odd,even,odd)$. Therefore, it has no boundary cohomology of $GL_3(\Z)$; only inner cohomology, denoted by $H^q_!$.


\subsubsection{$E_1$-page. Highest weight $(2a+1,2b,2b,2d-1)$.  Case $D1'$.}

{\begin{center}
\begin{small}
\scriptsize\renewcommand{\arraystretch}{2.2}
\begin{longtable}{|c|c|c|c|c|c|c|c|c|}
\hline
$E_1^{0,0}=0$ 
& $E_1^{1,0}=0$ & 
$E_1^{2,0}=0$
\\
\hline
$E_1^{0,1}=0$
& $E_1^{1,1}=0$
& $E_1^{2,1}=0$
\\
\hline
$E_1^{0,2}=(2a+2|H^2_!(2b+1,2c,2d-1))$
& $E_1^{1,2}=0$ 
& $E_1^{2,2}=0$
\\
\hline
$E_1^{0,3}=\left\{
	\begin{tabular}{ll}
	$(2a+2|H^3_!(2b+1,2c,2d-1))$\\
	$(2b|H^2_!(2a+3,2c,2d-1))$
	\end{tabular}
	\right.$
& $E_1^{1,3}=0$
& $E_1^{2,3}=0$\\
\hline
$E_1^{0,4}=
	\left\{
	\begin{tabular}{ll}
	$(2b|H^3_!(2a+3,2c,2d-1))$\\
	$(2c-2|H^2_!(2a+3,2b+2,2d-1))$
	\end{tabular}
	\right.$
& $E_1^{1,4}=0$
& $E_1^{2,4}=0$\\
\hline
$E_1^{0,5}=\left\{
	\begin{tabular}{ll}
	$(2c-2|H^3_!(2a+3,2b+2,2d-1))$\\
	$(2d-4|H^2_!(2a+3,2b+2,2c+1))$
	\end{tabular}
	\right.$
& $E_1^{1,5}=0$
& $E_1^{2,5}=0$
\\
\hline
$E_1^{0,6}=(2d-4|H^3_!(2a+3,2b+2,2c+1))$
& $E_1^{1,6}=0$
& $E_1^{1,6}=0$\\
\hline
\end{longtable}
\end{small}
\end{center}

The spectral sequence degenerated at $E_1$ level. Therefore the boundary cohomology is


\subsubsection{Boundary cohomology. Highest weight $(2a+1,2b,2c-1,2d-2)$.  Case $D1'$.}
The spectral sequence degenerated at $E_1$ level. Therefore the boundary cohomology is

\[H^q
=
\left\{
\begin{tabular}{lll}
$E_1^{0,2}=(2a+2|H^2_!(2b+1,2c,2d-1))$
		& $q=2$\\
$E_1^{0,3}=\left\{
	\begin{tabular}{ll}
	$(2a+2|H^3_!(2b+1,2c,2d-1))$\\
	$(2b|H^2_!(2a+3,2c,2d-1))$
	\end{tabular}
	\right.$
		& $q=3$\\	
$E_1^{0,4}=
	\left\{
	\begin{tabular}{ll}
	$(2b|H^3_!(2a+3,2c,2d-1))$\\
	$(2c-2|H^2_!(2a+3,2b+2,2d-1))$
	\end{tabular}
	\right.$	
		& $q=4$\\
$E_1^{0,5}=\left\{
	\begin{tabular}{ll}
	$(2c-2|H^3_!(2a+3,2b+2,2d-1))$\\
	$(2d-4|H^2_!(2a+3,2b+2,2c+1))$
	\end{tabular}
	\right.$
		& $q=5$\\ 
$E_1^{0,6}=(2d-4|H^3_!(2a+3,2b+2,2c+1))$
		& $q=6$
\end{tabular}
\right.
\]





\section{Euler characteristic of $GL_m(\Z)$}
\label{se:euler}

The homological Euler characteristic $\chi_h$ of a group $\Gamma$ with coefficients in a representation is defined as
\begin{equation}
\label{eq:hec}
\chi_h(\Gamma,V)=\sum_{i=0}^{\infty}\, (-1)^{i} \, \mr{dim} \, H^{i}(\Gamma,V). 
\end{equation}
For details on the above formula  see \cite{Br}, \cite{Se}. We recall the definition of orbifold Euler characteristic. If $\Gamma$ is torsion free, then the orbifold Euler characteristic is defined as $\chi_{orb}(\Gamma) = \chi_h(\Gamma)$. If $\Gamma$ has torsion elements and admits a finite index torsion free subgroup $\Gamma'$, then the orbifold Euler characteristic of $\Gamma$ is given by 
\begin{equation}\label{eq:orbeuler}
\chi_{orb}(\Gamma) = \frac{1}{[\Gamma : \Gamma']} \chi_h(\Gamma')\,.
\end{equation} 
One important fact is that, following Minkowski, every arithmetic group of rank greater than one contains a torsion free finite index subgroup and therefore the concept of orbifold Euler characteristic is well defined in this setting. If $\Gamma$ has torsion elements then we make use of the following  formula discovered by Chiswell in \cite{Ch}.
\begin{equation}\label{eq:hecT}
\chi_h(\Gamma,V )=\sum_{(A)} \chi_{orb}(C(A)) Tr(A^{-1}| V) .
\end{equation}
 Otherwise, we use the formula described in equation~\eqr{hec}. The sum runs over all the conjugacy classes in $\Gamma$ of its torsion elements $A$, denoted by $(A)$, and $C(A)$ denotes the centralizer of $A$ in $\Gamma$. From now on, orbifold Euler characteristic $\chi_{orb}$ will be simply denoted by $\chi$. Orbifold Euler characteristic has the following properties.
 \begin{enumerate}
\item  If $\Gamma$  is finitely generated  torsion free group  then $\chi(\Gamma)$ is defined as $\chi_h(\Gamma,\Q).$
\item If $\Gamma$ is finite of order $\left| \Gamma \right|$ then $\chi(\Gamma)=\frac{1}{\left| \Gamma \right|}$.
\item Let $\Gamma$, $\Gamma_1$ and  $\Gamma_2$ be groups such that 
$ 1 \longrightarrow \Gamma_1 \longrightarrow \Gamma \longrightarrow \Gamma_2 \longrightarrow 1$ is exact then $\chi(\Gamma)= \chi(\Gamma_1) \chi(\Gamma_2)$.
 \end{enumerate}

\begin{prop}
If $V$ is a representation of $GL_4(\Z)$, then for all $n$, we have 
\[H^i(SL_n(\Z),V) =   H^i(GL_n(\Z),V) +   H^i(GL_n(\Z),V\otimes det).\]
\end{prop}
\proof
If $V$ is a finite dimensional representation of $GL_n(\Z)$ then we can consider it as a representation of $SL_n(\Z)$. As a representation of $SL_n(\Z)$ the induced representation to $GL_n(\Z)$ is $Ind(V) = V + V\otimes det$. Therefore,
\[H^i(SL_n(\Z),V) =  H^i(GL_n(\Z),Ind(V))= H^i(GL_n(\Z),V) +   H^i(GL_n(\Z),V\otimes det).\]

\begin{cor}
\label{chi(SL,Q)}
For any integer $n\geq 2$
\[\chi_h(SL_n(\Z),V) =  \chi_h(GL_n(\Z),V) +   \chi_h(GL(\Z),V\otimes det).\]
\end{cor}
It follows directly from the previous proposition.

\begin{cor}
Let $n$ be an odd integer and let $V$ be a finite dimensional highest weight representation of $GL_n(\Z)$.

(a) If $-I$ acts trivially on $V$, then \[H^i(SL_n(\Z),V)=H^i(GL_n(\Z),V)\text{ and }\]
\[H^i(GL_n(\Z),V\otimes det)=0\]. Consequently,
\[\chi_h(SL_n(\Z),V) =  \chi_h(GL_n(\Z),V) .\]

(b) If $-I$ acts nontrivially on $V$, then $-I$ acts trivially on $V\otimes det$, and
\[H^i(SL_n(\Z),V)=H^i(GL_n(\Z),V\otimes det)\text{ and }\]
\[H^i(GL_n(\Z),V)=0\] Consequently,
\[\chi_h(SL_n(\Z),V) =   \chi_h(GL(\Z),V\otimes det).\]

\end{cor}

A key result that we are going to use is a computationally effective way for finding the homological Euler characteristic, developed in \cite{thesis}, \cite{EulerChar}. We know that when $\Gamma$ is $\mr{GL}_n(\Z)$ one has an expression of the form 
 \begin{equation}\label{eq:hnthor}
 \chi_{h}(\Gamma, V) = \sum_{A} Res(f_A) \chi(C(A)) Tr(A^{-1}| V)\,,
 \end{equation}
where $f_A$ denotes the characteristic polynomial of the matrix $A$.

Now we will explain~\eqref{eq:hnthor} in detail. 

\par Let us denote
\[ 
T_3 = \left( \begin{array}{rrr}
0 & 1  \\
-1 & -1 \\ 
\end{array}  \right),
T_4 = \left( \begin{array}{rrr}
0 & 1 \\
-1 & 0 \\ 
\end{array}  \right) \mbox{ and }
T_6 = \left( \begin{array}{rrr}
0 & -1  \\
1 & 1   \\ 
\end{array}  \right).
\]
The summation is over all block-diagonal matrices $A$, such that
\begin{itemize}
\item The blocks in the diagonal belong to the set $\left\{1, -1, T_3, T_4, T_6\right\}$.
\item The blocks $T_3, T_4$ and $T_6$ appear at most once, $-1$ appears at most twice, and $1$ appears at most twice.
\item A change in the order of the blocks in the diagonal does not count as a different element.
\end{itemize}
So, for example, if $n > 10$, the sum is empty and $\chi_{h}(GL_n(\Z), V) = 0$ for all finite dimensional representations.

In this case, one can see that every $A$ satisfying these properties has the same eigenvalues as $A^{-1}$. Even more, every such $A$ is conjugate, over $\mathbb{C}$, to $A^{-1}$ and therefore $Tr(A^{-1}| V) = Tr(A| V)$. We will use these facts in what follows.

\par Let us explain briefly the notation $Res(f)$. Let $f_1=\prod_i(x - \alpha_i)$ and $f_2=\prod_j(x-\beta_j)$ be two polynomials. Then by the resultant of $f_1$ and $f_2$, we mean $Res(f_1,f_2)=\prod_{i,j}(\alpha_i-\beta_j)$. If the characteristic polynomial $f$ is a power of an irreducible polynomial then we define $Res(f)=1$. Let $f=f_1 f_2\dots f_d$, where each  $f_i$ is a power of an irreducible polynomial over $\Q$ and they are relatively prime pairwise. Then, we define $Res(f)=\prod_{i<j} Res(f_i,f_j)$.

\par For any torsion free arithmetic subgroup $\Gamma \subset \mr{SL}_n(\R)$ we have the Gauss-Bonnet formula
$$\chi_h(\Gamma \backslash X) = \int_{\Gamma \backslash X} \omega_{GB}$$
where $\omega_{GB}$ is the Gauss-Bonnet-Chern differential form and $X = \mr{SL}_n(\mathbb{R}) / \mr{SO}(n, \mathbb{R})$, see \cite{Harder}. This differential form is zero if $n > 2$ and therefore for any torsion free congruence subgroup $\Gamma \subset \mr{SL}_n(\mathbb{Z})$, $\chi_h(\Gamma \backslash X) = 0$.
We will make use of this fact in the calculation of the homological Euler characteristic of $GL_4(\Z)$ and a some congruence subgroups $\Gamma_1(3,p)$ and $\Gamma_1(4,p)$ inside  $GL_3(\Z)$  and $GL_4(\Z)$, respectively.

\subsection{Torsion elements}\label{ss:torsion}

At this note, let $\Phi_n$ be the $n$-th cyclotomic polynomial then we list all the characteristic polynomials of torsion elements in $\mr{GL}_4(\Z)$  in the following table.

If a torsion element contains in its characteristic polynomial a factor of $\Phi_{_1}^3$ or $\Phi_{_2}^3$, then its centralizer will contain a subgroup commensurable to $GL_3(\Z)$.
However, the orbifold Euler characteristic of $GL_3(\Z)$ is zero. Thus, it will not contribute to the sum of orbifold Euler characteristics.
Similarly, if a torsion element contains in its characteristic polynomial a factor of $\Phi_{_1}^4$ or $\Phi_{_2}^4$, then its centralizer will be  $GL_4(\Z)$.
However, the orbifold Euler characteristic of $GL_4(\Z)$ is zero.
For all $\Phi_{_n}$, when $n=5,8,10,12$, we are going to show that again the Euler characteristics of their cetralizers is zero.
For such a torsion elements A, we have that the set $M(A)$ of all matrices $X$, that commute with such a torsion element $A$ contain polynomials in $A$. Since $\Phi_{_n}(A)=0$, we obtain that $M(A)\subset \Z[x]/(\Phi_{_n})$ has finite index in ring of integers of the cyclotomic field obtained by adjoining a primitive $n$-root of $1$. By the Dirichlet unit theorem we have that $\Z[x]/(\Phi_{_n})$ contains a free group $\Z$ inside the group of units. Moreover the centralizer $C(A)$ is exactly the invertible elements of $M(A)$, which is the group of units in $M(A)$. Since $M(A)$ has a finite index in $\Z[x]/(\Phi_{_n}$, we obtain that the centralizer $C(A)$ is the groups of units $M(A)$ which is a finite index in the group of units in the cyclotimic field $\Q[x]/(\Phi_{_n})$. Note that $\chi(\Z)$ it the same as the Euler characteristic of a circle, which is zero. Thus, $\chi(C(A))=0$.

Below will list all torsion elements together with their centralizers $C(A)$, orbifold Euler characteristic of their centralizers $\chi(C(A))$, and their resultants $R(f)$ and finally the product $\chi(C(A))R(f)$.

Also if $\Phi_{_n}^2$ is the characteristic polynomial of a torsion element then its centralizer is commensurable to $GL_2(Z[\xi_n])$ where $\xi_n$ is a primitive $3$-rd, $4$-th or $6$-th root of $1$. However, $\chi(C(A))=\chi(GL_2(Z[\xi_n]))=0$.

{\begin{center}
\scriptsize\renewcommand{\arraystretch}{2.2}
\begin{longtable}{|c|c|c|c|c|c|c|c|c|}
\hline
S.No.   & Polynimoial                       & Centralizer $C(A)$        & $det(A)$  & $\chi(C(A))$                              & $R(f)$        & $\chi(C(A))R(f)$  \\
\hline
\hline
3       & $\Phi_{_1}^{2} \Phi_{_2}^{2}$     & $GL_2(\Z)\times GL_2(\Z)$ & $1$       & $\frac{1}{24^2}$                           & $2^4$         & $\frac{16}{24^2}$\\
\hline
4       & $\Phi_{_1}^{2}\Phi_{_3}$          & $GL_2(\Z)\times C_6$      & $1$       & $-\frac{1}{24}\cdot\frac{1}{6}$           & $3^2$         & $-\frac{36}{24^2}$\\
\hline
5       & $\Phi_{_1}^{2}\Phi_{_4}$          & $GL_2(\Z)\times C_4$      & $1$       & $-\frac{1}{24}\cdot\frac{1}{4}$           & $2^2$         & $-\frac{24}{24^2}$\\
\hline
6       & $\Phi_{_1}^2 \Phi_{_6}$           & $GL_2(\Z)\times C_6$      & $1$       & $-\frac{1}{24}\cdot\frac{1}{6}$           & $1^2$         & $-\frac{4}{24^2}$\\
\hline
8       & $\Phi_{_1} \Phi_{_2} \Phi_{_3}$   & $C_2\times C_2 \times C_6$ & $-1$     & $\frac{1}{2}\cdot\frac{1}{2}\cdot\frac{1}{6}$  & $2\times 3$   & $\frac{1}{4}$\\
\hline
9       & $\Phi_{_1} \Phi_{_2} \Phi_{_4}$   & $C_2\times C_2 \times C_4$ & $-1$     & $\frac{1}{2}\cdot\frac{1}{2}\cdot\frac{1}{4}$  & $2^3$         & $\frac{1}{2}$\\
\hline
10      & $\Phi_{_1}\Phi_{_2} \Phi_{_6}$    & $C_2\times C_2 \times C_6$ & $-1$     & $\frac{1}{2}\cdot\frac{1}{2}\cdot\frac{1}{6}$  & $2\times 3$   & $\frac{1}{4}$\\
\hline
12      & $\Phi_{_2}^{^2} \Phi_{_3}$        & $GL_2(\Z)\times C_6$      & $1$       & $-\frac{1}{24}\cdot\frac{1}{6}$           & $1^2$         & $-\frac{4}{24^2}$\\
\hline
13      & $\Phi_{_2}^{2} \Phi_{_4}$         & $GL_2(\Z)\times C_4$      & $1$       & $-\frac{1}{24}\cdot\frac{1}{4}$           & $2^2$         & $-\frac{24}{24^2}$\\
\hline
14      & $\Phi_{_2}^{2}\Phi_{_6}$          & $GL_2(\Z)\times C_6$      & $1$       & $-\frac{1}{24}\cdot\frac{1}{6}$           & $3^2$         & $-\frac{36}{24^2}$\\
\hline
16      & $\Phi_{_3} \Phi_{_4}$             & $C_6\times C_4$           & $1$       & $\frac{1}{6}\cdot\frac{1}{4}$                  & $1$           & $\frac{24}{24^2}$\\
\hline
17      & $\Phi_{_3} \Phi_{_6} $            & $C_6\times C_6$           & $1$       & $\frac{1}{6}\cdot\frac{1}{6}$                  & $2^2$         & $\frac{64}{24^2}$\\
\hline
19      & $\Phi_{_4} \Phi_{_4}$             & $C_4\times C_6$           & $1$       & $\frac{1}{4}\cdot\frac{1}{6}$                  & $1$           & $\frac{24}{24^2}$\\
\hline
\caption{Torsion elements in $\mr{GL}_4(\Z)$.}\label{torsiojngl4}
\end{longtable}
\end{center}
}

\subsection{Coefficient systems $\Q$ and $det$}
Then the Euler characteristic $\chi_h(GL_4(\Z),\Q)$ is just the sum of the elements
\[\chi_h(GL_4(\Z),\Q)=\sum_A\chi(C(A))R(f),\]
which is the sum of the entries in the last column. Note that all fractions with denominators $24^2$ add up to zero. The remaining one are $\frac{1}{4}+\frac{1}{2}+\frac{1}{4}=1$.
Therefore, we obtain the following.
\begin{lemma}
\label{chi(GL,Q)}
\[
\chi_h(GL_4(\Z),\Q)=1.
\]
\end{lemma}
For the representation $det$, we have that $Tr(A|det)=det(A)$. Therefore,
\[\chi_h(GL_4(\Z),\det)=\sum_A\det(A)\chi(C(A))R(f).\]
Note that all the entries with denominators $24^2$ come from torsion element with determinant $1$, and all entries with denominators $2$ or $4$ come from torsion elements with determinant $-1$. Again, the sum of $\det(A)\chi(C(A))R(f)=\chi(C(A))R(f)$ over torsion elements of determinant $=1$ add up to zero. The remaining toriosn elements have determinant $-1$. Therefore
\begin{align*}\chi_h(\mr{GL}_4(\Z),\det) &=\sum_A\det(A)\chi(C(A))R(f)\\
& =\sum_{A:\det(A)=-1}\det(A)\chi(C(A))R(f)\\
& =-\sum_{A:\det(A)=-1}\chi(C(A))R(f) \\ 
&=-\left(\frac{1}{4}+\frac{1}{2}+\frac{1}{4}\right)\\
& =-1.
\end{align*}
Therefore, we obtain the following.
\begin{lemma}
\label{chi(GL,det)}
\[
\chi_h(GL_4(\Z),det)=-1.
\]
\end{lemma}

From Lemmas \ref{chi(GL,Q)}  and \ref{chi(GL,Q)}, and Corollary \ref{chi(SL,Q)}, we have that
\begin{cor}
\[\chi_h(\mr{SL}_4(\Z),\Q)=0.\]
\end{cor}


\section{Eisenstein Cohomology of $GL_4(\Z)$}

\begin{thm}
Let $V_\lambda$ be a highest weight representation of $GL_n(\Z)$,
Denote by $V^*_\lambda$ the dual representation.
Then $H^q_\partial(GL_n(\Z),V^*_\lambda)$ is non-cannonically isomorphic to $H^q_\partial(GL_n(\Z),V_\lambda)$ 
\end{thm}

For the boundary cohomology, we have a version of the Serre's duality.
\begin{thm}
Let $V_\lambda$ be a highest weight representation of $GL_n(\Z)$,
Denote by $V^*_\lambda$ the dual representation.
If $N$ is the dimension of the boundary of the Borel-Serre compactification of $GL_n(\Z)\backslash GL_n(\R)/O_n(\R)\times \R_{>0}$,
then we have the following perfect pairing for the boundary cohomology of $GL_n(\Z)$.
\[H^q_\partial(GL_n(\Z),V_\lambda)\otimes H^{N-q}_\partial(GL_n(\Z),V^*_\lambda\otimes det^{\otimes(n-1)})\rightarrow  H^N_\partial(GL_n(\Z),det^{\otimes(n-1)}).\]
\end{thm}
In this theorem, $\tilde{V}^*_\lambda$ is the dual sheaf of $\tilde{V}$ and $\det^{\otimes(n-1)}$ is the dualizing sheaf. 

\begin{thm}
The Eisenstein cohomology $H^q_{Eis}(GL_n(\Z),V_\lambda)$ is an anisotropic parts of the pairing. 
More specifically, 
if we denote by $\widetilde{H}^q_{Eis}(GL_n(\Z),V_\lambda)$ the dual space under the pairing,
then the direct sum of $\widetilde{H}^q_{Eis}(GL_n(\Z),V_\lambda)$ and $H^{N-q}_{Eis}(GL_n(\Z),V^*_\lambda\otimes det^{\otimes(n-1)})$ is the entire boundary cohomolog $H^{N-q}_{\partial}(GL_n(\Z),V^*_\lambda\otimes det^{\otimes(n-1)})$ and the intersection of $\widetilde{H}^q_{Eis}(GL_n(\Z),V_\lambda)$ and $H^{N-q}_{Eis}(GL_n(\Z),V^*_\lambda\otimes det^{\otimes(n-1)})$ is trivial.
 \end{thm}

\begin{cor}
For any highest weight representation $V_\lambda$, we have 
\[H^6_{Eis}(GL_4(\Z),V_\lambda)=H^6_\partial(GL_4(\Z),V_\lambda)=H^6(GL_4(\Z),V_\lambda).\]
\end{cor}
\proof
From Serre's duality, we have that $H^2_\partial(GL_4(\Z),V_\lambda\otimes \det)$ is dual to $H^6_{Eis}(GL_4(\Z),V_\lambda)$. 
On the other hand, $H^2_{Eis}(GL_4(\Z),\lambda\otimes de)=0$, 
because  the non-trivial cohomology classes occur only in dimensions $3,4,5,6$. 
Since, the Eisenstein cohomology is an anisotropic in the boundary cohomology 
and  $H^2_{Eis}(GL_4(\Z),\lambda\otimes det)=0$, we mist have  $H^6_{Eis}(GL_4(\Z),\lambda)=H^6_\partial(GL_4(\Z),V_\lambda)$.

For the second equality, we use that $H^6(GL_4(\Z),V_\lambda)=H^6_{Eis}(GL_4(\Z),V_\lambda)+H^6_!(GL_4(\Z),V_\lambda)$.
However the inner cohomology is concentrated in degrees $4$ and $5$. Therefore,  $H^6(GL_4(\Z),V_\lambda)=H^6_{Eis}(GL_4(\Z),V_\lambda)$.

\begin{conjecture} The Euler characteristic of the Eisenstein cohomology \[\sum_q dim H^q_{Eis}(GL_4(\Z),V_\lambda)\] contains only terms induced by the inner cohomology of $GL_1$ and $GL_2$, but not $GL_3$ or $GL_4$. 
\end{conjecture}

All computations of Euler characteristics so far point to this conjecture.
If the conjecture is false, then we can find dimensions of parts of inner cohomology of $GL_3(\Z)$, which is even more interesting.

\begin{thm}
Let $V_\lambda$ be a highest weight representation of $GL_4$. Then
$\epsilon_\lambda=\chi_h(GL_4(\Z),V_\lambda)-
\chi_h(GL_4(\Z),V_\lambda\otimes det)$
is bounded.
\end{thm}
\proof The homological Euler characteristics are computed via a sum over torsion elements. For a torsion element, the trace $Tr(A|V_\lambda)$ will change sign compared to $Tr(A|V_\lambda\otimes det)$, only when $A$ has determinant $-1$. This can happen only in three cases. Namely, when $A$ is of block-diagonal form $[1,-1,T_3]$, or $[1,-1,T_4]$ or $[1,-1,T_6]$. For those matrices $A$, we have that $Tr(A|Sym^a)$ is bounded with respect to $a$. By Weyl formula, we have that the trace for general $\lambda$ is a determinant of $2\times 2$ or $3\times 3$ matrix whose entries are of the type $Tr(A|Sym^a)$. Therefore, $\epsilon_\lambda$ is bounded. 

\begin{conjecture}
If a term $S_{2k}$ or $S_{2k}\otimes S_{2l}$ appears in the boundary cohomology of $GL_4(\Z)$, then it either stays in the Eisenstein cohomology or it dispears completely.
\end{conjecture}

More generally,
\begin{conjecture}
Let $S_P$ be a Levi quotient of a parabolic subgroup $P$ of $GL_m(\Z)$. Let $H^k_!(S_P,V)$ be inner cohomology of $S_P$. If $H^k_!(S_P,V)$ appears in the boundary cohomology of $GL_m(\Z)$ with coefficients in a highest weight representation then it either stays in the Eisenstein cohomology or it completely disapears.
\end{conjecture}

Using those consideration for determining the Eisenstein cohomology we obtain the following results.



\subsection{Eisenstein cohomology of $GL_4(\Z)$ for the representations $A1$ and $A1'$.}
Using Corollary 73, we obtain that  the $6$-th Eisenstein cohomology is the same as the boundary cohomology.
\[H^6_{Eis}(GL_4(\Z),(2a+1,2b+1,2c+1,2d+1))=
\left\{
	\begin{tabular}{lll}
	$(H^3_!(GL_3(\Z),(2b,2c,2d)|2a+4))$\\
	$+(2d-2|H^3_!(GL_3(\Z),(2a+2,2b+2,2c+2)))$\\
	$+(2c-1,2d-1|2a+3,2b+3)$\\
	$+(2d-2|\overline{2b+1,2c+1}|2a+4)$
	\end{tabular}
	\right.\]

If we assume Conjecture 74, then $H^5_{Eis}(GL_4(\Z),(2a+1,2b+1,2c+1,2d+1))$ should contain terms induced from the inner cohomology of $GL_3(\Z)$, so that, the inner cohomology of $GL_3(\Z)$ in $H^6_{Eis}$ and $H^5_{Eis}$ cancel in the Euler characteristic.

From Serre's duality, $H^3_{\partial}(V_\lambda)$ is dual to $H^5_{\partial}(V_\lambda\otimes det)$. Since $H^5_{Eis}(V_\lambda\otimes det)$ contain the inner cohomology of $GL_3(\Z)$, we obtain that $H^3_{Eis}(V_\lambda)$ does not. Therefore,
$H^3_{Eis}(GL_4(\Z),(2a,2b,2c,2d))=H^3_{Eis}(V_\lambda)=0$. 
Also $H^3(V_\lambda\otimes det)=0$. Therefore $H^5(V_\lambda)=0$.
Thus, for $q\neq 4$, we have determined $H^q_{Eis}(V_\lambda)$ and $H^q_{Eis}(V_\lambda\otimes det)$.

In order to determine the dimensions of  $H^4_{Eis}(V_\lambda)$ and $H^4_{Eis}(V_\lambda\otimes det)$, we need to examine Euler characteristics.
Denote by \[\chi_h(\Gamma,V)=\sum_q= dim H^q(\Gamma,V),\] the homological Euler characteristic of the arithmetic group $\Gamma$ with coefficients in the representation $V$.
We have
$\chi_h(GL_4(\Z),V )=\sum_{(A)} \chi_{orb}(C(A)) Tr(A^{-1}| V),$
where the sum is over conjugacy classes $(A)$ of torsion elements.

A result from (thesis), implies that
$\chi_h(GL_4(\Z),V )=\sum_{f} R(f)C(A)Tr(A|V)$.

Then 
\begin{eqnarray}
&&\chi_h(GL_4(\Z),V\otimes det)=\\
&&=\sum_{f} R(f)C(A)Tr(A|V\otimes det)\\
&&=\sum_{f, det(A)=1} R(f)C(A)Tr(A|V\otimes det)+\sum_{f, det(A)=-1} R(f)C(A)Tr(A|V\otimes det)=\\
&&=\sum_{f, det(A)=1} R(f)C(A)Tr(A|V)-\sum_{f, det(A)=-1} R(f)C(A)Tr(A|V).
\end{eqnarray}
Therefore 
\begin{equation}
\chi_h(GL_4(\Z),V)-\chi_h(GL_4(\Z),V\otimes det)=2\sum_{f,det(A)=-1}R(f)C(A)Tr(A|V).
\end{equation}
In $GL_4(\Z)$ there are characteristic polynomials of torsion elements such that the determinant of the torsion elements in  $-1$. Those characteristic polynomials are 
$\Phi_1\Phi_2\Phi_3$, $\Phi_1\Phi_2\Phi_4$ and $\Phi_1\Phi_2\Phi_6$. 
\begin{lemma} For $k=3,4,6$ we have that the trace of those torsion elements acting on the symmetric powers of the standard representations,
$Tr(\Phi_1\Phi_2\Phi_k|Sym^nV_3)$,
is bounded.
\end{lemma}

\begin{lemma}
For $k=3,4,6$ we have that the trace of those torsion elements acting on the symmetric powers of the standard representations,
$Tr(\Phi_1\Phi_2\Phi_k|V_\lambda)$,
is bounded for every finite dimensional highest weight representation $V_\lambda$.
\end{lemma}

\begin{cor}
$\chi_h(GL_4(\Z),V)-\chi_h(GL_4(\Z),V\otimes det)$ is bounded, where  $V$ runs through all highest weight representations.
\end{cor}
\proof We have that $\chi_h(GL_4(\Z),V)-\chi_h(GL_4(\Z),V\otimes det)= 2\sum_{f,det(A)=-1}R(f)C(A)Tr(A|V)$, which from the previous Lemma is bounded.

We will approximate the dimension of the space of holomorphic cusp forms $S_k$ for the classical modular group $SL_2(\Z)$. 
In our notation, we have $(\overline{a,b})=H^1(GL_2(\Z),(a,b))=H^1(GL_2(\Z),Sym^{a-b}V_2)=S_{a-b+2}$, where $V_2$ is the standard two-dimensional representation. 
We will approximate the dimension of $S_{a-b+2}$ by $\frac{a-b}{12}$. Then the error is bounded

\begin{lemma}
Up to a bounded error, we have can approximate \[[dim H^6_{Eis}(A1) - dim H^5{Eis}(A1)] - [dim H^6_{Eis}(A1') - dim H^5{Eis}(A1')]\sim \frac{2b-2c}{12}.\]
\end{lemma}
\proof
We will approximate the dimension of the space of holomorphic cusp forms $S_k$ for the classical modular group $SL_2(\Z)$. 
In our notation, we have $(\overline{a,b})=H^1(GL_2(\Z),(a,b))=H^1(GL_2(\Z),Sym^{a-b}V_2)=S_{a-b+2}$, where $V_2$ is the standard two-dimensional representation. 
We will approximate the dimension of $S_{a-b+2}$ by $\frac{a-b}{12}$. Then the error is bounded.
From our computation so far, we have $dim H^6_{Eis}(A1') - dim H^5{Eis}(A1')= dim(2c-1,2d-1|2a+3,2b+3) +dim(2d-2|\overline{2b+1,2c+1}|2a+4)$.
Note that $dim(2c-1,2d-1|2a+3,2b+3)=(1+dim S_{2c-2d+2})(1+dim S_{2a-2b})$ and $dim(2d-2|\overline{2b+1,2c+1}|2a+4)=dim S_{2b-2c}$
Also $dim H^6_{Eis}(A1) - dim H^5{Eis}(A1)=dim(2c-2,2d-2|2a+2,2b+2)=(dim S_{2c-2d+2})(dim S_{2a-2b})$.
Therefore, 
\begin{eqnarray*}
&&[dim H^6_{Eis}(A1) - dim H^5{Eis}(A1)] - [dim H^6_{Eis}(A1') - dim H^5{Eis}(A1')] =\\
&& -dim S_{2b-2c}-dim S_{2c-2d+2} - dim S_{2a-2b} - 1\sim\\
&&\sim -\frac{2b-2c}{12}-\frac{2c-2d+2}{12}-\frac{2a-2b}{12}=\\
&&= -\frac{2a-2d+2}{12}.
\end{eqnarray*}
which is approximately, $-\frac{2a-2d+2}{12}$.
\qed

From the above Lemma, it follows that $dim H^4_{Eis}(A1)-dim H^4_{Eis}(A1')$ is approximately  $\frac{2a-2d+2}{12}$.

We have
$dim H^4_{\partial}(A1)=dim(\overline{2b-1,2c-1}|2a+2,2d)+dim(2c-2|\overline{2a+1,2d-1}|2b+2)+(2a,2d-2|\overline{2b+1,2c+1})$.
Similarly $dim H^4_{\partial}(A1')=dim(2b,2c|\overline{2a+3,2d+1})+dim(2b|2a+2,2d|2c+2)+dim(\overline{2a+1,2d-1}|2b+2,2c+2)$.

Since $dim(\overline{2b-1,2c-1}|2a+2,2d)=(dim S_{2b-2c+2})(dim S_{2a-2d})$, by the boundedness of $\chi_h(GL_4(\Z),V)-\chi_h(GL_4(\Z),V\otimes det)$, we obtain that 
$H^4_{\partial}(A1)$ contain only one copy of $(\overline{2b-1,2c-1}|2a+2,2d)$ and $H^4_{\partial}(A1')$ only one copy of $(2b,2c|\overline{2a+3,2d+1})$. From the previous Lemma (differences of H5 and H6), and the boundedness of $\chi_h(GL_4(\Z),V)-\chi_h(GL_4(\Z),V\otimes det)$, we obtain that $H^4_{\partial}(A1)$ contain the entire $(2b|2a+2,2d|2c+2)$ and $H^4_{\partial}(A1')$ does not. Therefore, we obtain the following.
\begin{thm}
The Eisenstein cohomology of $GL_(\Z)$ with coefficients in the highest weight representations $V_\lambda$, where $\lambda$ is $(2a,2b,2c,2d)$ or $(2a+1,2b+2,2c+1,2d+1)$ is given below. Recall that by $A1$ we mean any representation with highest weight $(2a,2b,2c,2d)$ (as  a character on the maximal split torus of $GL_4$). Similarly, by $A1$ we mean any representation with highest weight $(2a+1,2b+1,2c+1,2d+1)$. In the cases $A1$ and $A1'$, we have $a>b>c>d$.  
\[H^q_{Eis}(GL_4(\Z),A1)
=
\left\{
\begin{tabular}{lll}
$0$ 
		& $q=3$\\
$\left\{
	\begin{tabular}{lll}
	$(2a,2d-2|\overline{2b+1,2c+1})$\\
	$(2c-2|\overline{2a+1,2d-1}|2b+2)$
	\end{tabular}
\right\}$	& $q=4$\\
$0$ 
		& $q=5$\\
$(2c-2,2d-2|2a+2,2b+2)$ 
		& $q=6$
\end{tabular}
\right.
\]
\[H^q_{Eis}(GL_4(\Z),A1')
=
\left\{
\begin{tabular}{lll}
$0$
		& $q=3$\\	
$(\overline{2a+1,2d-1}|2b+2,2c+2)$
		& $q=4$\\
$\left\{
	\begin{tabular}{lll}
	$(H^2_!(GL_3(\Z),(2b,2c,2d))|2a+4)$\\
	$(2d-2|H^2_!(GL_3(\Z),(2a+2,2b+2,2c+2)))$
	\end{tabular}
\right\}$
		& $q=5$\\ 
$\left\{
	\begin{tabular}{lll}
	$(H^3_!(GL_3(\Z),(2b,2c,2d))|2a+4)$\\
	$(2d-2|H^3_!(GL_3(\Z),(2a+2,2b+2,2c+2))))$\\
	$(2c-1,2d-1|2a+3,2b+3)$\\
	$(2d-2|\overline{2b+1,2c+1}|2a+4)$
	\end{tabular}
	\right\}$
		& $q=6$
\end{tabular}
\right.
\]
Up to semisiplicity, the Eisenstein cohomology is
\[H^q_{Eis}(GL_4(\Z),A1)
=
\left\{
\begin{tabular}{lll}
$0$ 
		& $q=3$\\
$S_{2a-2d} + S_{2a-2d}\otimes S_{2b-2c+2} $
		& $q=4$\\
$0$ 
		& $q=5$\\
$S_{2c-2d}\otimes S_{2a-2b+2}$
		& $q=6$
\end{tabular}
\right.
\]
\[H^q_{Eis}(GL_4(\Z),A1')
=
\left\{
\begin{tabular}{lll}
$0$
		& $q=3$\\	
$S_{2a-2d}\otimes S_{2b-2c+2}$
		& $q=4$\\
$\left\{
	\begin{tabular}{lll}
	$S_{(2b,2c,2d)}(GL_3(\Z))$\\
	$S_{(2a+2,2b+2,2c+2)}(GL_3(\Z))$
	\end{tabular}
	\right\}$
		& $q=5$\\ 
$\left\{
	\begin{tabular}{lll}
	$S_{(2b,2c,2d)}(GL_3(\Z))$\\
	$S_{(2a+2,2b+2,2c+2)}(GL_3(\Z))$\\
	$M_{2c-2d+2}\otimes M_{2a-2a+2}$\\
	$S_{2b-2c+2}$
	\end{tabular}
	\right\}$
		& $q=6$
\end{tabular}
\right.
\]

\end{thm}




\subsection{Eisenstein cohomology of $GL_4(\Z)$ for the representations $A2$ and $A2'$}

\[H^q_\partial(A2)
=
\left\{
\begin{tabular}{lll}
	$\left\{
	\begin{tabular}{ll}
	$E_2^{1,1}=(2a|2b, 2c|2c)$\\
	$E_2^{0,2}=(\overline{2a,2b,2c}|2c)$
	\end{tabular}
	\right\}$
		& $q=2$\\
$E_2^{0,3}=(\overline{2a,2b,2c}|2c)$
		& $q=3$\\	
$E_2^{0,4}=
	\left\{
	\begin{tabular}{lll}
	$(2a,2c-2|\overline{2b+1,2c+1})$\\
	$(\overline{2b-1,2c-1}|2a+2,2c)$\\
	$(2c-2|\overline{2a+1,2c-1}|2b+2)$
	\end{tabular}
	\right\}$
		& $q=4$\\
$E_2^{0,5}=(2c-2|2c-2|2a+2,2b+2)$
		& $q=5$\\ 
$0$
		& $q=6$
\end{tabular}
\right.
\]

\[H^q_\partial(A2')
=
\left\{
\begin{tabular}{lll}
$0$
	& $q=2$\\
$E_2^{1,2}=(\overline{2a+1,2b+1}|2c|2c+2)$
		& $q=3$\\	
$\left\{
\begin{tabular}{llll}
$E_2^{0,4}=\left\{
	\begin{tabular}{lll}
	$(\overline{2a+1,2c-1}|2b+2,2c+2)$\\
	$(2b,2c|\overline{2a+3,2c+1})$
	\end{tabular}
	\right.$\\
$E_2^{1,3}=(2b|2a+2,2c|2c+2)$
\end{tabular}
\right\}$	
		& $q=4$\\
$E_2^{0,5}=(2c-2|\overline{2a+2,2b+2,2c+2})$
		& $q=5$\\ 
$E_2^{0,6}=
	\left\{
	\begin{tabular}{lll}
	$(2c-2|\overline{2a+2,2b+2,2c+2})$\\
	$(2c-2|\overline{2b+1,2c+1}|2a+4)$
	\end{tabular}
	\right.
	$
		& $q=6$
\end{tabular}
\right.
\]

From the Euler characteristics, it follows that

\[H^q_{Eis}(GL_4(\Z),A2)
=
\left\{
\begin{tabular}{lll}
$E_2^{0,4}=
	\left\{
	\begin{tabular}{lll}
	$(2a,2c-2|\overline{2b+1,2c+1})$\\
	$(2c-2|\overline{2a+1,2c-1}|2b+2)$
	\end{tabular}
	\right\}$
		& $q=4$\\
$E_2^{0,5}=(2c-2|2c-2|2a+2,2b+2)$
		& $q=5$\\ 
$0$
		& $q=6$
\end{tabular}
\right.
\]

\[H^q_{Eis}(GL_4(\Z),A2')
=
\left\{
\begin{tabular}{lll}
$E_2^{0,4}=(\overline{2a+1,2c-1}|2b+2,2c+2)$	
		& $q=4$\\
$E_2^{0,5}=(2c-2|\overline{2a+2,2b+2,2c+2})$
		& $q=5$\\ 
$E_2^{0,6}=
	\left\{
	\begin{tabular}{lll}
	$(2c-2|\overline{2a+2,2b+2,2c+2})$\\
	$(2c-2|\overline{2b+1,2c+1}|2a+4)$
	\end{tabular}
	\right.
	$
		& $q=6$
\end{tabular}
\right.
\]

In terms of inner cohomology the result is
\[H^q_{Eis}(GL_4(\Z),A2)
=
\left\{
\begin{tabular}{lll}
$S_{2a+2c+4}S_{2b-2c+2}+S_{2a-2c+4}$
		& $q=4$\\
$S_{2a-2b+2}$
		& $q=5$\\ 
$0$
		& otherwise
\end{tabular}
\right.
\]

\[H^q_{Eis}(GL_4(\Z),A2')
=
\left\{
\begin{tabular}{lll}
$S_{2a-2c+4}S_{2b-2c+2}$
		& $q=4$\\
$S_{2a,2b,2c}(SL_3(\Z))$
		& $q=5$\\ 
$S_{2a,2b,2c}(SL_3(\Z))+S_{2b-2c+2}$
		& $q=6$\\
$0$
		& otherwise
\end{tabular}
\right.
\]




\subsection{Eisenstein cohomology. Cases $A3$ and $A3'$.}

Form the boundary cohomology we have that $H^2(A3)$ is dual to $H^6(A3')$ and $H^2(A3')$ is dual to $H^6(A3)$. 
The Eisenstein cohomology imbedded in $H^2(A3) +H^6(A3')$ and it has half of the dimension.
Since the second group cohomology of $GL_4(\Z)$ always vanishes, We obtain that $H^2_{Eis}(A3)=0$, $H^6_{Eis}(A3)=H^6_\partial(A3)$ and
$H^2_{Eis}(A3')=0$, $H^6_{Eis}(A3')=H^6_\partial(A3')$. 

\[H^q_\partial(GL_4(\Z),A3)=
\left\{
\begin{tabular}{llll}
	$\left\{
	\begin{tabular}{lll}
	$E_2^{0,2}=(2a,2b|2b,2d)$\\
	$E_2^{1,1}=
		\left\{
		\begin{tabular}{lll}
		$(2a,2b|2b|2d)$\\
		$(2a|2b|2b,2d)$
		\end{tabular}
		\right\}$\\
	$E_2^{2,0}=(2a|2b|2b|2d)$
	\end{tabular}
		\right\}$
		& $q=2$\\
	$0$ & $q=3$\\
	$0$ & $q=4$\\
	$E_2^{1,4}=(2b-2|\overline{2a+1,2d-1}|2b+2)$
	& $q=5$\\
	$E_2^{0,6}=(2b-2,2d-2|2a+2,2b+2)$ & $q=6$\\
	$0$ & $q=7$\\
	$0$ & $q=8$
\end{tabular}
\right.
\]

\[H^q_\partial(GL_4(\Z),A3')
=
\left\{
\begin{tabular}{lll}
$E_2^{0,2}=(\overline{2a+1,2b+1}|\overline{2b+1,2d+1})$ 
		& $q=2$\\
$E_2^{0,3}=(2b|2a+2,2d|2b+2)$
		& $q=3$\\	
$0$
		& $q=4$\\
$0$
		& $q=5$\\ 
$E_2^{0,6}=\left\{
	\begin{tabular}{lll}
	$(\overline{2b-1,2d-1}|\overline{2a+3,2b+3})$\\
	$(\overline{2b-1,2d-1}|2b+2|2a+4)$\\
	$(2d-2|2b|\overline{2a+3,2b+3})$
	\end{tabular}
	\right\}$
		& $q=6$
\end{tabular}
\right.
\]

We are going to use that the difference between the Euler characteristic of $GL_4(\Z)$  with coefficients in $A3$ minus the one in $A3'$ is bounded. (Theorem 75)
Then, we apply Conjecture 76.
 
We have that $dim H^6_{Eis}(A3)=\frac{2b-2d}{12}$ up to  small constant.
For the representation $A3'$, we have that $H^6_{Eis}(A3')=\frac{2b-2d}{12}\times\frac{2a-2b}{12} + \frac{2b-2d}{12} + \frac{2a-2b}{12}$
The remaining part of the Eisenstein cohomology is embedded in $H^5(A3)+H^3(A3')$. If we assume the nonsplitting conjecture. Then either $dim H^3(A3')=\frac{a-d}{12}$
 or $dim H^3(A3')=0$. If $dim H^3(A3')=\frac{a-d}{12}$, then
$\chi_h(GL_4(\Z),A3')=dim H^6(A3')-dim H^3(A3')=\frac{2b-2d}{12}\times\frac{2a-2b}{12} + \frac{2b-2d}{12} + \frac{2a-2b}{12} -\frac{a-d}{12}=  \frac{2b-2d}{12}\times\frac{2a-2b}{12}$,
which is the same as
$\chi_h(GL_4(\Z),A3)=dim H^6(A3)=\frac{2b-2d}{12}\times\frac{2a-2b}{12}$.
If $dim H^3(A3')=0$,
 then
$\chi_h(GL_4(\Z),A3')=dim H^6(A3')-dim H^3(A3')=\frac{2b-2d}{12}\times\frac{2a-2b}{12} + \frac{2b-2d}{12} + \frac{2a-2b}{12}$
and
$\chi_h(GL_4(\Z),A3)=dim H^6(A3)-dim H^5(A3)=\frac{2b-2d}{12}\times\frac{2a-2b}{12}  - \frac{a-d}{12}$,
which are not the same. Therefore, we are in the first case.

Instead of assuming this conjecture, we can rely on the heuristic that $H^q_Eis$ prefers not to have ghost classes.
In this case $H^5_{Eis}(A3)$ is a potentially ghost class since in does not come from $E_\infty^{0,q}$. It comes from $E_\infty^{1,4}$.
Using that heuristic, we obtain that $H^5_{Eis}(A3)=0$. Therefore, $H^3_{Eis}(A3)=H^3_{\partial}(A3)$.

Therefore

\[H^q_{Eis}(GL_4(\Z),A3)=
\left\{
\begin{tabular}{llll}
	$0$ & $q=3$\\
	$0$ & $q=4$\\
	$0$ & $q=5$\\
	$E_2^{0,6}=(2b-2,2d-2|2a+2,2b+2)$ & $q=6$\\
\end{tabular}
\right.
\]

\[H^q_{Eis}(GL_4(\Z),A3')
=
\left\{
\begin{tabular}{lll}
$E_2^{0,3}=(2b|2a+2,2d|2b+2)$
		& $q=3$\\	
$0$
		& $q=4$\\
$0$
		& $q=5$\\ 
$E_2^{0,6}=\left\{
	\begin{tabular}{lll}
	$(\overline{2b-1,2d-1}|\overline{2a+3,2b+3})$\\
	$(\overline{2b-1,2d-1}|2b+2|2a+4)$\\
	$(2d-2|2b|\overline{2a+3,2b+3})$
	\end{tabular}
	\right\}$
		& $q=6$
\end{tabular}
\right.
\]

In terms of inner cohomology

 \[H^q_{Eis}(GL_4(\Z),A3)=
\left\{
\begin{tabular}{llll}
	$S_{2a-2b+2}S_{2b-2d+2}$ 	& $q=6$\\
	$0$						& otherwise
\end{tabular}
\right.
\]

\[H^q_{Eis}(GL_4(\Z),A3')
=
\left\{
\begin{tabular}{lll}
$S_{2a-2d+4}$
		& $q=3$\\	
$S_{2a-2b+2}S_{2b-2d+2}
+S_{2b-2b+2}
+S_{2a-2b+2}$
		& $q=6$\\
$0$
		& otherwise\\
		
\end{tabular}
\right.
\]




\subsection{Eisenstein cohomology. Cases $A4$ and $A4'$.}

$H^q_\partial(GL_4(\Z),A4)=
\left\{
\begin{tabular}{llll}
	$0$ & $q=0$\\
	$0$ & $q=1$\\
	$E_2^{1,1}=(2a|2a,2c|2d)$ 
		& $q=2$\\
	$0$ & $q=3$\\
	$E_2^{0,4}=
		\left\{
		\begin{tabular}{lll}
			$(2a,2d-2|\overline{2a+1,2c+1})$\\
			$(2a,2c|\overline{2a+3,2d+1})$\\
			$(2a|2a+2,2d|2c+2)$
		\end{tabular}
		\right\}
		$
		& $q=4$\\
	$E_2^{0,5}=(2c-2,2d-2|2a+2,2a+2)$ & $q=5$\\
	$0$ & $q=6$\\
	$0$ & $q=7$\\
	$0$ & $q=8$
\end{tabular}
\right.
$

\[H^q(GL_4(\Z),A4')
=
\left\{
\begin{tabular}{lll}
$0$
		& $q=2$\\
$E_2^{1,2}=(2a|2a+2|\overline{2c+1,2d+1})$
		& $q=3$\\	
$E_2^{0,4}=
	\left\{
	\begin{tabular}{lll}
	$(\overline{2a+1,2d-1}|2a+2,2c+2)$\\
	$(2a,2c|\overline{2a+3,2d+1})$
	\end{tabular}
	\right\}$
		& $q=4$\\
$0$
		& $q=5$\\ 
$E_2^{0,6}=(2d-2|\overline{2a+1,2c+1}|2a+4)$
		& $q=6$
\end{tabular}
\right.
\]

Again $H^2_{Eis}$ in both cases must vanish. Therefore $H^6_{Eis}=H^6_\partial$.
Using Euler characteristics argument we obtain

$H^q_\partial(GL_4(\Z),A4)=
\left\{
\begin{tabular}{llll}
$S_{2a-2d+4}S_{2a_2c+2}+S_{2a-2d+4}$
		& $q=4$\\
$S_{2c-2d+2}$ 
		& $q=5$\\
$0$ 		& otherwise
\end{tabular}
\right.
$

\[H^q(GL_4(\Z),A4')
=
\left\{
\begin{tabular}{lll}
$S_{2a-2d+4}S_{2a-2c+2}+S_{2a-2d+4}$
		& $q=4$\\
$S_{2a,2c,2d}(SL_3(\Z))$
		& $q=5$\\
$S_{2a,2c,2d}(SL_3(\Z))+S_{2a-2c+2}$
		& $q=6$\\
$0$
		& otherwise
\end{tabular}
\right.
\]




\subsection{Eisenstein cohomology of $GL_4(\Z)$ with coefficients in symmetric powers\\
Symmetric powers and their twist}

Note that $A5$ is the family $Sym^{2a-2b}$ and $A5'$ is $Sym^{2a-2b}\otimes det$.

For $Sym^{2a-2b}$, we have
\[H^q_\partial(GL_4(\Z),A5)=
\left\{
\begin{tabular}{llll}
	$S_{2a-2b+4}+S_{2a-2b+2}+\C$
 		& $q=5$\\
	$0$ 
		& otherwise
	\end{tabular}
\right.
\]

For $Sym^{2a-2b}\otimes det$, we have

\[H^q_\partial(GL_4(\Z),A5)=
\left\{
\begin{tabular}{llll}
	$S_{2a-2b+4}+S_{2a-2b+2}+\C$
 		& $q=3$\\
	$0$ 
		& otherwise
	\end{tabular}
\right.
\]
Using Euler characteristics, we find

\[H^q_{Eis}(GL_4(\Z),Sym^{2a-2b})=
\left\{
\begin{tabular}{llll}
	$S_{2a-2b+2}$
 		& $q=5$\\
	$0$ 
		& otherwise
	\end{tabular}
\right.
\]

\[H^q_{Eis}(GL_4(\Z),Sym^{2a-2b}\otimes det)=
\left\{
\begin{tabular}{llll}
	$S_{2a-2b+4}+\C$
 		& $q=3$\\
	$0$ 
		& otherwise
	\end{tabular}
\right.
\]


\section{Cohomology of $GL_4(\Z)$ with coefficients in $det$ }

We would like to compute the cohomology of $GL_4(\Z)$ with coefficients in $det$. 

In order to do that, first we need the following result
by Elbas-Vincent, Gangle and Soule \cite{EVGS}
\begin{thm}
\label{SL4}
\[
H^i(SL_4(\Z),\Q)=
\left\{
\begin{tabular}{lll}
$\Q$ & for $i=0$ or $3$\\
\\
$0$ & otherwise
\end{tabular}
\right.
\]
\end{thm}

From Corollary \ref{chi(SL,Q)}, we know that
\[\dim[H^i(GL_4(\Z),det)]=  \dim[H^i(SL_4(\Z),\Q)]  -  \dim[H^i(GL_4(\Z),\Q)].\]
From the Theorem \ref{SL4}
$\dim[H^i(SL_4(\Z),\Q)]$ completely. We need to examine $\dim[H^i(GL_4\Z),\Q)].$

\begin{prop}
\[
H^i(GL_4(\Z),\Q)=
\left\{
\begin{tabular}{lll}
$\Q$ & for $i=0$\\
\\
$0$ & otherwise
\end{tabular}
\right.
\]
\end{prop}
\proof
Obviously, $H^0(GL_4(\Z),\Q)=\Q.$ Also,
by Theorem \ref{SL4}, we have that 
$H^i(GL_4(\Z),\Q)=0$ for $i\neq 0, 3$ and $H^3(GL_4(\Z),\Q)=0$ or $\Q$, since $H^i(GL_4(\Z),\Q)$ is a direct summand of $H^i(SL_4(\Z),\Q)=0$.
There are only two possibilities $H^3(GL_4(\Z),\Q)=0$ or $\Q$. If $H^3(GL_4(\Z),\Q)= \Q$, then
$\chi_h(GL_4(\Z),\Q)=0$; however, $\chi_h(GL_4(\Z),\Q)=1$. Therefore, $H^3(GL_4(\Z),\Q)$ cannot be $\Q$. The only other possibility is  $H^3(GL_4(\Z),\Q)= 0$,
which is the statement of the Proposition.

Since 
$\dim[H^i(GL_4(\Z),det)]=  \dim[H^i(SL_4(\Z),\Q)]  -  \dim[H^i(GL_4(\Z),\Q)],$
we obtain the following.
\begin{cor}
\[
H^i(GL_4(\Z),det)=
\left\{
\begin{tabular}{lll}
$\Q$ & for $i=3$\\
\\
$0$ & otherwise
\end{tabular}
\right.
\]
\end{cor}





\renewcommand{\em}{\textrm}

\begin{small}
   
\end{small}


\end{document}